\documentclass[reqno]{amsart}

\usepackage{amssymb, amsfonts}
\usepackage{verbatim}
\usepackage{fullpage}

\usepackage[hidelinks]{hyperref}
\usepackage{url}
\usepackage{cleveref}
\usepackage{mathrsfs}
\usepackage{color}   

\hypersetup{
    colorlinks=true, 
    linktoc=all,     
    linkcolor=blue,  
}

\usepackage{tikz-cd}

\usepackage{mathtools}
\usepackage{amssymb, amsfonts}
\usepackage[utf8]{inputenc}
\usepackage{amsthm}
\usepackage{amsmath}
\usepackage{amsxtra}
\usepackage{amscd}
\usepackage{amssymb}
\usepackage{color}
\usepackage{mathtools}
\usepackage{xspace}
\usepackage{amsfonts}
\usepackage{mathrsfs}
\usepackage{todonotes}
\usepackage{enumerate}
\usepackage{multirow}
\usepackage{tikz}
\usepackage{hyperref}
\usepackage{cancel}
\usepackage{tikz-cd}
\usepackage[all]{xy}
\usepackage{amscd}
\usepackage{caption}
\usepackage{epsfig}
\usepackage{graphicx,transparent}
\usepackage{here}
\usepackage{hyperref}
\usepackage{setspace}
\usepackage{wasysym}
\usepackage{xparse}
\usepackage{manfnt}
\usepackage{fullpage}

\usepackage{url}
\usepackage{cleveref}
\usepackage{mathrsfs}

\usepackage{tikz-cd}

\usepackage{mathtools}

\usepackage{amsmath}
\usepackage{tikz}
\usepackage{mathdots}
\usepackage{yhmath}
\usepackage{cancel}
\usepackage{color}
\usepackage{siunitx}
\usepackage{array}
\usepackage{multirow}
\usepackage{amssymb}
\usepackage{tabularx}
\usepackage{extarrows}
\usepackage{booktabs}
\usetikzlibrary{fadings}
\usetikzlibrary{patterns}
\usetikzlibrary{shadows.blur}
\usetikzlibrary{shapes}

\newtheorem{theorem}[equation]{Theorem}
\newtheorem{lemma}[equation]{Lemma}
\newtheorem{proposition}[equation]{Proposition}
\newtheorem{corollary}[equation]{Corollary}

\theoremstyle{definition}
\newtheorem{definition}[equation]{Definition}
\newtheorem{example}[equation]{Example}

\theoremstyle{remark}
\newtheorem{remark}[equation]{Remark}

\numberwithin{equation}{section}

\DeclareMathOperator{\lcm}{lcm}
\DeclareMathOperator{\Ind}{Ind}

\newcommand{\cR}{\mathcal{R}}
\newcommand{\cC}{\mathcal{C}}
\newcommand{\cD}{\mathcal{D}}
\newcommand{\cP}{\mathcal{P}}
\newcommand{\cM}{\mathcal{M}}
\newcommand{\cQ}{\mathcal{Q}}

\newcommand{\cH}{\mathcal{H}}
\newcommand{\cF}{\mathcal{F}}
\newcommand{\ZZ}{\mathbb{Z}}
\newcommand{\PP}{\mathbb{P}}
\newcommand{\CC}{\mathbb{C}}

\newcommand{\RR}{\mathbb{R}}

\newcommand{\MIN}{single-zero }
\newcommand{\NMIN}{double-zero }

\begin{document}

\title[]{Connected components of strata of residueless meromorphic differentials}
\author[M. Lee]{Myeongjae Lee}
\address{Stony Brook University, Department of Mathematics, Stony Brook, NY 11794-3651}
\email{myeongjae.lee@stonybrook.edu}
\thanks{}

\begin{abstract}
Generalized strata of meromorphic differentials are loci in the usual strata of differentials, where certain sets of residues sum up to zero. They appear naturally in the boundary of the multi-scale compactification of the usual strata. Enumerating the connected components of generalized strata is necessary to understand the boundary complex of the multi-scale compactification. In this paper we classify connected components of the strata of residueless meromorphic differentials, which are the strata with the maximum possible number of conditions imposed on the residues of the poles. 
\end{abstract}

\maketitle

\tableofcontents

\section{Introduction} \label{sec:intro}

In the present paper, we investigate the flat surfaces that arise from meromorphic differentials on compact Riemann surfaces with vanishing residues at each pole. We prove that the connected components of strata of residueless meromorphic differentials are classified by the well-known topological invariants called {\em hyperellipticity}, {\em spin parity}, and {\em rotation number}. This is the first step of a more general project on the classification of connected components of generalized strata, which are loci of the usual strata where the residues of the poles satisfy certain linear conditions. These generalized strata appear in the description of the boundary of the moduli space of {\em multi-scale differentials}, which is a smooth compactification of the stratum $\PP\cH_g(\mu)$ with normal crossing boundary. The final goal of the study in this direction is to compute the boundary complex of the moduli space of multi-scale differentials, which may lead to new results on the top-weight cohomology of $\PP\cH_g(\mu)$, parallel to the recent results \cite{CGP} on $\mathcal{M}_g$ and \cite{brandt2022topweight} on $\mathcal{A}_g$. 

For integers $m,n> 0$ and $g\geq 0$, let $\mu \coloneqq (a_1,\dots, a_m, -b_1,\dots, -b_n)$ be a partition of $2g-2$, where $-b_n\leq \dots \leq -b_1 <0 \leq a_1\leq \dots \leq a_m$. Recall that the stratum of meromorphic differentials of type $\mu$, denoted by $\cH_g (\mu)$, is defined to be the moduli space of meromorphic differentials on genus $g$ compact Riemann surfaces with orders of zeroes and poles prescribed by $\mu$. More precisely, an element of $\cH_g (\mu)$ is a pair $\left((X, \boldsymbol{z}\sqcup\boldsymbol{p}),\omega\right)$ where $(X, \boldsymbol{z}\sqcup\boldsymbol{p})\in \cM_{g,m+n}$ is a $(m+n)$-pointed compact Riemann surface of genus $g$, and $\omega$ is a meromorphic differential on $X$, where $\boldsymbol{z}=\{z_1,\dots,z_m\}$, $\boldsymbol{p}=\{p_1,\dots,p_n\}$ are the sets of zeroes and poles of $\omega$ with orders prescribed by $\mu$. In other words,  $\operatorname{div}(\omega)=\sum_i a_i z_i - \sum_j b_j p_j$. An element $\left((X, \boldsymbol{z}\sqcup\boldsymbol{p}),\omega\right)$ of the stratum is called a {\em flat surface}, since the differential $\omega$ defines a flat metric on $X\setminus \boldsymbol{p}$ with conical singularities at the points in $\boldsymbol{z}$. In the present paper, we simply denote the flat surface by $(X,\omega)$, or $X$ when we are not focusing on the differential $\omega$ itself and no confusion can arise. The space $\cH_g (\mu)$ has a smooth complex orbifold structure given by the period map $\operatorname{Per} : H_1 (X\setminus{\boldsymbol{p}}, \boldsymbol{z}; \ZZ) \to \CC$ obtained by integrating the meromorphic form $\omega$. The local coordinate neighborhood can be identified with a $\CC$-vector space $H^1 (X\setminus{\boldsymbol{p}}, \boldsymbol{z}; \CC)$ of dimension $2g+m+n-2$, and in particular $\operatorname{dim}_{\mathbb{C}}\cH_g (\mu)=2g+m+n-2$.

The connected components of the usual strata $\cH_g(\mu)$ of differentials are well-known. In \cite{kozo1}, Kontsevich and Zorich classified the connected components of the strata of holomorphic differentials. In this case, it turns out that there are at most three connected components of each stratum, distinguished by two topological invariants called {\em spin parity} and {\em hyperellipticity}. The main tools in \cite{kozo1}, which will also be used in the present paper, are the flat-geometric surgeries called {\em breaking up a zero} and {\em bubbling a handle}. In \cite{lanneauquad}, Lanneau classified the connected components of the strata of holomorphic quadratic differentials. In the quadratic case, there are at most two connected components for each stratum. In \cite{boissymero} and \cite{boissy}, Boissy classified the connected components of the strata of meromorphic differentials and the strata of meromorphic differentials with marked horizontal separatrices. As in the holomorphic case, there are at most three connected components for each stratum, unless $g=1$. For $g=1$, a new topological invariant called {\em rotation number} is required to classify the connected components, and strata with arbitrarily many components exist.

In the present paper, we will expand the discussion on the connected components to certain types of loci of $\cH_g(\mu)$ that we introduce now. Let $R$ be a partition of the set $\boldsymbol{p}$ of poles of the differential. A meromorphic differential $(X,\omega)\in \cH_g (\mu)$ is said to satisfy {\em the residue condition} given by $R$ if $\sum_{p\in P} \operatorname{res}_{p} \omega = 0$ for each part $P$ of $R$. In particular, if a singleton $\{p\}$ is a part of $R$, then $\operatorname{res}_{p} \omega=0$, and the pole $p$ is said to be {\em residueless}. A {\em generalized stratum} $\cH^R_g (\mu)$ is defined to be the subspace of $\cH_g (\mu)$ consisting of differentials satisfying the residue conditions given by $R$. If $R$ is a partition $\boldsymbol{p}=P_1 \sqcup \dots \sqcup P_r$ and $\mu_i=(-b_j)_{p_j\in P_i}$, the generalized stratum $\cH^R_g (\mu)$ is also denoted by $\cH_g (a_1,\dots, a_m ; \mu_1 ;\dots; \mu_r)$. The residue $\operatorname{res}_{p}\omega$ at the pole $p$ is obtained by integrating $\omega$ along a small circle around $p$. For each part $P_i$, let $\alpha_i$ be a closed curve only enclosing the poles in $P_i$, with trivial homology class $[\alpha_i]$ in $H_1(X, \boldsymbol{z};\ZZ)$. The residue condition given by $R$ is the linear condition $\int_{\alpha_i}\omega=0$ for each $i$. So in the period coordinates of $\cH_g (\mu)$, the generalized stratum $\cH^R_g (\mu)$ is a linear subvariety given by the vector subspace $H^1 (X\setminus{\boldsymbol{p}}, \boldsymbol{z}; \CC)^R$ of $H^1 (X\setminus\boldsymbol{p}, \boldsymbol{z}; \CC)=\operatorname{Hom}(H_1(X\setminus{\boldsymbol{p}}, \boldsymbol{z}; \ZZ),\CC)$, consisting of the linear functions that vanish on the classes $[\alpha_i]$. We denote $H_1(X\setminus{\boldsymbol{p}}, \boldsymbol{z}; \ZZ)^R \coloneqq H_1(X\setminus{\boldsymbol{p}}, \boldsymbol{z}; \ZZ)/\langle [\alpha_i] \rangle$. Thus the period coordinates on $\cH^R_g (\mu)$ can be identified with $\operatorname{Hom}(H_1(X\setminus{\boldsymbol{p}}, \boldsymbol{z}; \ZZ)^R,\CC)$. If $|R|$ is the number of parts of $R$, the codimension of $\cH^R_g (\mu)$ in $\cH_g (\mu)$ is equal to $|R|-1$, if it is nonempty. 

If $R$ is the finest possible partition consisting of $n$ singletons, then the residue condition given by $R$ is that every pole of $\omega$ is residueless. We denote the corresponding generalized stratum $\cH^R _g (\mu) = \cH_g (a_1,\dots,a_m;-b_1;\dots;-b_n)$ also by $\cR_g (\mu)$, and call it a {\em residueless stratum}. In the present paper, we will simply call it a {\em stratum} when no confusion can arise. The dimension of the residueless stratum $\cR_g(\mu)$ is equal to $2g+m-1$. In this case, the quotient space $H_1 (X\setminus{\boldsymbol{P}}, \boldsymbol{z}; \ZZ)^R$ can be identified with $H_1 (X, \boldsymbol{z}; \ZZ)$ by the map $H_1 (X\setminus{\boldsymbol{p}}, \boldsymbol{z}; \ZZ) \to H_1 (X, \boldsymbol{z}; \ZZ)$ induced by the inclusion $X\setminus{\boldsymbol{p}} \hookrightarrow X$. If there is only one pole (which is then necessarily of order $>1$), then automatically $\cH_g (\mu)=\cR_g (\mu)$.  

If $m=1$, then the residueless genus zero stratum $\cR_0(\mu)$ is empty. The finest possible partition $R$ such that $\cH^R_0(\mu)$ is nonempty consists of $n-2$ singletons and one part of cardinality two. This case also plays an important role in the present paper, so we introduce the notation for it. By relabeling the poles, assume that $\{p_{n-1},p_n\}$ is the only part of $R$ with two elements. Then we denote $\cH^R_0 (\mu)$ by $\cR_0 (a_1, -b_1,\dots, -b_{n-2}; -b_{n-1}, -b_n)$. 

A simple pole cannot be residueless. So if $R$ contains a singleton consisting of a simple pole, then $\cH^R_g (\mu)=\emptyset$. In particular, $\cR_g (\mu)=\emptyset$ if $b_1=1$. Throughout the paper, we assume that a stratum is strictly meromorphic (i.e. $n>0$) and that the orders of all residueless poles are at least two.

It is a bit harder to describe the deformations of flat surfaces in the generalized strata than in the usual strata, since we need to make sure that the deformation does not change the residue condition. This is why the arguments in \cite{kozo1} and \cite{boissymero} cannot be directly applied to the generalized strata. In order to deal with this difficulty, we will be benefited from the existence of $GL^+(2,\mathbb{R})$-action and the multi-scale compactification of the generalized strata. The surgeries from \cite{kozo1} will be interpreted as the degeneration to the boundary of the multi-scale compactification, introduced in \cite{BCGGM} and applied to the residueless cases in \cite{mresidueless}. The degeneration technique is motivated by the principal boundary introduced in \cite{emz}, and its interpretation in terms of the compactification in \cite{ccprincipal}. We will prove that we can approach to the principal boundary by applying certain $GL^+(2,\mathbb{R})$-action from a general point. Then we will be much benefited from the general construction of the multi-scale compactification, which allows us to navigate around the boundary. 

\subsection{Main results}

As mentioned above, we will prove that hyperellipticity, spin parity and rotation number still suffice to distinguish all connected components of the generalized strata, with few exceptional cases. However, there could be {\em many} hyperelliptic connected components, depending on the singularity type $\mu$. Before giving the statements classifying the connected components, we introduce the topological invariants of connected components.

\begin{definition}
A flat surface $(X,\omega)$ is called {\em hyperelliptic} if $X$ has an involution $\sigma$ such that $X/\sigma \cong \PP^1$ and $\sigma^* \omega = -\omega$. A connected component $\cC$ of $\cR_g (\mu)$ is said to be {\em hyperelliptic} if every flat surface contained in $\cC$ is hyperelliptic. Otherwise $\cC$ is said to be {\em non-hyperelliptic}.
\end{definition}

\begin{definition}
A stratum $\cR_g(\mu)$ is said to be of {\em even type} if all zeroes and poles have even orders. 
\end{definition}

\begin{definition} \label{def:ramiprof}
For a stratum $\cR_g(\mu)$ with $m\leq 2$ zeroes, an involution $\cP$ on $\boldsymbol{z}\sqcup \boldsymbol{p}$ is called a {\em ramification profile} of $\cR_g(\mu)$ if the following holds
\begin{itemize}
    \item If $m=1$, then $\cP$ fixes the unique zero $z$. If $m=2$, then $a_1=a_2$, and $\cP$ interchanges $z_1$ and $z_2$.
    \item If $\cP(p_i)=p_j$, then $b_i = b_j$. 
    \item $\cP$ fixes at most $2g+2$ marked points, only of even orders.
\end{itemize}
\end{definition}

Note that for a given stratum $\cR_g(\mu)$, a ramification profile $\cP$ is determined by its restriction to $\boldsymbol{p}$. Since the poles are labeled by $\{1,\dots,n\}$, we will identify $\cP$ with an involution (i.e. an element of order two) in $Sym_n$ when no confusion can arise. Our first result is the classification of the hyperelliptic connected components of $\cR_g(\mu)$. 

\begin{theorem} \label{mainhyper}
For a stratum $\cR_g(\mu)$ of genus $g>0$, there is a one-to-one correspondence between the hyperelliptic connected components and the ramification profiles of $\cR_g(\mu)$.
\end{theorem}

\begin{example}
We observe that it is possible for $\cR_g (\mu)$ to have multiple ramification profiles. For example, consider the stratum $\cR_g (bn+2g-2,-b^n)$, where the exponent $n$ means that the stratum has $n$ poles of order $b>1$, as usual. Assume that $n>2g+1$ and $b$ is even. For any $1\leq r\leq 2g+2$ such that $n+1-r$ is even, $\mu$ has a ramification profile $\cP$ that fixes exactly $r-1$ poles. In particular, there exist $\frac{n!}{2^{(n-r)/2}}$ such ramification profiles. 
\end{example}

For strata of genus $g>1$, the non-hyperelliptic connected components of $\cR_g(\mu)$ are classified by the following

\begin{theorem}\label{main1}
For $g>1$ and a stratum $\cR_g (\mu)$ is of even type, then $\cR_g (\mu)$ has {\em two} non-hyperelliptic connected components distinguished by spin parity. Otherwise if $g>1$ and $\cR_g (\mu)$ is not of even type, then $\cR_g (\mu)$ has a {\em unique} non-hyperelliptic connected component.
\end{theorem}

For strata of genus $g=1$, the non-hyperelliptic connected components of $\cR_1(\mu)$ are classified by the following

\begin{theorem}\label{main2}
For a stratum $\cR_1 (\mu)$ of genus one, denote $d \coloneqq \operatorname{gcd}(a_1,\dots, a_m, b_1,\dots, b_n)$ and let $r$ be a positive integer divisor of $d$. Then $\cR_1 (\mu)$ has a unique non-hyperelliptic connected component $\cC_r$ with rotation number $r$, except for the following cases:

\begin{itemize}
    \item The strata $\cR_1(r,-r)$ does not have a non-hyperelliptic component with rotation number $r$.
    \item The strata $\cR_1(2n,-2^n)$ and $\cR_1(n,n,-2^n)$ have no non-hyperelliptic components.
    \item The strata $\cR_1(2r,-2r),\cR_1(2r,-r,-r),\cR_1(r,r,-2r)$ and $\cR_1(r,r,-r,-r)$ have no non-hyperelliptic components with rotation number $r$.
    \item The stratum $\cR_1(12,-3^4)$ has {\em two} non-hyperelliptic components with rotation number 3.
\end{itemize}
\end{theorem}

Since $\cR_1(r,-r)=\cH_1(r,-r)$, the first genus one exceptional cases follows from \cite{boissymero}. The second cases are treated in \Cref{exception0} and \Cref{nnspecial}. The third cases are treated in \Cref{exception1}, \Cref{exception2} and \Cref{nonexistspecial}. The last case, $\cR_1 (12,-3^4)$, is treated in \Cref{special1} and \cite{LT}.

In summary, except for these special cases listed above, as in the case of usual meromorphic strata $\cH_g(\mu)$, the connected components of the residueless stratum $\cR_g(\mu)$ can be classified by hyperellipticity (though now with multiple hyperelliptic components with different ramification profiles), and spin parity (if $g>1$) or rotation number (if $g=1$). 

\subsection{Outline of the paper}

\begin{itemize}
    \item In \Cref{sec:fs}, we give basic definitions related to flat surfaces, and recall the $GL^+(2,\RR)$-action on the strata and related concepts such as cores and polar domains. We will classify zero-dimensional (projectivized) generalized strata. 
    \item In \Cref{sec:msc}, we recall the definition and the properties of the multi-scale compactification of the generalized strata. Then we describe how we can shrink a collection of parallel saddle connections using the contraction flow. 
    \item In \Cref{sec:pb}, we recall the definition of the principal boundary of strata and how a flat surface degenerates to the principal boundary. We will explain that two surgeries introduced in \cite{kozo1} can be considered as {\em smoothing} processes from certain multi-scale differentials, and they can be reversed by degeneration into the principal boundary under certain conditions.
    \item In \Cref{sec:hc}, we describe hyperelliptic components of the stratum and their principal boundary. 
    \item In \Cref{sec:ssc}, we prove the existence of a flat surface with a multiplicity one saddle connection for all connected component but hyperelliptic components with $2g+2$ fixed marked points. 
    \item In \Cref{sec:g1m}, we classify the non-hyperelliptic components of genus one \MIN strata.
    \item In \Cref{sec:chc}, we classify the hyperelliptic components of strata, completing the proof of \Cref{mainhyper}.
    \item In \Cref{sec:g1n}, we classify the connected components of genus one multiple-zero strata, completing the proof of \Cref{main2}.
    \item In \Cref{sec:hg}, we classify the non-hyperelliptic components of strata of higher genus, completing the proof of \Cref{main1}.
\end{itemize}

\paragraph{\bf Acknowledgements.} This research was partially supported by Kwanjeong Educational Foundation and also by Simons Foundation International, LTD. The author would like to thank his advisor, Samuel Grushevsky for introducing him to the theory of strata of differentials, and encouraging him to work on this project. The author is grateful to Benjamin Dozier, Corentin Boissy and Yiu Man Wong for many valuable discussions. The author would like to thank to Martin M\"oller and Guillaume Tahar for useful comments on an earlier version of this text. The author would also like to show gratitude to reviewers for carefully reading the proofs and for providing comments that greatly improved the manuscript. 

\section{Flat surfaces and their deformations} \label{sec:fs}

In this section, we will recall and introduce basic properties of flat surfaces and their deformations that will be used in later sections. Recall that a {\em saddle connection} of a flat surface $(X,\omega)$ is a straight line with respect to the flat structure connecting two (possibly identical) zeroes of $\omega$ that does not contain any other zeroes of $\omega$ in its interior. The saddle connections play a main role in understanding the flat structure of $(X,\omega)$ and their deformations. 

\subsection{\texorpdfstring{$\operatorname{GL}^{+}(2,\RR)$}{TEXT}-action and the contraction flow}

There is a natural $\operatorname{GL}^{+}(2,\RR)$-action on the meromorphic stratum $\cH_g (\mu)$. For $u\in X\setminus{(\boldsymbol{z}\sqcup \boldsymbol{p})}$, let $z=x+iy$ be a local flat coordinate at $u$ given by $\omega$. That is, $z(u)=0$ and $\omega=dz$ in a neighborhood of $u$.   

For a matrix 
$M=\begin{pmatrix}
a & b \\
c & d 
\end{pmatrix}\in \operatorname{GL}^{+}(2,\RR)$, we can associate another complex coordinate $$z'=(ax+by)+i(cx+dy)$$ at $u$. This can be done for any $u\in X\setminus{(\boldsymbol{z}\sqcup \boldsymbol{p})}$, and these local patches give a new complex structure $X'$. Also the new flat structure given by $z'$ is equivalent to the meromorphic differential $\omega'$ on $X'$, locally determined by $\omega'=dz'$. This new flat surface $(X',\omega')$ is also contained in the stratum $\cH_g (\mu)$. The $\operatorname{GL}^{+}(2,\RR)$-action on $\cH_g (\mu)$ is defined by $M\circ (X,\omega)\coloneqq(X',\omega')$. A remarkable property of the $\operatorname{GL}^{+}(2,\RR)$-action is that the action preserves the straight lines. In other words, if $\gamma$ is a straight line on $(X,\omega)$, then its image in $M\circ (X,\omega)$ is also a straight line. 

Given two distinct directions $\alpha, \theta \in S^1$, the contraction flow $C^t_{\alpha,\theta}$ is given by contracting $\theta$ direction and preserving $\alpha$ direction of flat surfaces. More precisely, it is the action of the semigroup of matrices of the form
$$
C^t_{\alpha,\theta} = \begin{pmatrix}
\cos\theta & \sin\theta \\
\cos\alpha & \sin\alpha 
\end{pmatrix}\begin{pmatrix}
e^{-t} & 0 \\
0 & 1 
\end{pmatrix}\begin{pmatrix}
\cos\theta & \sin\theta \\
\cos\alpha & \sin\alpha 
\end{pmatrix}^{-1}
\in \operatorname{GL}^{+}(2,\RR)$$ for $t\in \RR_+$.

If $(X,\omega)$ satisfies some residue condition $R$, then $M\circ (X,\omega)$ also satisfies $R$ for any $M\in \operatorname{GL}^{+}(2,\RR)$. In other words, the generalized stratum $\cH^R_g(\mu)$ is a $\operatorname{GL}^{+}(2,\RR)$-invariant subvariety of $\cH_g(\mu)$. In particular, the generalized stratum $\cH^R_g(\mu)$ also has the contraction flows. 

\subsection{Flat surfaces with degenerate core} \label{subsec:core}

For a flat surface $(X,\omega)$, a subset $Y\subset X$ is said to be {\em convex} if any straight line joining two points in $Y$ is also contained in $Y$. The {\em convex hull} of a subset $Y$ is the smallest convex subset of $X$ containing $Y$. Recall from \cite{Tah} that the {\em core} $C(X)$ of $(X,\omega)$ is defined to be the convex hull of $\boldsymbol{z}$. In particular, $C(X)$ contains all zeroes and saddle connections of $(X,\omega)$. In \cite{Tah}, Tahar established the following properties of the core, allowing us to decompose every flat surface into the core and the polar domains:

\begin{proposition} \label{degenerate}
For any flat surface $(X,\omega)\in \cH_g (\mu)$, $\partial C(X)$ is a finite union of saddle connections. The complement $X\setminus C(X)$ has exactly $n$ connected components, each of which is homeomorphic to a disk containing one pole $p_i$ of $\omega$. 
\end{proposition}

For a pole $p$ of $\omega$, the connected component of $X\setminus C(X)$ containing $p$ is called the {\em polar domain} of $p$.

A flat surface $(X,\omega)$ is said to have {\em degenerate core} if the core $C(X)$ has empty interior. Since $C(X)$ is closed, it is equivalent to saying $C(X)=\partial C(X)$. By \cite[Lemma 5.15]{Tah}, we can construct a flat surface with degenerate core contained in any connected component of a stratum. 

A consequence of the above proposition is the following
    
\begin{proposition} \label{saddlehomology}
For any flat surface $(X,\omega)$, there exist a finite collection of saddle connections $\gamma_1,\dots,\gamma_N$ of $(X,\omega)$ such that their homology classes generate $H_1 (X\setminus{\boldsymbol{p}}, \boldsymbol{z} ; \ZZ)$. 
\end{proposition}

\begin{proof}
First, we will prove this for flat surfaces $(X,\omega)$ with degenerate core. Then $C(X)$ is a union of finitely many saddle connections. By \Cref{degenerate}, $X\setminus C(X)$ is a disjoint union of polar domains. Each polar domain is homeomorphic to a disk containing one pole. Any path in $X\setminus \boldsymbol{p}$ between two zeroes can be homotoped to a union of saddle connections in $C(X)$. Therefore, the saddle connections of $(X,\omega)$ generate $H_1 (X\setminus{\boldsymbol{p}}, \boldsymbol{z} ; \ZZ)$.

In general, any flat surface $(X,\omega)$ can be deformed to a flat surface $(X',\omega')$ with degenerate core by a contraction flow $C^t_{\alpha,\theta}$ with general $\alpha,\theta$ as $t\to \infty$ (See \cite[Lemma 5.15]{Tah}). Any saddle connection $\gamma'$ of $(X',\omega')$ is a limit of some saddle connection $\gamma$ of $(X,\omega)$. Since $\operatorname{GL}^{+}(2,\RR)$-action preserves the homology class of saddle connections, $\gamma'$ and $\gamma$ have the same homology class. Therefore, we can find a set of saddle connections of $(X,\omega)$ generating $H_1 (X\setminus{\boldsymbol{p}}, \boldsymbol{z} ; \ZZ)$.
\end{proof}
    
\subsection{Flat surfaces in zero-dimensional (projectivized) strata}

A stratum $\cH^R_g (\mu)$ has a natural $\CC^{\star}$ action given by scaling of the differentials.  The projectivized generalized stratum $\PP\cH^R_g(\mu)$ is defined by the quotient $\cH^R_g (\mu)/\CC^{\star}$. Let $\pi:\cH^R_g(\mu)\to \PP\cH^R_g(\mu)$ be the quotient map. Flat surfaces in zero-dimensional projectivized strata will play a role of building blocks. Each connected component of such a stratum is just a point, thus the number of connected components can be computed by counting isomorphism classes of flat surfaces up to scaling. By looking at the dimension formula $\dim \PP\cR_g (\mu) = 2g+m+n-r-2$, we see that there are two types of zero-dimensional projectivized generalized strata. 

The first cases are the genus zero residueless strata with two zeroes. That is, $(g,m)=(0,2)$ and $r=n$. In general, such a stratum is of the form $\PP\cR_0 (a_1,a_2,-b_1,\dots,-b_n)$. By \cite[Proposition 2.3]{ccprincipal}, each connected component of this stratum correspond to a {\em configuration of type I} of parallel saddle connections. For the future use, we can paraphrase the proposition there into the following

\begin{proposition} \label{zerodim1} 
Let $(\PP^1,\omega)\in \cR_0 (a_1,a_2,-b_1,\dots,-b_n)$. Then it has exactly $n$ saddle connections joining $z_1$ and $z_2$, parallel to each other. Also, $(\PP^1,\omega)$ is uniquely determined up to scaling by following information:

\begin{itemize}
    \item A cyclic order on $\boldsymbol{p}$ given by a permutation $\tau\in Sym_n$.
    \item A tuple of integer ${\bf C}= (C_1,\dots,C_n)$ such that $1\leq C_i\leq b_i-1$ for each $i$, satisfying $\sum_i C_i = a_1+1$. Remark that if we denote $D_i \coloneqq b_i-C_i$, then $\sum_i D_i = a_2+1$. 
\end{itemize}
\end{proposition}

The second cases of zero-dimensional projectivized generalized strata are of the form $\PP\cR_0 (a,-b_1,\dots,-b_{n-1} ; -b_{n-1},-b_n)$. That is, $(g,m)=(0,1)$ and $|R|=n-1$. By \cite[Proposition 3.8]{ccprincipal}, each component of this stratum is given by a {\em configuration of type II} of parallel saddle connections. We can paraphrase the proposition there into the following

\begin{proposition} \label{zerodim2}
Let $(\PP^1,\omega)\in \cR_0(a,-b_1,\dots,-b_{n-2} ; -b_{n-1},-b_n)$. Then it has $n-1$ saddle connections, parallel to each other. Also, $(\PP^1,\omega)$ is determined uniquely up to scaling by following information:

\begin{itemize}
    \item A permutation $\tau \in Sym_{n-2}$ on the set of $n-2$ residueless poles.
    \item A tuple of integer ${\bf C}= (C_1,\dots,C_{n-2})$ such that $1\leq C_i\leq b_i-1$ for each $i$.
\end{itemize}
\end{proposition}

\subsection{Polar domains}

In this subsection, we define two types of polar domains that will be used for some constructions in the later sections. 

\begin{figure}
    \centering
    \tikzset{every picture/.style={line width=0.75pt}} 

\begin{tikzpicture}[x=0.75pt,y=0.75pt,yscale=-1,xscale=1]

\draw [color={rgb, 255:red, 200; green, 200; blue, 200 }  ,draw opacity=1 ]   (321.16,46.72) -- (321.16,189.22) ;
\draw [fill={rgb, 255:red, 200; green, 200; blue, 200 }  ,fill opacity=1 ]   (321.16,46.72) .. controls (236.31,65.35) and (233.31,167.35) .. (321.16,189.22) ;
\draw [fill={rgb, 255:red, 200; green, 200; blue, 200 }  ,fill opacity=1 ]   (321.16,46.72) .. controls (400.31,68.85) and (408.31,158.35) .. (321.16,189.22) ;
\draw    (304.81,184.44) .. controls (314.31,173.44) and (325.81,172.85) .. (335.31,182.85) ;
\draw    (303.81,52.35) .. controls (316.81,61.85) and (325.81,61.94) .. (337.81,53.44) ;
\draw   (316.34,118.09) .. controls (316.34,117.02) and (317.19,116.15) .. (318.24,116.15) .. controls (319.29,116.15) and (320.14,117.02) .. (320.14,118.09) .. controls (320.14,119.16) and (319.29,120.03) .. (318.24,120.03) .. controls (317.19,120.03) and (316.34,119.16) .. (316.34,118.09) -- cycle ;
\draw   (253.32,112.59) -- (257.53,106.52) -- (260.28,113.37) ;
\draw   (379.82,111.58) -- (382.07,104.55) -- (386.7,110.31) ;
\draw [fill={rgb, 255:red, 200; green, 200; blue, 200 }  ,fill opacity=1 ]   (117.66,178.72) .. controls (0.95,115.55) and (67.18,73.82) .. (117.17,74.61) .. controls (167.16,75.41) and (237.5,102.96) .. (122.65,178.72) ;
\draw    (104.31,170.94) .. controls (115.31,159.85) and (124.31,159.85) .. (134.81,169.85) ;
\draw   (110.84,110.59) .. controls (110.84,109.52) and (111.69,108.65) .. (112.74,108.65) .. controls (113.79,108.65) and (114.64,109.52) .. (114.64,110.59) .. controls (114.64,111.66) and (113.79,112.53) .. (112.74,112.53) .. controls (111.69,112.53) and (110.84,111.66) .. (110.84,110.59) -- cycle ;
\draw   (70.91,152.02) -- (68.91,144.91) -- (75.94,147.16) ;
\draw  [fill={rgb, 255:red, 0; green, 0; blue, 0 }  ,fill opacity=1 ] (318.66,189.22) .. controls (318.66,187.86) and (319.78,186.77) .. (321.16,186.77) .. controls (322.54,186.77) and (323.65,187.86) .. (323.65,189.22) .. controls (323.65,190.57) and (322.54,191.67) .. (321.16,191.67) .. controls (319.78,191.67) and (318.66,190.57) .. (318.66,189.22) -- cycle ;
\draw  [fill={rgb, 255:red, 0; green, 0; blue, 0 }  ,fill opacity=1 ] (318.66,46.72) .. controls (318.66,45.36) and (319.78,44.27) .. (321.16,44.27) .. controls (322.54,44.27) and (323.65,45.36) .. (323.65,46.72) .. controls (323.65,48.07) and (322.54,49.17) .. (321.16,49.17) .. controls (319.78,49.17) and (318.66,48.07) .. (318.66,46.72) -- cycle ;
\draw  [fill={rgb, 255:red, 0; green, 0; blue, 0 }  ,fill opacity=1 ] (117.66,178.72) .. controls (117.66,177.36) and (118.78,176.27) .. (120.16,176.27) .. controls (121.54,176.27) and (122.65,177.36) .. (122.65,178.72) .. controls (122.65,180.07) and (121.54,181.17) .. (120.16,181.17) .. controls (118.78,181.17) and (117.66,180.07) .. (117.66,178.72) -- cycle ;

\draw (319.65,69.67) node    {$2\pi D$};
\draw (321.65,163.67) node    {$2\pi C$};
\draw (324.22,106.76) node [anchor=north west][inner sep=0.75pt]  [font=\scriptsize]  {$p$};
\draw (120.65,137.67) node    {$2\pi ( b-1) +\pi $};
\draw (118.72,99.26) node [anchor=north west][inner sep=0.75pt]  [font=\scriptsize]  {$p$};
\draw (292.67,217.67) node [anchor=north west][inner sep=0.75pt]    {$P_{2}( C,D)$};
\draw (101.67,219.17) node [anchor=north west][inner sep=0.75pt]    {$P_{1}( b)$};

\end{tikzpicture} 
    \caption{Polar domains of type I and II} \label{fig201}
\end{figure}
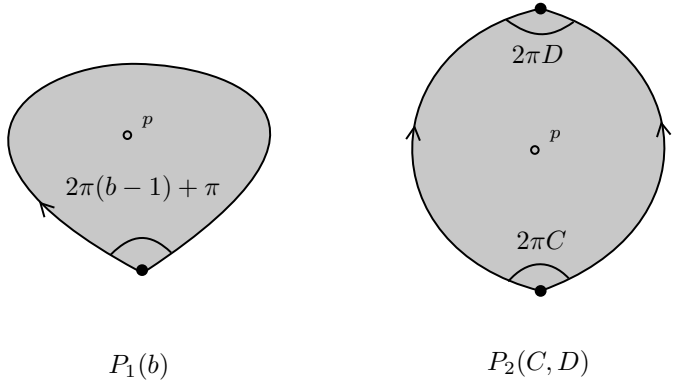

\begin{definition}
Let $(\PP^1, \omega)\in\cH_0 (b+1,-1,-b)$, whose residue at the simple pole is equal to 1. By cutting along the unique saddle connection, the Riemann sphere $\PP^1$ is separated into two regions. The region containing the pole of order $b$ is called  the {\em polar domain of type I} and denoted by $P_1(b)$. 
\end{definition}

Note that the boundary of $P_1(b)$ is a straight line joining the unique singularity $z$ to itself, forming an angle equal to $2\pi (b-1)+\pi$. The polar domain $P_1(b)$ is depicted in the left of \Cref{fig201}.

\begin{definition}
Let $(\PP^1, \omega)\in \cH_0 (C-1,D-1,-b)$, whose period over the unique saddle connection is equal to 1. The surface obtained by cutting along the unique saddle connection is called the {\em polar domain of type II} and denoted by $P_2 (C,D)$.
\end{definition}

Note that $P_2(C,D)$ is bounded by two straight lines joining two singularities $z_1,z_2$. They form two angles at $z_1$ and $z_2$, equal to $2\pi C$ and $2\pi D$, respectively. The polar domain $P_2(C,D)$ depicted in the right of \Cref{fig201}.
    
\section{The multi-scale compactification and degenerations of flat surfaces} \label{sec:msc}

The projectivized stratum $\PP \cH_g (\mu)$ has a smooth compactification $\PP \overline{\cH}_g (\mu)$ with normal crossings boundary, called the (projectivized) {\em moduli space of multi-scale differentials}. This is constructed in \cite{BCGGM} by Bainbridge, Chen, Gendron, Grushevsky and Möller. The construction can be generalized by a simple modification to the projectivized generalized stratum $\PP\cH^R_g(\mu)$, and we obtain a smooth compactification $\PP\overline{\cH}^R_g(\mu)$. For example, Mullane dealt with the strata of residueless multi-scale differentials $\PP\overline{\cR}_g(\mu)$ in \cite{mresidueless}. 

Since $\PP \overline{\cH}^R_g(\mu)$ is a smooth compactification with normal crossings boundary divisor, the closure $\overline{\cC}$ of a connected component $\cC \subset \cH^R_g(\mu)$ in $\overline{\cH}^R_g(\mu)$ is also a connected component therein. Therefore, there is a one-to-one correspondence between the connected components of $\cH^R_g(\mu)$ and the connected components of $\overline{\cH}^R_g(\mu)$. 

In this section, we will briefly recall the notion of the generalized stratum of multi-scale differentials $\overline{\cH}^R_g(\mu)$ and discuss how a flat surface in $\cH^R_g(\mu)$ degenerates to the boundary of $\overline{\cH}^R_g(\mu)$. The multi-scale differentials in the boundary consist of flat surfaces in the strata with smaller dimensions and a combinatorial datum called (the equivalence classes of) {\em prong-matchings}. Therefore the degeneration will provide us a way to use the induction on the dimension of the strata. 

\subsection{The moduli space of multi-scale differentials}

We recall some notions related to the multi-scale differentials.

\paragraph{\bf Enhanced level graph} Let $(\overline{X},\boldsymbol{z}\sqcup \boldsymbol{p})\in \overline{\cM}_{g,n+m}$ be a stable $\boldsymbol{z}\sqcup \boldsymbol{p}$-pointed curve. Recall that the {\em enhanced level structure} on the dual graph $\Gamma$ of $(\overline{X},\boldsymbol{z}\sqcup \boldsymbol{p})$ is given by 

1. A weak order $\preceq$ on the set of vertices $V(\Gamma)$. It is equivalent to a surjective {\em level function} $\ell : V(\Gamma)\to \{0,-1,\dots,-L\}$. An edge $e\in E(\Gamma)$ is called {\em vertical} if it is joining vertices in the distinct levels. Otherwise $e$ is called {\em horizontal}. We denote the set of vertical edges of $\Gamma$ by $E^v(\Gamma)$.

2. An assignment of a positive integer $\kappa_e$ for each edge $e\in E^v(\Gamma)$. 

For each $v\in V(\Gamma)$, we denote $g_v$ the genus of the irreducible component $X_v$ of $\overline{X}$. Then it satisfies $$2g_v-2=\sum_{z_i\mapsto v}a_i - \sum_{p_j\mapsto v}b_j + \sum_{e\in E^+(v)}(\kappa_e-1) - \sum_{e\in E^-(v)}(\kappa_e+1)$$ where the first (second, respectively) sum is over all zeroes (poles) incident to $v$, and $E^+(v)$ ($E^-(v)$, respectively) are the set of edges incident to $v$ that are going from $v$ to a lower (upper) level vertex. 

\paragraph{\bf Twisted differentials} A twisted differential on $(\overline{X},\boldsymbol{z}\sqcup \boldsymbol{p})\in \overline{\cM}_{g,n+m}$ compatible with the enhanced level graph $\Gamma$ is a collection of meromorphic differentials $\eta=\{\eta_v\}$, one for each irreducible component $X_v$ of $\overline{X}$, compatible with $\Gamma$. There are several conditions that ensure that $(\overline{X},\boldsymbol{z}\sqcup \boldsymbol{p},\eta)$ is compatible with $\Gamma$. One condition is that at the node corresponding to a vertical edge $e$, the differential $\eta_{+}$ on the upper component has a zero of order $\kappa_e-1$ and the differential $\eta_{-}$ on the lower component has a pole of order $-\kappa_e-1$. The other condition is the Global Residue Condition, which forces a sum of residues of certain poles at the nodes is equal to zero. See \cite{BCGGM} for the full detail. 

\paragraph{\bf Prong-matching} At the node $q$ corresponding to a vertical edge $e\in E^v(\Gamma)$, the upper level component has a zero $q^+$ of order $\kappa_e-1$. That is, the cone angle at the node is equal to $2\pi\kappa_e$. The lower level component has a pole $q^-$ of order $-\kappa_e-1$, which also has the cone angle $2\pi\kappa_e$. At each of $q^+$ and $q^-$, there are exactly $\kappa_e$ {\em prongs}, that is, choices of a horizontal direction. The {\em prong-matching} at $q$ is an orientation-reversing one-to-one correspondence between the prongs at $q^+$ and the prongs at $q^-$. There are exactly $\kappa_e$ prong-matchings, usually indexed by $\ZZ/\kappa_e\ZZ$. The set of prong-matchings $P_{\Gamma}\coloneqq \prod_{e\in E^v(\Gamma)} \ZZ/\kappa_e\ZZ$ is called the {\em prong rotation group}. The {\em level rotation action} is a homomorphism $\ZZ^L\to \prod_{e\in E^v(\Gamma)} P_{\Gamma}$ given by $\boldsymbol{n}\mapsto (n_{\ell (e^+)} - n_{\ell (e^-)} \mod \kappa_e)_{e\in E^v(\Gamma)}$. Two prong-matchings are called {\em equivalent} if they are contained in the same coset of the image of the level rotation action. 

\paragraph{\bf Multi-scale differentials}

The moduli space $\overline{\cH}^R_g(\mu)$ parameterizes {\em multi-scale differentials} $(\overline{X},\boldsymbol{z},\boldsymbol{p}, \eta, Pr)$ of type $\mu$ with residue condition $R$, consisting of the following data: 

\begin{itemize}
    \item A stable pointed curve $(\overline{X},\boldsymbol{z}\sqcup \boldsymbol{p})\in \overline{\cM}_{g,n+m}$ with an enhanced level structure on the dual graph $\Gamma$ of $\overline{X}$.
    \item A twisted differential $(\overline{X},\boldsymbol{z},\boldsymbol{p}, \eta)$ of type $\mu$, compatible with the enhanced level graph $\Gamma$ and satisfying the residue condition $R$.
    \item A prong-matching equivalence class $Pr$.
\end{itemize}

In this paper, we simply denote this multi-scale differential by $(\overline{X},\eta)$ or $\overline{X}$, when no confusion can arise.

The boundary divisors of $\overline{\cH}^R_g (\mu)$ are the closures of subspaces $D_{\Gamma}$ of multi-scale differentials compatible with $\Gamma$, where $\Gamma$ ranges over all enhanced level graphs with two levels and no horizontal edges, or with one level and one horizontal edge. 

\subsection{Plumbing construction}

Let $(\overline{X},\eta)\in \partial\overline{\cH}^R_g(\mu)$. The neighborhood of $(\overline{X},\eta)$ can be described by the {\em plumbing construction}. We can plumb any horizontal node, or plumb the level transition, that is, the collection of all nodes between chosen levels. The moduli parameters and smoothing parameters form a nice system of complex-analytic coordinates, see \cite{DozierSaddle} for detail. Here, we recall the explicit description for simple cases that we will use in this paper. Since we are only plumbing one horizontal node or the level transition between two levels, there is only one smoothing parameter, that we denote $t\in \CC$ here. More precisely, we will treat plumbing of a horizontal node and plumbing of the level transition between two levels, while there are at most two poles with non-zero residues at the bottom level. These cases can be described as combination of the following local constructions.

\paragraph{\bf Horizontal node}

A horizontal node $h$ represents a pair of simple poles $h_1$ and $h_2$ with opposite residues $\pm r$. The neighborhoods of each simple pole is a half-infinite flat cylinder, where the period of the closed geodesic enclosing the cylinder measures the residue of the pole. Since the residues at two poles are opposite, we can cut each cylinder along a closed geodesic and glue two boundaries. 

We can make this more precise using the standard coordinates. We can find a unique coordinate $u$ at $h_1$ so that $\omega=\frac{r}{u}du$. Also we can find a unique coordinate $v$ at $h_2$ so that $\omega=-\frac{r}{v}dv$. For some $\epsilon>0$, the standard coordinate neighborhoods contains each of the disks $U\coloneqq \{|u|<\epsilon\}$ and $V\coloneqq \{|v|<\epsilon\}$. We can remove two small disks $\{|u|\leq |t|\}\subset U$ and $\{|v|\leq |t|\}\subset V$, containing $q^+$ and $q^-$ respectively. For remaining points $u\in U$ and $v\in V$, we glue $u$ and $v$ whenever $uv=s$. Note that $\omega$ is invariant under the coordinate change $v=\frac{s}{u}$. 

\paragraph{\bf Node between two levels, no residue}

Assume the simplest possible situation, where we only have two irreducible components $(X_0,\omega_0)$ and $(X_{-1},\omega_{-1})$, at the level 0 and -1, respectively. Suppose that there is only one node $q$ between them. The prong rotation group is isomorphic to $\ZZ/\kappa\ZZ$. Fix a prong-matching, and suppose that two prongs $v^-$ and $v^+$, respectively at $q^-$ and $q^+$, are matched by this prong-matching. 

If the top level component $X_0$ only contains, if any, residueless marked poles, then by Global Residue Condition, the residues of the pole of $\omega_{-1}$ at $q$ is equal to zero. Let $\kappa$ be the number of prongs at $q$. We denote the scaling parameter by $s$, so we need to plumb at $q$ between two flat surfaces $(X_0,\omega_0)$ and $(X_{-1},s\omega_{-1})$ for small $|s|$. 

We can choose a standard coordinate $v$ at $q^-$ so that $\omega_{-1}=d(v^{-\kappa})=-\kappa v^{-\kappa-1}$. There are $\kappa$ choice of such coordinates, because whenever $v$ satisfies the above, $\xi^i_{\kappa} v$ also satisfies the same equation. For the given prong $v^-$ (i.e. a horizontal direction) at $q^-$, we can choose a unique standard coordinate $v$ such that $\{v\in \RR_+\}$ is the ray toward the direction of the prong $v^-$. Similarly, we can uniquely choose a standard coordinate $u$ at $q^+$, so that $\omega_0=d(u^{\kappa})=\kappa u^{\kappa-1}du$ and $\{u\in \RR_+\}$ is the ray toward the direction of the prong $v^+$. 

For some $\epsilon>0$, the standard coordinate neighborhoods contains each of the disks $U\coloneqq \{|u|<\epsilon\}$ and $V\coloneqq \{|v|<\epsilon\}$. We can remove two small disks $\{|u|\leq |s|\}\subset U$ and $\{|v|\leq |s|\}\subset V$, containing $q^+$ and $q^-$ respectively. For remaining points $u\in U$ and $v\in V$, we glue $u$ and $v$ whenever $uv=s$. Consequently we can identify $\omega_0=d\left(\frac{s}{v}\right)^{\kappa} = s^{\kappa} d(v^{-\kappa})=s^{\kappa} \omega_{-1}$. 

\paragraph{\bf Node between two levels with a non-zero residue}

Now we assume that $X_0$ contains a non-residueless pole, say $p$. Then Global Residue Condition does not apply and the pole $q^-$ of $\omega_{-1}$ may have a nonzero residue $r\neq 0$. By scaling $\omega_{-1}$, we may assume that $r=1$. In order to plumb at $q$, we need to choose a {\em modification differential} $\xi$ on $X_0$ that has only two (simple) poles at $p$ and $q^+$, so that the residue of $\xi$ at $q^+$ is equal to $-1$. 

We can choose a standard coordinate $v$ at $q^-$ so that $\omega_{-1}=-(v^{-\kappa} -1)\frac{dv}{v}$. There are $\kappa$ choice of such coordinates, because whenever $v$ satisfies the above, $\xi^i_{\kappa} v$ also satisfies the same equation. For the given prong $v^-$ (i.e. a horizontal direction) at $q^-$, we can choose a unique standard coordinate $v$ such that $\{v\in \RR_+\}$ is the ray toward the direction of the prong $v^-$. Similarly, we can uniquely choose a standard coordinate $u$ at $q^+$, so that $\omega_0+s^{\kappa}\xi= (u^{\kappa}-s^{\kappa})\frac{du}{u}$ and $\{u\in \RR_+\}$ is the ray toward the direction of the prong $v^+$. 

As in the previous case, we glue two annulus centered at $q^+$ and $q^-$ by identifying $u$ and $v$ whenever $uv=s$. Consequently we can identify $\omega_0+s^{\kappa}\xi=s^{\kappa} \omega_{-1}$. 

A pole of $\omega_0$ is also a pole of the differential $\omega_0+s^{\kappa} \xi$ with the same order. It has one additional simple pole at $q$. A zero $z$ of $\omega_0$ is no longer a zero of $\omega_0+s^{\kappa} \xi$. However, $\omega_0+s^{\kappa} \xi$ has $\kappa$ distinct zeroes in a small neighborhood of $z$. They are given by $u=s\xi^i_{\kappa}$ for $i=1,\dots, \kappa$ where $\xi_{\kappa}$ is a primitive $\kappa$-th root of unity. Thus they are all contained in the small disk $\{|u|\leq |s|\}$ which was removed. 

\paragraph{\bf Pair of nodes between two levels, with Global Residue Condition}

Suppose that there are two nodes $s_1$ and $s_2$ between $X_0$ and $X_{-1}$. If the top level component $(X_0,\omega_0)$ only contains, if any, residueless marked poles, then by Global Residue Condition, the residues of the poles of $\omega_{-1}$ at $s^-_1$ and $s^-_2$ are opposite to each other. By scaling $\omega_{-1}$, we assume that the residue at $s^-_1$ is equal to $1$.  

Let $\kappa_1$ and $\kappa_2$ be the numbers of prongs at the nodes $s_1, s_2$, respectively. We denote $\kappa=\lcm(\kappa_1,\kappa_2)$. A prong-matching between two levels is represented by an element of a prong-rotation group $\ZZ/\kappa_1\ZZ \times \ZZ/\kappa_2\ZZ$. Fix a prong-matching, and suppose that the prongs $v^-$ ($w^-$, resp) at $s^-_1$ ($s^-_2$), is matched to $v^+$ ($w^+$) at $s^+_1$ ($s^+_2$), by this prong-matching. 

To plumb two nodes $s_1$ and $s_2$, we need to choose a {\em modification differential} $\xi$ on $X_0$ that has only two (simple) poles at $s^+_1$ and $s^+_2$, so that the residue of $\xi$ at $s^+_1$ is equal to $1$. 

We can choose a standard coordinate $v$ at $s^-_1$, so that $\omega_{-1}=-(v^{-\kappa_1} -1)\frac{dv}{v}$, uniquely determined by the prong $v^-$ (i.e. a horizontal direction) at $s^-_1$. Similarly, we can choose a standard coordinate $u$ at $s^+_1$, so that $\omega_0-s^{\kappa}\xi= (u^{\kappa_1}-s^{\kappa})\frac{du}{u}$, uniquely determined by the prong $v^+$. Similar to the previous situation, we glue two annulus centered at $s^+_1$ and $s^-_1$ by identifying $u$ and $v$ whenever $uv=s^{\frac{\kappa}{\kappa_1}}$. Consequently we can identify $\omega_0-s^{\kappa}\xi=s^{\kappa} \omega_{-1}$. This plumbs the node $s_1$. Simultaneously, we can plumb the node $s_2$ in the same way, using the prongs $w^-$ and $w^+$. 

\subsection{Parallel saddle connections} 

The main tool that we will use in the proof of the main theorems is the degeneration to the principal boundary of $\cH^R_g(\mu)$. In \cite{kozo1}, the principal boundary is defined to be the set of surfaces obtained by shrinking a family of parallel saddle connections in a flat surface in the given stratum. In \cite{ccprincipal}, the principal boundary is described as a certain subspace of the twisted differentials in the boundary of the incidence variety compactification of $\cH^R_g(\mu)$. The description can be refined for the case of the multi-scale compactification $\overline{\cH}^R_g(\mu)$. There can be various ways to shrink a given saddle connection, but in this paper we will use the contraction flow that contracts the direction of the saddle connection. 

Two saddle connections of $(X,\omega)$ are said to be {\em parallel} if the ratio of the periods of $\omega$ over them is a real number. In particular, if two saddle connections are homologous (i.e. they represent the same class in $H_1 (X\setminus \boldsymbol{p}, \boldsymbol{z}; \ZZ)$), then the periods of $\omega$ over them are equal, thus they are parallel in particular. The converse is not necessarily true, since the periods of $\omega$ over two non-homologous saddle connections can be $\RR$-proportional. For general flat surfaces in $\cH^R_g(\mu)$, it is still true that non-homologous saddle connections are not parallel, due to the fact that the periods of $\omega$ over a certain set of saddle connections provide complex local coordinates for the stratum $\cH^R_g(\mu)$, and we can always slightly perturb the period coordinates in the stratum. Since the period coordinates give a map $\operatorname{Per} : H_1 (X\setminus \boldsymbol{p}, \boldsymbol{z}; \ZZ)^R \to \CC$, we have the following 

\begin{proposition} \label{parallel}
For any stratum $\cH^R_g (\mu)$, there exists an open dense subset $W\subset \cH^R_g (\mu)$ such that for any flat surface $(X,\omega)\in W$, two saddle connections of $(X,\omega)$ are parallel if and only if they are homologous in $H_1 (X\setminus \boldsymbol{p}, \boldsymbol{z}; \ZZ)^R$.
\end{proposition}

\begin{proof}
For the usual strata $\cH_g(\mu)$, this is proved in \cite[Proposition 3.1]{emz}. The same argument applies to the generalized strata. 
\end{proof}

The following definition of multiplicity of saddle connection is analogous to \cite{lanneauquad}, but adjusted to the setup of the generalized strata. 

\begin{definition}
A saddle connection $\gamma$ of $X\in \cH^R_g(\mu)$ is said to have {\em multiplicity} $k$ if there exists exactly distinct $k$ saddle connections of $X$ homologous to $\gamma$ in $H_1 (X\setminus{\boldsymbol{P}}, \boldsymbol{z}; \ZZ)^R$.
\end{definition}

If two saddle connections $\gamma_1$ and $\gamma_2$ are homologous in $H_1 (X\setminus{\boldsymbol{P}}, \boldsymbol{z}; \ZZ)^R$, then they have the same endpoints. Moreover, the closed curve $\gamma_1 \cup \gamma_2$ has trivial homology class. In other words, $X\setminus (\gamma_1 \cup \gamma_2)$ has two connected components, because any non-separating closed curve represents nonzero homology class. In the case of residueless strata $\cR_g(\mu)$, recall that $H_1 (X\setminus{\boldsymbol{P}}, \boldsymbol{z}; \ZZ)^R$ can be identified with $H_1 (X, \boldsymbol{z}; \ZZ)$. Therefore, \Cref{parallel} implies that two saddle connections of a general residueless flat surface are parallel if and only if they are homologous in $H_1 (X, \boldsymbol{z}; \ZZ)$.

\subsection{Shrinking a saddle connection}

Let $\cC$ be a connected component of $\cH^R_g (\mu)$, and let $(X,\omega)\in \cC$ be a general flat surface in the sense of \Cref{parallel}. For a given saddle connection $\gamma$ on $(X,\omega)$, we now describe how to {\em shrink} it. 

\begin{proposition}\label{short}
For any $\epsilon >0$, the flat surface $(X,\omega)\in \cC$ can be continuously deformed within its $\operatorname{GL}^{+}(2,\RR)$-orbit so that $|\gamma|<\epsilon|\gamma'|$ for any other saddle connection $\gamma'$ not homologous to $\gamma$. 
\end{proposition}

\begin{proof}
Since the length of the saddle connection $\gamma$ is equal to $\left|\int_{\gamma}\omega\right|$, we have $\int_{\gamma}\omega \neq 0$. By scaling $\omega$ if necessary, we may assume that the period over $\gamma$ is equal to $1$. Fix $L>0$ such that $\frac{1}{L}<\epsilon$. There are only finitely many saddle connections $\alpha_1,\dots, \alpha_M$ in $(X,\omega)$ of length $<L$. 

Let $C^t$ be the contraction flow that contracts the real direction and preserves the imaginary direction. By applying $C^t$ to $(X,\omega)$, we obtain a family of flat surfaces $(X^t,\omega^t)\coloneqq C^t \circ (X,\omega) \subset \cC$ containing a saddle connection $\gamma^t$ that comes from $\gamma$. The period of $\omega^t$ over $\gamma^t$ is equal to $e^{-t}$. Any saddle connection $\beta$ of $(X,\omega)$of length $>L$ deforms to a saddle connection $\beta^t$ of length $>Le^{-t}$. So for such $\beta$, we have $$\frac{|\gamma^t|}{|\beta^t|}<\frac{1}{L}<\epsilon.$$

For the saddle connections $\alpha_j$, we denote $\int_{\alpha_j}\omega \coloneqq a_j+ib_j$ for some real $a_j$ and $b_j\neq 0$ for each $j=1,\dots,M$. Let $$\delta \coloneqq \frac{1}{2} \operatorname{min}_{1\leq j\leq M} b_j.$$ 

By applying $C^t$, we obtain $$\int_{\alpha^t_j } \omega^t =a_j e^{-t} +i b_j.$$ For large enough $t$, the real part of the period becomes negligible and the length of $\alpha^t_j$ is then larger than $\frac{1}{2}b_j$. Also, we may assume that $e^{-t} < \delta \epsilon$. Therefore, the ratio between $|\gamma^t|$ and $|\alpha^t_j|$ is smaller than
$$\frac{2e^{-t}}{b_j}<\frac{2 \delta \epsilon}{2\delta} =\epsilon.$$ 
\end{proof}

\section{The principal boundary of residueless strata} \label{sec:pb}

In \cite{emz}, Eskin-Masur-Zorich described the principal boundary of strata of abelian differentials. It is a subspace of $\cH_g(\mu)$ that parameterises flat surfaces containing a collection of short parallel saddle connections. In \cite{ccprincipal}, D. Chen and Q. Chen described the principal boundary in terms of the twisted (holomorphic) differentials, as a certain subspace in the boundary of the incidence variety compactification. In this section, we will introduce the principal boundary and describe them as boundary strata of the multi-scale compactification --- in particular, in terms of the corresponding enhanced level graphs. The principal boundary of type I is obtained by shrinking a collection of parallel saddle connections joining two distinct zeroes. Thus it is only defined for the multiple-zero strata. The principal boundaries of type II is obtained by shrinking a collection of saddle connections joining a zero to itself. They are defined for any strata, but we will only describe them for \MIN strata as we only need these cases. 

At the end of this section, we prove that every connected component $\cC$ of $\cR_g (\mu)$ with $g>0$ has some principal boundary in its closure $\overline{\cC}$ in $\overline{\cR}_g (\mu)$. More precisely, $\overline{\cC}$ contains some principal boundary of type I if $\cR_g (\mu)$ is a multiple-zero stratum, and some principal boundary of type II if $\cR_g (\mu)$ is a \MIN stratum.

\subsection{Configurations of type I} \label{subsec:ct1}

Let $X\in\cR_g(\mu)$ be a general flat surface in the sense of \Cref{parallel}, with at least two zeroes. Suppose $X$ has a multiplicity $k$ saddle connection joining $z_1$ and $z_2$. That is, there are precisely $k$ saddle connections $\gamma_1,\dots, \gamma_k$ in the given direction, labeled in the clockwise order at $z_1$. We set $\gamma_{k+1}\coloneqq \gamma_1$ for convenience. We observe that the angles between $\gamma_i$ and $\gamma_{i+1}$ are $2\pi C_i$ at $z_1$ and $2\pi D_i$ at $z_2$, for some positive integers $C_i,D_i$. Denote ${\bf C}\coloneqq(C_1,\dots, C_k)$, ${\bf D}\coloneqq(D_1,\dots, D_k)$. The angles satisfy $\sum_i C_i = a_1 +1$ and $\sum_i D_i = a_2 +1$.

Let $\boldsymbol{p}_i$ ($\boldsymbol{z}_i$, respectively) denote the set of poles (zeroes), contained in the region bounded by $\gamma_i$ and $\gamma_{i+1}$. Also denote $\boldsymbol{z}_{-1}\coloneqq \{z_1,z_2\}$. Then $\boldsymbol{z}_{-1}\sqcup \boldsymbol{z}_1\sqcup\dots \sqcup \boldsymbol{z}_k$ is a partition of $\boldsymbol{z}$ and $\boldsymbol{p}_1\sqcup\dots\sqcup \boldsymbol{p}_k$ is a partition of $\boldsymbol{p}$. 

The data $\cF\coloneqq (a_1,a_2, {\bf C}, {\bf D}, \{\boldsymbol{z}_i\}, \{\boldsymbol{p}_i\})$ given by the collection of saddle connections $\gamma_i$ is said to be a {\em configuration of type I}. We say that $X$ {\em has a configuration $\cF$} if there exists a collection of saddle connections of $X$ that gives $\cF$. 

\subsection{Graphs of configurations of type I} \label{subsec:gt1}

Given a configuration $\cF$ of type I, we introduce the {\em configuration graph} $\Gamma(\cF)$ to describe the enhanced level graph of the multi-scale differentials in the principal boundary corresponding to $\cF$. This graph is already introduced in \cite{CMSZ}, as a {\em backbone} graph. 

For each $i=1,\dots,k$, consider a non-negative integer $$g_i \coloneqq \frac{1}{2}\left[ \sum_{z_j \in \boldsymbol{z}_i} a_j - \sum_{p_j \in \boldsymbol{p}_i} b_j + C_i + D_i\right].$$ If $\boldsymbol{z}_i=\emptyset$, $g_i=0$ and $\boldsymbol{p}_i=\{q_i\}$ for some pole $q_i$, then the region bounded by $\gamma_i$ and $\gamma_{i+1}$ is isomorphic to the polar domain $P_2(C_i,D_i)$. In this case, the order of the pole $q_i$ is equal to $C_i+D_i$. Let $J\subset\{1,\dots,k\}$ be the set of indices for which the region bounded by $\gamma_i$ and $\gamma_{i+1}$ is {\em not} isomorphic to such a polar domain, and denote $\boldsymbol{p}_{-1}=\{q_i|i\notin J\}$.

We define the {\em configuration graph} $\Gamma(\cF)$ as follows:
\begin{itemize}
    \item The set of vertices is $V(\cF)\coloneqq \{v_{-1}\}\cup \{v_i|i\in J\}$. 
    \item The set of edges is $E(\cF)\coloneqq \{e_i|i\in J\}$ where each $e_i$ is joining $v_{-1}$ and $v_i$. 
    \item  The vertex $v_i$ has half-edges marked by $\boldsymbol{z}_i$ and $\boldsymbol{p}_i$, for each $i\in \{-1\}\cup J$.
    \item We assign to each vertex $v_i$, $i\in J$, an integer $g_i\geq 0$. 
    \item We assign to each edge $e_i$, $i\in J$, an integer $\kappa_i \coloneqq C_i + D_i-1>0$.
    \item The level function $\ell : V(\cF)\to \{0,-1\} $ is given by $\ell(v_{-1})=-1$ and $\ell(v_i)=0$ for each $i\in J$.
\end{itemize}
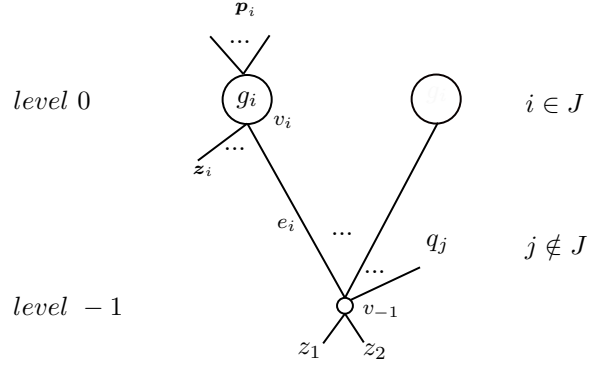
\begin{figure}
    \centering
    \tikzset{every picture/.style={line width=0.75pt}} 

\begin{tikzpicture}[x=0.75pt,y=0.75pt,yscale=-1,xscale=1]

\draw   (186.83,169.5) .. controls (186.83,167.29) and (188.62,165.5) .. (190.83,165.5) .. controls (193.04,165.5) and (194.83,167.29) .. (194.83,169.5) .. controls (194.83,171.71) and (193.04,173.5) .. (190.83,173.5) .. controls (188.62,173.5) and (186.83,171.71) .. (186.83,169.5) -- cycle ;
\draw    (179.83,188.5) -- (190.83,173.5) ;
\draw    (200.33,188) -- (190.83,173.5) ;
\draw    (141.58,77.5) -- (190.83,165.5) ;
\draw    (228.7,150.1) -- (193.83,166.5) ;
\draw    (237.7,77.1) -- (190.83,165.5) ;
\draw    (141.58,77.5) -- (116.64,96.22) ;
\draw    (122.83,33.5) -- (139.58,52.5) ;
\draw    (152.83,33) -- (139.58,52.5) ;

\draw    (141.5, 65.4) circle [x radius= 12.38, y radius= 12.38]   ;
\draw (141.5,65.4) node   [align=left] {$\displaystyle g_{i}$};
\draw (22,163.17) node [anchor=north west][inner sep=0.75pt]    {$level\ -1$};
\draw (22.17,59.33) node [anchor=north west][inner sep=0.75pt]    {$level\ 0$};
\draw (189.49,133) node   [align=left] {...};
\draw (137.82,36) node   [align=left] {...};
\draw (205.82,152.33) node   [align=left] {...};
\draw (237.6,138.41) node    {$q_{j}$};
\draw (135.34,90) node   [align=left] {...};
\draw (134,15.17) node [anchor=north west][inner sep=0.75pt]  [font=\scriptsize]  {$\boldsymbol{p}_{i}$};
\draw (112.64,97.22) node [anchor=north west][inner sep=0.75pt]  [font=\scriptsize]  {$\boldsymbol{z}_{i}$};
\draw (155.17,124.17) node [anchor=north west][inner sep=0.75pt]  [font=\footnotesize]  {$e_{i}$};
\draw (153.33,72) node [anchor=north west][inner sep=0.75pt]  [font=\footnotesize]  {$v_{i}$};
\draw (165.33,186.5) node [anchor=north west][inner sep=0.75pt]    {$z_{1}$};
\draw (198.33,187) node [anchor=north west][inner sep=0.75pt]    {$z_{2}$};
\draw (297.32,139.67) node    {$j\notin J$};
\draw (296.68,67.17) node    {$i\in J$};
\draw (198.33,167.5) node [anchor=north west][inner sep=0.75pt]  [font=\footnotesize]  {$v_{-1}$};
\draw  [color={rgb, 255:red, 9; green, 2; blue, 2 }  ,draw opacity=1 ]  (236.83, 65.17) circle [x radius= 12.38, y radius= 12.38]   ;
\draw (230.33,57.17) node [anchor=north west][inner sep=0.75pt]  [color={rgb, 255:red, 249; green, 249; blue, 249 }  ,opacity=1 ] [align=left] {$\displaystyle g_{i}$};

\end{tikzpicture}
    \caption{The graph $\Gamma(\cF)$ of a configuration $\cF$ of type I}\label{fig401}
\end{figure}

The enhanced level graph $\Gamma(\cF)$ is shown in \Cref{fig401}. It has two levels and no horizontal edges. The principal boundary $D(\Gamma(\cF))$ of $\overline{\cR}_g (\mu)$ is the subspace of multi-scale differentials compatible with the level graph $\Gamma(\cF)$. For a multi-scale differential $(\overline{X},\eta)\in D(\Gamma(\cF))$, we denote the irreducible component corresponding to the vertex $v_i$ by $(X_i,\eta_i)$ for each $i\in J\cup\{-1\}$.

\subsection{Configuration of type II} \label{subsec:ct2}

Assume that $\cR_g(\mu)$ is a \MIN stratum with unique zero $z$ of order $a$. Let $X\in \cR_g (\mu)$ be a general flat surface in the sense of \Cref{parallel}. Suppose that $X$ has a multiplicity $k$ saddle connection. That is, there are precisely $k$ saddle connections $\gamma_1,\dots, \gamma_k$ in the given direction, labeled in the clockwise order at $z$. Remark that the homology class $[\gamma_i]$ is nontrivial in $H_1(X;\ZZ)$ and $X\setminus (\cup_i \gamma_i)$ has $k$ connected regions. There is precisely one region with a pair of holes boundary, which is the union of $\gamma_1$ and $\gamma_k$. We observe that the angles between $\gamma_i$ and $\gamma_{i+1}$ are $2\pi C_i$ and $2\pi D_i$, for some positive integers $C_i,D_i$. Denote ${\bf C}\coloneqq(C_1,\dots, C_k)$, ${\bf D}\coloneqq(D_1,\dots, D_k)$. The angle bounded by $\gamma_1$ is $2\pi Q_1 + \pi$ and the angle bounded by $\gamma_k$ is $2\pi Q_2 +\pi$ for some non-negative integers $Q_1,Q_2$. Also these angles satisfy $\sum_i (C_i+D_i) + Q_1 + Q_2 = a$.

Let $\boldsymbol{p}_i$ denote the set of poles contained in the region bounded by $\gamma_i$ and $\gamma_{i+1}$. Also, Let $\boldsymbol{p}_0$ denote the set of poles contained in the region bounded by $\gamma_1$ and $\gamma_k$. Then $\boldsymbol{p}_0\sqcup\dots\sqcup \boldsymbol{p}_{k-1}$ is a partition of $\boldsymbol{p}$. 

The data $\cF\coloneqq (a, {\bf C},{\bf D},Q_1,Q_2, \{\boldsymbol{p}_i\})$ given by the collection of saddle connections $\gamma_i$ is said to be a {\em configuration of type II}. We say that $X$ {\em has a configuration $\cF$} if there exists a collection of saddle connections of $X$ that forms $\cF$. 

\subsection{Graphs of configurations of type II} \label{subsec:gt2}

Given a configuration $\cF$ of type II, we introduce the {\em configuration graph} $\Gamma(\cF)$ to describe the enhanced level graph of the multi-scale differentials in the principal boundary corresponding to $\cF$. 

For each $i=1,\dots,k-1$, consider a non-negative integer $$g_i\coloneqq \frac{1}{2}\left[C_i + D_i- \sum_{p_j \in \boldsymbol{p}_i} b_j \right]$$ and  $$g_0 \coloneqq \frac{1}{2}\left[Q_1 + Q_2 - \sum_{p_j\in \boldsymbol{p}_0} b_j\right].$$ If $g_i=0$ and $\boldsymbol{p}_i=\{q_i\}$ for some pole $q_i$, then the region bounded by $\gamma_i$ and $\gamma_{i+1}$ is isomorphic to the polar domain $P_2(C_i,D_i)$. Let $J\subset\{1,\dots,k-1\}$ be the set of indices for which the region bounded by $\gamma_i$ and $\gamma_{i+1}$ is {\em not} isomorphic to such a polar domain, and denote $\boldsymbol{p}_{-1}=\{q_i|i\notin J\}$.

Suppose that $X$ has a configuration $\cF$. All possible local patterns at $z$ are given in \cite[Sec.~3.1]{ccprincipal}. In other words, the regions of $X\setminus (\cup_i \gamma_i)$, listed in a clockwise order at $z$, are one of the following three possibilities:

(i) A cylinder, followed by surfaces of genus $g_i$ for each $i\in J$ with a figure eight boundary, followed by a cylinder. Two cylinders at the beginning and the end are the same. This is the case when $Q_1=Q_2=0$.

(ii) A cylinder, followed by surfaces of genus $g_i$ for each $i\in J$ with a figure eight boundary, followed by another surface with a pair of holes boundary. This case cannot happen because the cylinder at the beginning and the surface at the end should be the same. Therefore, it is impossible to have the cases $Q_1=0$ and $Q_2>0$, or $Q_1>0$ and $Q_2=0$.

(iii) A surface of genus $g_0$ with a pair of holes boundary, followed by surfaces of genus $g_i$ for each $i\in J$ with a figure eight boundary, followed by the surface of genus $g_0$. This is the case when $Q_1,Q_2>0$.

If $Q_1=Q_2=0$, then the local pattern at $z$ follows (i) above, and in this case we define the {\em configuration graph} $\Gamma(\cF)$ as follows:

\begin{itemize}
    \item The set of vertices is $V(\cF)\coloneqq \{v_{-1}\}\cup \{v_i|i\in J\}$. 
    \item The set of edges is $E(\cF)\coloneqq \{f\} \cup \{e_i|i\in J\}$ where each $e_i$ is joining $v_{-1}$ and $v_i$, and $f$ is joining $v_{-1}$ to itself.
    \item Each vertex $v_i$ has half-edges marked by $\boldsymbol{p}_i$.
    \item We assign to each $v_i$, $i\in  J$, an integer $g_i\geq 0$. 
    \item We assign to each $e_i$, $i\in J$, an integer $\kappa_i \coloneqq C_i+D_i-1>0$.
    \item The level function $\ell:V(\cF)\to \{0,-1\}$ is defined by $\ell(v_{-1})=-1$ and $\ell(v_i)=0$ for each $i\in J$.
\end{itemize}
The enhanced level graph $\Gamma(\cF)$ is depicted in \Cref{fig402}.

\begin{figure}
    \centering
    \tikzset{every picture/.style={line width=0.75pt}} 

\begin{tikzpicture}[x=0.75pt,y=0.75pt,yscale=-1,xscale=1]

\draw   (190.67,168.67) .. controls (190.67,166.46) and (192.46,164.67) .. (194.67,164.67) .. controls (196.88,164.67) and (198.67,166.46) .. (198.67,168.67) .. controls (198.67,170.88) and (196.88,172.67) .. (194.67,172.67) .. controls (192.46,172.67) and (190.67,170.88) .. (190.67,168.67) -- cycle ;
\draw    (194.67,189.67) -- (194.67,172.67) ;
\draw    (140.92,78.67) -- (194.67,164.67) ;
\draw    (235.7,149.2) -- (194.67,164.67) ;
\draw    (235.7,77.7) -- (194.67,164.67) ;
\draw    (124.17,34.67) -- (140.92,53.67) ;
\draw    (154.17,34.17) -- (140.92,53.67) ;
\draw     ;
\draw    (190.67,165.67) .. controls (147.2,143.9) and (149.7,193.9) .. (190.67,168.67) ;

\draw (21.83,161) node [anchor=north west][inner sep=0.75pt]    {$level\ -1$};
\draw (21.33,60.5) node [anchor=north west][inner sep=0.75pt]    {$level\ 0$};
\draw (298.68,67.33) node    {$i\in J$};
\draw (298.32,137.83) node    {$j\notin J$};
\draw    (140.92, 65.92) circle [x radius= 12.38, y radius= 12.38]   ;
\draw (140.92,65.92) node   [align=left] {$\displaystyle g_{i}$};
\draw (193.16,127.67) node   [align=left] {...};
\draw (140.16,37.67) node   [align=left] {...};
\draw (211.66,149.67) node   [align=left] {...};
\draw (240.6,137.41) node    {$q_{j}$};
\draw (141.24,22.43) node  [font=\scriptsize]  {$\boldsymbol{p}_{i}$};
\draw (157.62,119.26) node  [font=\footnotesize]  {$e_{i}$};
\draw (159.23,78.26) node  [font=\footnotesize]  {$v_{i}$};
\draw (195.11,193.17) node    {$z$};
\draw (209.38,169.26) node  [font=\footnotesize]  {$v_{-1}$};
\draw (149.11,168.26) node  [font=\footnotesize]  {$f$};
\draw  [color={rgb, 255:red, 9; green, 2; blue, 2 }  ,draw opacity=1 ]  (236.83, 64.83) circle [x radius= 12.38, y radius= 12.38]   ;
\draw (230.33,56.83) node [anchor=north west][inner sep=0.75pt]  [color={rgb, 255:red, 249; green, 249; blue, 249 }  ,opacity=1 ] [align=left] {$\displaystyle g_{i}$};

\end{tikzpicture} 
    \caption{The graph $\Gamma(\cF)$ of a configuration $\cF$ of type II when $Q_1=Q_2=0$} \label{fig402}
\end{figure}
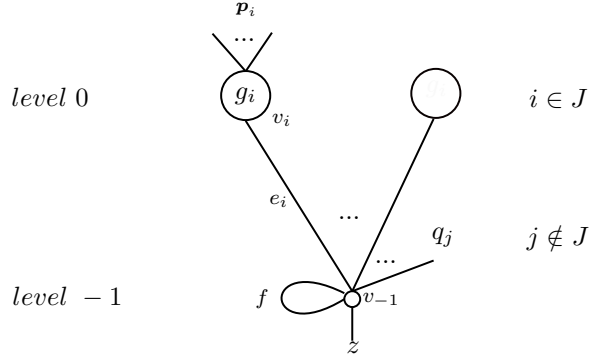

\begin{remark}
Note that it is possible that $J=\emptyset$ if $g=1$. In this case there exists no top level component, and $\Gamma(\cF)$ is in fact a single-level graph. We will still call the level containing $v_{-1}$ the {\em bottom level} for convenience. 
\end{remark}

If $Q_1,Q_2>0$, then $X$ can be constructed by the pattern (iii) above, and we define the configuration graph $\Gamma(\cF)$ as follows:
\begin{itemize}
    \item The set of vertices is $V(\cF)\coloneqq \{v_{-1},v_0\}\cup \{v_i|i\in J\}$. 
    \item The set of edges is $E(\cF)\coloneqq \{f_1,f_2\} \cup \{e_i|i\in J\}$ where each $e_i$ is joining $v_{-1}$ and $v_i$ and $f_1$ and $f_2$ is joining $v_{-1}$ and $v_0$.
    \item Each vertex $v_i$ has half-edges marked by $\boldsymbol{p}_i$.
    \item We assign to each $v_i$, $i\in \{0\}\cup J$, an integer $g_i\geq 0$. 
    \item We assign to each $e_i$, $i\in \{0\}\cup J$, an integer $\kappa_i \coloneqq C_i + D_i-1>0$. Also we assign $Q_1$ and $Q_2$ to $f_1$ and $f_2$, respectively. 
    \item The level function $\ell :V(\cF)\to\{0,-1\}$ is defined by $\ell(v_{-1})=-1$ and $\ell(v_i)=0$ for any $i\in \{0\}\cup J$.
\end{itemize}
The enhanced level graph $\Gamma(\cF)$ is described in \Cref{fig403}; it has two levels and no horizontal edge. 

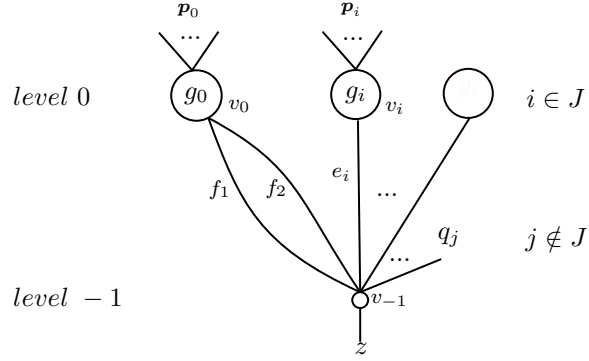
\begin{figure}
    \centering
    \tikzset{every picture/.style={line width=0.75pt}} 

\begin{tikzpicture}[x=0.75pt,y=0.75pt,yscale=-1,xscale=1]

\draw   (194.83,169) .. controls (194.83,166.79) and (196.62,165) .. (198.83,165) .. controls (201.04,165) and (202.83,166.79) .. (202.83,169) .. controls (202.83,171.21) and (201.04,173) .. (198.83,173) .. controls (196.62,173) and (194.83,171.21) .. (194.83,169) -- cycle ;
\draw    (198.83,190) -- (198.83,173) ;
\draw    (197.7,78.3) -- (198.83,165) ;
\draw    (239.7,148.3) -- (198.83,165) ;
\draw    (254.2,76.8) -- (198.83,165) ;
\draw    (179.83,33.5) -- (196.58,52.5) ;
\draw    (209.83,33) -- (196.58,52.5) ;
\draw    (122.2,77) .. controls (138.2,120.5) and (147.7,141.5) .. (198.83,165) ;
\draw    (122.2,77) .. controls (167.95,101.5) and (168.2,119.5) .. (198.83,165) ;
\draw    (97.33,34) -- (114.08,53) ;
\draw    (127.33,33.5) -- (114.08,53) ;

\draw    (196.58, 65.25) circle [x radius= 12.38, y radius= 12.38]   ;
\draw (196.58,65.25) node   [align=left] {$\displaystyle g_{i}$};
\draw (22.33,161) node [anchor=north west][inner sep=0.75pt]    {$level\ -1$};
\draw (22.33,60) node [anchor=north west][inner sep=0.75pt]    {$level\ 0$};
\draw (205.33,113.5) node [anchor=north west][inner sep=0.75pt]   [align=left] {...};
\draw (195.32,37) node   [align=left] {...};
\draw (218.82,148.5) node   [align=left] {...};
\draw (243.76,137.24) node    {$q_{j}$};
\draw (194.41,22.26) node  [font=\scriptsize]  {$\boldsymbol{p}_{i}$};
\draw (189.29,105.59) node  [font=\footnotesize]  {$e_{i}$};
\draw (216.4,72.09) node  [font=\footnotesize]  {$v_{i}$};
\draw (199.28,193.5) node    {$z$};
\draw (297.51,67) node    {$i\in J$};
\draw (213.55,169.59) node  [font=\footnotesize]  {$v_{-1}$};
\draw    (116.25, 65.74) circle [x radius= 12.73, y radius= 12.73]   ;
\draw (116.25,65.74) node   [align=left] {$\displaystyle g_{0}$};
\draw (128.03,113.09) node  [font=\footnotesize]  {$f_{1}$};
\draw (156.53,112.59) node  [font=\footnotesize]  {$f_{2}$};
\draw (138,71.09) node  [font=\footnotesize]  {$v_{0}$};
\draw (297.82,137.83) node    {$j\notin J$};
\draw (113.32,36.5) node   [align=left] {...};
\draw (112.43,22.76) node  [font=\scriptsize]  {$\boldsymbol{p}_{0}$};
\draw  [color={rgb, 255:red, 9; green, 2; blue, 2 }  ,draw opacity=1 ]  (253.33, 64.67) circle [x radius= 12.38, y radius= 12.38]   ;
\draw (246.83,56.67) node [anchor=north west][inner sep=0.75pt]  [color={rgb, 255:red, 249; green, 249; blue, 249 }  ,opacity=1 ] [align=left] {$\displaystyle g_{i}$};

\end{tikzpicture} 
    \caption{The graph $\Gamma(\cF)$ of a configuration $\cF$ of type II when $Q_1,Q_2>0$} \label{fig403}
\end{figure}

The principal boundary $D(\Gamma(\cF))$ of $\overline{\cR}_g (\mu)$ is the subspace of multi-scale differentials compatible with the level graph $\Gamma(\cF)$. For a multi-scale differential $(\overline{X},\eta)\in D(\Gamma(\cF))$, we denote the irreducible component corresponding to the vertex $v_i$ by $(X_i,\eta_i)$ for each $i\in J\cup\{-1,0\}$.

\subsection{Degeneration to the principal boundary}

Suppose that $X\in \cR_g(\mu)$ has a configuration $\cF$ formed by $k$ parallel saddle connections $\gamma_1,\dots, \gamma_k$. By \Cref{short}, for any $\epsilon>0$, $X$ can be continuously deformed to another flat surface, which by abuse of notation we will henceforth denote $X$, such that the saddle connections $\gamma_i$ have length $<\epsilon$ and any other saddle connections of $X$ have length $>3\epsilon$. By \cite[Thm.~2.1 and Thm.~3.4]{ccprincipal}, we have the following

\begin{proposition} \label{shrink}
Suppose that $X\in \cR_g(\mu)$ has a configuration $\cF$. Then there exists a continuous degeneration to the principal boundary $D(\cF)\subset \partial\overline{\cR}_g(\mu)$. Conversely, any multi-scale differential in $D(\cF)$ can be smoothed to a flat surface in $\cR_g(\mu)$ that has a configuration $\cF$.
\end{proposition}

In particular, if $\cR(\mu)$ is a multiple-zero stratum, then any connected component $\cC$ of $\cR(\mu)$ has a principal boundary of type I. In case of a \MIN stratum, any connected component $\cC$ has a principal boundary of type II.

\subsection{Degeneration to a multi-scale differential with two irreducible components}

Using the degeneration to the principal boundary, we can prove the following statement that will be useful when we apply the induction on the dimension of $\cR_g(\mu)$.

\begin{proposition} \label{break}
Suppose that $\dim_{\mathbb{C}} \cR_g(\mu)> 2$. Then for any connected component $\cC$ of $\cR_g (\mu)$, there exists a multi-scale differential $\overline{Y}\in \partial\overline{\cC}$ satisfying the following:

\begin{itemize}
    \item $\overline{Y}$ has exactly two irreducible components $Y_0$ and $Y_{-1}$ which are at different levels. Let $\boldsymbol{p_0}$ and $\boldsymbol{p_{-1}}$ denote the set of marked poles contained in $Y_0$ and $Y_{-1}$, respectively. 
    \item The components $Y_0$ and $Y_{-1}$ intersect at only one node $q$. 
    \item Moreover, if the number of zeroes $m>1$, for any given pair of zeroes, say $z_1,z_2$ by relabeling zeroes, the bottom level component $Y_{-1}$ contains $z_1,z_2$.
\end{itemize}
\end{proposition}

In other words, if $\dim_{\mathbb{C}} \cR_g(\mu)> 2$, we can break a flat surface in $\cC$ into two flat surfaces contained in the residueless strata of smaller dimensions. The level graph of $\overline{Y}$ is illustrated in \Cref{fig405}.

\begin{figure}
    \centering
    \tikzset{every picture/.style={line width=0.75pt}} 

\begin{tikzpicture}[x=0.75pt,y=0.75pt,yscale=-1,xscale=1]

\draw    (138.67,198.83) -- (149.67,183.83) ;
\draw    (159.17,198.33) -- (149.67,183.83) ;
\draw    (133.92,78.83) -- (150.17,151.33) ;
\draw    (181.14,133.08) -- (150.17,151.33) ;
\draw    (133.92,78.83) -- (108.67,96.83) ;
\draw    (117.17,33.33) -- (133.92,52.33) ;
\draw    (147.17,32.83) -- (133.92,52.33) ;

\draw    (133.07, 65.57) circle [x radius= 13.09, y radius= 13.09]   ;
\draw (133.07,65.57) node   [align=left] {$\displaystyle g_{0}$};
\draw (22.33,160.17) node [anchor=north west][inner sep=0.75pt]    {$level\ -1$};
\draw (21.17,60.5) node [anchor=north west][inner sep=0.75pt]    {$level\ 0$};
\draw (133.16,38.83) node   [align=left] {...};
\draw (160.16,135.33) node   [align=left] {...};
\draw (126.66,90.83) node   [align=left] {...};
\draw (128.17,16.83) node [anchor=north west][inner sep=0.75pt]  [font=\scriptsize]  {$\boldsymbol{p}_{0}$};
\draw (118.17,99.83) node [anchor=north west][inner sep=0.75pt]  [font=\scriptsize]  {$\boldsymbol{z}_{0}$};
\draw (144.67,104.83) node [anchor=north west][inner sep=0.75pt]  [font=\footnotesize]  {$e$};
\draw (125.17,196.33) node [anchor=north west][inner sep=0.75pt]    {$z_{1}$};
\draw (158.17,196.83) node [anchor=north west][inner sep=0.75pt]    {$z_{2}$};
\draw    (150.38, 167.57) circle [x radius= 15.81, y radius= 15.81]   ;
\draw (150.38,167.57) node   [align=left] {$\displaystyle g_{-1}$};
\draw (240.19,167.57) node    {$( Y_{-1} ,\zeta _{-1})$};
\draw (232.58,69.07) node    {$( Y_{0} ,\zeta _{0})$};
\draw (165.67,117.33) node [anchor=north west][inner sep=0.75pt]  [font=\scriptsize]  {$\boldsymbol{p}_{-1}$};
\draw (149.66,195.83) node   [align=left] {...};
\draw (143.17,217.83) node [anchor=north west][inner sep=0.75pt]  [font=\scriptsize]  {$\boldsymbol{z}_{-1}$};

\end{tikzpicture} 
    \caption{Breaking a flat surface into two irreducible components} \label{fig405}
\end{figure}
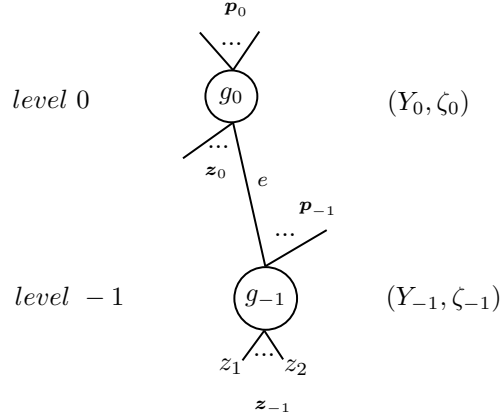

\begin{proof}
We will navigate the boundary of $\overline{\cC}$, until we obtain $\overline{Y}\in \partial\overline{\cC}$ satisfying the desired condition. Use the induction on $\dim_{\mathbb{C}} \cR_g(\mu)=2g+m-1$. 

First, we assume $m>1$. We need to show that there exists $X\in \cC$ that contains a saddle connection joining $z_1$ and $z_2$. If $m=2$, then this is trivial by \Cref{saddlehomology}. So assume that $m>2$. By \Cref{saddlehomology}, any flat surface $X\in \cC$ has a saddle connection $\gamma$ joining $z_1$ and $z_j$ for some $j\neq 1$. If $j=2$, then it is done. So assume $j>2$. Suppose that $\gamma$ is contained in a collection of parallel saddle connections that forms a configuration $\cF$ of type I. By \Cref{shrink}, we can shrink $\gamma$ and other saddle connections in the collection, obtaining $\overline{X}\in D(\cF)\subset \partial\overline{\cC}$. Then $\overline{X}$ has unique bottom level component $X_{-1}$ containing only two zeroes $z_1,z_j$, and $|J|$ top level components $X_i$ for $i\in J$. There exists some top level component $X_i$ that contains $z_2$. Since $X_i$ contains less than $m$ zeroes, by induction hypothesis, there exists a saddle connection $\gamma_i$ in $X_i$ joining $z_2$ and the node $q_i$ between $X_i$ and $X_{-1}$. At $q^+_i$, there exists a outgoing prong $u$ that corresponds to $\gamma_i$. At $q^-_i$, there are at least one incoming prong $v$ that comes from $z_2$. See \Cref{fig603}. We can choose a prong-matching of $\overline{X}$ that sends $v$ to $u$. By plumbing the level transition with this prong-matching, we obtain a flat surface in $\cC$ and $\gamma_i$ deforms to a saddle connection joining $z_1$ and $z_2$.

So we can assume that $\gamma$ is joining $z_1$ and $z_2$. By shrinking $\gamma$, we obtain $\overline{X}$ such that $X_{-1}$ contains only two zeroes $z_1$ and $z_2$. If $|J|=1$, then we can take $\overline{Y}=\overline{X}$. If $|J|>1$, then we further degenerate $\overline{X}$ to a three-level multi-scale differential $\overline{X'}\in \partial\overline{\cC}$ as follows. We keep one top level component, say $X_1$. We send $X_{-1}$ to the lowest level -2 and all other components are sent to the level -1. Plumb the level transition between the levels -1 and -2. As a result, we obtain $\overline{Y}$ as desired. 

Now we assume $m=1$. Then we have $g>1$ by assumption. We choose a flat surface $X\in \cC$ and a saddle connection $\gamma$ of $X$. Then $\gamma$ is contained in a collection of parallel saddle connections that forms a configuration $\cF$ of type II. By shrinking $\gamma$, we can obtain $\overline{X}\in D(\cF)\subset  \partial\overline{\cC}$. If $Q_1=Q_2=0$, then the bottom level component has a pair of simple poles. If $J=\emptyset$, then we have $g=1$ and $\dim_{\mathbb{C}} \cR_g(\mu)=2$, which contradicts the assumption. If $J\neq \emptyset$, then there exists a top level component, say $X_1$. As in the previous paragraph, we can degenerate $\overline{X}$ further into three-level differential so that $X_1$ is the only component at level 0. By plumbing the level transition between the levels -1 and -2, we obtain $\overline{Y}$ as desired. 

If $Q_1>0$ or $Q_2>0$, then there exists a top level component $X_0$ that intersects $X_{-1}$ at two nodes. If $J\neq \emptyset$, we can obtain $\overline{Y}$ by the same argument as in the previous paragraph. If $J=\emptyset$ and $X_0$ is the unique top level component, then the genus of $X_0$ is equal to $g-1>0$. Therefore by induction hypothesis, we can degenerate $X_0$ into a multi-scale differential $\overline{Z}$ with two irreducible components intersecting at one node. Together with $X_{-1}$ at level -2, they form a three-level multi-scale differential. By plumbing the level transition between the levels -1 and -2, we obtain $\overline{Y}$ as desired. 
\end{proof}

\subsection{Breaking up a zero and merging zeroes} \label{subsec:breakmerge}

Now we can explain how the two main surgeries in \cite{kozo1}, breaking up a zero and bubbling a handle, are related to the degeneration to the principal boundary of $\cR_g (\mu)$.

\begin{remark}
Let $\gamma$ be a multiplicity one saddle connection in $X\in \cC$, which forms a configuration $\cF$. In this case, the bottom level component $X_{-1}$ is contained in the stratum $\cR_0 (a_1, a_2, -a_1 -a_2-2)$ (if $\gamma$ has two distinct endpoints) or $\cR_0 (a ;-Q_1-1,-Q_2-1)$ (if $\gamma$ is a simple closed curve). Both of these strata are connected, which makes it easy to keep track of the connected component $\cC$. Therefore, multiplicity one saddle connections will play an important role in classification of the connected components. 
\end{remark}

The first surgery used in \cite{kozo1} is called {\em breaking up a zero}. For a flat surface $X\in \cR_g (\mu)$ with a zero $z$ of order $a>0$, breaking up the zero $z$ constructs a flat surface $X'\in \cR _g (\mu')$ where $\mu'$ is obtained by replacing $a$ with two integers $a_1,a_2\geq 0$ such that $a_1+a_2=a$. This surgery from the point of view of the multi-scale compactification is described in \cite{ChenGendronComponents}. Let $(\PP^1, \eta_{-1})$ be the unique (up to scaling) element of $\cH_0 (a_1,a_2,-a-2)$. We identify the unique zero $z$ of $X$ with the pole $p\in \PP^1$ of $\omega_{-1}$ to obtain a multi-scale differential in $\partial\overline{\cR}_g (\mu')$. The enhanced level graph of the differential, illustrated in \Cref{fig406}, consists of two vertices at distinct levels, and one edge connecting them (Note that there is a unique prong-matching equivalence class). By plumbing the level transition, we obtain a flat surface in $X'\in \cR_g (\mu')$. This surgery is called {\em breaking up a zero} $z$. The connected component of $X'$ depends only on $a_1$ and the connected component $\cC$ of $X$. 

\begin{figure}
    \centering
    \tikzset{every picture/.style={line width=0.75pt}} 

\begin{tikzpicture}[x=0.75pt,y=0.75pt,yscale=-1,xscale=1]

\draw   (218.83,121.67) .. controls (218.83,119.46) and (220.62,117.67) .. (222.83,117.67) .. controls (225.04,117.67) and (226.83,119.46) .. (226.83,121.67) .. controls (226.83,123.88) and (225.04,125.67) .. (222.83,125.67) .. controls (220.62,125.67) and (218.83,123.88) .. (218.83,121.67) -- cycle ;
\draw    (217.7,78.9) -- (222.83,117.67) ;
\draw    (217.08,79.17) -- (191.83,97.17) ;
\draw    (199.83,39.67) -- (216.58,58.67) ;
\draw    (229.83,39.17) -- (216.58,58.67) ;
\draw     ;
\draw    (211.83,141.17) -- (222.83,126.17) ;
\draw    (232.33,140.67) -- (222.83,126.17) ;
\draw    (66.08,69.17) -- (74.83,86.67) ;
\draw    (66.08,69.17) -- (40.83,87.17) ;
\draw    (49.33,29.17) -- (66.08,48.17) ;
\draw    (79.33,28.67) -- (66.08,48.17) ;
\draw   (66.83,159.67) .. controls (66.83,157.46) and (68.62,155.67) .. (70.83,155.67) .. controls (73.04,155.67) and (74.83,157.46) .. (74.83,159.67) .. controls (74.83,161.88) and (73.04,163.67) .. (70.83,163.67) .. controls (68.62,163.67) and (66.83,161.88) .. (66.83,159.67) -- cycle ;
\draw    (70.64,134.13) -- (70.83,155.67) ;
\draw     ;
\draw    (59.83,179.17) -- (70.83,164.17) ;
\draw    (80.33,178.67) -- (70.83,164.17) ;
\draw    (386.08,106.67) -- (386.83,126.17) ;
\draw    (386.08,106.67) -- (360.83,124.67) ;
\draw    (368.33,66.67) -- (385.08,85.67) ;
\draw    (398.33,66.17) -- (385.08,85.67) ;
\draw    (386.08,106.67) -- (402.33,124.67) ;
\draw    (109.14,104.63) -- (166.14,104.63) ;
\draw [shift={(168.14,104.63)}, rotate = 180] [color={rgb, 255:red, 0; green, 0; blue, 0 }  ][line width=0.75]    (10.93,-3.29) .. controls (6.95,-1.4) and (3.31,-0.3) .. (0,0) .. controls (3.31,0.3) and (6.95,1.4) .. (10.93,3.29)   ;
\draw    (270.64,102.13) -- (327.64,102.13) ;
\draw [shift={(329.64,102.13)}, rotate = 180] [color={rgb, 255:red, 0; green, 0; blue, 0 }  ][line width=0.75]    (10.93,-3.29) .. controls (6.95,-1.4) and (3.31,-0.3) .. (0,0) .. controls (3.31,0.3) and (6.95,1.4) .. (10.93,3.29)   ;

\draw    (216.66, 68.67) circle [x radius= 10, y radius= 10]   ;
\draw (216.66,68.67) node   [align=left] {$\displaystyle g$};
\draw (215.82,48.17) node [anchor=south] [inner sep=0.75pt]   [align=left] {...};
\draw (211.32,96.67) node [anchor=south] [inner sep=0.75pt]   [align=left] {...};
\draw (206.1,146.91) node    {$a_{1}$};
\draw (239.1,147.41) node    {$a_{2}$};
\draw    (65.66, 58.67) circle [x radius= 10, y radius= 10]   ;
\draw (65.66,58.67) node   [align=left] {$\displaystyle g$};
\draw (64.82,38.67) node [anchor=south] [inner sep=0.75pt]   [align=left] {...};
\draw (63.32,80.17) node   [align=left] {...};
\draw (77.78,90.17) node    {$a$};
\draw (59.1,186.41) node    {$a_{1}$};
\draw (86.6,186.41) node    {$a_{2}$};
\draw (48.33,118.17) node [anchor=north west][inner sep=0.75pt]    {$-a-2$};
\draw (68.83,100.67) node [anchor=north west][inner sep=0.75pt]    {$+$};
\draw (238.17,97.76) node  [font=\footnotesize]  {$a+1$};
\draw    (384.66, 96.17) circle [x radius= 10, y radius= 10]   ;
\draw (384.66,96.17) node   [align=left] {$\displaystyle g$};
\draw (383.32,69.17) node   [align=left] {...};
\draw (378.32,124.17) node [anchor=south] [inner sep=0.75pt]   [align=left] {...};
\draw (385.6,130.91) node    {$a_{1}$};
\draw (409.6,130.91) node    {$a_{2}$};
\draw (136.13,93.5) node   [align=left] {Identify};
\draw (295.63,91) node   [align=left] {Plumb};
\draw (300.63,113) node   [align=left] {level transition};
\draw (136.13,114.5) node   [align=left] {zero \& pole};

\end{tikzpicture} 
    \caption{Breaking up a zero} \label{fig406}
\end{figure}
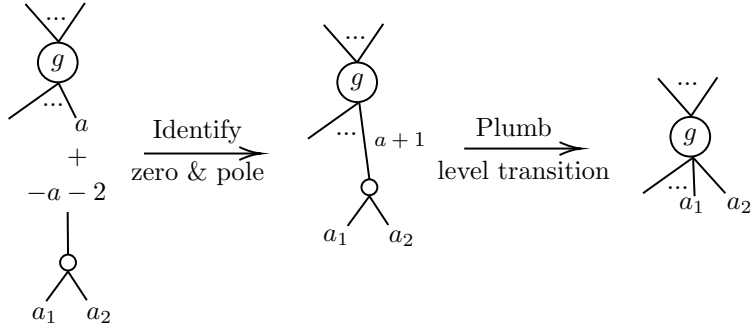

Conversely, if $\overline{\cC}$ contains a multi-scale differential $\overline{X}$ with level graph given in the middle of \Cref{fig406}, then a flat surface in $\cC$ can be obtained by breaking up a zero of the top level component $X_1$. Therefore a degeneration to this boundary divisor is the inverse operation to breaking up a zero. We will call this operation {\em merging zeroes} $z_1$ and $z_2$. We have the following result about merging zeroes. 

\begin{proposition}
\label{merging}
Suppose that $X\in \cR_g (\mu)$ is a general flat surface and it has two zeroes $z_1, z_2$ of orders $a_1, a_2$, respectively. If there exists a multiplicity one saddle connection $\gamma$ joining $z_1$ and $z_2$, then we can merge two zeroes $z_1$ and $z_2$ to obtain a flat surface with a zero of order $a_1 + a_2$. 
\end{proposition}

\begin{proof}
By applying \Cref{shrink} in the case $k=1$, we can easily see that the enhanced level graph $\Gamma(\cF)$ is equal to the enhanced level graph that appears in the middle of \Cref{fig406}. Thus $X$ can be obtained from a flat surface with a zero of order $a_1+a_2$ by breaking up this zero.
\end{proof}

\subsection{Bubbling and unbubbling a handle} \label{subsec:bubble}

The second surgery used in \cite{kozo1} is called {\em bubbling a handle}. For a flat surface $X\in \cR_g (\mu)$ with a zero $z$ of order $a$, bubbling a handle at $z$ produces a flat surface $X'\in \cR_{g+1} (\mu')$ where $\mu'$ is obtained by replacing $a$ with $a+2$.

Bubbling a handle can also be considered as plumbing a node from a certain multi-scale differential. Take a flat surface $(\PP^1 ,\eta_{-1})\in \cR_0 (a +2, -a -2; -1, -1)$. Note that by \Cref{zerodim2}, up to scaling, such a flat surface is uniquely determined by the angle $2\pi s$, $1\leq s\leq a+1$, between the two half-infinite cylinders. Denote $p\in \PP^1$ the residueless pole of order $a+2$. We identify the zero $z$ of $X$ and $p\in \PP^1$. We also identify two simple poles in $\PP^1$ to obtain a multi-scale differential in $\partial\overline{\cR}_g (\mu')$. The enhanced level graph is illustrated in \Cref{fig407} (Note that there is a unique prong-matching equivalence class). By plumbing the horizontal edge and the level transition, we obtain a flat surface $X'\in \cR_{g+1} (\mu')$. This surgery is called {\em bubbling a handle} at $z$. The connected component of $X'$ depends only on the connected component $\cC$ of $X$ and the choice of $1\leq s\leq a+1$. We denote this connected component by $\cC\oplus_{z} s$, as in \cite{lanneauquad} and \cite{boissymero}. Note that the flat surface $(\PP^1 ,\eta_{-1})\in \overline{\cH}_1 (a+2,-a-2)$ we used for bubbling a handle is contained in the boundary of the connected component of $\cH_1(a+2,-a-2)$ with rotation number $\gcd(a+2,s)$. The definition of rotation number will be recalled in \Cref{subsec:rot}. Therefore, we have $\cC'\oplus_z s_1=\cC'\oplus_z s_2$ when $\gcd(a+2,s_1)=\gcd(a+2,s_2)$. 

\begin{figure}
    \centering
    \tikzset{every picture/.style={line width=0.75pt}} 

\begin{tikzpicture}[x=0.75pt,y=0.75pt,yscale=-1,xscale=1]

\draw    (241.2,95.2) -- (247.5,123.83) ;
\draw    (240.75,94.83) -- (215.5,112.83) ;
\draw    (223,56.33) -- (239.75,75.33) ;
\draw    (253,55.83) -- (239.75,75.33) ;
\draw    (83.75,66.83) -- (92.5,84.33) ;
\draw    (83.75,66.83) -- (58.5,84.83) ;
\draw    (66.5,27.33) -- (83.25,46.33) ;
\draw    (96.5,26.83) -- (83.25,46.33) ;
\draw    (88,147.83) -- (88,164.33) ;
\draw     ;
\draw    (408.25,125.33) -- (388.2,142.7) ;
\draw    (391.5,75.83) -- (408.25,94.83) ;
\draw    (421.5,75.33) -- (408.25,94.83) ;
\draw    (408.25,125.33) -- (424.5,143.33) ;
\draw    (88,172.33) -- (88,189.33) ;
\draw    (127.64,116.63) -- (184.64,116.63) ;
\draw [shift={(186.64,116.63)}, rotate = 180] [color={rgb, 255:red, 0; green, 0; blue, 0 }  ][line width=0.75]    (10.93,-3.29) .. controls (6.95,-1.4) and (3.31,-0.3) .. (0,0) .. controls (3.31,0.3) and (6.95,1.4) .. (10.93,3.29)   ;
\draw    (294.64,115.62) -- (351.64,115.62) ;
\draw [shift={(353.64,115.62)}, rotate = 180] [color={rgb, 255:red, 0; green, 0; blue, 0 }  ][line width=0.75]    (10.93,-3.29) .. controls (6.95,-1.4) and (3.31,-0.3) .. (0,0) .. controls (3.31,0.3) and (6.95,1.4) .. (10.93,3.29)   ;
\draw   (84,168.33) .. controls (84,166.12) and (85.79,164.33) .. (88,164.33) .. controls (90.21,164.33) and (92,166.12) .. (92,168.33) .. controls (92,170.54) and (90.21,172.33) .. (88,172.33) .. controls (85.79,172.33) and (84,170.54) .. (84,168.33) -- cycle ;
\draw    (84.89,165.75) .. controls (65.27,153.54) and (67.71,183.46) .. (85.33,171.17) ;
\draw     ;
\draw    (247.5,131.83) -- (247.5,148.83) ;
\draw   (243.5,127.83) .. controls (243.5,125.62) and (245.29,123.83) .. (247.5,123.83) .. controls (249.71,123.83) and (251.5,125.62) .. (251.5,127.83) .. controls (251.5,130.04) and (249.71,131.83) .. (247.5,131.83) .. controls (245.29,131.83) and (243.5,130.04) .. (243.5,127.83) -- cycle ;
\draw    (244.39,125.25) .. controls (224.77,113.04) and (227.21,142.96) .. (244.83,130.67) ;

\draw    (240.25, 85.08) circle [x radius= 10, y radius= 10]   ;
\draw (240.25,85.08) node   [align=left] {$\displaystyle g$};
\draw (238.99,66.34) node [anchor=south] [inner sep=0.75pt]   [align=left] {...};
\draw (235.99,113.84) node [anchor=south] [inner sep=0.75pt]   [align=left] {...};
\draw    (83.25, 56.58) circle [x radius= 10, y radius= 10]   ;
\draw (83.25,56.58) node   [align=left] {$\displaystyle g$};
\draw (82.99,36.84) node [anchor=south] [inner sep=0.75pt]   [align=left] {...};
\draw (79.49,84.34) node [anchor=south] [inner sep=0.75pt]   [align=left] {...};
\draw (96.44,89.84) node    {$a$};
\draw (86.39,138.34) node    {$-a-2$};
\draw (82,106.33) node [anchor=north west][inner sep=0.75pt]    {$+$};
\draw (261.34,107.93) node  [font=\footnotesize]  {$a+1$};
\draw    (408.25, 110.08) circle [x radius= 15.53, y radius= 15.53]   ;
\draw (408.25,110.08) node  [font=\LARGE,color={rgb, 255:red, 255; green, 255; blue, 255 }  ,opacity=1 ] [align=left] {$\displaystyle g$};
\draw (406.99,85.84) node [anchor=south] [inner sep=0.75pt]   [align=left] {...};
\draw (406.99,141.84) node [anchor=south] [inner sep=0.75pt]   [align=left] {...};
\draw (427.94,151.34) node    {$a+2$};
\draw (85.94,197.34) node    {$a+2$};
\draw (154.63,104.5) node   [align=left] {Identify};
\draw (321.63,104.49) node   [align=left] {Plumb};
\draw (248.44,156.84) node    {$a+2$};
\draw (154.63,127) node   [align=left] {zero \& pole};
\draw (326.13,134.49) node   [align=left] {level transition \&\\horizontal node};
\draw (400.95,111) node    {$g$};
\draw (411.95,110) node    {$+1$};

\end{tikzpicture} 
    \caption{Bubbling a handle} \label{fig407}
\end{figure}
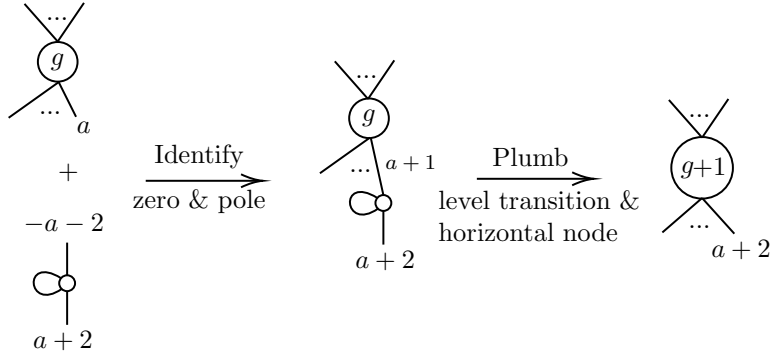

Conversely, if $\overline{\cC}$ contains a multi-scale differential with level graph in the middle of \Cref{fig407}, then a flat surface in $\cC$ can be obtained by bubbling a handle at a zero $z$ of the top level component. If $\cD$ is the connected component containing the top level component, then $\cC=\cD\oplus_{z} s$ for some $s$. Therefore a degeneration to this boundary divisor is the inverse operation of bubbling a handle. We will call this operation {\em unbubbling a handle}. Unlike merging zeroes, the operation of unbubbling a handle requires vanishing of periods of {\em two} independent classes in $H_1 (X, \ZZ)$. So we will achieve this by shrinking two non-parallel multiplicity one saddle connections in \Cref{sec:hg}. 

\section{Hyperelliptic components} \label{sec:hc}

Recall that a connected component $\cC$ of $\cR_g (\mu)$ is called {\em hyperelliptic} if every flat surface $(X,\omega)$ in $\cC$ is hyperelliptic. That is, the curve $X$ has an involution $\sigma$ such that $X/\sigma \cong \PP^1$ and $\sigma^* \omega = -\omega$. In this section, we will enumerate the hyperelliptic components of $\cR_g (\mu)$.

The involution $\sigma$ above permutes the poles, giving an involution $\mathcal{P}\in Sym_n$. It is immediate that $\mathcal{P}$ is a ramification profile of $\cR_g (\mu)$ if $m\leq 2$, as defined in \Cref{def:ramiprof}. In this case, the connected component $\cC$ is said to {\em have the ramification profile} $\cP$. Obviously, the ramification profile is a topological invariant of hyperelliptic components. That is, two hyperelliptic components are distinct if they have different ramification profiles. The number $0\leq r\leq 2g+2$ of marked points (i.e. poles and zeroes) fixed by $\cP$ is also a topological invariant. The hyperelliptic connected component $\cC$ is said to have {\em $r$ fixed marked points} if the number of poles plus the number of zeroes fixed by the ramification profile $\cP$ of $\cC$ is equal to $r$.

\subsection{Existence of hyperelliptic components}

We now prove the existence part of \Cref{mainhyper}.

\begin{proposition} \label{hyperprofile1}
A stratum $\cR_g (\mu)$ with $g>0$ has a hyperelliptic component if and only if $\cR_g(\mu)$ has $m\leq 2$ zeroes and it has a ramification profile. Moreover, for each ramification profile $\cP$ of $\cR_g(\mu)$, there exists at least one hyperelliptic component with ramification profile $\cP$.
\end{proposition}

\begin{proof}
We first deal with the case of \MIN strata. Then the profile $\cP$ must fix the unique zero $z$. By relabeling the poles, we may assume that $\cP$ fixes $1,\dots,r-1$ and $\cP(r+2i)=r+2i-1$ for each $i=1,\dots,f$ where $f=\frac{n-r-1}{2}$.

Denote $$\nu \coloneqq (a-1,-b_1-1, \dots, -b_{r-1} -1, -1^{2g+2-r}, -2b_{r},-2b_{r+2} \dots, -2b_{r+2f}).$$ This is a partition of $-4$, so there exists a stratum $\cQ_0 (\nu)$ of meromorphic quadratic differentials on $\PP^1$. Since a quadratic differential on $\PP^1$ is uniquely given (up to scaling) by the locations of its poles and zeroes, $\PP \cQ_0 (\nu)$ is isomorphic to $\cM_{0,2g+2+f}$. So $\dim_{\mathbb{C}} \cQ_0 (\nu)=2g+f$. Let $w$ denote the unique zero of order $a-1$, $u_i$ denote the pole of order $-b_i-1$, $v_j$ denote the pole of order $-2b_{r+2j}$, and $s_{\ell}$ denote the simple poles of the quadratic differentials in $\cQ_0 (\nu)$. The 2-residues of a quadratic differential $\xi\in \cQ_0 (\nu)$ automatically vanish at all odd order poles $u_i$ and $s_{\ell}$. Recall from \cite{getaquadratic} that the residual application $\mathscr{R}^2_0 (\nu) : \cQ_0 (\nu) \to \CC^f$ is a map sending a quadratic differential $\xi$ to $(\operatorname{Res}^2_{v_j}\xi)_j$. Let $\cQ^R_0 (\nu)$ denote the subspace of $\cQ_0 (\nu)$ of quadratic differentials with vanishing 2-residues at all $v_j$. By \cite[Theorem~1.3]{getaquadratic}, $\mathscr{R}^2_0 (\nu)$ is surjective and $\cQ^R_0 (\nu)\subset\cQ_0 (\nu)$ is a nonempty subspace of codimension $f$. So $\dim_{\mathbb{C}} \cQ^R_0 (\nu)=2g$.

Given a quadratic differential $\xi \in \cQ_0 (\nu)$, consider a double covering $\phi : C\to \PP^1$ ramified at the first $2g+2$ marked points. Note that $\phi$ is ramified exactly over the poles of odd orders. The curve $C$ is hyperelliptic (or elliptic if $g=1$) and $\phi$ is unique up to post-composing the hyperelliptic involution $\sigma$ on $C$, assuming $g>1$. The pullback $\phi^{\star} \xi$ is a square of an abelian differential $\omega$ on $C$, uniquely determined up to multiplication by $\pm 1$ (see \cite{lanneauquad}). We have

\begin{equation*}
    \operatorname{ord}_{\phi(x)} \xi =
    \begin{cases*}
        \operatorname{ord}_x \omega -1& if $\sigma$ fixes $x$ \\
        2\operatorname{ord}_x \omega & otherwise
    \end{cases*}
\end{equation*}

We label the preimages of $v_j$ by $p_{r+2j}$ and $p_{r+2j+1}$ and the preimage of $u_i$ by $p_i$. Then $(C,\omega)$ is contained in the stratum $\cH_g (\mu)$. Since $\phi$ is compatible with $\sigma$, we have $\sigma^* \omega=-\omega$ and therefore $(C,\omega)$ is a hyperelliptic flat surface. Also, $\operatorname{Res}^2_{\phi(x)} \xi =(\operatorname{Res}_x \omega)^2$. Therefore if $(\PP^1,\xi) \in \cQ^R_0 (\nu)$, then $(C,\omega) \in \cR_g(\mu)$. This gives an injective morphism $\Phi: \PP\cQ^R _0 (\nu) \to \PP\cR_g (\mu)$ whose image $\operatorname{Im} \Phi$ consists of hyperelliptic flat surfaces. Since $\dim_{\mathbb{C}} \PP \cQ^R _0 (\nu)=\dim_{\mathbb{C}} \PP \cR_g (\mu)=2g$, $\operatorname{Im} \Phi$ is a connected component of $\PP\cR_g (\mu)$.

Suppose now that $\cR_g (\mu)$ is a \NMIN stratum. So $m=2$ and $a_1=a_2=a$. By relabeling the poles, we may assume that $\cP$ fixes $1,\dots, r$ and $\cP(r+2i-1)=r+2i$ for each $i=1,\dots, f$, where $f=\frac{n-r}{2}$. As in the \MIN hyperelliptic case, denote $$\nu \coloneqq (2a, -b_1 -1, \dots, -b_r -1, -1^{2g+2-r}, -2b_{r+2}, \dots, -2b_f)$$ and we find a hyperelliptic component using the double covering of $\PP^1$ ramified at the odd poles of a quadratic differentials in $\cQ^R_0 (\nu)$.

Conversely, suppose that $\cR_g (\mu)$ has a hyperelliptic connected component $\cC$. The hyperelliptic involution gives an involution $\cP$ on the set of poles and zeroes. Therefore, in order to prove that $\cR_g (\mu)$ is of hyperelliptic type, it is enough to show that $m\leq 2$ and $\cP$ interchanges two zeroes when $m=2$. 

For any $(X,\omega)\in \cC$, there exists a hyperelliptic involution $\sigma$ such that $\sigma^\star \omega = -\omega$. Let $\phi :X\to X/\sigma \simeq \PP^1$ be the quotient map. There exists a unique quadratic differential $\xi$ on $\PP^1$ such that $\phi^\star \xi=(\omega)^2$. If $\sigma$ fixes a point $p\in X$, choose a local complex coordinate $x$ on~$X$ at $p$ such that locally $\sigma$ is given by $x\mapsto -x$. If the order of $\omega$ at $p$ is equal to $b$, then $\omega = (c_b x^b + c_{b+1}x^{b+1}+\dots) dx$ and $\sigma^\star \omega = (c_b (-x)^b + \dots) d(-x) = ((-1)^{b+1} c_b x^b +\dots) dx =-\omega$. Therefore, $c_i=0$ for all $i$ even. In particular, $b$ is even and $p$ is residueless if $b<0$. Also, $\operatorname{ord}_{\phi(p)}\xi=\operatorname{ord}_p \omega-1=b-1$ is odd. If $\sigma(p_1)=p_2$, then $p_1$ and $p_2$ have the same order and $\operatorname{ord}_{\phi(p_1)}\xi= 2\operatorname{ord}_{p_1} \omega$ is even. Moreover, $\operatorname{Res}^2_{\phi(p_1)}\xi =(\operatorname{Res}_{p_1} \omega)^2$. Therefore, the 2-residue of $\xi$ at $\phi(p_1)$ vanishes if and only if $p_1$ is residueless. In particular, the singularity type $\nu$ and the 2-residue condition of $\xi$ is completely determined by $\cP$. 

Since all orders of singularities of $\xi$ are integers, they are constant along deformations of $(X,\omega)\in \cC$ and thus we obtain a morphism $\Phi' : \PP\cC \to \PP \cQ^R_0 (\nu)$. Since $\PP\cC$ is connected, the image is also contained in a connected component of $\PP \cQ^R_0 (\nu)$. We denote this connected component by $\PP \cD$. Conversely, for any quadratic differential $(\PP^1,\xi)\in \cD$, consider the double covering $X\to \PP^1$ ramified at the odd poles. The pullback $\phi^\star \xi$ is a square of an abelian meromorphic differential $\omega$ on $X$, uniquely determined up to multiplication by $\pm 1$, which defines a morphism $\Phi: \PP\cD \to \PP\cR_g (\mu)$. The morphisms $\Phi,\Phi'$ are the inverses to each other, thus $\PP\cC$ and $\PP\cD$ are isomorphic. In particular, the dimensions of $\cR_g (\mu)$ and $\cQ^R_0 (\nu)$ must be equal. 

The dimension of $\cR_g (\mu)$ is equal to $2g+m-1$. Suppose that $\cP$ fixes $m_1$ zeroes and $n_1$ poles, and interchanges $m_2$ pairs of zeroes and $n_2$ pairs of poles. Then $m=m_1+2m_2$, $n=n_1+2n_2$. The dimension of $\cQ_0 (\nu)$ is equal to $2g+m_2+n_2$. The codimension of $\cQ^R_0 (\nu)$ in $\cQ_0 (\nu)$ is equal to the number of even order poles. There are exactly $n_2$ of them. So $\dim_{\mathbb{C}}\cQ^R_0 (\nu) =2g+m_2$ and we must have $2g+m-1=2g+m_2$. That is, $m_1+m_2=1$. So $\cR_g (\mu)$ has either a unique zero ($m_1=1,m_2=0$), or a pair of zeroes interchanged by $\cP$ ($m_1=0,m_2=1$).
\end{proof}

In order to complete the proof of \Cref{mainhyper}, we need to show that for a given ramification profile $\cP$ of $\mu$, there is a unique corresponding hyperelliptic component. In \Cref{sec:chc}, together with the techniques developed throughout \Cref{sec:ssc} and \Cref{sec:g1m}, we will prove the remaining part of \Cref{mainhyper} by using the description in this subsection.

\subsection{Boundary of hyperelliptic components} \label{subsec:pbhc}

In this subsection, we will give an immediate description of some boundary elements of hyperelliptic components. Let $\cC$ be a hyperelliptic component of $\cR_g (\mu)$. The symmetry given by the hyperelliptic involution allows us to describe the multi-scale differentials in $\partial\overline{\cC}$. They also have an involution compatible with the level structure and the prong-matching equivalence class. For example, if $\dim_{\mathbb{C}} \cR_g(\mu)>2$, then by \Cref{break} there exists a two-level differential $\overline{Y}$ with two irreducible component intersecting at a node $q$. Then both components have the symmetries compatible to each other. In conclusion, we have the following equivalent condition to being contained in a hyperelliptic component. 

\begin{lemma} \label{hyper1}
Suppose that $\dim_{\mathbb{C}} \cR_g(\mu)>2$. Let $\cC$ be a hyperelliptic connected component of $\cR_g (\mu)$, and $\overline{Y}\in \partial \overline{\cC}$ be a two-level differential obtained by \Cref{break}. Then both $Y_0$ and $Y_{-1}$ are hyperelliptic flat surfaces and their hyperelliptic involutions $\sigma_0$ and $\sigma_{-1}$ are compatible with the ramification profile of $\cC$. Moreover, both involutions fix the node $q$. 

Conversely, any multi-scale differential $(\overline{Y},\eta)$ satisfying the above conditions is contained in the principal boundary of {\em some} hyperelliptic component of $\cR_g (\mu)$ with ramification profile $\cP$. 
\end{lemma}

\begin{figure}
    \centering
    \tikzset{every picture/.style={line width=0.75pt}} 

\begin{tikzpicture}[x=0.75pt,y=0.75pt,yscale=-1,xscale=1]

\draw    (182.67,173.33) -- (193.67,158.33) ;
\draw    (203.17,172.83) -- (193.67,158.33) ;
\draw    (194.1,81.6) -- (194.67,126.33) ;
\draw    (220.14,102.9) -- (194.67,126.33) ;
\draw    (178.17,35.83) -- (194.92,54.83) ;
\draw    (208.17,35.33) -- (194.92,54.83) ;

\draw    (194.07, 68.07) circle [x radius= 13.09, y radius= 13.09]   ;
\draw (194.07,68.07) node   [align=left] {$\displaystyle g_{0}$};
\draw (48.66,146.84) node    {$level\ -1$};
\draw (42.21,67.84) node    {$level\ 0$};
\draw (194.16,40.83) node   [align=left] {...};
\draw (203.16,105.83) node   [align=left] {...};
\draw (194.76,25.09) node  [font=\scriptsize]  {$\boldsymbol{p}_{0}$};
\draw (185.11,94.93) node  [font=\footnotesize]  {$\kappa $};
\draw (181.49,181.07) node    {$a$};
\draw (205.49,181.57) node    {$a$};
\draw (150.99,145.57) node    {$\sigma _{-1} \curvearrowright $};
\draw (147.24,69.07) node    {$\sigma _{1} \curvearrowright $};
\draw (216.76,93.09) node  [font=\scriptsize]  {$\boldsymbol{p}_{-1}$};
\draw    (195.07, 142.57) circle [x radius= 15.81, y radius= 15.81]   ;
\draw (195.07,142.57) node   [align=left] {$\displaystyle g_{-1}$};

\end{tikzpicture} 
    \caption{The involutions of $\overline{Y}$} \label{fig501}
\end{figure}
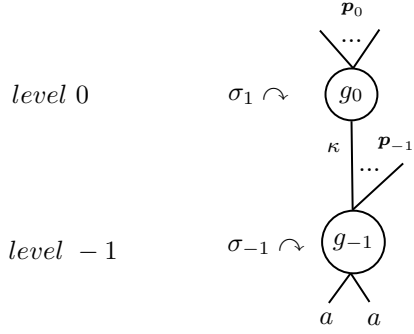

The level graph of $\overline{Y}$ and the involutions on each level are illustrated in \Cref{fig501}. Similarly, for a two-level differential with two irreducible components intersecting at two nodes, we also have the condition equivalent to being contained in the boundary of a hyperelliptic component. 
 
\begin{lemma} \label{hyper2}
Let $\cC$ be a hyperelliptic connected component of $\cR_g (\mu)$, and $\overline{Y}\in \partial \overline{\cC}$ be a two-level differential with two nodes $s_1,s_2$ between the levels. Then both $Y_0$ and $Y_{-1}$ are hyperelliptic flat surfaces and their hyperelliptic involutions $\sigma_0$ and $\sigma_{-1}$ are compatible with the ramification profile of $\cC$. Moreover, both involutions interchange the nodes $s_1,s_2$ and the prong-matching classes are compatible with the involution.

Conversely, any multi-scale differential $\overline{Y}$ satisfying the above conditions is contained in the principal boundary of {\em some} hyperelliptic component of $\cR_g (\mu)$ with ramification profile $\cP$. 
\end{lemma}

The level graph of $\overline{Y}$ and the involutions on each level are illustrated in \Cref{fig502}.

\begin{figure}
    \centering
    \tikzset{every picture/.style={line width=0.75pt}} 

\begin{tikzpicture}[x=0.75pt,y=0.75pt,yscale=-1,xscale=1]

\draw    (168.33,173) -- (168.33,156) ;
\draw    (160.6,107.47) -- (168.33,124) ;
\draw    (167.6,77.47) .. controls (144.6,88.97) and (140.6,108.47) .. (168.33,124) ;
\draw    (167.6,77.47) .. controls (186.6,84.97) and (199.6,105.97) .. (168.33,124) ;
\draw    (151.33,32) -- (168.08,51) ;
\draw    (181.33,31.5) -- (168.08,51) ;

\draw (21.83,136) node [anchor=north west][inner sep=0.75pt]    {$level\ -1$};
\draw (21.83,60.5) node [anchor=north west][inner sep=0.75pt]    {$level\ 0$};
\draw (170.82,110) node   [align=left] {...};
\draw (138.03,98.59) node  [font=\footnotesize]  {$Q$};
\draw (201.03,98.59) node  [font=\footnotesize]  {$Q$};
\draw (167.32,35.5) node   [align=left] {...};
\draw (167.43,20.76) node  [font=\scriptsize]  {$\boldsymbol{p}_{0}$};
\draw (168.99,179) node    {$a$};
\draw (171.76,100.09) node  [font=\scriptsize]  {$\boldsymbol{p}_{-1}$};
\draw (123.99,141.07) node    {$\sigma _{-1} \curvearrowright $};
\draw (120.24,66.07) node    {$\sigma _{0} \curvearrowright $};
\draw    (168.57, 64.07) circle [x radius= 13.09, y radius= 13.09]   ;
\draw (168.57,64.07) node   [align=left] {$\displaystyle g_{0}$};
\draw    (168.57, 140.1) circle [x radius= 15.81, y radius= 15.81]   ;
\draw (168.57,140.1) node   [align=left] {$\displaystyle g_{-1}$};
\draw (235.61,100.36) node  [font=\footnotesize]  {$Pr$};

\end{tikzpicture} 
    \caption{The involutions of $\overline{Y}$ with two nodes} \label{fig502}
\end{figure}

\section{Multiplicity one saddle connections} \label{sec:ssc}

In \Cref{subsec:breakmerge}, we showed that the existence of a certain set of multiplicity one saddle connections enables us to merge zeroes and unbubble a handle from a flat surface. In this section, we will show that every connected component of $\cR_g(\mu)$, except for the hyperelliptic components with $2g+2$ fixed marked points, contains a flat surface with a multiplicity one saddle connection. The main goal of this section is the following

\begin{theorem} \label{simple}
Let $\cC$ be a connected component of a residueless stratum $\cR_g(\mu)$ of genus $g>0$. Assume that $\cC$ is {\em not} a hyperelliptic component with $2g+2$ fixed marked points. Then for any given pair $z_i, z_j$ of (possibly identical when $m=1$) zeroes, $\cC$ contains a flat surface with a multiplicity one saddle connection joining $z_i$ and $z_j$. 
\end{theorem}

We use the induction on $\dim_{\mathbb{C}} \PP \cR_g (\mu) = 2g+m-2>0$. In order to initiate the induction, we will use the degeneration to the boundary of $\cC$, using the following

\begin{proposition} \label{breaksimple}
Let $\overline{X}$ be a multi-scale differential contained in $\partial\overline{\cC}$. Suppose that an irreducible component $X_i$ of $\overline{X}$ contains a multiplicity one saddle connection $\gamma_i$. Then by plumbing all horizontal nodes and the level transitions, $\gamma_i$ deforms into a multiplicity one saddle connection $\gamma$ of some flat surface $X\in \cC$.
\end{proposition}

\begin{proof}
By perturbation of the parameters for plumbing, we may assume that $X$ is general in the sense of \Cref{parallel}. Suppose that $X$ contains another saddle connection $\gamma'$ parallel to $\gamma$.  Then $\gamma$ and $\gamma'$ are homologous. If we degenerate $X$ to $\overline{X}$ by reversing the plumbing construction, $\gamma$ and $\gamma'$ remains homologous throughout the deformation. Thus they degenerate to parallel saddle connections on $X_i$, a contradiction. 
\end{proof}

That is, if we want to find a flat surface with a multiplicity one saddle connection, we can go to the boundary and look into the flat surfaces in each level. We prove one useful lemma below for future use. 

\begin{lemma}\label{suitableprong}
Let $\overline{Y}$ be a two-level multi-scale differential with two irreducible components $Y_0$ and $Y_{-1}$ at distinct levels, intersecting at one node $q$. Suppose that $\gamma$ is a saddle connection of $Y_0$ joining $q$ to itself, and $Y_{-1}$ is a genus zero residueless flat surface with two zeroes, $z_1$ and $z_2$. Then by plumbing the level transition with a suitable choice of prong-matching at $q$, $\gamma$ deforms to a saddle connection joining $z_1$ and $z_2$.  
\end{lemma}

\begin{proof}
Let $\kappa$ be the number of prongs at $q$. Then there are $\kappa$ outgoing prongs in $Y_0$ at $q$, denoted by $u_1,\dots,u_{\kappa}$. We can label them in clockwise order so that the saddle connection $\gamma$ encloses from $u_1$ to $u_t$. There are $\kappa$ incoming prongs in $Y_{-1}$ at $q$, denoted by $v_1,\dots, v_\kappa$ in counterclockwise order. We can assume that first $s$ prongs are coming from $z_1$ and the others are coming from $z_2$. Since $1<t\leq \kappa$, we can choose a prong-matching that sends $u_1$ to $v_m$ so that $1\leq m\leq s$ and $s+1\leq t+m-1\leq \kappa$. By plumbing the level transition with this prong-matching, $\gamma$ deforms to a saddle connection $\gamma'$ joining $z_1$ and $z_2$.
\end{proof}

We first deal with the base case: $g=1$ and $m=1$ (i.e. genus one \MIN strata). 

\begin{proposition} \label{simple1}
Let $\cC$ be a connected component of a genus one \MIN stratum $\cR_1(\mu)$. Assume that $\cC$ is {\em not} a hyperelliptic component with four fixed marked points. Then $\cC$ contains a flat surface with a multiplicity one saddle connection.
\end{proposition}

This case is very difficult compared to the remaining cases, because \Cref{break} is not applied. In order to use the degeneration techniques, we need to investigate the principal boundary of each connected component of $\PP\cR_1 (\mu)$. Since $\dim_{\mathbb{C}}\PP\cR_1(\mu)=1$, its boundary is a finite collection of points. We will prove that a connected component $\cC$ of $\PP\cR_1 (\mu)$ contains a principal boundary obtained by shrinking a multiplicity one saddle connection if $\cC$ is {\em not} a hyperelliptic component with four fixed marked points. The proof of \Cref{simple1} will be given in \Cref{subsec:simple_base}. 

\begin{remark}
One case that we can easily find a multiplicity one saddle connection is when a \MIN flat surface $X\in \cC$ contains a flat cylinder. The cross curve of the cylinder joining the unique zero to itself is automatically a multiplicity one saddle connection. The flat cylinder is closely related to the horizontal boundary of $\overline{\cC}$. Suppose that $\overline{X}\in \partial\overline{\cC}$ and the level graph $\Gamma(\overline{X})$ has a horizontal edge. Then we can plumb the node corresponding to this horizontal edge and obtain a flat surface in $\cC$ containing a flat cylinder. Therefore, if $\overline{\cC}$ contains a horizontal boundary divisor, then $\cC$ contains a flat surface with a multiplicity one saddle connection. 
\end{remark}

\subsection{The principal boundary of genus one \MIN strata} \label{subsec:pbg1m}

By \Cref{shrink}, every connected component of a \MIN stratum of dimension$>1$ has a principal boundary of type II. In the case of genus one \MIN strata, we can prove the following stronger statement. 

\begin{proposition} \label{bpb}
Each multi-scale differential in the boundary of a genus one \MIN stratum $\cR_1(\mu)$ is contained in the principal boundary of type II. Moreover, a two-level multi-scale differential in the boundary consists of two irreducible component intersecting at two nodes. 
\end{proposition}

\begin{proof}
Since the boundary of $\PP\overline{\cR}_1(\mu)$ is a finite collection of points, each multi-scale differential in the boundary is either horizontal or a two-level multi-scale differential. A horizontal multi-scale differential is obviously contained in the principal boundary of type II (see \Cref{fig403}). 

Let $\overline{X}$ be a two-level multi-scale differential. Since there is only one marked zero, it has only one bottom level component, which we denote by $X_{-1}$. The component $X_{-1}$ cannot admit a further degeneration, because $\dim_{\mathbb{C}}\PP\cR_1(\mu)=1$. So $X_{-1}$ is contained in a zero-dimensional (projectivized) generalized stratum. In particular, $X_{-1}$ contains only two poles $s_1$ and $s_2$ with nonzero residues. Since all marked poles are residueless, $s_1$ and $s_2$ are the nodes of $\overline{X}$. Also, each top level component gives a Global Residue Condition. Thus two nodes $s_1$ and $s_2$ must be contained in the same top level component, which we denote by $X_0$. If $X_{-1}$ has a pole other than $s_1$, $s_2$, then it is a residueless pole. Thus it is a marked pole or a unique node between $X_{-1}$ and some top level component. Therefore, if there exists a top level component other than $X_0$, then it is contained in a genus zero \MIN stratum. This is impossible because there exist no genus zero \MIN residueless flat surfaces. Therefore, $\overline{X}$ consists of two irreducible component intersecting at two nodes, thus contained in the principal boundary of type II (see \Cref{fig402}). 
\end{proof}

Let $\cR_1(\mu)$ be a genus one \MIN stratum. In order to prove \Cref{simple1}, we will navigate the boundary of the connected component $\cC$ of $\cR_1(\mu)$ until we find a principal boundary obtained by shrinking a multiplicity one saddle connection. In this subsection, we will give a complete description of the principal boundary of $\cR_1(\mu)$. Let $X\in \cR_1(\mu)$ be a general flat surface in the sense of \Cref{parallel} and let $\gamma$ be a saddle connection of $X$, forming a configuration $\cF$. By \Cref{shrink}, we can shrink $\gamma$ and obtain $\overline{X}\in D(\cF)$. Assume that $\overline{X}$ is not a horizontal boundary. The following lemma provides the combinatorial description of $\overline{X}$.

\begin{lemma} \label{genus1lemma}
A two-level multi-scale differential $\overline{X} \in \partial\overline{\cR}_1 (\mu)$ is determined by the following combinatorial data:
\begin{itemize}
\item An integer $1\leq t\leq n$, the number of marked poles on the unique top level component.
\item  A permutation $\tau$ on $\{1,\dots,n\}$, the set of marked poles.
\item A tuple of integers ${\bf C}=(C_1,\dots,C_n)$ such that $1\leq C_i \leq b_i -1$ for each $i$. The number of prongs at two nodes $s_1,s_2$ on the top component $X_0$ are given by $Q_1 = \sum_{i=1}^t C_{\tau(i)}$ and $Q_2=\sum_{i=1}^t b_{\tau(i)} -Q_1$.
\item A prong-matching equivalence class $Pr$ represented by a prong-matching $(u,v)\in \ZZ/Q_1\ZZ \times \ZZ/Q_2 \ZZ$. 
\end{itemize}
\end{lemma}

We denote this multi-scale differential by $X(t,\tau,{\bf C},Pr)$. 

\begin{proof}
By \Cref{bpb}, $\overline{X}$ consists of two irreducible components. The number of marked poles contained in the top level component $X_0$ determines the datum $1\leq t\leq n$.

We define $\tau$ and $C_i$ as follows. By \Cref{zerodim1}, there are $t$ parallel saddle connections $\alpha_0, \dots, \alpha_{t-1}$ from $s_1$ to $s_2$, labeled clockwise at $s_1$. By cutting the surface $X_0$ along all $\alpha_i$, we obtain $t$ connected components of $X_0 \setminus (\alpha_0\cup\dots\cup\alpha_{t-1})$. The component bounded by $\alpha_{i-1}$ and $\alpha_i$ contains one pole, which we denote by $p_{\tau(i)}$. This component is isomorphic to the polar domain $P_2 (C_{\tau(i)}, b_{\tau(i)} - C_{\tau(i)})$ for some integer $1\leq C_{\tau(i)} \leq b_{\tau(i)} -1$. Now we have $k$ angles $2\pi C_{\tau(i)}$ for $i=1,\dots,t$ at $s_1$, given by the saddle connections. Therefore the total angle $2\pi Q_1$ at $s_1$ is equal to $2\pi \sum_{i=1}^t C_{\tau(i)}$, and we have $Q_1=\sum_{i=1}^t C_{\tau(i)}$ and $Q_2 = \sum_{i=1}^t b_{\tau(i)} -Q_1$. The bottom level component $X_{-1}$ contains the other $n-t$ marked poles and two more unmarked poles at $s_1$ and $s_2$ of orders $Q_1+1$ and $Q_2+1$, respectively. By \Cref{zerodim2}, there are $n-t+1$ parallel saddle connections $\beta_{t},\dots, \beta_{n}$, labeled in clockwise order at $z$ so that $\beta_{t}$ bounds the polar domain of $s_1$. By cutting $X_{-1}$ along all $\beta_i$, we obtain $n-t+2$ connected components of $X_{-1} \setminus (\beta_t\cup\dots\cup\beta_n)$. Two components bounded by each of $\beta_n$ and $\beta_t$ are isomorphic to the polar domains $P_1(Q_2+1)$ and $P_1(Q_1+1)$, respectively. For $i=t+1,\dots,n$, the component bounded by $\beta_{i-1}$ and $\beta_i$ contains only one pole, denoted by $p_{\tau(i)}$. This component is isomorphic to the polar domain $P_2 (C_{\tau(i)}, b_{\tau(i)}-C_{\tau(i)})$ for some integer $1\leq C_{\tau(i)} \leq b_{\tau(i)}-1$. It is easy to see that $\tau$ is a permutation on $\{1,\dots, n\}$, and thus we have defined $C_i$ for each $i$ above. 

To determine $Pr$, we will label the prongs at the nodes. By scaling the differential of $X_0$ if necessary, we may assume that the periods of $\alpha_i$ are equal to $-1$. In $X_{-1}$, consider the $Q_1$ incoming prongs at $s_1$. They can be represented by half-infinite rays emanating from $z$. Let $v^-_1,\dots, v^-_{Q_1}$ denote them in clockwise order at $z$. Similarly, there are $Q_2$ outgoing prongs at $s_2$, denoted by $w^-_1,\dots, w^-_{Q_2}$ in counterclockwise order at $z$. In $X_0$, there are $Q_1$ prongs at $s_1$ denoted by $v^+_0,\dots, v^+_{Q_1 -1}$ in clockwise order, where $\alpha_0$ is in the direction of $v^+_0$. Similarly, there are $Q_2$ incoming prongs at $s_2$ denoted by $w^+_0,\dots, w^+_{Q_2-1}$ in counterclockwise order, where $\alpha_0$ is in the direction of $w^+_0$. The prongs at the nodes are illustrated in \Cref{fig601}. The prong rotation group $P_{\Gamma(\cF)}$ is isomorphic to $\ZZ/Q_1\ZZ \times \ZZ/Q_2\ZZ$. A prong-matching is determined by the images of $v^-_1$ and $w^-_{Q_2}$. If they are mapped to $v^+_u$ and $w^+_v$, respectively, we identify the prong-matching with an element $(u,v)\in \ZZ/Q_1\ZZ \times \ZZ/Q_2\ZZ$. This represents the prong-matching equivalence class $Pr$ of $\overline{X}$. 
\end{proof}

\begin{figure}
    \centering
    \input{diagram601} 
    \caption{The prongs at the nodes of $\overline{X}$} \label{fig601}
\end{figure}

\begin{remark}
We also introduce more notation related to the differential $X(t,\tau, {\bf C}, Pr)$ in \Cref{genus1lemma} for later uses. We denote $D_i \coloneqq b_i-C_i$ and ${\bf D}\coloneqq (D_1,\dots,D_n)$. Also, we denote $c_i \coloneqq \sum_{j=1} ^i C_{\tau(j)}$ and $d_i \coloneqq \sum_{j=1} ^i D_{\tau(j)}$. For convenience, we denote $c_0=d_0=0$. Then the saddle connection $\alpha_i$ lies between the prongs $v^+_{c_i-1}$ and $v^+_{c_i}$ at $s_1$ and between $w^+_{d_i -1}$ and $w^+_{d_i}$ at $s_2$ as depicted in \Cref{fig601}. 
\end{remark}

\begin{remark}\label{expression1}
The combinatorial data in \Cref{genus1lemma} is not uniquely determined by $\overline{X}$. It is only unique up to the choice of the labeling of two nodes $s_1, s_2$ and the labeling of the $t$ saddle connections $\alpha_i$ of $X_0$. We can describe the relations between data that give the same multi-scale differential. 

We can relabel the saddle connections $\alpha_i$ so that the cyclic order is remained unchanged. Any such relabeling is generated by shifting the labeling by one. In this case, new labeling gives a permutation $\tau'=\tau \circ \tau_1$ on the poles where $\tau_1 = \left(\begin{smallmatrix}
1 & 2 & \dots & t & t+1 & \dots & n\\
2 & 3 & \dots & 1 & t+1 & \dots & n
\end{smallmatrix}\right)$. Since other information is unchanged, we have $X(t,\tau, {\bf C}, Pr) = X(t,\tau \circ \tau_1, {\bf C}, Pr)$.

We can change the labeling of two nodes $s_1$ and $s_2$. Then the angles $2\pi C_i$ and $2\pi D_i$ also change the roles. The saddle connections $\alpha_i$ and $\beta_j$ are relabeled in the inverse order. So new labeling gives a permutation $\tau''=\tau\circ \tau_2$ on the poles where $\tau_2 = \left(\begin{smallmatrix}
1 & \dots & t & t+1 & \dots & n\\
t & \dots & 1 & n & \dots & t+1
\end{smallmatrix}\right)$. The prong-matching $(u,v)\in \ZZ/Q_1\ZZ \times \ZZ/Q_2\ZZ$ is sent to $(-v,-u)\in \ZZ/Q_2\ZZ\times \ZZ/Q_1\ZZ$ with new labeling, so we have $X(t,\tau, {\bf C}, [(u,v)]) = X(t,\tau \circ \tau_2, {\bf D}, [(-v,-u)])$.
\end{remark}

\begin{remark}
The level rotation group $\ZZ$ acts on the prong rotation group $ \ZZ/Q_1\ZZ \times \ZZ/Q_2\ZZ$ by $k\cdot (u,v)=(u+k,v-k)$. Recall that two prong-matchings are said to be {\em equivalent} if the level rotation action transforms one prong-matching into the other. So the number of prong-matching equivalence classes is equal to $\gcd (Q_1,Q_2)$. 
\end{remark}

There also exist multi-scale differentials in the horizontal boundary of $\cR_1(\mu)$. Those differentials are given by the elements of the stratum $\cR_0 (a,-b_1,\dots,-b_n;-1,-1)$, except that two simple poles are unmarked (i.e, switching the labeling of two simple poles does not change the differential as a boundary point of $\cR_1(\mu)$). The stratum $\PP \cR_0 (a,-b_1,\dots,-b_n;-1,-1)$ is zero-dimensional and the flat surfaces in this stratum are described in \Cref{zerodim2}. We obtain the following

\begin{lemma} \label{genus1hor}
A horizontal multi-scale differential $\overline{X}$ in the boundary of $\cR_1 (\mu)$ is given by the following combinatorial data:
\begin{itemize}
    \item A permutation $\tau$ on $\{1,\dots,n\}$.
    \item A tuple of integers ${\bf C}=(C_1,\dots, C_n)$ such that $1\leq C_i \leq b_i -1$ for each $i$. 
\end{itemize}
\end{lemma}

We denote this multi-scale differential by $X(0,\tau, {\bf C})$.

\begin{remark}
The data in \Cref{genus1hor} is only unique up to the choice of the labeling of two simple poles. By changing the roles of two simple poles, we have $X(0,\tau, {\bf C}) = X(0,\tau\circ \tau_2, {\bf D})$ as in \Cref{expression1}, where $\tau_2 = 
\left(\begin{smallmatrix}
1 & 2 & \dots & n \\
n & n-1 & \dots & 1
\end{smallmatrix}\right)$ is a permutation inverting the order.
\end{remark}

\subsection{Plumbing and saddle connections of genus one \MIN flat surfaces}

Let $\overline{X}=X(t,\tau, {\bf C}, Pr)$ be a two-level multi-scale differential in $\partial\overline{\cR}_1(\mu)$, given by the combinatorial data in \Cref{genus1lemma}. Recall that the top level component $X_0$ has $t$ saddle connections $\alpha_i$, $i=0,\dots, t-1$ and the bottom level component $X_{-1}$ has $n-t+1$ saddle connections $\beta_j$, $j=t,\dots,n$. By rescaling the differential $X_0$, we may assume that the period of $\alpha_i$ are equal to $-1$. We can obtain a flat surface $(X,\omega)\in\cR_1 (\mu)$ by plumbing construction with a prong-matching $(u,v)\in Pr$ and the smoothing parameter $s=\epsilon e^{i\theta}\in \mathbb{C}$. We will denote this flat surface by $\overline{X}_s(u,v)$. The periods of the saddle connections $\beta'_j$ deformed from $\beta_j$ are equal to $s$. Also, recall that the periods of $\alpha_i$ are set to be $-1$. Each saddle connection in the components of $\overline{X}$ is the limit of saddle connections in $\overline{X}_s(u,v)$ as $|s|=\epsilon \to 0$. So there exists a saddle connection $\alpha'_i$ that converges to $\alpha_i$. For small $s$, we have $\operatorname{Im} s >0$ if and only if $\operatorname{Im} \left( \int_{\beta'_j} \omega / \int_{\alpha'_i} \omega \right) <0$. We can describe the configuration of the saddle connections of $\overline{X}_s(u,v)$. 

\begin{proposition} \label{saddledegenerate}
Let $\overline{X}=X(t,\tau, {\bf C},[(u,v)])\in \partial\overline{\cR}_1(\mu)$. Consider $\overline{X}_s(u,v)$ for $\operatorname{Im} s\leq 0$. For the saddle connections $\beta_j$ and $\alpha_i$ of each irreducible components of $\overline{X}$, there exists a unique saddle connection $\beta'_j$ and $\alpha'_i$ of the flat surface $\overline{X}_s (u,v)$ that degenerates to $\beta_j$ and $\alpha_i$, respectively, as $|s| \to 0$.
\end{proposition}

\begin{proof}
If $\beta'_j$ degenerates to $\beta_j$, then $\overline{X}$ is obtained by shrinking $\beta'_j$. If another saddle connection $\beta''_j$ of $\overline{X}_s (u,v)$ degenerates to $\beta_j$, then the homology classes $[\beta'_j]$ and $[\beta''_j]$ in $H_1(X\setminus \boldsymbol{p},\boldsymbol{z};\ZZ)$ differ by a multiple of the vanishing class, which is equal to $[\beta'_j]$. In particular, $[\beta''_j]$ is an integer multiple of $[\beta'_j]$. However, since both saddle connections are simple closed curves in $\overline{X}_s (u,v)$, their homology classes are primitive. This is a contradiction unless $[\beta'_j]=[\beta''_j]$. So $\beta'_j$ and $\beta''_j$ are homologous, and thus they degenerate to distinct parallel saddle connections in $\overline{X}$ by shrinking $\beta'_j$ and $\beta''_j$ simultaneously. This is a contradiction and therefore $\beta'_j$ is unique. 

Now suppose that $\alpha'_i$ and $\alpha''_i$ degenerate to $\alpha_i$. Then the homology classes $[\alpha'_i]$ and $[\alpha''_i]$ are equal in $H_1(X,\boldsymbol{z};\ZZ)$ or they differ by a multiple of the vanishing class $[\beta'_n]$. If $[\alpha'_i]=[\alpha''_i]$, then $\overline{X}_s (u,v)$ contains a flat parallelogram bounded by $\alpha'_i,\alpha''_i,\beta'_t$ and $\beta'_n$. This means $\operatorname{Im} \left( \int_{\beta'_j} \omega / \int_{\alpha'_i} \omega \right)<0$, a contradiction. So we assume that $[\alpha'_i]=[\alpha''_i]+k[\beta'_n]$ for some $k$. We may assume that an angle between $\alpha'_i$ and $\alpha''_i$ at $z$ is smaller than $\pi$, as this angle converges to zero as $|s|\to 0$. In particular, $\alpha'_i$ and $\alpha''_i$ are two sides of a flat triangle. The last side of this triangle is $\beta'_t$ or $\beta'_n$. This means $\operatorname{Im} \left( \int_{\beta'_j} \omega / \int_{\alpha'_i} \omega \right)<0$, a contradiction. 
\end{proof}

The following proposition determines parallel saddle connections of $\overline{X}_s(u,v)$.

\begin{proposition} \label{genus1parallel}
Let $\overline{X}=X(t,\tau, {\bf C}, [(u,v)])$ be a two-level multi-scale differential in $\partial\overline{\cR}_1(\mu)$. By relabeling the saddle connections, we may assume $c_{t-1}<u\leq Q_1$. If $\operatorname{Im} s\leq 0$, then two saddle connections $\alpha'_i$ and $\alpha'_j$ of $\overline{X}_s(u,v)$ for $i<j$ are parallel if and only if $d_j \leq v<Q_2$ or $0\leq v< d_i$.
\end{proposition}

\begin{proof}
Since the surface $\overline{X}_s(u,v)$ has genus one, two simple closed curves in $\overline{X}_s(u,v)$ are homologous if and only if the intersection number between them is equal to zero. Two saddle connections $\alpha'_i$ and $\alpha'_j$ intersect only at the unique zero $z$. So the intersection number between them is equal to zero if and only if they do not intersect at $z$ transversely. We will determine the intersection number in terms of the prong-matching $(u,v)$. 

First, suppose that $d_j \leq v<Q_2$. At the node $s_1$, the prongs $v^+_{c_i}, v^+_{c_j}$ are matched to $v^-_{c_i+Q_1-u+1}, v^-_{c_j+Q_1-u+1}$, respectively. At $s_2$, the prongs $w^+_{d_i}, w^+_{d_j}$ are matched to $w^-_{d_i+Q_2-v}, w^-_{d_j+Q_2-v}$, respectively. The saddle connection $\alpha'_i$ is coming out from $z$ along $v^-_{c_i+Q_1-u+1}$ and going into $z$ along $w^-_{d_i+Q_2-v}$ in \Cref{fig601}. Similarly, $\alpha'_j$ is coming out from $z$ along $v^-_{c_j+Q_1-u+1}$ and going into $z$ along $w^-_{d_j+Q_2-v}$. Note that $1\leq c_i+Q_1-u +1< c_j+Q_1-u+1 \leq Q_1$ and $1\leq d_i+Q_2-v < d_j+Q_2-v\leq Q_2$. Thus $\alpha'_i$ and $\alpha'_j$ do not intersect transversely at $z$. By the same argument, this is also true under the assumption $0\leq v< d_i$. Therefore, the intersection number between $\alpha'_i$ and $\alpha'_j$ is zero. Thus $\alpha'_i$ and $\alpha'_j$ are parallel.

Conversely, suppose that $d_i\leq v<d_j$. At the node $s_1$, the prongs $v^+_{c_i}$ and $ v^+_{c_j}$ are matched to $v^-_{c_i+Q_1-u+1}$ and $v^-_{c_j+Q_1-u+1}$, respectively. At the node $s_2$, the prongs $w^+_{d_i}$ and $w^+_{d_j}$ are matched to $w^-_{d_i+Q_2-v}$ and $w^-_{d_j-v}$, respectively. Note that $1\leq c_i+Q_1-u+1< c_j+Q_1-u+1 \leq Q_1$ and $1\leq d_j-v < d_i+Q_2-v \leq Q_2$. By the similar argument as in the previous paragraph, $\alpha'_i$ and $\alpha'_j$ intersect transversely at $z$ and the intersection number between $\alpha'_i$ and $\alpha'_j$ is equal to one. Thus $\alpha'_i$ and $\alpha'_j$ are not parallel to each other. 
\end{proof}

The configuration of saddle connections $\overline{X}_s(u,v)$ in various cases is depicted in \Cref{fig602}. The shaded area is $C(\overline{X}_s(u,v))$, the core of this flat surface. The other regions are polar domains of the poles. The definitions of the core and the polar domain are recalled in \Cref{subsec:core}. 

Suppose that $c_{t-1
}<u\leq c_t$. Then in \Cref{fig602}, $\alpha'_{t-1}$ is drawn vertically and $\beta'_n$ is drawn horizontally. Since the smoothing parameter $s$ is small, $\int_{\beta'_n}\omega / \int_{\alpha'_{t-1}} \omega$, the ratio between the periods of saddle connections drawn vertically and horizontally, is close to $-s$. In fact, $-\int_{\beta'_n}\omega / \int_{\alpha'_{t-1}} \omega$ is also a part of smooth coordinate system of the stratum that converges to 0 as $s\to 0$. Therefore, by change of the coordinate, we can set up this ratio as a new smoothing parameter, still denoted by $s$ by abuse of notation. It is immediate that this new parameter can take values in the entire $\CC$. From now on, let $\overline{X}_s(u,v)$ denote the flat surface obtained from $\overline{X}$ with this new smoothing parameter. 

\begin{figure}
    \centering
    \input{diagram602} 
    \caption{The saddle connections of $\overline{X}_s(u,v)$} \label{fig602}
\end{figure}

As a consequence we have the following corollary, which will be crucial for the proof of existence of a multiplicity one saddle connection. In particular, we have a criterion for a two-level multi-scale differential to contain a flat surface with a multiplicity one saddle connection in a coordinate neighborhood. 

\begin{corollary} \label{genus1parallelcor}
Let $\overline{X}=X(t,\tau, {\bf C}, [(u,v)])$ be a two-level multi-scale differential in $\partial\overline{\cR}_1(\mu)$. Assume that there exist $i$ such that $c_{i-1}< u \leq c_i$ and $d_i \leq v < d_{i+1}$, or $c_i< u \leq c_{i+1}$ and $d_{i-1} \leq v < d_i$. Then $\alpha'_i$ is a multiplicity one saddle connection of $\overline{X}_s(u,v)$ for $\operatorname{Im} s\leq 0$. 
\end{corollary}

\begin{proof}
Suppose that $c_{i-1}< u \leq c_i$ and $d_i \leq v < d_{i+1}$. The other case follows from this by relabeling the nodes $s_1,s_2$. By relabeling the saddle connection if necessary, we may assume that $i=0$. If $t=1$, then $\alpha_0$ is a unique saddle connection of $X_0$, thus $\alpha'_0$ is obviously a multiplicity one saddle connection. So suppose $t\geq 2$. By \Cref{genus1parallel}, the saddle connections each $\alpha'_j$ for $1\leq j\leq t-1$ is not parallel to $\alpha'_0$. Therefore $\alpha'_0$ is a multiplicity one saddle connection. 
\end{proof}

From the configurations of saddle connections described in \Cref{fig602}, we can see that there are at most three collections of parallel saddle connections on the flat surface $\overline{X}_s(u,v)$ obtained by plumbing construction. By shrinking $\beta'_i$, we obtain the original two-level multi-scale differential $\overline{X}$. The following lemma describes the result of shrinking other two collections of saddle connection, providing two ways to navigate the multi-scale differentials in the boundary of a given connected component $\cC$. 

\begin{lemma} \label{genus1connect}
For $\overline{X}=X(t,\tau,{\bf C},[(u,v)])\in \partial\overline{\cC}$, we can find two other multi-scale differentials $T_1^{(u,v)}\overline{X}$ and $T_2^{(u,v)}\overline{X}$ in $\partial\overline{\cC}$ given by the following. 

By relabeling the saddle connection, we may assume that $c_{t-1} < u \leq Q_1$. If $d_{j-1}\leq v < d_j$ for some $j$, then $T_1^{(u,v)} \overline{X}=X(t',\tau', {\bf C'}, Pr')$ is given by the following combinatorial data:

\begin{itemize}
    \item $t'=n-(j-1)$ poles on the top level component.
    \item A permutation $\tau'\in Sym_n$ defined by 
\begin{equation*}
    \tau' (i)=
    \begin{cases*}
        i+j-1 & if $1\leq i \leq t'$ \\
        j-(i-t') & if $t'+1\leq i \leq n$
    \end{cases*}
\end{equation*}
    \item A set of integers 

\begin{equation*}
    C'_i =
    \begin{cases*}
      C_i & if $j+1\leq i \leq t-1$ or $t+1\leq i \leq n$ \\
      D_i        & if $1\leq i\leq j-1$ \\
      C_j+v-d_{j-1} & if $i=j$ \\
      u-c_{t-1} & if $i=t$.
    \end{cases*}
\end{equation*}
In particular, $Q'_1 \coloneqq \sum_{i=1}^{t'} C'_i = u+v+c_n-c_t-c_{j-1}-d_{j-1}$ and $Q'_2\coloneqq \sum_{i=1}^{t'} b_{\tau'(i)}-Q'_1 = d_n+c_t-u-v$.
    \item A prong-matching $(v-d_{j-1},d_t-v)\in \ZZ/Q'_1 \ZZ \times \ZZ/Q'_2 \ZZ$.
\end{itemize}

Also, $T_2^{(u,v)} \overline{X}=X(t'',\tau'', {\bf C''}, Pr'')$ is given by the following combinatorial data:

\begin{itemize}
    \item $n-(t-1-j)$ poles on the top level component.
    \item A permutation $\tau''$ defined by

\begin{equation*}
    \tau'' (i)=
    \begin{cases*}
        i & if $1\leq i \leq j$ \\
        n+(j+1-i) & if $j+1\leq i \leq j+(n-t+1) $ \\
        i-(n-t+1) & if $j+(n-t+1)+1 \leq i \leq n$
    \end{cases*}
\end{equation*}
    \item A set of integers 

\begin{equation*}
    C''_i =
    \begin{cases*}
      C_i & if $1\leq i \leq j-1$ or $j+1\leq i \leq t-1$ \\
      D_i        & if $t+1\leq i\leq n$ \\
      C_j+(d_j-v-1) & if $i=j$ \\
      c_t-u+1 & if $i=t$.
    \end{cases*}
\end{equation*}
In particular, $Q''_1 \coloneqq \sum_{i=1}^{t''} C''_i = c_j+d_j+d_n+c_t-d_t-u-v$ and $Q''_2\coloneqq  \sum_{i=1}^{t''} b_{\tau''(i)}-Q''_1 =u+v+c_n+d_t-c_t-c_j-d_j$.
    \item A prong-matching $(c_j-1,-D_t+1)\in \ZZ/Q''_1 \ZZ \times \ZZ/Q''_2 \ZZ$.
\end{itemize}
\end{lemma}

\begin{proof}
Consider the flat surface $\overline{X}_s(u,v)$ for $\operatorname{Im} s\leq 0$. If $c_{k-1} < u \leq c_k$ and $d_{j-1} \leq v < d_j$, then $\overline{X}_s(u,v)$ has three collections of parallel saddle connections $\{\beta'_t,\dots,\beta'_n\}$, $\{\alpha'_0,\dots, \alpha'_{j-1}\}$ and $\{\alpha'_j,\dots, \alpha'_{k-1}\}$. By shrinking the first, we obtain $\overline{X}$ again. By shrinking the second and the third, we obtain $T_1^{(u,v)} \overline{X}$ and $T_2^{(u,v)} \overline{X}$, respectively in $\partial \overline{\cC}$.

More precisely, we have $\overline{X}_s(u,v)=\left(T_1^{(u,v)} \overline{X}\right)_{s+1}(v-d_{j-1},d_t-v)= \left(T_2^{(u,v)}\overline{X}\right)_{-s^{-1}}(c_j-1,-D_t+1)$. The flat surface $\overline{X}_s(u,v)$ is drawn in three different ways in \Cref{fig606}. By shrinking saddle connections horizontally drawn in \Cref{fig606}, we can obtain three different multi-scale differentials $\overline{X}, T_1^{(u,v)}\overline{X}$ and $T_2^{(u,v)} \overline{X}$.

\begin{figure}
    \centering
    \input{diagram606} 
    \caption{Proof of \Cref{genus1connect}} \label{fig606}
\end{figure}

\end{proof}

We can characterize the hyperelliptic connected components of $\cR_1(\mu)$ in terms of the combinatorial data given in \Cref{genus1lemma}. 

\begin{lemma} \label{genus1hyper}
Let $\cC$ be a hyperelliptic component of $\cR_1(\mu)$. Suppose $\overline{X}=X(t,\tau, {\bf C}, Pr)\in \partial\overline{\cC}$. If $(c_i,v)\in Pr$ for some $0\leq i\leq t-1$, then $v=d_j$ for some $0\leq j\leq t-1$. 
Conversely for any connected component $\cC$ of $\cR_1(\mu)$, if every $X(t,\tau, {\bf C}, Pr) \in \partial\overline{\cC}$ satisfies the above condition on $Pr$, then $\cC$ is a hyperelliptic component. 
\end{lemma}

\begin{proof}
By relabeling the poles if necessary, we may assume that $\tau=Id$. The first part easily follows from the existence of the involution $\sigma_0$ on the top level component. Since $\sigma_0$ sends saddle connections to saddle connections, the prong $v^+_{c_i}$ corresponding to the saddle connection $\alpha_i$ must be mapped to $w^+_{d_j}$ for some $j$ by $\sigma_0$. Since the involutions are compatible with $Pr$, we must have $(c_i,d_j)\in Pr$ and $v=d_j$. 

We now prove the converse. First of all, we prove that if $X(t,Id, {\bf C}, Pr)$ satisfies the condition on $Pr$, then the top level component has an involution $\sigma_0$ that interchanges the two nodes $s_1,s_2$. Suppose that $(c_i, d_j)\in Pr$ for some $i,j$. If $C_{i+1}<D_j$, then by the level rotation action, we obtain a prong-matching $(c_{i+1},d_j-C_{i+1})\in Pr$ and $d_{j-1}<d_j-C_{i+1}<d_j$, a contradiction. Similarly if $C_{i+1}>D_j$, a prong-matching $(c_i+D_j, d_{j-1})\in Pr$ gives a contradiction after interchanging the labeling of $s_1$ and $s_2$. So $C_{i+1}=D_j$ and $(c_{i+1},d_{j-1})\in Pr$. By repeating this, we can conclude that $C_{i+k}=D_{j+1-k}$ for any $k$. In particular, $Q_1 = \sum_{i=1}^t C_i = \sum_{j=1}^t D_j = Q_2$ and the orders of zeroes in $X_0$ at the nodes $s_1,s_2$ are equal. 

If $i+j$ is even, then by relabeling the saddle connections if necessary, we may assume $i+j=0\in \ZZ/t\ZZ$ and $(0,0)\in Pr$. Since $C_k=D_{t+1-k}$ for each $k$, there exists an involution $\sigma_0$ on the top level component that interchanges pairs of saddle connections $\alpha_k$ and $\alpha_{t+1-k}$. That is, $\sigma_0$ interchanges the pair of poles $p_k, p_{t+1-k}$, for each $k=1,\dots,t$. If $t$ is odd, the pole $p_{\frac{t+1}{2}}$ is fixed. 

If $i+j$ is odd, then by relabeling the saddle connections if necessary, we may assume $m=1$ and $(0,d_1)\in Pr$. Since $C_k=D_{t+2-k}$ for each $k$, there exists an involution $\sigma_0$ that interchanges the pair of poles $p_k, p_{t+2-k}$, for each $k=1,\dots,t+1$. Therefore in any case, the top level component $X_0$ has an involution that interchanges two nodes $s_1$ and $s_2$. 

If $(0,0)\in Pr$, then we can take $\overline{X'}\coloneqq T_1^{(0,0)} \overline{X} \in \partial\overline{\cC}$. The top level component contains all $n$ poles and has an involution $\sigma'_0$ interchanging two nodes $s'_1,s'_2$. The bottom level component $X'_{-1}$ is contained in the hyperelliptic stratum $\cH_0 (a,-\frac{a}{2}-1,-\frac{a}{2}-1)$, so it obviously has a unique involution $\sigma'_{-1}$ interchanging $s_1,s_2$. Note that $(0,d_t)\in Pr'$ and $C_k=D_{t+1-k}$ for any $k=1,\dots,t$. By level rotation action, we obtain a prong-matching $(c_{\frac{t}{2}},d_{\frac{t}{2}})\in Pr'$ when $t$ is even and $(c_{\frac{t-1}{2}}, d_{\frac{t+1}{2}})\in Pr'$ when $t$ is odd. So the involutions $\sigma'_0$, $\sigma'_{-1}$ are compatible with the prong-matching class $Pr'$ and $\cC$ is hyperelliptic.

If $(0,d_1)\in Pr$, then we can take $\overline{X'}=T_1^{(0,d_1)} \overline{X} \in \partial\overline{\cC}$. The top level component $X'_0$ contains $n-1$ poles and has an involution $\sigma'_0$ interchanging $s_1,s_2$ as above. In this case, the bottom level component $X'_{-1}$ contains one marked pole $p_1$ and is contained in the stratum $\cR_1 (a,-b_1 ;-\frac{a-b_1}{2}-1,-\frac{a-b_1}{2}-1)$. The pole $p_1$ of order $b_1$ is residueless since every prescribed pole is residueless. Note that $C_1=D_1$ by assumption. So $X'_{-1}$ has a unique involution $\sigma'_{-1}$ that interchanges $s_1,s_2$. Note that $(-D_1,d_t)\in Pr'$ and $C_{k+1}=D_{t-k+1}$ for any $k=1,\dots,t-1$. By level rotation action, we obtain a prong-matching $(c_{\frac{t}{2}-1},d_{\frac{t}{2}+1})\in Pr'$. So the involutions $\sigma'_0$, $\sigma'_{-1}$ are compatible with the prong-matching class $Pr'$ and $\cC$ is hyperelliptic.
\end{proof}

\subsection{Existence of a multiplicity one saddle connection --- base case} \label{subsec:simple_base}

We can finally prove \Cref{simple1} for hyperelliptic components of genus one \MIN strata. 

\begin{lemma}
Let $\cC$ be a hyperelliptic component of a genus one \MIN stratum $\cR_1(\mu)$ with less than four fixed marked points. Then $\cC$ contains a flat surface with a multiplicity one saddle connection.
\end{lemma}

\begin{proof}
Assume the contrary --- that $\cC$ does not contain any flat surface with a multiplicity one saddle connection --- and consider $\overline{X}=X(t,\tau, {\bf C}, Pr)\in \partial\overline{\cC}$ as before. The top level component $X_0$ has an involution $\sigma_0$ interchanging two nodes. If $X_0$ does not contain two fixed poles, then there exists $1\leq m\leq t$ such that $(c_m,d_m)\in Pr$ and the flat surface $\overline{X}(c_m,d_m)\in \cC$ contains a multiplicity one saddle connection by \Cref{genus1parallelcor}. So $X_0$ contains two fixed poles. 

By relabeling the poles if necessary, we may assume that $\tau=Id$ and the pole $p_1\in X_0$ is one of the two fixed poles contained in $X_0$. Then $(0,c_1)\in Pr$ and we can take another multi-scale differential $\overline{X'}=T_1^{(0,c_1)} \overline{X}\in \partial\overline{\cC}$, as in \Cref{genus1connect}, so that the pole $p_1$ is now contained in the bottom level component $X'_{-1}$. Since the top level component $X'_0$ still contains two fixed poles by the argument of the previous paragraph, we can conclude that $\cC$ has three fixed marked poles. Since the unique zero $z$ is always fixed by the involution, $\cC$ has four fixed marked points. 
\end{proof}

Finally, we prove \Cref{simple1} for a non-hyperelliptic component $\cC$ by showing that there exists a multi-scale differential in $\partial \overline{\cC}$ satisfying the assumption of \Cref{genus1parallelcor}. 

\begin{proof}[Proof of \Cref{simple1}]
Let $\cC$ be a non-hyperelliptic component of $\cR_1(\mu)$. Assume the contrary --- that $\cC$ does not contain any flat surface with a multiplicity one saddle connection. By \Cref{genus1hyper}, there exists a multi-scale differential $\overline{X}=X(t, \tau, {\bf C}, Pr)\in \partial\overline{\cC}$ with a prong-matching $(u,v)\in Pr$ that satisfies $u=c_i$ and $d_j< v<d_{j+1}$ for some $i,j$. By relabeling the saddle connections if necessary, we may assume that $i=0$. Also by relabeling the poles, we may assume that $\tau=Id$.

If $j=0$, then by \Cref{genus1parallelcor}, $\overline{X}_s(0,v)\in \cC$ for $\operatorname{Im} s\leq 0$ has a multiplicity one saddle connection. Thus we only have to deal with the case $j>0$. We can choose $\overline{X}$ and $(0,v)\in Pr$ such that this $j>0$ is minimal among all possible choices in $\partial\overline{\cC}$. If $t=1$, then $0<j\leq t-1=0$, a contradiction. So $t>1$. 

Suppose that $C_1<v-d_j$. Then by the level rotation action, we have $(c_1,v-C_1)\in Pr$. If $j>1$, then this contradicts to the minimality of $j$ since $d_j< v-C_1 < d_{j+1}$ and $0<j-1<j$. So we have $j=1$ and the flat surface $\overline{X}_s (c_1,v-C_1)$ for $\operatorname{Im} s\leq 0$ has a multiplicity one saddle connection $\alpha'_1$ by \Cref{genus1parallelcor}. Now suppose that $C_1>v-d_j$. Then we have $(v-d_j, d_j)\in Pr$. If $j>1$, then this contradicts to the minimality of $j$ by relabeling the nodes $s_1$ and $s_2$, since $0<v-d_j<c_1$. So we have $j=1$ and the flat surface $\overline{X}_s (v-d_j, d_j)$ for $\operatorname{Im} s\leq 0$ has a multiplicity one saddle connection $\alpha'_1$ by \Cref{genus1parallelcor}. Finally, suppose that $C_1=v-d_j$. Then we have $(c_1,d_j)\in Pr$. If $j=1$, then $\overline{X}_s(c_1,d_1)$ has a multiplicity one saddle connection $\alpha'_1$ by \Cref{genus1parallelcor}. 

Therefore, we may assume $(c_1,d_j)\in Pr$ and $j>1$. By repeating the argument as above, we have $C_i=D_{j+2-i}$ for each $2\leq i <\frac{j}{2}+1$. Also, by relabeling the nodes $s_1$ and $s_2$ and considering the minimality of $j$, we can further obtain the same equation for $\frac{j}{2}+1\leq i \leq j$. Now we take $\overline{X'}=T_2^{(0,v)} \overline{X}\in \partial\overline{\cC}$ as defined in \Cref{genus1connect}. It has a prong-matching $(c_{j+1},-D_1)\in \ZZ/Q'_1 \ZZ \times \ZZ/Q'_2 \ZZ$. If $D_1<C_{j+1}$, then by the level rotation action, we have $(c_j + (C_{j+1}-D_1), 0)\in Pr'$. Since $c'_{j-1}<c_j + (C_{j+1}-D_1)<c'_j$, this contradicts the minimality of $j$ by relabeling the nodes. If $D_1>C_{j+1}$, then we have $(c_j,-(D_1-C_{j+1}))\in Pr'$. Since $-D_{t'}=-D_1<-(D_1-C_{j+1})<0$, this also contradicts the minimality of $j$. Therefore we have $D_1=C_{j+1}$. 

Now we consider the prong-matching $(c_{j+1}, 0)\in Pr$ of $\overline{X}$. We can take $\overline{X''}=T_1^{(c_{j+1}, 0)}\overline{X} \in\partial\overline{\cC}$. By the same argument as above, we can show $C_1=D_{j+1}$. Then $v=d_j+C_1=d_j+D_{j+1}=d_{j+1}$, a contradiction. 
\end{proof}

\subsection{Existence of a multiplicity one saddle connection}

The next step is to prove \Cref{simple} for hyperelliptic components of genus $g>0$ \MIN strata. 

\begin{lemma}\label{simplehyper}
Let $\cC$ be a hyperelliptic component of a \MIN stratum $\cR_g(\mu)$ with less than $2g+2$ fixed marked points. Then $\cC$ contains a flat surface with a multiplicity one saddle connection.
\end{lemma}

\begin{proof}
We use induction on $g>0$. Assume the contrary that any flat surface in $\cC$ has a multiplicity one saddle connection. By \Cref{break}, we can obtain a two-level multi-scale differential $\overline{Y}\in \partial\overline{\cC}$ with two irreducible components $Y_0,Y_{-1}$ intersecting at one node $q$. By \Cref{hyper1}, both $Y_0$ and $Y_{-1}$ are residueless \MIN hyperelliptic flat surfaces. In particular, their genera, which we denote by $g_{-1}$ and $g_0$, satisfy $1\leq g_{-1},g_0\leq g-1$ and $g=g_{-1}+g_0$. If any of these components has a multiplicity one saddle connection, then we obtain a flat surface in $\cC$ with a multiplicity one saddle connection by \Cref{breaksimple}. Therefore we suppose not. By induction hypothesis, $Y_0$ and $Y_{-1}$ have $2g_0+2$ and $2g_{-1}+2$ fixed marked points, respectively. The node $q$ must be fixed by the involutions of both components, so $\cC$ have $(2g_{-1}+1)+(2g_0+1)=2g+2$ fixed marked points. This is a contradiction.  
\end{proof}

We now observe that even for the strata where a flat surface with a multiplicity one saddle connection does not exist, there always exists a flat surface with a pair of saddle connections, which has no other parallel saddle connection. This will be used in the induction step below to deal with the situation where this stratum appears in the boundary of $\overline{\cR}_g(\mu)$.

\begin{proposition} \label{double}
Let $\cC$ be a hyperelliptic component of $\cR_g(\mu)$ of genus $g>0$ and $p_1$ be a pole fixed by the ramification profile of $\cC$. Then there exists a flat surface $X\in \cC$ and a multiplicity {\em two} saddle connections. More precisely, there exist a pair of parallel saddle connections $\gamma_1, \gamma_2$ of $X$ bounding the polar domain of $p_1$, and there does not exist any other saddle connection parallel to them. 
\end{proposition}

\begin{proof}
We use the induction on $\dim_{\mathbb{C}} \cR_g(\mu)=2g+m-1$. First, we take care of the base case --- genus one \MIN strata. First, we need to find $\overline{X}\in \partial\overline{\cC}$ such that $p_1$ is contained in the top level component. Suppose that $p_1\in X_{-1}$ for some $\overline{X}=X(t,\tau, {\bf C}, [(0,v)])\in \partial\overline{\cC}$. Then $d_{i-1}\leq v<d_i$ for some $i$. The multi-scale differential $T_1^{(0,v)}\overline{X}\in \partial\overline{\cC}$ as defined in \Cref{genus1connect} contains $p_1$ in the top level component. 

So we can always find $\overline{X}\in \partial\overline{\cC}$ such that $p_1\in X_0$. By relabeling the saddle connections, we may assume that $\tau(1)=1$. Then $C_1=D_1$ and $(0,D_1)\in Pr$. By \Cref{genus1parallel}, the flat surface $\overline{X}_s(0,D_1)$ for $\operatorname{Im} s\leq 0$ has a desired pair of parallel saddle connections $\alpha'_0$ and $\alpha'_1$. 

Now assume that $\dim_{\mathbb{C}} \cR_g(\mu)>2$. by \Cref{break}, there exists a two-level multi-scale differential $\overline{Y}\in \partial\overline{\cC}$ consisting of two hyperelliptic irreducible components $Y_0,Y_{-1}$ intersecting at one node $q$. The genera $g_0,g_{-1}$ of two components satisfy $g=g_0+g_{-1}$. Also, since the node $q$ must be fixed by hyperelliptic involutions, the top level component $Y_0$ is a \MIN hyperelliptic flat surface. Thus $g_0>0$. 

We will prove that $\cC$ contains $\overline{Y}$ with $p_1\in Y_0$. If $Y_0$ contains $p_1$, then we can use the induction hypothesis on the stratum containing $Y_0$ so that it has a desired pair of saddle connections. Assume the contrary that $Y_{-1}$ contains $p_1$.

If $g_{-1}>0$, then we can use the induction hypothesis on the stratum containing $Y_{-1}$ and deform $Y_{-1}$ so that it has a desired pair of saddle connections. So we may assume that $Y_{-1}$ contains $p_1$ and $g_{-1}=0$. Since there do not exist genus zero \MIN residueless flat surfaces, we automatically have $m=2$.

If $Y_0$ has a multiplicity one saddle connection, then by \Cref{breaksimple}, we can obtain a flat surface $(X,\omega)\in \cC$ with a multiplicity one saddle connection. By shrinking this multiplicity one saddle connection, we obtain another multi-scale differential $\overline{Y'}$ that contains $p_1$ in the top level component. If $Y_0$ does not have a multiplicity one saddle connection, then by \Cref{simplehyper}, $Y_0$ has $2g+2$ fixed marked points. By relabeling the poles, we may assume that $p_2$ is a fixed pole in $Y_0$. By induction hypothesis on the stratum containing $Y_0$, $Y_0$ can be deformed so that it contains a pair of parallel saddle connections bounding the polar domain of $p_2$ and there do not exist any other saddle connections parallel to them. After plumbing the level transition of $\overline{Y}$, we obtain a flat surface in $\cC$ with a pair of saddle connections with the same property. By shrinking them, we obtain another multi-scale differential $\overline{Y'}$ that only contains $p_2$ in the bottom level component. So we can always assume that $p_1\in Y_0$. 

However, since $Y_0$ is a \MIN residueless flat surface, it have the genus $g_0>0$. By induction hypothesis on the stratum containing $Y_0$, $Y_0$ can be deformed so that it has a desired pair of saddle connections bounding the polar domain of $p_1$. By plumbing the level transition of $\overline{Y}$, we obtain a flat surface in $\cC$ that still has the desired pair of saddle connections. 
\end{proof}

We are now ready to prove \Cref{simple} for an arbitrary connected component of $\cR_g(\mu)$.

\begin{proof}[Proof of \Cref{simple}]
The case of hyperelliptic components are given by \Cref{simplehyper}. So we may assume that $\cC$ is non-hyperelliptic component. We use the induction on $\dim_{\mathbb{C}} \cR_g (\mu)\geq 2$ with \Cref{simple1} giving the base case $\dim_{\mathbb{C}} \cR_g(\mu)=2$. Suppose $\dim_{\mathbb{C}} \cR_g (\mu)> 2$. Assume the contrary --- that $\cC$ does not contain any flat surface with a multiplicity one saddle connection joining $z_i$ and $z_j$.

By \Cref{break}, there exists a two-level multi-scale differential $\overline{Y}\in\partial\overline{\cC}$ with two irreducible components $Y_{-1}$ and $Y_0$ of genera $g_{-1}$ and $g_0$, respectively, intersecting at one node $q$. Each component is contained in a residueless stratum with dimension smaller than $\dim_{\mathbb{C}}\cR_g (\mu)$. 

First, we assume that $m=1$ and $z_i=z_j=z_1$. Then both components $Y_0$ and $Y_{-1}$ are residueless \MIN flat surfaces. So $g_0,g_{-1}>0$. If both components are hyperelliptic, then $\cC$ is also hyperelliptic by \Cref{hyper1}. So at least one of $Y_0$ or $Y_{-1}$ is contained in a non-hyperelliptic component. By induction hypothesis, that component can be deformed so that is has a saddle connection. After plumbing construction, we can obtain a flat surface in $\cC$ with a multiplicity one saddle connection by \Cref{breaksimple}. 

If $m>1$, then by relabeling the zeroes, we may assume that $(z_i,z_j)=(z_1,z_2)$. By \Cref{break}, we can further assume that the bottom level component $Y_{-1}$ contains $z_1$ and $z_2$. If $Y_{-1}$ has nonzero genus, then by induction hypothesis $Y_{-1}$ can be continuously deformed to have a multiplicity one saddle connection joining $z_1$ and $z_2$. So we assume that $g_{-1}=0$. If $Y_{-1}$ contains more than two zeroes, then we can repeatedly apply \Cref{break} to $Y_{-1}$ so that the bottom level component contains only two zeroes $z_1$ and $z_2$. By plumbing for all level transitions except the bottom one, we may assume that $Y_{-1}$ only contains two zeroes $z_1$ and $z_2$. 

Now we assume that $m=2$. Then $Y_0$ has a unique zero at the node $q$. Suppose that $Y_0$ contains a multiplicity one saddle connection $\gamma$ joining the node $q$ to itself. By \Cref{suitableprong}, there exists a suitable prong-matching at $q$, so that $\gamma$ deforms to a multiplicity one saddle connection $\gamma'$ joining $z_1$ and $z_2$ after plumbing construction. So we may assume that any continuous deformation of $Y_0$ does not have a multiplicity one saddle connection. By induction hypothesis, $Y_0$ is then contained in a hyperelliptic component with $2g+2$ fixed marked points. By \Cref{double}, we can deform $Y_0$ so that it has a pair of parallel saddle connections $\gamma_1,\gamma_2$ with multiplicity two. Also, they bound the polar domain of a fixed pole, say $p_1$. Therefore, both angles at $q$ in the polar domain bounded by $\gamma_1$ and $\gamma_2$ must be equal to $2\pi \frac{b_1}{2}$. By plumbing the level transition, we obtain a flat surface in $\cC$ with saddle connections $\gamma'_1$ and $\gamma'_2$, deformed from $\gamma_1$ and $\gamma_2$. By \Cref{suitableprong}, we can choose a suitable prong-matching for plumbing so that $\gamma'_1$ is joining $z_1$ and $z_2$. By assumption, $\gamma'_1$ must not be a multiplicity one saddle connection and thus $\gamma'_2$ must be parallel to $\gamma_1$. In particular, $\gamma'_2$ is also joining $z_1$ and $z_2$. By shrinking $\gamma'_1$ and $\gamma'_2$, we obtain a multi-scale differential $\overline{Y'}$ with two components intersecting at one node $q'$, whose bottom level component $Y'_{-1}$ contains only one marked pole $p_1$. The top level component $Y'_0$ has unique zero at $q'$. By the same argument as in the previous paragraph, $Y'_0$ is hyperelliptic flat surface with $2g+2$ fixed marked points. Also, $Y'_{-1}$ is contained in the stratum $\cR_0 (a_1,a_2,-b_1,b_1-a_1-a_2-2)$ and $Y'_{-1}$ has two parallel saddle connections that come from $\gamma_1$ and $\gamma_2$ between $z_1$ and $z_2$. Still the angles of the polar domain of $p_1$ are equal to $2\pi \frac{b_1}{2}$. 

Summarizing the previous paragraphs, we can reduce to the case when $Y_0$ is a \MIN hyperelliptic flat surface with $2g+2$ marked points and $Y_{-1}$ is genus zero flat surface containing two zeroes $z_1,z_2$ and only one marked pole $p_1$. Also, the two angles in the polar domain of $p_1$ are equal to $2\pi \frac{b_1}{2}$. If $a_1=a_2$, then $Y_{-1}$ is a genus zero hyperelliptic flat surface and the node $q$ is fixed by the involution. By \Cref{hyper1}, $\cC$ is then a hyperelliptic component with $2g+2$ fixed marked points, which contradicts the assumption. So we must have $a_1<a_2$. Denote the number of prongs at $q$ by $Q\coloneqq a_1+a_2-b_1+1$. The two angles in the polar domain of $q$ are given by $2\pi Q_1$ at $z_1$ and $2\pi Q_2$ at $z_2$, where $Q_1\coloneqq a_1+1- \frac{b_1}{2}$ and $Q_2 \coloneqq a_2+1-\frac{b_1}{2}$. Consider the incoming prongs $v^-_i$, $i=1,\dots, Q$ at $q$. We can label them in counterclockwise order so that $v^-_1,\dots, v^-_{Q_1-1}$ are coming from $z_1$ and $v^-_{Q_1},\dots, v^-_{Q}$ are coming from $z_2$ (see \Cref{fig603} for the configuration of the prongs at $q$ in $Y_{-1}$).

\begin{figure}
    \centering
    \tikzset{every picture/.style={line width=0.75pt}} 

\begin{tikzpicture}[x=0.75pt,y=0.75pt,yscale=-1,xscale=1]

\draw  [fill={rgb, 255:red, 0; green, 0; blue, 0 }  ,fill opacity=1 ] (197.65,107.28) .. controls (197.65,105.93) and (198.77,104.83) .. (200.14,104.83) .. controls (201.52,104.83) and (202.64,105.93) .. (202.64,107.28) .. controls (202.64,108.63) and (201.52,109.73) .. (200.14,109.73) .. controls (198.77,109.73) and (197.65,108.63) .. (197.65,107.28) -- cycle ;
\draw  [fill={rgb, 255:red, 0; green, 0; blue, 0 }  ,fill opacity=1 ] (435.65,106.28) .. controls (435.65,104.93) and (436.77,103.83) .. (438.14,103.83) .. controls (439.52,103.83) and (440.64,104.93) .. (440.64,106.28) .. controls (440.64,107.63) and (439.52,108.73) .. (438.14,108.73) .. controls (436.77,108.73) and (435.65,107.63) .. (435.65,106.28) -- cycle ;
\draw    (195.82,98.56) -- (186.71,72.36) ;
\draw [shift={(186.06,70.47)}, rotate = 70.83] [color={rgb, 255:red, 0; green, 0; blue, 0 }  ][line width=0.75]    (10.93,-3.29) .. controls (6.95,-1.4) and (3.31,-0.3) .. (0,0) .. controls (3.31,0.3) and (6.95,1.4) .. (10.93,3.29)   ;
\draw    (197.65,116.56) -- (188.69,143.57) ;
\draw [shift={(188.06,145.47)}, rotate = 288.36] [color={rgb, 255:red, 0; green, 0; blue, 0 }  ][line width=0.75]    (10.93,-3.29) .. controls (6.95,-1.4) and (3.31,-0.3) .. (0,0) .. controls (3.31,0.3) and (6.95,1.4) .. (10.93,3.29)   ;
\draw    (442.89,95.73) -- (458.82,75.54) ;
\draw [shift={(460.06,73.97)}, rotate = 128.27] [color={rgb, 255:red, 0; green, 0; blue, 0 }  ][line width=0.75]    (10.93,-3.29) .. controls (6.95,-1.4) and (3.31,-0.3) .. (0,0) .. controls (3.31,0.3) and (6.95,1.4) .. (10.93,3.29)   ;
\draw   (238.5,41.62) .. controls (238.5,40.55) and (239.35,39.68) .. (240.4,39.68) .. controls (241.44,39.68) and (242.3,40.55) .. (242.3,41.62) .. controls (242.3,42.69) and (241.44,43.56) .. (240.4,43.56) .. controls (239.35,43.56) and (238.5,42.69) .. (238.5,41.62) -- cycle ;
\draw    (442.39,113.73) -- (455.56,136.74) ;
\draw [shift={(456.56,138.47)}, rotate = 240.22] [color={rgb, 255:red, 0; green, 0; blue, 0 }  ][line width=0.75]    (10.93,-3.29) .. controls (6.95,-1.4) and (3.31,-0.3) .. (0,0) .. controls (3.31,0.3) and (6.95,1.4) .. (10.93,3.29)   ;
\draw   (310.5,110.95) .. controls (310.5,109.88) and (311.35,109.01) .. (312.4,109.01) .. controls (313.44,109.01) and (314.3,109.88) .. (314.3,110.95) .. controls (314.3,112.03) and (313.44,112.9) .. (312.4,112.9) .. controls (311.35,112.9) and (310.5,112.03) .. (310.5,110.95) -- cycle ;
\draw    (200.14,107.28) .. controls (228.94,177.38) and (409.35,176.38) .. (438.14,106.28) ;
\draw    (200.14,107.28) .. controls (244.89,34.23) and (379.39,33.73) .. (438.14,106.28) ;
\draw   (312.19,49.31) -- (317.18,51.94) -- (312.19,54.58) ;
\draw   (314.39,156.71) -- (319.38,159.34) -- (314.39,161.98) ;

\draw (80.67,91.27) node [anchor=north west][inner sep=0.75pt]    {$Y_{-1}$};
\draw (216.96,106.6) node    {$z_{1}$};
\draw (418.96,108.6) node    {$z_{2}$};
\draw (317.2,173.82) node    {$\gamma _{1}$};
\draw (182.57,153.57) node  [font=\scriptsize]  {$v_{1}^{-}$};
\draw (179.87,60.4) node  [font=\scriptsize]  {$v_{Q_{1} -1}^{-}$};
\draw (469.17,143.7) node  [font=\scriptsize]  {$v_{Q}^{-}$};
\draw (450.64,103.73) node  [rotate=-90] [align=left] {...};
\draw (471.47,63.1) node  [font=\scriptsize]  {$v_{Q_{1}}^{-}$};
\draw (313.55,97.62) node  [font=\normalsize]  {$p_{1}$};
\draw (184.64,105.73) node  [rotate=-267.53] [align=left] {...};
\draw (241.78,27.22) node  [font=\normalsize]  {$q$};
\draw (313.2,32.82) node    {$\gamma _{2}$};

\end{tikzpicture} 
    \caption{The prongs in $Y_{-1}$ at $q$} \label{fig603}
\end{figure}
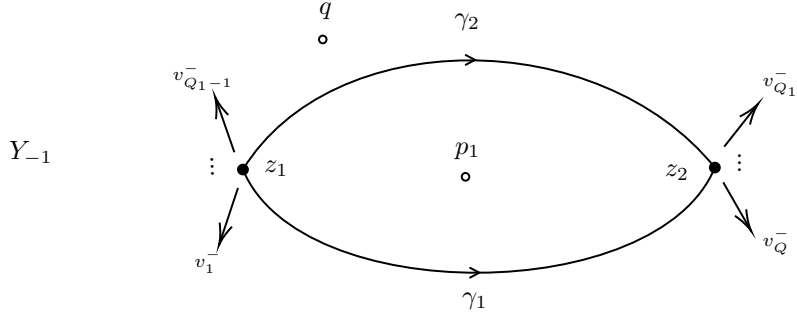

The order of zero of $Y_0$ at $q$ is equal to $Q-1$. By \Cref{double}, we can deform $Y_0$ so that it contains a pair of parallel saddle connections $\alpha_1$ and $\alpha_2$ with multiplicity two, bounding the polar domain of a fixed pole, say $p_2$. First, assume that $Y_0$ is a \MIN hyperelliptic flat surface. Then $q$ is a unique zero of $Y_0$ and $\alpha_1$, $\alpha_2$ are joining $q$ to itself. Since they are parallel, they intersect at $q$ non-transversely. The two angles in the polar domain of $p_2$ bounded by $\alpha_1$ and $\alpha_2$ are equal to $2\pi \frac{b_2}{2}$. Also there are two other angles at $q$ bounded by $\alpha_1$, and $\alpha_2$ themselves. These angles are both equal to $2\pi \frac{Q-b_2}{2}$. We can label the outgoing prongs in $Y_0$ at $q$ in clockwise order, $v^+_i$ for $1\leq i \leq Q-1$ so that $v^+_1,\dots, v^+_{\frac{b_2}{2}}$ are lying between $\alpha_1$ and $\alpha_2$. Then $v^+_{\frac{Q}{2}+1},\dots, v^+_{\frac{Q+b_2}{2}}$ are also lying between $\alpha_2$ and $\alpha_1$, on the other side. The prong-matching at $q$ is determined by the image of $v^-_1$. We can identify the prong-matching that sends $v^-_1$ to $v^+_u$ with $u\in \ZZ/Q\ZZ$ . We will find a proper prong-matching $u$ such that the flat surface $\overline{Y}_s(u)$, obtained by plumbing the level transition with small smoothing parameter $s\in \mathbb{C}$, has a multiplicity one saddle connection. Note that $\frac{Q}{2}=\frac{a_1+a_2-b_1+1}{2} > \frac{2a_1-b_1}{2} =a_1-\frac{b_1}{2}=Q_1-1$. So the set $U\coloneqq\{u|\frac{Q}{2}+1\leq u \leq Q-(Q_1-1)\}$ is nonempty. Moreover, for any $\frac{Q}{2}+1\leq v < Q$, there exists $u\in U$ such that $u\leq v<u+(Q_1-1)$. Since $g_0>0$, we have $Q-1-b_2\geq 2g_0-2\geq 0$. That is, $b_2<Q$ and thus $\frac{Q}{2}+1\leq \frac{Q+b_2}{2} < Q$. So we can take $u\in U$ such that $u\leq \frac{Q+b_2}{2}<u+(Q_1-1)$. Then the flat surface $\overline{Y}(u)$ has a multiplicity one saddle connection $\alpha'_1$, deformed from $\alpha_1$, joining $z_1$ and $z_2$ (see \Cref{fig604}).

\begin{figure}
    \centering
    \input{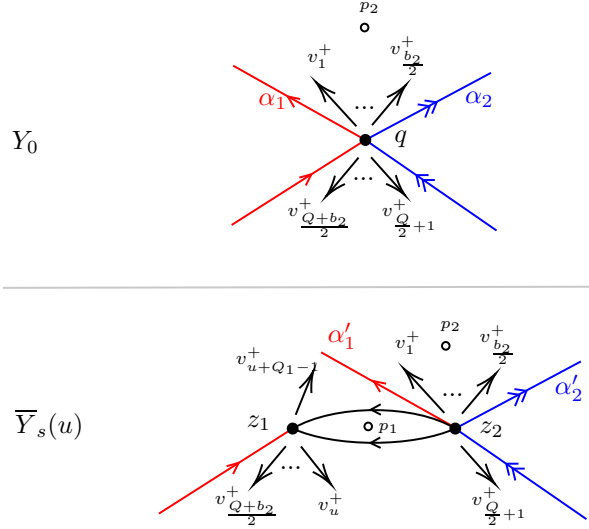} 
    \caption{The prongs at $q$ in \MIN $Y_0$, and after plumbing} \label{fig604}
\end{figure}

Now suppose that $m>2$. In this case, we assume that the bottom level component $Y_{-1}$ of $\overline{Y}$ contains $z_2$ and $z_3$ instead. Then $Y_0$ has at least two zeroes $z_1$ and $q$. If $Y_0$ contains a multiplicity one saddle connection $\gamma$ joining $z_1$ and $q$, then by plumbing with proper choice of prong-matching, $\gamma$ deforms to a multiplicity one saddle connection joining $z_1$ and $z_2$. So we may assume that $Y_0$ cannot be continuously deformed to have a multiplicity one saddle connection joining $z_1$ and $q$. By induction hypothesis, $Y_0$ is \NMIN hyperelliptic flat surface with $2g+2$ fixed marked points. In particular, the order of $q$ is equal to $a_1$. Note that this order is at least $a_2+a_3$, since $Y_{-1}$ is a genus zero flat surface. That is, $a_1\geq a_2+a_3$. Now we repeat this assuming that $Y_{-1}$ contains $z_1$ and $z_3$ instead. In this setting, we also have $a_2\geq a_1+a_3$. This is a contradiction, thus there is a flat surface in $\cC$ with a multiplicity one saddle connection joining $z_1$ and $z_2$. 

\end{proof}

\section{Genus one \MIN strata} \label{sec:g1m}

In this section we classify all non-hyperelliptic connected components of the genus one \MIN stratum $\cR_1 (a,-b_1,\dots,-b_n)$, proving \Cref{main2} for the \MIN cases. Recall that by \Cref{bpb}, every multi-scale differential in the boundary of $\cR_1(\mu)$ is in the principal boundary. As in \Cref{sec:ssc}, we will navigate the principal boundary to prove that the non-hyperelliptic connected components of $\cR_1(\mu)$ are classified by rotation number. 

In \Cref{subsec:pbhc1}, we give a criterion that determines whether a multi-scale differential is contained or not contained in the boundary of a hyperelliptic component. In \Cref{subsec:special}, we first deal with some special strata that our general strategy cannot be applied to. In \Cref{subsec:general}, we finally prove \Cref{main2}. 

In each step, the combinatorial description of multi-scale differentials introduced in \Cref{subsec:pbg1m} will play an important role. Recall that any multi-scale differential in the boundary of $\cR_1 (\mu)$ can be given by the combinatorial datum $X(t,\tau, {\bf C}, Pr)$. This combinatorial datum is uniquely determined only up to the cyclic order of the poles and the markings of the nodes $s_1$ and $s_2$. Throughout this section, we assume that $C_i$ and $b_i$ are indexed by the elements of the cyclic group $\ZZ/n\ZZ$. For example, $C_{n+i}=C_{i}$ for each integer $i$. 

Throughout most of the discussion, we will fix $t=n$ and we will drop it from the notation when no confusion can arise. Also from now on, the bold notation ${\bf X},{\bf X'},\dots$ are reserved for multi-scale differentials of the form ${\bf X}=X(\tau, {\bf C}, Pr)=X(n,\tau, {\bf C}, Pr)$.

\subsection{Rotation number} \label{subsec:rot}

Denote $d\coloneqq \gcd(a,b_1,\dots,b_n)$. Recall from \cite[Definition 4.2]{boissymero} that the {\em rotation number} $r$ of $X\in \cH_1(\mu)$ is defined by the formula $$r \coloneqq \gcd (d, \operatorname{Ind}\alpha, \operatorname{Ind}\beta)$$ where $\alpha, \beta$ are simple closed curves on $X$ which form a symplectic basis of $H_1 (X,\ZZ)$, and the index $\operatorname{Ind}\alpha$ of a closed curve $\alpha$ is defined to be the degree of Gauss map $G_\alpha : S^1 \to S^1$ with respect to the flat structure on $X$. The rotation number is a deformation invariant, thus it is constant in a connected component of $\cH_1 (\mu)$. By \cite[Theorem 4.3]{boissymero}, the connected components of $\cH_1(\mu)$ are indexed by their rotation number. That is, $\cH_1(\mu)=\coprod_{r|d}\cC_r$, where $\cC_r$ is the connected component consisting of flat surfaces of rotation number $r$, except there is no $\cC_d$ for $\cH_1(d,-d)$. 

Note that $\PP\overline{\cR}_1 (\mu)$ is a smooth compactification and the rotation number is a topological invariant of connected components of $\cR_1 (\mu)$. So for a multi-scale differential $\overline{X} \in \partial\overline{\cR}_1 (\mu)$, the rotation number of a flat surface in $\cR_1 (\mu)$ near $\overline{X}$ is constant. We will call this number the {\em rotation number} of $\overline{X}$. We can compute the rotation number of a multi-scale differential in $\partial\overline{\cR}_1 (\mu)$ in terms of its combinatorial data.

\begin{proposition} \label{rotation}
The rotation number of $X(t,\tau, {\bf C}, [(0,v)])\in \partial\overline{\cR}_1 (\mu)$ is equal to $$\gcd(d,Q_1,\sum_{i=1}^n C_{\tau(i)} + v).$$
\end{proposition}

\begin{proof}
By relabeling the poles if necessary, we may assume that $\tau=Id$. By plumbing two nodes of $\overline{X}$ with prong-matching $(0,v)$, we obtain a flat surface $\overline{X}(0,v)\in \cR_1 (\mu)$, whose saddle connections are described by \Cref{saddledegenerate}. Also see \Cref{fig602}. The index of the saddle connection $\alpha'_0$ is equal to $Q_1+Q_2+\sum_{i=t+1}^n D_i-v = Q_1+\sum_{i=1}^n D_i -v$. The index of the saddle connection $\beta'_1$ is equal to $Q_1$. Since two curves $\alpha'_0$ and $\beta'_1$ form a symplectic basis of $\overline{X}(0,v)$, the rotation number of $\overline{X}(0,v)$ is equal to $\gcd(d,\Ind \alpha'_0,\Ind\beta'_1)=\gcd(d,Q_1+\sum_{i=1}^n D_i -v,Q_1)=\gcd(d,Q_1,\sum_{i=1}^n C_i + v)$.
\end{proof}

In particular, if $t=n$, then the rotation number of $\overline{X}=X(Id, {\bf C}, [(0,v)])$ is equal to $\gcd(d,Q_1,v)$.

\subsection{The principal boundary of hyperelliptic components} \label{subsec:pbhc1}

If $\cC$ is a hyperelliptic connected component of $\cR_1(\mu)$, then multi-scale differentials in the boundary of $\cC$ must satisfy properties that come from the hyperelliptic involution. In particular, in case of multi-scale differentials $X(\tau,{\bf C},Pr)$ containing all poles in the top level component, we can determine whether it is contained in the boundary of a hyperelliptic or a non-hyperelliptic component by the following propositions. 

\begin{proposition} \label{genus1hyper1}
Let $\cC$ be a hyperelliptic component of $\cR_1(\mu)$ with the ramification profile $\cP$ fixing two or three marked points. After relabeling the poles and the saddle connections so that $\tau=Id$, $X(Id, {\bf C}, Pr)\in \partial\overline{\cC}$ satisfies the following:

\begin{itemize}
    \item $\cP(i)=-i$. In particular, $b_i = b_{-i}$.
    \item $C_i+C_{-i}=b_i$ for each $i\in \ZZ/n\ZZ$. In particular, $Q_1=Q_2=\frac{a}{2}$.
    \item The prong-matching class $Pr$ is represented by $(0,d_{n-1})\in \ZZ/Q_1 \ZZ \times \ZZ/Q_2 \ZZ$.
    
\end{itemize}

Conversely, if a multi-scale differential $X(Id, {\bf C}, Pr)\in \partial\overline{\cR}_1(\mu)$ is given by the data satisfying the above conditions, it lies in the boundary of some hyperelliptic component of $\cR_1(\mu)$ with ramification profile $\cP$. 
\end{proposition}

\begin{proof}
The top level component $X_0$ has the hyperelliptic involution $\sigma_0$ that interchanges two zeroes. By relabeling the saddle connections on $X_0$, we may assume that $\sigma_0$ fixes the pole $p_{\tau(n)}$. By relabeling the poles, we may assume that $\tau=Id$. For each $i$, $\sigma_0$ sends the saddle connection $\alpha_i$ to $\alpha_{n-1-i}$, so it also sends the pole $p_i$ the pole $p_{n-i}$. Therefore $\cP(i)=-i$. Moreover, the angle $2\pi C_i$ between $\alpha_{i-1}$ and $\alpha_{i}$ at $s_1$ is equal to the angle $2\pi D_{n-i}$ between $\alpha_{n-i-1}$ and $\alpha_{n-i}$ at $s_2$. Therefore, $C_{-i}=b_{-i}-D_{-i}=b_i-C_{i}$. The prong $v^-_1$ is sent to $w^-_{Q_2}$ by $\sigma_{-1}$, and the prong $v^+_0$ is sent to $w^+_{d_{n-1}}$ by $\sigma_0$ as the saddle connection $\alpha_0$ is sent to $\alpha_{-1}$. Since the prong-matching is compatible with the involutions $\sigma_0$ and $\sigma_{-1}$, the prong-matching is represented by $(0,d_{n-1})$. 

Conversely, if $\overline{X}=X(Id, {\bf C}, Pr)$ is given by the data satisfying the conditions in the proposition, then it is immediate that the top level component $X_0$ has an involution $\sigma_0$ compatible with $\cP$. The bottom level component $X_{-1}\in \cH_0(a,-Q-1,-Q-1)$ also has an involution $\sigma_{-1}$, and the prong-matching $Pr$ is compatible with the involutions $\sigma_0$ and $\sigma_{-1}$. So $\overline{X}$ is contained in the boundary of the hyperelliptic component with ramification profile $\cP$ by \Cref{hyper2}.
\end{proof}

\begin{remark}
Under the assumptions of the above proposition, the number of fixed marked points of the component $\cC$ is determined by the parity of $n$. If $n$ is even, then $\cP$ fixes three marked points. If $n$ is odd, then $\cP$ fixes two marked points. 
\end{remark}

\begin{proposition} \label{genus1hyper0}
Let $\cC$ be a hyperelliptic component of $\cR_1(\mu)$ with ramification profile $\cP$ fixing one marked point. After relabeling the poles and the saddle connections so that $\tau=Id$, $X(Id, {\bf C}, Pr)\in \partial\overline{\cC}$ satisfies the following:

\begin{itemize}
    \item $\cP(i+1)=-i$. In particular, $b_{i+1} = b_{-i}$.
    \item $C_{i+1}+C_{-i}=b_i$ for each $i\in \ZZ/n\ZZ$. In particular, $Q_1=Q_2=\frac{a}{2}$.
    \item The prong-matching class is represented by $(0,0)\in \ZZ/Q_1 \ZZ \times \ZZ/Q_2 \ZZ$.
\end{itemize}

Conversely, if a multi-scale differential $X(Id, {\bf C}, Pr)\in \partial\overline{\cR}_1(\mu)$ satisfies above conditions, it lies in the boundary of some hyperelliptic component of $\cR_1(\mu)$ with ramification profile $\cP$. 
\end{proposition}

\begin{proof}
The top level component $X_0$ has the hyperelliptic involution $\sigma_0$ that interchanges two nodes $s_1,s_2$. By relabeling the saddle connections if necessary, we may assume that $\sigma_0$ interchanges two poles $p_{\tau(n)}$ and $p_{\tau(1)}$. The proof follows from the same argument as \Cref{genus1hyper1}.
\end{proof}

\subsection{Connected components of special strata} \label{subsec:special}

In this subsection, we will classify non-hyperelliptic components of certain strata that cannot be dealt with the general strategy which we will follow in the next subsection. They are exactly the \MIN exceptional cases in \Cref{main2}     \smallskip

\paragraph{\bf The stratum ${\bf \cR_1 (2n,-2^n)}$}

Consider the case when the order of every pole is equal to -2. Since $d\coloneqq \gcd(b_i)=2$, there are two possible rotation numbers, $r=1,2$ for flat surfaces in $\cR_1(2n,-2^n)$. We have the following result that makes this case special. 

\begin{proposition} \label{exception0}
{\em Every} flat surface in $\cR_1(2n,-2^n)$ is hyperelliptic. Therefore this stratum has {\em no} non-hyperelliptic components. 
\end{proposition}

\begin{proof}
If $n=1$, then $\cR_1(2,-2)=\cH_1 (2,-2)$, and this stratum is already known to be connected and hyperelliptic by \cite{boissymero}. Suppose $n>1$. Assume the contrary that $\cC$ is a non-hyperelliptic component of $\cR_1(2n,-2^n)$. By \Cref{simple}, we could find a multi-scale differential ${\bf X}=X(\tau, {\bf C}, [(0,v)])\in \partial\overline{\cC}$. Since all $b_i=2$, we have $C_i=1$ for each $i$ and $Q_1=Q_2=n$. By relabeling the poles, we may assume that $\tau=Id$.

Suppose that $n$ is odd. We can apply the level rotation action to obtain a prong-matching of the form $(u,u-1)\in Pr$ as follows. If $v$ is odd, we can simply take $(\frac{v+1}{2},\frac{v-1}{2})\in Pr$. If $v$ is even, then since $n$ is odd, we can take $(\frac{n-b+1}{2},\frac{n-b-1}{2})\in Pr$. Since $C_i=1$ for each $i$, we have $c_i=d_i=i$ and therefore $(u,u-1)=(c_u,d_{u-1})\in Pr$. By relabeling the saddle connections and the poles, we may assume that $(0,n-1)\in Pr$. By \Cref{genus1hyper1}, ${\bf X}$ is contained in the boundary of a hyperelliptic component with two fixed marked points.

Now suppose that $n$ is even. The proof follows from the same argument as above, and by using \Cref{genus1hyper1} or \Cref{genus1hyper0}. If $v$ is odd, then ${\bf X}$ has one fixed marked points. If $v$ is even, then ${\bf X}$ has three fixed marked points. 
\end{proof}

As already mentioned in the proof, the stratum $\cR_1(2,-2)$ is connected. For $n>1$, up to relabeling the poles, there are two ramification profiles of $(2n,-2^n)$, determined by the number of fixed marked points. A ramification profile $\cP$ can fix one or three (resp. two or four) marked points if $n$ is even (resp. odd). In \Cref{sec:chc}, we will prove that hyperelliptic components of strata are classified by ramification profile. Thus $\cR_1(2n,-2^n)$ for $n>1$ has exactly two (hyperelliptic) connected components, up to relabeling the poles. It is easy to see that the rotation number is two if and only if $\cP$ fixes one or three marked points.

\paragraph{\bf The strata ${\bf \cR_1 (2r,-2r)}$ and ${\bf \cR_1 (2r,-r,-r)}$}

In these two cases, the connected components with rotation number $r$ are the exceptions. For rotation numbers other than $r$, we can follow the general strategy in the next subsection.  

\begin{proposition} \label{exception1}
Let $\cC$ be a connected component of $\cR_1 (2r,-2r)$ with rotation number $r$. Then $\cC$ is hyperelliptic.
\end{proposition}

This is an immediate consequence of \cite{boissymero}, since $\cR_1(2r,-2r)=\cH_1(2r,-2r)$ and the rotation number of the hyperelliptic component of $\cH_1(2r,-2r)$ is equal to $r$. We present below another proof more in the spirit of the current paper, using the description of the principal boundary of $\cR_1 (2r,-2r)$.

\begin{proof}
The stratum has two marked points, so by \Cref{simple1}, we can always find a flat surface in $\cC$ with a multiplicity one saddle connection. So there exists a multi-scale differential $X(Id, C_1, Pr)\in \partial\overline{\cC}$. Let $(0,v)\in Pr$ for some $0\leq v\leq D_1-1$. By \Cref{rotation}, we have $r=\gcd(2r,C_1,v)$. In particular, $r$ divides both $C_1$ and $v$. Since $1\leq C_1\leq 2r-1$, we have $C_1=D_1=r$. Since $0\leq v\leq r-1$, we have $v=0$. By \Cref{genus1hyper1}, $\cC$ is then a hyperelliptic component.
\end{proof}

\begin{proposition} \label{exception2}
Let $\cC$ be a connected component of $\cR_1 (2r,-r,-r)$ with rotation number $r$. Then $\cC$ is hyperelliptic.
\end{proposition}

For completeness, we give two proofs, the first using the principal boundary as above, and the second just investigating the geometry of the situation directly.

\begin{proof}[First proof]
The stratum has three marked points, so by \Cref{simple1}, we can always find a flat surface in $\cC$ with a multiplicity one saddle connection. So there exists a multi-scale differential $X(\tau, {\bf C}, Pr)\in \partial\overline{\cC}$. Let $(0,v)\in Pr$ for some $0\leq v\leq D_1+D_2-1$. By relabeling the saddle connections, we may assume that $\tau=Id$. By \Cref{rotation}, we have $r=\gcd(r,C_1+C_2,v)$. In particular, $r$ divides both $C_1+C_2$ and $v$. Since $2\leq C_1+C_2\leq 2r-2$, we have $C_1+C_2=r$. Therefore, $C_1=r-C_2=D_2$ and $C_2=r-C_1=D_1$. Since $0\leq v\leq r-1$, we have $v=0$. By \Cref{genus1hyper0}, $\cC$ is a hyperelliptic component.
\end{proof}

\begin{proof}[Second proof]
In fact, we can see that the connected component $\cH_1(2r,-r,-r)$ with rotation number $r$, which is unique by \cite{boissymero}, is hyperelliptic. Consider the map $\phi_r : \PP \cH_1(2,-1,-1)\to \PP\cH_1(2r,-r,-r)$ defined by $fdz\to f^r dz$. The stratum $\cH_1(2,-1,-1)$ is connected and hyperelliptic. Since $\phi_r$ preserves zeroes and poles of the differential, it also preserves hyperellipticity of flat surfaces. The dimensions of $\PP \cH_1(2,-1,-1)$ and $\PP\cH_1(2r,-r,-r)$ are equal, so the image of $\phi_r$ must be a hyperelliptic component of $\PP\cH_1(2r,-r,-r)$ and the rotation number is equal to $r$. However, by \cite{boissymero}, there exists unique connected component of rotation number $r$. Thus $\cR_1(2r,-r,-r) \subset \cH_1(2r,-r,-r)$ and we can conclude that any connected component of $\cR_1(2r,-r,-r)$ with rotation number $r$ is hyperelliptic. 
\end{proof}

\paragraph{\bf The stratum $\cR_1 (12,-3^4)$}

This is the strangest case, since $\cR_1 (12,-3^4)$ has {\em two} non-hyperelliptic components $\cC_3^1$ and $\cC_3^2$ with rotation number $3$. In order to prove this, we need to describe the projective structure of $\cR_1 (12,-3^4)$ given by the period coordinates in full detail. This will be given in another paper \cite{LT} with G. Tahar. Here we give an upper bound of the number of connected components.

\begin{proposition}\label{special1}
The stratum $\cR_1 (12,-3^4)$ has at most {\em two} non-hyperelliptic components with rotation number $3$.
\end{proposition}

\begin{proof}
Let ${\bf X}=X(\tau,{\bf C},[(0,v)])$ be a two-level multi-scale differential of rotation number 3, not contained in the boundary of some hyperelliptic component. Then $Q_1=\sum_i C_i=6$. So $\{C_i\}=\{1,1,2,2\}$. By relabeling the saddle connections, we may assume that $C_{\tau(1)}=1$. If $C_{\tau(3)}=1$, then ${\bf X}$ is always hyperelliptic, a contradiction. Thus $C_{\tau(3)}=2$. If $C_{\tau(2)}=1$, then we have $v=3$ because otherwise $v=0$ and ${\bf X}$ is hyperelliptic. Similarly if $C_{\tau(2)}=2$, then we have $v=0$. Therefore by relabeling the saddle connections again, we may assume that $C_{\tau(1)}=C_{\tau(2)}=1$, $C_{\tau(3)}=C_{\tau(4)}=2$ and $v=3$. Under this assumption, ${\bf X}$ is only determined by the permutation $\tau$, so we denote it by $X_{\tau}$. 

By relabeling the nodes, we obtain $X_{Id}=X_{(14)(23)}$. Also, we can have $T_2^{(0,2)}T_1^{(3,0)}X_{Id} = X_{(132)}\sim X_{Id}$. By symmetry, we can also have $X_{(423)}\sim X_{Id}$. Therefore, for any $\tau\in Alt_4$, $X_{\tau}\sim X_{Id}$. Thus we can conclude that there are at most two non-hyperelliptic components of $\cR_1(12,-3^4)$ with rotation number $r=3$. 
\end{proof}

\paragraph{\bf The stratum $\cR_1 (2n+2,-2^{n-1},-4)$ for odd $n$ and $r=2$}

In fact, this case satisfies \Cref{main2} although we cannot use the same strategy in the next subsection. So we prove that there exists a unique non-hyperelliptic component $\cC_2$ of rotation number $2$, directly for this case.

\begin{proposition}\label{special2}
The stratum $\cR_1 (2n+2,-2^{n-1},-4)$ for odd $n$ has a unique non-hyperelliptic component with rotation number $2$.
\end{proposition}

\begin{proof}
By relabeling the poles, we may assume $b_n=4$. Let ${\bf X}=X(\tau,{\bf C},[(0,v)])$ be a two-level multi-scale differential of rotation number 2, not contained in the boundary of some hyperelliptic component. Then ${\bf C}=(1,\dots,1,2)$. By relabeling the saddle connections, we may assume that $\tau(n)=n$. If $v=n-1$, then ${\bf X}$ is hyperelliptic, a contradiction. Under this assumption, ${\bf X}$ is determined by $\tau$ and an even number $0\leq v<n-1$, so we denote it by $X_{\tau,v}$. We have $T_1^{(2,n-v)} T_1^{(v,1)}T_2^{(-1,v+1)} X_{Id,v} = X_{\tau,0}$ for $\tau=(12\dots v+1)$. 

Note that $T_1^{(\frac{v}{2}-1,\frac{v}{2}+1)} X_{Id,v} = T_2^{(\frac{v}{2}+1,\frac{v}{2}-1)} X_{(\frac{v}{2}\frac{v}{2}+1),v}$. Therefore we have $X_{\tau,0}= X_{Id,0}$ for each $\tau=(i,i+1)$, $1\leq i\leq n-2$. Thus $\cR_1(2n+2,-4,-2^{n-1})$ has a unique non-hyperelliptic component with rotation number $2$.
\end{proof}

We call the strata introduced in this subsection by {\em special} strata. They are the only exceptions for the proof of existence and uniqueness of non-hyperelliptic component in the next subsection. Three non-hyperelliptic components introduced in this subsection ($\cC_3^1,\cC_3^2$ of $\cR_1(12,-3^4)$ and $\cC_2$ of $\cR_1(2n+2,-2^{n-1},-4)$ for odd $n$) are called {\em special} connected components. 

\subsection{Classification of non-hyperelliptic components} \label{subsec:general}

Now we will prove \Cref{main2} for \MIN strata. Throughout this subsection, we suppose that $\cR_1(\mu)$ is none of the strata that is dealt with in \Cref{subsec:special}. Suppose that $\cC$ is a non-hyperelliptic component of $\cR_1(\mu)$. By \Cref{simple}, we can always find a flat surface $X\in \cC$ with a multiplicity one saddle connection. By shrinking the multiplicity one saddle connection, we can obtain a multi-scale differential ${\bf X}=X(\tau, {\bf C}, Pr)\in \partial\overline{\cC}$ by \Cref{shrink}. Conversely, any multi-scale differential of the form $X(\tau, {\bf C}, Pr)\in \partial\overline{\cC}$ can be obtained by shrinking a multiplicity one saddle connection of a flat surface in $\cC$. 

Let $r|d$ where $d=\gcd(b_1,\dots,b_n)$ as usual. We will explicitly construct a multi-scale differential in $\partial\cR_1(\mu)$ with rotation number $r$ that is {\em not} contained in the boundary of hyperelliptic components. This proves the {\em existence} of a non-hyperelliptic component with rotation number $r$. 

\begin{proposition} \label{nonhyperexist}
Let $\cR_1(\mu)$ be a genus one \MIN stratum, not one of special strata treated in \Cref{subsec:special}. Suppose that $n>1$ and $r|d$. There exists a multi-scale differential $X(\tau, {\bf C}, Pr)\in \partial\overline{\cR}_1(\mu)$ with rotation number $r$, that is not contained in the boundary of any hyperelliptic component. 
\end{proposition}

\begin{proof}
It suffices to show that there exists a combinatorial data ${\bf C}=(C_1,\dots,C_n)$, $1\leq C_i\leq b_i-1$, such that $Q_1=\sum_i C_i <\frac{a}{2}$ and $r|Q_1$. Then we can construct a multi-scale differential $X(Id,{\bf C},[(0,r)])$ with rotation number $r$. Since $Q_1\neq Q_2=a-Q_1$, by \Cref{hyper2}, ${\bf X}$ is not contained in the boundary of any hyperelliptic component. Given a number $n\leq Q\leq a-n$, we can always find ${\bf C}$ such that $\sum_i C_i=Q$. So we need to find $Q$ satisfying $n\leq Q\leq a-n$, $Q< \frac{a}{2}$ and $r|Q$.

Assume that $\frac{a}{r}$ is odd and take $Q\coloneqq \frac{a-r}{2}=r\left(\frac{a}{r}-1\right)/2$. It is sufficient to prove $n\leq Q$, because then $Q\leq 2Q-n=a-r-n<a-n$ follows. First, suppose that $r>2$. Then $a=\sum_i b_i \geq rn$ since $r|b_i$ for each $i$. So $Q\geq (n-1)\frac{r}{2}\geq n$ since $n>1$. Now suppose that $r=2$. Since $b_n>2$ and $2|b_i$ for each $i$, we have $a\geq 2n+2$ and thus $Q\geq \frac{2n+2-2}{2}= n$. Finally when $r=1$, then again $a\geq 2n+1$ and $Q\geq \frac{2n+1-1}{2}=n$. 

Assume that $\frac{a}{r}$ is even and take $Q\coloneqq \frac{a}{2}-r = r\left(\frac{a}{r}-2\right)/2$. Again, we need to prove $n\leq Q$. First, suppose that $r>3$. Since $r|b_i$ for each $i$, we have $a\geq rn$. So $Q\geq (n-2)\frac{r}{2} \geq n$ if $n>3$. If $n\leq 3$, then $a\geq 4r$ since $\frac{a}{r}$ is even and $\mu=(2r,-r,-r)$ is excluded. Thus $Q\geq r>n$. Now suppose that $r=1$. Then $a$ is even and thus $a=\sum_i b_i\geq 2n+2$ since $b_n>2$. So $Q\geq \frac{2n+2}{2}-1=n$. 

Suppose that $r=3$, and assume the contrary that $Q<n$. Then $\frac{3n-6}{2}\leq Q<n$ and thus $n<6$. If $n=5$, then since $\frac{a}{3}$ is even, we have $a\leq 18$ and thus $Q\geq 5$. If $n=3$, then similarly $a\geq 12$ and $Q\neq 3$. If $n\leq 2$, then since $\mu\neq (6,-6)$ or $(6,-3,-3)$, we have $a\geq 12$ and $Q\geq 2$. So we have $n=4$ and $3\leq Q<n$, thus $a=12$ and $\mu=(12,-3^4)$, which is excluded by assumption. 

Finally, suppose that $r=2$, and assume the contrary that $Q<n$. Since $\mu\neq (2n,-2^n)$, we have $a\geq 2n+1$. Then $\frac{2n-3}{2}\leq Q<n$ and thus $Q=n-1$. Since $2|Q$, $n$ is odd. Thus $a=2n+2$ and $\mu=(2n+2,-2^{n-1},-4)$ for odd $n$, which is excluded by assumption.   
\end{proof}

It only remains to prove the {\em uniqueness} of non-hyperelliptic component of $\cR_1(\mu)$ with given rotation number $r|d$.

\begin{theorem} \label{main2minimal}
Let $\cR_1(\mu)$ be a genus one \MIN stratum, not one of special strata treated in \Cref{subsec:special}. Suppose that $n>1$ and $r|d$. There exists a {\em unique} non-hyperelliptic component $\cC_r$ of $\cR_1(\mu)$ with rotation number $r$. 
\end{theorem}

\begin{proof}
By \Cref{nonhyperexist}, there exists a non-hyperelliptic component of $\cR_1(\mu)$ with rotation number $r$. By \Cref{simple}, we can find a multi-scale differential ${\bf X}=X(\tau, {\bf C},Pr)\in \partial\overline{\cR}_1(\mu)$. We denote by ${\bf X} \sim {\bf X'}$ if ${\bf X}$ and ${\bf X'}\coloneqq X(\tau', {\bf C'},Pr')$ are contained in the boundary of the same connected component of $\cR_1(\mu)$. Our goal is to prove that ${\bf X} \sim {\bf X'}$ for {\em any} pair ${\bf X},{\bf X'}$ of non-hyperelliptic multi-scale differentials with rotation number $r$. To this end, we need to be able to show that all combinatorial data of a multi-scale differential can be changed to another one within the connected component, and we do it step by step as follows to achieve the goal. 

The following statements holds for non-hyperelliptic multi-scale differentials ${\bf X}$ and ${\bf X'}$ with rotation number $r$.
\begin{itemize}
    \item ${\bf X} \sim {\bf X''}$ for some ${\bf X''}$ satisfying $0<Q''_2-Q''_1$. See \Cref{unbalance}.
    \item ${\bf X} \sim {\bf X'''}$ for some ${\bf X'''}$ satisfying $0<Q'''_2-Q'''_1\leq 2r$. See \Cref{Qreduction}.
    \item ${\bf X} \sim {\bf X'}$ whenever $\tau=\tau'$, $Q_1=Q'_1$ and $Pr=Pr'$. See \Cref{Cchoice}.
    \item ${\bf X} \sim {\bf X'}$ whenever $\tau=\tau'$. See \Cref{Pchoice}.
    \item ${\bf X} \sim {\bf X'}$. See \Cref{Tchoice}.
\end{itemize}
\end{proof}

The rest of this section is devoted to proving various ingredients that we need to prove each step for \Cref{main2minimal}. We fix ${\bf X}=X(\tau, {\bf C},Pr)\in \partial \overline{\cC}$ as in the statement, and use the operators $T_1,T_2$ defined in \Cref{sec:ssc} to navigate in the boundary of $\cC$. We will also denote ${\bf X'}, {\bf X''}$, and $\overline{X'},\overline{X''}$, etc for other elements of $\partial \overline{\cC}$ and write their combinatorial data correspondingly. 

Following two lemmas provide useful tool to connect multi-scale differentials in $\partial\overline{\cC}$ given by combinatorial data. 

\begin{lemma} \label{move12}
Let $\overline{X}=X(t,Id,{\bf C},[(u,v)])$. Suppose that $c_{t-1}< u\leq c_t$ and $d_{j-1}\leq v < d_j$.

If $D_t,C_j\geq 2$, let $\overline{X'}=X(t,Id, {\bf C'},[(u,v)])$ be a multi-scale differential given by 
\begin{equation*}
    C'_i =
    \begin{cases*}
    C_t+1         & if $i=t$ \\
    C_j-1         & if $i=j$ \\
    C_i           & otherwise.
    \end{cases*}
\end{equation*}
Then $\overline{X'}\sim \overline{X}$. 

Similarly if $C_t,D_j\geq 2$, let $\overline{X''}=X(t,Id, {\bf C''},Pr)$ be a multi-scale differential given by 
\begin{equation*}
    C''_i =
    \begin{cases*}
      C_t-1         & if $i=t$ \\
      C_j+1         & if $i=j$ \\
      C_i           & otherwise 
    \end{cases*}
\end{equation*}
Then $\overline{X''}\sim \overline{X}$. 
\end{lemma}

In other words, given a prong-matching satisfying the conditions above, we can increase (decrease) $C_t$ and decrease (increase, resp) $C_j$ by one, while any other data are unchanged, in the same connected component. 

\begin{proof}
If $D_t,C_j\geq 2$, it is straightforward to check that $T_1^{(u-1,v+1)} \overline{X'} = T_1^{(u,v)} \overline{X}$. So $\overline{X'}\sim \overline{X}$. Also if $C_t,D_j\geq 2$, we have $T_2^{(u+1,v-1)} \overline{X''}=T_2^{(u,v)} \overline{X}$. So $\overline{X''}\sim \overline{X}$.
\end{proof}

\begin{lemma} \label{insert}
Let $\overline{X}=X(t,Id,{\bf C},[(0,v)])$ for some $t<n$ and $d_{j-1}<v<d_j$ for some $j$. Then there exists $\overline{X'}=X(t+1,Id,{\bf C'},Pr')\sim \overline{X}$ such that $Q'_1=Q_1+b_{t+1}$ and $Q'_2=Q_2$. 
\end{lemma}

\begin{proof}
We can take $\overline{X'}=T_2^{(v-d_{j-1}-1,d_t-v+1)} T_1^{(0,v)} \overline{X}$ and it is straightforward to see that $\overline{X'}$ satisfies the conditions.
\end{proof}

\begin{lemma} \label{insertunbalance}
Let $\overline{X}$ satisfy $Q_1<Q_2$. Then $\overline{X}$ has a prong-matching $(c_i,v)\in Pr$ for some $i$, such that $v\neq d_j$ for any $j$.
\end{lemma}

\begin{proof}
Let $S\coloneqq \{(c_i,d_j)\in Pr|i,j=1,\dots,n\}$ and $q\coloneqq \gcd(Q_1,Q_2)$. For each $d_j$, there are $\frac{Q_1}{q}$ prong-matching given by $(u,d_j)\in Pr$ for some $u$. So $|S|\leq n \frac{Q_1}{q}$. If there exists no prong-matching $(c_i,v)\in Pr$ such that $d_j<v<d_{j+1}$, then $|S| = n \frac{Q_2}{q}$. However, since $Q_1<Q_2$, we have $|S|\leq n \frac{Q_1}{q} < n\frac{Q_2}{q}=|S|$, a contradiction. 
\end{proof}

We are now ready to prove the first step of the proof in \Cref{main2minimal}.

\begin{proposition} \label{unbalance}
Let $\cC$ be a non-hyperelliptic component of a genus one \MIN stratum $\cR_1 (\mu)$, not one of special strata treated in \Cref{subsec:special}. Then there exists a multi-scale differential $X(\tau, {\bf C},Pr)\in \partial\overline{\cC}$ such that $Q_1 < Q_2$. 
\end{proposition}

\begin{proof}
Assume the contrary --- that every ${\bf X}=X(\tau,{\bf C},Pr)\in \partial\overline{\cC}$ satisfies $Q_1=Q_2=\frac{a}{2}$. Fix a multi-scale differential ${\bf X}$. We will navigate the boundary of $\cC$ using ${\bf X}$ to show that $\cR_1(\mu)$ is one of the strata dealt in \Cref{subsec:special}, a contradiction. We can find a prong-matching $(c_i,d_i+v)\in Pr$, $v\neq 0$. Otherwise ${\bf X}$ is hyperelliptic by \Cref{genus1hyper}, a contradiction. We can further assume that $(c_i,d_i+v)$ is chosen so that $|v|>0$ is minimal among all such prong-matchings in $Pr$. By relabeling the poles, the nodes and the saddle connections, we may assume that $\tau=Id$, $i=0$ and $v>0$. So $(0,v)\in Pr$. If $v>b_1$, then $(c_1,d_1+(v-b_1))\in Pr$ with $0<v-b_1<v$. This contradicts the minimality of $|v|$, so $v\leq b_1$. Then $(c_1,d_1-(b_1-v))\in Pr$ with $0\leq b_1-v$. If $v<b_1$, then $v\leq b_1-v$ by the minimality of $|v|$, thus $v\leq \frac{b_1}{2}$. 

First, suppose that $v\leq \frac{b_1}{2}$. If $v<D_1$, then we can take ${\bf X'}=T_1^{(0,v)} {\bf X}\in \partial\overline{\cC}$. This satisfies $C'_1=C_1+v$ and $C'_i=C_i$ for any $i\neq 1$. Therefore, $Q'_1-Q'_2=(Q_1+v)-(Q_2-v)=2v>0$. Similarly, if $v<C_1$, we can take ${\bf X'}=T_1^{(v,0)} {\bf X}\in \partial\overline{\cC}$, satisfying $Q'_2-Q'_1=2v>0$. This contradicts to the assumption, so $v=C_1=D_1=\frac{b_1}{2}$. If $D_i=C_{2-i}$ for each $i=1,\dots,n$, then $\cC$ is hyperelliptic by \Cref{genus1hyper}, a contradiction. Let $i>1$ be the smallest such that $C_i\neq D_{2-i}$ or $D_i\neq C_{2-i}$. Since $c_{i-1}=\sum_{j=1}^{i-1} C_j=\sum_{j=1}^{i-1} D_{2-j}=Q_2-d_{n+2-i}+D_1$, we have $(c_{i-1},d_{n+2-i})\in Pr$. We take $\overline{X'}=T_2^{(c_{i-1},d_{n+2-i})} {\bf X}\in \partial\overline{\cC}$, satisfying $Q'_1=Q'_2=\frac{a}{2}-c_{i-1}$ and $(0,0)\in Pr'$. The bottom level component $X'_{-1}$ has $2i-3$ marked poles $p_{n-i+3},\dots, p_n, p_1,\dots, p_{i-1}$. Suppose $C_i<D_{2-i}$. The other possible cases ($C_i>D_{2-i}$, $D_i<C_{2-i}$, or $D_i>C_{2-i}$) can be treated in a similar way. We have a prong-matching $(C_i-1,-C_i+1)\in Pr'$. If $C_i>1$, then can apply \Cref{move12} to reduce $C_i$ and increase $C_{2-i}$ by one. So we may assume $C_i=1$. Then we take $\overline{X''}=T_2^{(1,-1)} \overline{X'}$ satisfying $Q''_1=\frac{a}{2}$, $Q''_2=\frac{a}{2}-b_i$. The bottom level component $X''_{-1}$ has only one marked pole $p_i$. After relabeling the saddle connections so that $\tau''(1)=1$, we have $(C_1-c_{i-1}-1, c_{i-1})\in Pr''$. Since $C_j=D_{2-j}$ and $D_j=C_{2-j}$ for all $j<i$, we have $(0,C_1-1)\in Pr''$. The process of obtaining $\overline{X'}$ and $\overline{X''}$ are depicted in each row of \Cref{fig701}. 

If $b_1>2$, then $C_1=\frac{b_1}{2}>1$, thus $0<C_1-1<c_1$. By \Cref{insert}, we can obtain some ${\bf X'''}\in \partial\overline{\cC}$ with $Q'''_1- Q'''_2=2b_i>0$, a contradiction. If $b_1=2$, then $C_1=1$ and $(0,0)\in Pr''$. In this case, we take ${\bf X'''}=T_1^{(0,0)}\overline{X''}$ satisfying $Q'''_1-Q'''_2=2>0$, a contradiction. The process of obtaining ${\bf X'''}$ for two cases are depicted in each row of \Cref{fig702}.

\begin{figure}
    \centering
    \input{diagram701} 
    \caption{Proof of \Cref{unbalance}, part I} \label{fig701}
\end{figure}

\begin{figure}
    \centering
    \input{diagram702} 
    \caption{Proof of \Cref{unbalance}, part II} \label{fig702}
\end{figure}

Now we suppose that $v=b_1$. Then $(c_2,d_2-b_2)\in Pr$. By the minimality of $|v|$, we have $b_1=v\leq b_2$. Suppose that $C_1\geq D_2$ and $D_1\geq C_2$. Then $b_1=C_1+D_1\geq C_2+D_2=b_2$, thus $b_1=b_2$. That is, $C_1=D_2$ and $D_1=C_2$. If $C_i=D_{3-i}$ for each $i=1,\dots,n$, then $\cC$ is hyperelliptic by \Cref{genus1hyper}, a contradiction. Let $i>2$ be the smallest such that $C_i\neq D_{3-i}$ or $D_i\neq C_{3-i}$. We can repeat the argument in the previous paragraph, to get a contradiction. So we suppose $C_1<D_2$ or $D_1<C_2$ holds. Suppose $C_1<D_2$. The other possible case $D_1<C_2$ can be treated in a similar way. If $C_1>1$, we can apply \Cref{move12} with the prong-matching $(c_1-1,d_1+1)\in Pr$ to reduce $C_1$ and increase $C_2$ by one. So we may assume $C_1=1$. Since $d_1<b_1=C_1+D_1<d_2$, we have a prong-matching $(0,v)\in Pr$ such that $d_1<v<d_2$. We take $\overline{X'} = T_2^{(0,v)} {\bf X}$ satisfying $Q'_1=b_2$, $Q'_2=b_1$ and $(c_2, 0)\in Pr'$. The top level component $X'_0$ contains only two marked poles $p_1,p_2$. Assume that $q'\coloneqq \gcd(b_1,b_2)<b_1$. If $b_1>2$, then $q'\leq \frac{b_1}{2} < b_1-1$. By the level rotation action, we have $(c_2,q')\in Pr'$ with $0< q' < d'_1=b_1-1$. We take ${\bf X''}=T_2^{(c_2,q')} \overline{X'}$. If $b_1=2$, then $C_1=D_1=1$, $q'=1$ and $b_2>2$. Therefore, $(c_2+1,0)\in Pr'$ and we take ${\bf X''}=T_2^{(c_2,q')} \overline{X'}$. In both cases, ${\bf X''}$ satisfies $Q''_2=Q_2-q'$ and $Q''_1=Q_1+q'$, a contradiction. Therefore $\gcd(b_1,b_2)=b_1$, and thus $b_1|b_2$. We take $\overline{X'''}=T_1^{(0,v)} {\bf X} \in \partial\overline{\cC}$, satisfying $Q'''_1=Q_1$, $Q'''_2=Q_2-b_1$ and $(1,0)\in Pr'''$. The bottom level component $X'''_{-1}$ contains only one marked pole $p_1$. Suppose that $D_2>2$. Then $d'''_1=D_2-1>1$. Consider a prong-matching $(0,1)\in Pr'''$. By \Cref{insert}, there exists ${\bf X''''}\in \partial\overline{\cC}$ with $Q''''_2-Q''''_1=2b_1>0$, a contradiction. So we reduce to the case when $D_2=2$ and $C_1=1$. In particular, $D_2-C_1=1$.

Since we have $b_1|b_2$, $C_2=b_2-2\geq b_1-2=D_1-1$. If $C_2=D_1$, then $b_2=b_1+1$, thus $b_1=1$. This is a contradiction. So $C_2>D_1$ or $C_2=D_1-1$. If $C_2>D_1$, then by the same argument as in the previous paragraph, we can deduce that $C_2-D_1=1$. So $b_2=b_1+2$, thus $b_1=2$ and $b_2=4$. If $C_2=D_1-1$, then $b_1=b_2$. We will prove that $b_i=b_1$ for $i=3,\dots,n$ in both cases. Consider again $\overline{X'''}$ in the previous paragraph. Note that $\gcd(Q'''_1,Q'''_2)=\gcd(b_1,\frac{a}{2})$ divides $b_1$. So by the level rotation action, we obtain $(c'''_1+2,0)\in Pr'''$ since $c'''_1+1=b_2$ is divisible by $b_1$. If $D_3>1$, then $d'''_2=1+D_3>2$ and $(c'''_1,2)\in Pr'''$. By \Cref{insert}, we obtain ${\bf X''''}$ satisfying $Q''''_2-Q''''_1=2b_1>0$. So $D_3=1$ and $d'''_2=2$. Consider a prong-matching $(c'''_1,d'''_2)\in Pr'''$. We take $\overline{X'''''}=T_2^{(c'''_1,d'''_2)} \overline{X'''}$, satisfying $Q'''''_1=b_3$, $Q'''''_2=b_1$. The top level component $X'''''_0$ contains only two marked poles $p_1$ and $p_3$. By repeating the same argument applied to $\overline{X'}$ as in the above, we can deduce that $b_3=b_1$ and there exists a prong-matching $(c'''_2,3)\in Pr'''$. By repeating this, we can conclude that $b_i=b_1$ and $D_i=1$ for each $i=3,\dots, n$. 

If $b_1>2$, then $b_i=b_1$ for each $i$ and $\mu=(nb_1, -b_1^n)$. In this case $Q'_1=(n-1)(b_1-1)$ and $Q'_2=n-1$. Since $Q'_1-Q'_2=b_1$, we have $b_1=\frac{2(n-1)}{n-2})$. The only integer solution for this equation is $n=4$, $b_1=3$. Therefore, $\mu=(12, -3^4)$ and the rotation number of $\cC$ is equal to $\gcd(d,Q_1,v)=3$. This is a contradiction. 

If $b_1=2$ and $b_2=4$, we have $\mu=(2n+2,-2^{n-1},-4)$. Obviously we have $c_1=d_1=1$, $c_2=d_2=3$ and $c_i=d_i=i+1$ for all $i>2$. If $n$ is even, then $Q_1=Q_2=n+1$ is odd. We have a prong-matching $(\frac{n}{2}+1,\frac{n}{2}+2)\in Pr$ by the level rotation action. This contradicts the minimality of $|v|$, thus $n$ is odd and the rotation number of $\cC$ is equal to $\gcd(d,Q_1,v)=2$. This is a contradiction.
\end{proof}

To proceed to the second step in the proof of \Cref{main2minimal}, we need the following lemma.

\begin{lemma} \label{Qlemma}
Let ${\bf X} \in \partial\overline{\cC}$ and $Q_2-Q_1>0$. Choose a prong-matching $(c_i,d_i+v)\in Pr$ with minimal $|v|>0$. If $|v|<\frac{Q_2-Q_1}{2}$, then there exists ${\bf X'}\in \partial\overline{\cC}$ such that $0\leq Q'_2-Q'_1=Q_2-Q_1-2|v|$.
\end{lemma}

This lemma states that we can reduce $Q_2-Q_1>0$ by $2|v|$ whenever we have $|v|<\frac{Q_2-Q_1}{2}$.

\begin{proof}
Assume that $v>0$. The case $v<0$ can be proven in a symmetric way. By relabeling the poles and the saddle connections, we may assume that ${\bf X}=X(Id,{\bf C},[(0,v)])$. By the level rotation action, $(c_1,d_1+(v-b_1))\in Pr$. By the minimality of $|v|$, we have $v\leq b_1$. 

First, suppose that $v=b_1$. Then $(c_1,d_1)\in Pr$. Also $(c_2,d_2-b_2)\in Pr$. By minimality of $|v|$, we have $v=b_1\leq b_2$. Assume that $C_1<D_2$. Then $d_1<v<d_2$ and we can take $\overline{X''}=T_1^{(0,v)} {\bf X}$, satisfying $Q''_2-Q''_1=v$. The bottom level component $X''_{-1}$ contains only one pole $p_1$. By \Cref{insertunbalance} and \Cref{insert}, we can obtain ${\bf X'}\in \partial\overline{\cC}$ with $Q'_2-Q'_1=Q_2-Q_1-2v$. This is the end of the proof, so we may assume that $C_1\geq D_2$. By \Cref{move12} with $(c_1-1,d_1+1)\in Pr$, if $D_2>1$, then we can reduce $D_2$ and increase $D_1$ by one. So we may assume $D_2=1$. If $C_1=1$, then $(0,d_2)\in Pr$. We take $\overline{X''}=T_1^{(0,d_2)}{\bf X}$, satisfying $Q''_1=Q_1-b_2$, $Q''_2=Q_2-b_1$. The bottom level component $X''_{-1}$ contains two marked poles $p_1,p_2$. Since $Q''_2-Q''_1=Q_2-Q_1+(b_2-b_1)>0$, we can apply \Cref{insertunbalance} and \Cref{insert} to obtain a multi-scale differential $\overline{X'''} \in \partial\overline{\cC}$, satisfying $Q'''_1=Q_1$, $Q'''_2=Q_2-b_1$. The bottom level component $X'''_{-1}$ contains only one marked pole $p_1$. Again, since $Q'''_2-Q'''_1=Q_2-Q_1-v>0$, we can apply \Cref{insertunbalance} and \Cref{insert} and obtain a desired multi-scale differential ${\bf X'}\in \partial\overline{\cC}$, satisfying $Q'_2-Q'_1=Q_2-Q_1-2v$. Finally, we reduced to the case when $D_2=1$ and $C_1>1$.

We use the induction on $C_1>1$. Consider a prong-matching $(c_1-2, d_2+1)\in Pr$. If $D_i=1$ for each $i=2,\dots,n$, then $Q_2-Q_1\leq d_2-c_2=D_1+1-C_1-C_2=(D_1+2-C_1)-b_2\leq b_1-b_2\leq 0$, a contradiction. Let $i$ be the smallest number larger than $2$ such that $D_i>1$. We use the induction on $i$ to reduce $C_1$. Consider a prong-matching $(c_1-(i-2),d_1+(i-2))=(c_1+2-i,d_{i-1})\in Pr$. There exists $j\leq i-4$ that $c_{n-j-1}\leq c_1+2-i<c_{n-j}$. If $D_{n-j}>1$, then we can use \Cref{move12} to increase $C_{n-j}$ and decrease $C_{i-1}$ by one. So we assume that $D_k=1$ for all $k=n-j,\dots,n$. If $c_{n-j-1}<c_1+2-i$, then $c_{n-j-1}\leq c_1+1-i$. So by \Cref{move12}, we can decrease $C_{n-j}$ and increase $C_i$ by one. That is, we again have $D_{n-j}>1$. So we assume that $c_{n-j-1}=c_1+2-i$. Now consider $(v+j+1,n-j-1)=(v+j+1,d_{n-j-1})\in Pr$. Since $c_2=C_1+C_2=C_1+b_2-1\geq C_1+b_1-1>b_1=v$, we have $c_2< v+j+1<c_{j+2}\leq c_{i-2}$. There exists $2<k<i$ such that $c_{k-1}\leq v+j+1 < c_k$. So by \Cref{move12}, we can increase $C_{n-j-2}$ and decrease $C_k$ by one. That is, now $D_k>1$ and we can use the induction hypothesis. 

Now, we suppose that $v<b_1$. Then $(c_1,d_1-(b_1-v)\in Pr$ and $b_1-v\geq v$ by the minimality of $|v|$. So $v\leq \frac{b_1}{2}$. If $v<D_1$, then we take ${\bf X'}=T_1^{(0,v)} {\bf X}$. Similarly, if $v<C_1$, we take ${\bf X'}=T_1^{(v,0)} {\bf X}\in \partial\overline{\cC}$. In both cases, ${\bf X'}$ satisfies $Q'_2-Q'_1=Q_2-Q_1-2v$. So we may assume that $v=C_1=D_1=\frac{b_1}{2}$. Again, we take $\overline{X''}=T_1^{(0,v)}{\bf X}$, satisfying $Q''_1=Q_1-v$, $Q''_2=Q_2-v$. The bottom level component $X''_{-1}$ contains only one marked pole $p_1$. By \Cref{insertunbalance}, \Cref{insert}, we can obtain a desired multi-scale differential ${\bf X'}\in \partial\overline{\cC}$, satisfying $Q'_2-Q'_1=Q_2-Q_1-2v$.
\end{proof}

Now we are ready to prove step (2) of \Cref{main2minimal}.

\begin{proposition} \label{Qreduction}
Let $\cC$ be a non-hyperelliptic component of $\cR_1(\mu)$ with rotation number $r$. Then there exists ${\bf X}=X(\tau,{\bf C},Pr) \in \partial\overline{\cC}$ such that $Q_2-Q_1 \leq 2r$. 
\end{proposition}

\begin{proof}
By \Cref{unbalance}, there exists ${\bf X}\in \partial\overline{\cC}$ such that $Q_1<Q_2$. Suppose that ${\bf X}$ has minimal $Q_2-Q_1>0$ among all elements in $\partial\overline{\cC}$ with $t=n$. We will prove that $Q_2-Q_1 \leq 2r$. Let $q=\gcd(Q_1,Q_2)$ and choose a prong-matching $(c_i,d_i+v)\in Pr$ with minimal $|v|>0$. By relabeling the poles, the nodes and saddle connections, we may assume that $Q_1<Q_2$, $i=0$ and $\tau=Id$. Thus $(0,v)\in Pr$ and the rotation number $r$ is equal to $\gcd(d,Q_1,v)$. 

If $0<Q_2-Q_1-2v$, then by \Cref{Qlemma}, we can further reduce $Q_2-Q_1>0$ and this contradicts to the minimality of $Q_2-Q_1>0$. So $v\geq \frac{Q_2-Q_1}{2}\geq  \frac{q}{2}$. By level rotation action, we have $(0,v+kq)\in Pr$ for any integer $k$. By minimality of $|v|$, $v=q$ or $v=\frac{q}{2}$. We want to prove that $v=r$.

First, suppose that $v=q$. Then we have a prong-matching $(0,0)\in Pr$. So $(c_1,d_1-b_1)\in Pr$. If $b_1$ is not divisible by $q$, then $(c_1,d_1-b')\in Pr$ when $b'>0$ is the remainder of $b_1$ divided by $q$. This contradicts the minimality of $|v|$, so $q|b_1$ and $(c_1,d_1)\in Pr$. We can apply the same argument to each $b_i$ and conclude that $q|b_i$ for each $i$, thus $q|d$. Therefore $r=\gcd(d,Q_1,v)=q=v$. 

Now, suppose that $v=\frac{q}{2}$. Then we have a prong-matching $(0,\frac{q}{2})\in Pr$. If $v>b_1$, then $(c_1,d_1+(v-b_1))\in Pr$, a contradiction to the minimality of $|v|$. So $v\leq b_1$ and we have $(c_1,d_1-(b_1-\frac{q}{2}))\in Pr$. Let $v'\geq 0$ be the remainder of $b_1-\frac{q}{2}$ divided by $q$. Then $(c_1,d_1-v')\in Pr$ and we must have $v'=0$ or $v'=\frac{q}{2}$ by the minimality of $|v|$. In any case, $b_1$ is divisible by $\frac{q}{2}$, so $(c_2,d_2-(v'+b_2))\in Pr$. We can apply the same argument to each $b_i$ and conclude that $\frac{q}{2}|b_i$ for each $i$, thus $\frac{q}{2}|d$. Therefore, $r=\gcd(d,Q_1,v)=\frac{q}{2}=v$. 
Therefore $v=r$ in any cases. So $r=v\geq\frac{q}{2}\geq \frac{Q_2-Q_1}{2}$ and thus $Q_2-Q_1\leq 2r$. 
\end{proof}

Since $d|Q_1+Q_2$ , we have $r|\gcd(Q_2,Q_1)$ by \Cref{rotation} and the following discussion. If $Q_2-Q_1\leq 2r$, then $\gcd(Q_1,Q_2)=\gcd(Q_2-Q_1,Q_1)\leq 2r$ and thus $\gcd(Q_1,Q_2) =r$ or $2r$.  Moreover, if $\gcd(Q_1,Q_2)=2r$, then we also have $Q_2-Q_1=2r$. 

We need the following lemma for the third step in the proof of \Cref{main2minimal}. 

\begin{lemma} \label{Creduction}
Let $X(t,Id, {\bf C},[(0,v)])\in \partial\overline{\cR}_1(\mu)$ with rotation number $r$ and suppose that $Q_1\neq Q_2$ and $\gcd(Q_1,Q_2)=r$. For any choice of the integers $1\leq C'_i\leq b_i-1$, such that $\sum_{i=1}^t C'_i=Q_1$ and $C'_i=C_i$ for each $t<i\leq n$, we have $X(t,Id, {\bf C'},[(0,v)])\sim X(t,Id, {\bf C},[(0,v)])$.
\end{lemma}

\begin{proof}
The difference between $X(t,Id, {\bf C},Pr)$ and $X(t,Id, {\bf C'},Pr)$ can be measured by $D \coloneqq \sum_{i=1}^t |C_i - C'_i|$. It is obvious that $D=0$ if and only if two multi-scale differentials are equal. If $t=1$, then $C_1=C'_1=Q_1$, thus $D=0$. We use the induction on $D$. If $D>0$, then we can find $j,k$ such that $C_j<C'_j$ and $C_k>C'_k$. In particular, $C'_j, C_k>1$. By relabeling the poles if necessary, we may assume that $1=j<k\leq t$. If we can increase $C_1$ and decrease $C_k$ by one within $\overline{\cC}$, the difference $D$ is decreased and thus $X(t,Id, {\bf C'},Pr)\in \partial\overline{\cC}$ by induction hypothesis. 

First, suppose that $r>2$. We use the induction on $k>1$. Assume that $k=2$ and consider $V\coloneqq \{v|(u,v)\in Pr, 0<u\leq c_1\}$. If $v\in V$ for some $d_1\leq v<d_2$, then by \Cref{move12}, we can increase $C_1$ and decrease $C_2$ by one within $\overline{\cC}$. So assume that $v\notin V$ for any $d_1\leq v<d_2$. Since $\gcd(Q_1,Q_2)=r$, we have $C_1+D_2\leq r$. Similarly, if $U\coloneqq \{u|(u,v)\in Pr, 0< v\leq d_1\}$ and $u\notin U$ for any $c_1\leq u< c_2$, then $D_1+C_2\leq r$. Therefore $b_1+b_2 =C_1+D_2+D_1+C_2\leq 2r$, thus the equalities hold and $C_1+D_2=D_1+C_2=b_1=b_2=r$. Now we have $C_1=b_1-D_1=r-D_1=C_2$ and similarly $D_1=D_2$. Moreover, $(0,d_1),(c_1,0)\in Pr$. Therefore, $C_1=D_1=C_2=D_2=\frac{r}{2}$. If we can apply \Cref{move12} to increase $C_1$ and decrease $C_3$, and also to decrease $C_2$ and increase $C_3$, then we are done. If not, then by the argument as above, we must have $C_3=D_3=\frac{r}{2}$. By repeating this to other poles on the top level component, we have $Q_1=Q_2=\frac{rt}{2}$, a contradiction. 

Now assume that $k>2$. If $C_2>1$, then as above, we can increase $C_1$ and decrease $C_2$ by one. By induction hypothesis, we can now increase $C_2$ and decrease $C_k$ by one. This two steps correspond to increasing $C_1$ and decreasing $C_k$ by one, as desired. If $C_2=1$, then $C_2<b_2-1$ and we can decrease $C_k$ and increase $C_2$ by one, by induction hypothesis. Also we can increase $C_1$ and decrease $C_2$ by one, obtaining the desired result. 

Now suppose that $r=2$. By level rotation action, we have $(u,v+2\ell)\in Pr$ for each $\ell$ and $(u,v)\in Pr$. Since $C_k\geq 2$, we can choose $(0,v)\in Pr$ such that $c_{k-1}< v\leq  c_k$. Therefore by \Cref{move12}, we can increase $C_1$ and decrease $C_k$ by one. 
\end{proof}

The following proposition gives step (3) in the proof of \Cref{main2minimal}. It states that whenever $Q_1$ is fixed, the connected component is independent of the choice of ${\bf C}$.

\begin{proposition} \label{Cchoice}
Let ${\bf X}=X(\tau,{\bf C},[(0,v)])\in \partial\overline{\cR}_1(\mu)$ with rotation number $r$, satisfying $0<Q_2-Q_1 \leq 2r$. Then for any choice of the integers $1\leq C'_i\leq b_i-1$ such that $\sum_i C'_i=Q_1$, we have ${\bf X'}=X(\tau,{\bf C'},[(0,v)])\sim {\bf X}$. 
\end{proposition}

\begin{proof}
By the discussion following \Cref{Qreduction}, we have $\gcd(Q_1,Q_2)=r$ or $2r$. If $\gcd(Q_1,Q_2)=r$, then the proposition follows from \Cref{Creduction}. So assume that $\gcd(Q_1,Q_2)=Q_2-Q_1=2r$. By relabeling the poles, we may assume that $\tau=Id$. The prong-matching class $Pr$ contains exactly one of two prong-matchings $(0,0)$ and $(0,r)$. We will use \Cref{move12} to ${\bf X}$ in order to prove ${\bf X'}\sim {\bf X}$. For convenience, we will say a pair of poles $(p_i,p_j)$ of ${\bf X}$ is {\em adjustable} if we can increase $C_i$ by one and decrease $C_j$ by one by applying \Cref{move12}.

We will find a sufficient condition for a pair $(p_i,p_j)$ to be adjustable. Assume that $D_i>1$ and $(p_i,p_j)$ is not adjustable. By relabeling the poles, we may assume that $i=1$. Consider $V\coloneqq \{v|(u,v)\in Pr, 0<u\leq c_1\}$. If $d_{j-1}\leq v< d_j$, then $v\notin V$ by assumption. We can find $(0,v-1)\in Pr$ such that $v-1< d_j\leq v-1+2r$. Then $v-1< d_{j-1}$. Since $(u,v-1-u)\in Pr$ for each $0<u\leq c_1$, we can conclude that $d_j\leq v+2r-C_1$. Therefore, $D_j=d_j-d_{j-1} \leq 2r-C_1$. Similarly, if we consider $U\coloneqq\{u|(u,v)\in Pr, 0<v\leq D_1\}$, then we can obtain $C_j \leq 2r-D_1$. So $b_1+b_j= C_1+C_2+C_j+D_j\leq 4r$. The equality holds if and only if $C_1+D_j=D_1+C_j=2r$, $b_1+b_j=4r$, and $(0,d_{j-1})\in Pr$. If $b_1,b_j>r$, then the equality holds and we have $b_1=b_j=C_1+D_j=D_1+C_j=2r$. That is, $C_1=b_1-D_1=2r-(2r-C_j)=C_j$ and $D_1=b_1-C_1=2r-(2r-D_j)=D_j$. Also, if $b_1\geq 4r$ or $b_j\geq 4r$, then the pair $(p_1,p_j)$ is always adjustable. 

Now relabeling the poles again, we may assume that $C_1<C'_1$. Let $1<k\leq n$ be the smallest number such that $C'_k>C_k$. We use the induction on $k>1$ to prove that $(p_1,p_k)$ is adjustable. If this is possible, we can repeat changing $C_1$ and $C_k$ until we reach to ${\bf X'}$. Assume the contrary that we cannot achieve this. In particular, $(p_1,p_k)$ is not an adjustable pair. 

If $b_i\geq 4r$ for some $i$, then two pairs $(p_1,p_i)$ and $(p_i,p_k)$ are adjustable. By applying \Cref{move12} to these pairs, we can increase $C_1$ and decrease $C_k$ by one. So assume that $b_i\leq 3r$ for each $i$. If $k=2$, then we have $(0,d_1)\in Pr$.  That is, $r|d_1$ and $b_1>r$. We have $D_2=D_1\geq r$ and thus $b_2=C_2+D_2>r$. The only possible case is when $b_1=b_2=2r$, $C_1=C_2=r$ and $(0,r)\in Pr$. Since $(0,d_2)=(0,2r)\notin Pr$, $(p_1,p_3)$ is adjustable. If $b_3\neq 2r$ or $C_3\neq C_2=r$, then $(p_2,p_3)$ are also adjustable, a contradiction. So we have $b_3=2r$ and $C_3=r$. Now $(c_1,d_3)=(r,3r)\notin Pr$ and $(p_4,p_2)$ is adjustable. If $(p_3,p_4)$ is adjustable, then we have a chain $p_1$-$p_3$-$p_4$-$p_2$ of poles where consecutive pairs is adjustable. That is, we can increase $C_1$ and decrease $C_2$ by one, a contradiction. By repeating this argument, we can finally get $b_i=2r$ and $C_i=r$ for any $i$, which is also a contradiction because then $Q_1=Q_2=nr$. 

Now assume that $k>2$. Then consider two pairs $(p_1,p_i)$ and $(p_i,p_k)$ for some $1<i<k$. If $b_i\neq 2$, then we can use induction hypothesis to these two pairs to increase $C_1$ and decrease $C_k$ as desired. The third number $C_i$ will be decreased (or increased) by one, and then increased (or decreased) to its original position. If $b_i=2$, then $r=1$ or $r=2$. If $r=1$, then note that $b_1,b_k\geq 3$ because $C'_1,C_k>1$. In particular, $b_1+b_k\geq 6>4r$ and therefore $(p_1,p_k)$ is adjustable. If $r=2$, then $b_1=b_k=2r=4$ and $C_1=C_k=r=2$. However, since $Q_1>Q_2$, there exists some $p_i$ such that $b_i>4$ or $b_i=4$ and $C_i=1$. We can use two adjustable pairs $(p_1,p_i)$ and $(p_i,p_k)$ now. 
\end{proof}

The following proposition gives the fourth step in the proof of \Cref{main2minimal}. It states that whenever ${\bf C}$ is fixed, the connected component is independent of the choice of the prong-matching class $Pr$.

\begin{proposition} \label{Pchoice}
Let ${\bf X}=X(\tau,{\bf C},[(0,v)])\in \partial\overline{\cR}_1(\mu)$ with rotation number $r$, satisfying $0<Q_2-Q_1 \leq 2r$. Suppose that ${\bf X'}=X(\tau, {\bf C'},[(0,v')])$ with ${\bf C}={\bf C'}$ also has rotation number $r$. Then ${\bf X} \sim {\bf X'}$.
\end{proposition}

\begin{proof}
By the discussion following \Cref{Qreduction}, we have $\gcd(Q_1,Q_2)=r$ or $2r$. If $\gcd(Q_1,Q_2)=r$, then there exists a unique prong-matching class. So we assume that $\gcd(Q_1,Q_2)=Q_2-Q_1=2r$. If $2r| d$ and $(0,0)\in Pr$, then $r=\gcd(d,Q_1,0)$ is divisible by $2r$, a contradiction. So there exists unique prong-matching class. So we assume that $2r\nmid d$. That is, there exists $b_i$ such that $2r\nmid b_i$. By relabeling the poles, we may assume that $\tau=Id$ and $b_1=(2k+1)r$ for some $k$. There are exactly two prong-matching equivalence classes, shifted by $r$ from each other. So we may assume that $(0,(k+1)r)\in Pr$ and $(0,kr)\in Pr'$. 

Our goal is to find multi-scale differentials $\overline{X} \sim {\bf X}$ and $\overline{X'} \sim {\bf X'}$ containing all poles but $p_1$ in the top level component, satisfying $\tilde{Q}_2-\tilde{Q}_1=\tilde{Q'}_2-\tilde{Q'}_1=r$ and $\Tilde{C}_1=\Tilde{C'}_1$. Then we can apply \Cref{Creduction} to conclude that ${\bf X} \sim {\bf X'}$. 

Suppose that $b_i>2$ for each $i$. By \Cref{Cchoice}, the connected component of ${\bf X}$ is independent of the choice of ${\bf C}$. So we may assume that $C_1=(k+1)r-1$ and $C_2<b_2-r+1$. Also, we can assume that $C'_1=kr+1$ and $C'_n<b_n-1$. Since $b_2,b_n>2$, we can take $\tilde{X}=T_1^{(0,(k+1)r)} {\bf X}$ and $\tilde{X'}=T_2^{(kr+1,-1)} {\bf X'}$.

Suppose that $b_1=r=2$. Assume that $(0,2)\in Pr$ and $(0,0)\in Pr'$. Since $\mu\neq (2n,-2^n)$, by relabeling the poles of necessary, we may assume that $b_2>3$ and $C_2=1$. We take $\tilde{X}=T_1^{(0,2)} {\bf X}$. Similarly, if $b_2>4$, then we can take $\tilde{X'}=T_1^{(0,4)} {\bf X'}$. Assume that $b_2=4$. Then $T_1^{(0,4)} {\bf X'}$ contains two poles $p_1,p_2$ in the bottom level component. Since $Q'_2-Q'_1=4$ and $2|b_i$ for each $i$, we can find $j$ such that $D_j>2$. Then there exists $v$ such that $(0,v)\in Pr$ such that $d_{j-1}<v<d_j$. By \Cref{insert}, we can obtain a multi-scale differential $\tilde{X'}$ with all poles but $p_1$ in the top level component. We can conclude that ${\bf X} \sim {\bf X'}$.

Suppose that $r=1$ and $b_1=2$. By relabeling the poles, we may assume that $b_2>2$ and $C_2=1$. Suppose that $(0,1)\in Pr$ and $(0,2)\in Pr'$. Then we can take $\tilde{X}=T_1^{(0,1)}{\bf X}$ and $\tilde{X'}=T_1^{(0,2)}{\bf X'}$. Since $\gcd(\tilde{Q}_2,\tilde{Q}_1)=\gcd(\tilde{Q'}_2,\tilde{Q'}_1)=1$, we can apply \Cref{Creduction} and therefore ${\bf X} \sim {\bf X'}$.
\end{proof}

The following proposition gives the last step in the proof of \Cref{main2minimal}. It states that whenever ${\bf C}$ and $Pr$ are fixed, the connected component is independent of the choice of the permutation $\tau$.

\begin{proposition} \label{Tchoice}
Let ${\bf X}=X(Id, {\bf C},[(0,v)])\in \partial\overline{\cR}_1(\mu)$ with rotation number $r$, satisfying $0<Q_2-Q_1 \leq 2r$. Then $X(\tau, {\bf C},[(0,v)]))\sim {\bf X}$ for any permutation $\tau\in Sym_n$.
\end{proposition}

\begin{proof}
If $n\leq 2$, there is nothing to prove because we have only one choice of permutation $\tau$ as a cyclic order. So assume that $n>2$. Since the transpositions of the form $(i,i+1)$ generates the symmetric group $Sym_n$, it suffices to prove for the case $\tau=(i,i+1)$. By relabeling the saddle connections, we may assume that $\tau=(1,2)$. Let ${\bf X'}=X((1,2),{\bf C},Pr)$.

Case (1): Suppose that $q=r$, or $q=2r$ and $2r\nmid d$. 

Then by \Cref{Pchoice}, the connected component of ${\bf X}$ is independent of the choice of prong-matching as long as the rotation number is not changed. By assumption, $X(Id, {\bf C},[(0,0)])$ has rotation number $r$, so we may assume that $v=0$. If $2r|b_1+b_2$, then by \Cref{Cchoice}, we can change $C_1$ and $C_2$ within the connected component so that $c_2=d_2=\frac{b_1+b_2}{2}$. Then we have a prong-matching $(0,c_2)\in Pr$. It is straightforward that $T_2^{(c_2,0)} {\bf X'}=T_1^{(0,c_2)}{\bf X}$. Therefore ${\bf X'} \sim {\bf X}$. 

If $2r\nmid b_1+b_2$, then $b_1+b_2\geq 3r$. So $b_1\geq 2r$ or $b_2\geq 2r$. Assume that $b_2\geq 2r$. The case when $b_1 \geq 2r$ can be treated similarly. Then by \Cref{Cchoice}, we can change $C_1$ and $C_2$ within the connected component so that $d_2-c_2=r$, and also $D_2>r$. We have a prong-matching $(0,d_2)\in Pr$. Now we consider ${\bf X'}$. By \Cref{Cchoice}, we may assume that $C'_1=D_1$, $D'_1=C_1$, $C'_2=D_2-r$ and $D'_2=C_2+r$. We have a prong-matching $(0,r)\in Pr'$. Since $0<r<d'_1=C_2+r$, we can take $T_1^{(0,r)}{\bf X'}$. It has a prong-matching $(0,c_2)$ and we can take $T_1^{(0,c_2)} T_1^{(0,r)}{\bf X'}$. It is straightforward that $T_1^{(0,c_2)} T_1^{(0,r)}{\bf X'}=T_1^{(0,d_2)}{\bf X}$. Therefore ${\bf X'}\sim {\bf X}$.

Case (2): Suppose that $q=2r$ and $2r|d$. 

Then $(c_i,d_i+r)\in Pr$ for each $i=0,\dots, n-1$. Since $2r|b_i$ for each $i$, we have $b_i \geq 2r$. If $\frac{b_1+b_2}{2r}$ is odd, then by \Cref{Cchoice}, we may assume that $d_2=c_2=\frac{b_1+b_2}{2}$. We have a prong-matching $(0,d_2)\in Pr$ and we can deduce that ${\bf X'}\sim {\bf X}$ by the same argument as in Case (1). If $\frac{b_1+b_2}{2r}$ is even, then by \Cref{Cchoice}, we may assume that $d_2-c_2=2r$ and $D_1,D_2>r$. Then $d_2=\frac{b_1+b_2}{2}+r$ is an odd multiple of $r$, so $(0,d_2)\in Pr$. We can deduce that ${\bf X'}\sim {\bf X}$ by the same argument as in Case (1).
\end{proof}

\section{Classification of hyperelliptic components} \label{sec:chc}

In this section, we will complete the proof of \Cref{mainhyper}. First, we deal with genus one \MIN residueless strata $\cR_1(\mu)$. 

\begin{proposition} \label{hyperclass1}
Let $\cR_1 (\mu)$ be a genus one \MIN stratum. For each ramification profile $\cP$ of $\cR_1 (\mu)$, there exists a unique hyperelliptic component $\cC_{\cP}$ of $\cR_1 (\mu)$.
\end{proposition}

\begin{proof}
Case (1): First, suppose that $\cP$ fixes one marked point. 

After relabeling the poles, we may assume $\cP(i)=1-i$ for each $i$. By \Cref{genus1hyper0}, the boundary of any hyperelliptic component with ramification profile $\cP$ contains $X(\tau,{\bf C},Pr)$ for some $\tau$ satisfying $\cP(\tau(i))=\tau(1-i)$ and $C_{\tau(i)}+C_{\tau(1-i)}=b_{\tau(i)}$. That is, $\tau$ sends a pair of poles interchanged by $\cP$ to another such pair. Such permutations $\tau$ can be generated by the permutations of the form $(i,i+1)(1-i,-i)$, interchanging two pairs, and $(i,1-i)$, interchanging two poles in a pair, for $1\leq i\leq \frac{n}{2}$. Let ${\bf X}=X(Id,{\bf C},Pr)$ and ${\bf X'}=X(\tau,{\bf C},Pr)$. It suffices to show that ${\bf X} \sim {\bf X'}$ for each $\tau=(i,i+1)(1-i,-i)$ or $\tau=(i,1-i)$.

First, let $\tau=(i,i+1)(-i+1,-i)$. Then $T_2^{(C_{i+1},-C_{i+1})}T_1^{(-d_i,d_i)}{\bf X}=T_2^{(C_{i+1},-C_{i+1})}T_1^{(-d_{i-1},d_{i-1})}{\bf X'}$. The path in $\overline{\cR}_1(\mu)$ connecting ${\bf X}$ and ${\bf X'}$ is illustrated in \Cref{fig801}. Thus ${\bf X} \sim {\bf X'}$.

\begin{figure}
    \centering
    \input{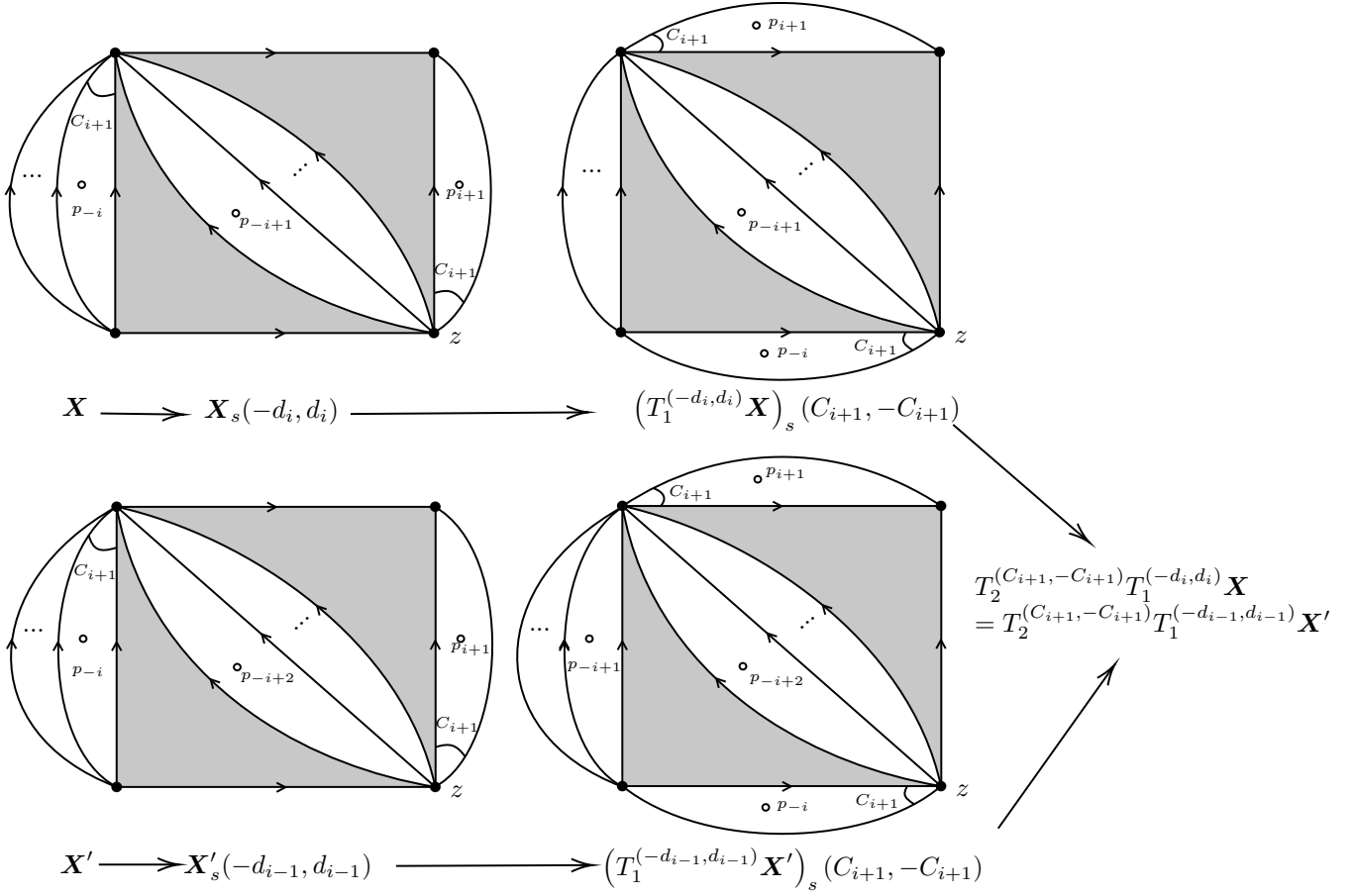} 
    \caption{The path in $\overline{\cR}_1(\mu)$ when $\tau=(i,i+1)(1-i,-i)$} \label{fig801}
\end{figure}

Now let $\tau=(i,-i+1)$. By applying \Cref{move12} to ${\bf X'}$ if necessary, we may assume that $D'_i=C'_{-i+1}=C_i$. Then $T_2^{(C_i,-C_i)} T_1^{(-d_{i-1},d_{i-1})} {\bf X} = T_1^{(-D_i,D_i)}T_2^{(c_{i-1},-c_{i-1})}{\bf X'}$. The path in $\overline{\cR}_1(\mu)$ connecting ${\bf X}$ and ${\bf X'}$ is illustrated in \Cref{fig802}. Thus ${\bf X} \sim {\bf X'}$.

\begin{figure}
    \centering
    \input{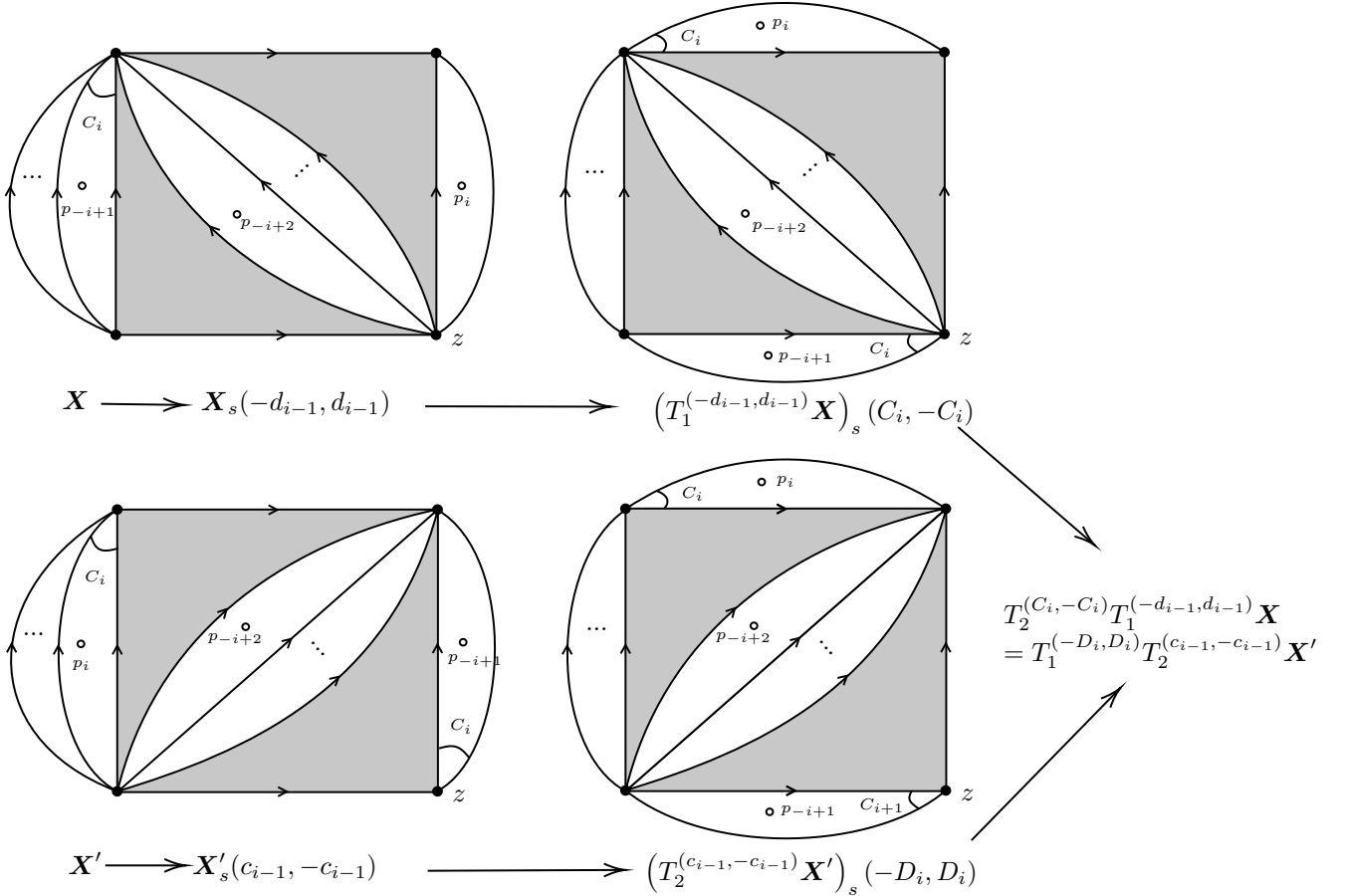} 
    \caption{The path in $\overline{\cR}_1(\mu)$ when $\tau=(i,1-i)$} \label{fig802}
\end{figure}

Case (2): Now suppose that $\cP$ fixes two or three marked points. 

After relabeling the poles, we may assume $\cP(i)=2-i$ for each $i$. The set of fixed marked poles is $\{p_1, p_{\frac{n+2}{2}}\}$ if $n$ is even, or $\{p_1\}$ if $n$ is odd. \Cref{genus1hyper0}, the boundary of any hyperelliptic component with ramification profile $\cP$ contains $X(\tau,{\bf C},Pr)$ for some $\tau$ satisfying $\cP(\tau(i))=\tau(2-i)$ and $C_{\tau(i)}+C_{\tau(2-i)}=b_{\tau(i)}$. Such permutations $\tau$ form a subgroup of $Sym_n$, generated by the permutations of the form $(i,i+1)(2-i,1-i)$ and $(i,2-i)$, for $2\leq i\leq \frac{n}{2}$. The rest of the proof is completely analogous to the proof of Case (1). 

Case (3): Finally suppose that $\cP$ fixes four marked points. 

By relabeling the poles, we may assume that $p_1,p_{\frac{n+1}{2}}$ and $p_n$ are the fixed poles and $\cP(i)=2-i$ for $i=1,\dots,n-1$. By \Cref{genus1hyper0}, the boundary of any hyperelliptic component with ramification profile $\cP$ contains $X(n-1,\tau,{\bf C},Pr)$ for some $\tau$ satisfying $\cP(\tau(i))=\tau(2-i)$ and $C_{\tau(i)}+C_{\tau(2-i)}=b_{\tau(i)}$ for each $i=1,\dots,n-1$. Such permutations $\tau$ form a subgroup of $Sym_n$, generated by the permutations of the form $(i,i+1)(2-i,1-i)$ and $(i,n+1-i)$ for $2\leq i\leq \frac{n+1}{2}$. Let ${\bf X}=X(n-1,Id,{\bf C},Pr)$ and ${\bf X'}=X(n-1,\tau,{\bf C},Pr)$. The rest of the proof is completely analogous to the proof of Case (1). 
\end{proof}

Now we prove \Cref{mainhyper} for any $\cR_g(\mu)$ with $g>0$ inductively. 

\begin{proof}[Proof of \Cref{mainhyper}]
We use the induction on $\dim \cR_g(\mu)=2g+m-1$. The base case is $g=m=1$ and this is already treated in \Cref{hyperclass1}. So we assume that $2g+m>3$. Fix a ramification profile $\cP$ of $\cR_g(\mu)$ and let $\cC$ be any hyperelliptic component of $\cR_g(\mu)$ with ramification profile $\cP$.

If $\cP$ fixes less than $2g+2$ marked points, then $\cC$ contains a flat surface with a multiplicity one saddle connection. By shrinking this saddle connection, we obtain a two level multi-scale differential $\overline{X}$ and the top level component $X_0$ is a hyperelliptic component with ramification profile $\cP$ contained in a stratum of dimension $2g+m>2$. By induction hypothesis, the connected component containing $X_0$ is unique. If $m=2$, then the bottom level component $X_{-1}$ is contained in a connected stratum $\cH_0(a,a,-2a-2)$, and $X_0$ and $X_{-1}$ intersect at a unique node. Thus there is a unique prong-matching equivalence class of $\overline{X}$ and therefore $\cC$ is unique. If $m=1$, then $X_{-1}$ is still contained in a connected stratum $\cH_0(a,-\frac{a}{2}-1,-\frac{a}{2}-1)$, and $X_0$ and $X_{-1}$ intersect at two nodes. However, by hyperellipticity of $\cC$, there is still a unique possible prong-matching equivalence class of $\overline{X}$ because the prong-matching commutes with hyperelliptic involutions. Therefore $\cC$ is also unique in this case. 

Now suppose that $\cP$ fixes $2g+2$ marked points. By relabeling the poles if necessary, we may assume that $p_1$ is one of the fixed poles. First, assume that $\cR_g(\mu)$ is a \MIN stratum. So it is $2g$-dimensional. By \Cref{double}, there exists a flat surface $X\in \cC$ that contains a pair of parallel saddle connections with multiplicity two bounding the polar domain of $p_1$. By shrinking them, we obtain a multi-scale differential $\overline{X}\in \partial\overline{\cC}$. The bottom level component $X_{-1}$ contains $p_1$, and is contained in the stratum $\cR_0(a,-b_1;-\frac{a-b_1}{2}-1,-\frac{a-b_1}{2}-1)$. By hyperellipticity of $\overline{X}$, the component $X_{-1}$ is uniquely determined up to re-scaling the differential. Specifically, an element of this stratum is determined by two angles of the polar domain of $p_1$, and it is hyperelliptic if and only if two angles are equal to each other. 

The top level component $X_0$ is a \NMIN hyperelliptic flat surface of genus $g-1$ with $2g$ fixed marked points. It is contained in a hyperelliptic component of a stratum of dimension $2g-1$, whose ramification profile $\cP_0$ is induced by $\cP$. By induction hypothesis, $X_0$ is contained in a unique connected component with ramification profile $\cP_0$. By hyperellipticity of $\cC$, there is a unique possible prong-matching equivalence class of $\overline{X}$ because the prong-matching commutes with hyperelliptic involutions. Combining this with the uniqueness of the connected components containing $X_0$ and $X_{-1}$, we can conclude that $\cC$ is also unique for given $\cP$.
\begin{figure}
    \centering
    \tikzset{every picture/.style={line width=0.75pt}} 

\begin{tikzpicture}[x=0.75pt,y=0.75pt,yscale=-1,xscale=1]

\draw    (324.67,126.83) -- (335.67,111.83) ;
\draw    (345.17,126.33) -- (335.67,111.83) ;
\draw    (335.64,70.34) -- (335.36,103.8) ;
\draw    (319.17,25.33) -- (335.92,44.33) ;
\draw    (349.17,24.83) -- (335.92,44.33) ;
\draw   (331.67,107.83) .. controls (331.67,105.62) and (333.46,103.83) .. (335.67,103.83) .. controls (337.88,103.83) and (339.67,105.62) .. (339.67,107.83) .. controls (339.67,110.04) and (337.88,111.83) .. (335.67,111.83) .. controls (333.46,111.83) and (331.67,110.04) .. (331.67,107.83) -- cycle ;
\draw    (136.17,26.33) -- (152.92,45.33) ;
\draw    (166.17,25.83) -- (152.92,45.33) ;
\draw    (140.67,86.83) -- (151.67,71.83) ;
\draw    (161.17,86.33) -- (151.67,71.83) ;
\draw    (186.36,60.3) -- (251.36,59.81) ;
\draw [shift={(253.36,59.8)}, rotate = 179.57] [color={rgb, 255:red, 0; green, 0; blue, 0 }  ][line width=0.75]    (10.93,-3.29) .. controls (6.95,-1.4) and (3.31,-0.3) .. (0,0) .. controls (3.31,0.3) and (6.95,1.4) .. (10.93,3.29)   ;
\draw    (319.93,284.07) -- (330.93,269.07) ;
\draw    (340.43,283.57) -- (330.93,269.07) ;
\draw    (330.91,227.58) -- (330.63,261.04) ;
\draw    (314.43,182.57) -- (331.18,201.57) ;
\draw    (344.43,182.07) -- (331.18,201.57) ;
\draw   (326.93,265.07) .. controls (326.93,262.86) and (328.72,261.07) .. (330.93,261.07) .. controls (333.14,261.07) and (334.93,262.86) .. (334.93,265.07) .. controls (334.93,267.28) and (333.14,269.07) .. (330.93,269.07) .. controls (328.72,269.07) and (326.93,267.28) .. (326.93,265.07) -- cycle ;
\draw    (135.43,182.24) -- (152.18,201.24) ;
\draw    (165.43,181.74) -- (152.18,201.24) ;
\draw    (139.93,242.74) -- (150.93,227.74) ;
\draw    (160.43,242.24) -- (150.93,227.74) ;
\draw    (344.75,255.03) -- (330.93,261.07) ;
\draw    (191.96,214.54) -- (256.96,214.05) ;
\draw [shift={(258.96,214.04)}, rotate = 179.57] [color={rgb, 255:red, 0; green, 0; blue, 0 }  ][line width=0.75]    (10.93,-3.29) .. controls (6.95,-1.4) and (3.31,-0.3) .. (0,0) .. controls (3.31,0.3) and (6.95,1.4) .. (10.93,3.29)   ;
\draw    (39.25,304.44) -- (682.45,306.04) ;
\draw    (40.45,151.64) -- (683.65,153.24) ;

\draw    (335.07, 57.57) circle [x radius= 12.9, y radius= 12.9]   ;
\draw (335.07,57.57) node  [font=\Large] [align=left] {$\displaystyle g$};
\draw (335.16,30.33) node   [align=left] {...};
\draw (322.49,133.57) node    {$z_{1}$};
\draw (346.49,134.07) node    {$z_{2}$};
\draw (295.24,58.57) node    {$\sigma _{0} \curvearrowright $};
\draw    (152.07, 58.57) circle [x radius= 12.9, y radius= 12.9]   ;
\draw (152.07,58.57) node  [font=\Large] [align=left] {$\displaystyle g$};
\draw (152.16,31.33) node   [align=left] {...};
\draw (114.24,59.57) node    {$\sigma \curvearrowright $};
\draw (138.49,93.57) node    {$z_{1}$};
\draw (162.49,94.07) node    {$z_{2}$};
\draw (366,49.75) node [anchor=north west][inner sep=0.75pt]    {$\in \mathcal{C}_{\mathcal{P}_{0}}$};
\draw (367,101.25) node [anchor=north west][inner sep=0.75pt]    {$\in \mathcal{H}_{0}( a,a,-2a-2)$};
\draw (294.58,107.24) node    {$\sigma _{-1} \curvearrowright $};
\draw (198.87,42.98) node [anchor=north west][inner sep=0.75pt]   [align=left] {shrink};
\draw    (330.34, 214.81) circle [x radius= 12.9, y radius= 12.9]   ;
\draw (330.34,214.81) node  [font=\Large] [align=left] {$\displaystyle g$};
\draw (330.42,187.58) node   [align=left] {...};
\draw (317.76,290.81) node    {$z_{1}$};
\draw (341.76,291.31) node    {$z_{2}$};
\draw    (151.34, 214.48) circle [x radius= 12.9, y radius= 12.9]   ;
\draw (151.34,214.48) node  [font=\Large] [align=left] {$\displaystyle g$};
\draw (151.42,187.24) node   [align=left] {...};
\draw (113.51,215.48) node    {$\sigma \curvearrowright $};
\draw (137.76,249.48) node    {$z_{1}$};
\draw (161.76,249.98) node    {$z_{2}$};
\draw (347.03,249.67) node  [font=\scriptsize]  {$p_{1}$};
\draw (205.93,196.49) node [anchor=north west][inner sep=0.75pt]   [align=left] {shrink};
\draw (293.48,214.72) node    {$\sigma _{0} \curvearrowright $};
\draw (290.64,268.22) node    {$\sigma _{-1} \curvearrowright $};
\draw (361.73,205.7) node [anchor=north west][inner sep=0.75pt]    {$\in \mathcal{C}_{\mathcal{P}_{0}}$};
\draw (360.57,259.03) node [anchor=north west][inner sep=0.75pt]    {$\in \mathcal{R}_{0}( a,a,-b_{1} ,b_{1} -2a-2)$};
\draw (42.93,281.6) node [anchor=north west][inner sep=0.75pt]    {$2g+2$};
\draw (88.13,281) node [anchor=north west][inner sep=0.75pt]   [align=left] {fixed marked points};
\draw (39.73,128.8) node [anchor=north west][inner sep=0.75pt]    {$< 2g+2$};
\draw (103.33,128.2) node [anchor=north west][inner sep=0.75pt]   [align=left] {fixed marked points};
\draw (219.3,70.5) node   [align=left] {simple s.c.};
\draw (224.97,224.5) node   [align=left] {pair of s.c.};

\end{tikzpicture} 
    \caption{Navigating in the boundary of a \MIN hyperelliptic component} \label{fig803}
\end{figure}
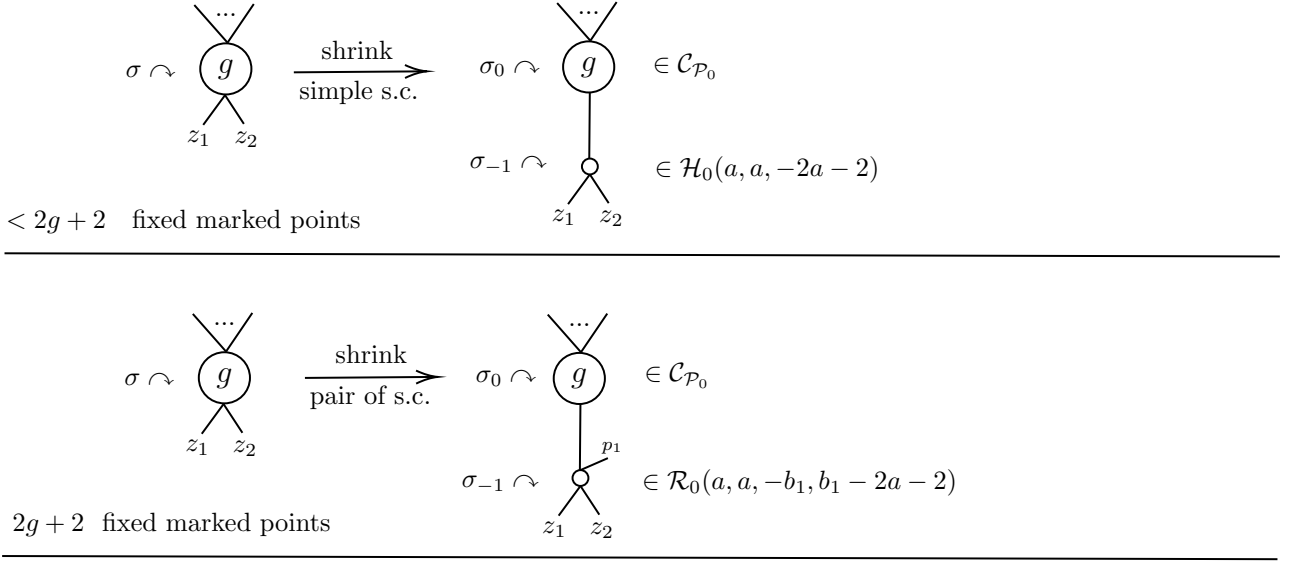

Now assume that $\cR_g(\mu)$ is a \NMIN stratum. So it is $(2g+1)$-dimensional. By \Cref{double}, there exists a flat surface $X\in \cC$ that contains a pair of parallel saddle connections with multiplicity two bounding the polar domain of $p_1$. By shrinking them, we obtain $\overline{X}\in \partial\overline{\cC}$. See \Cref{fig803} for the description of the level graph of $\overline{X}$. The bottom level component $X_{-1}$ contains $p_1$, and is contained in the stratum $\cR_0(a,a,-b_1,b_1-2a-2)$. Up to re-scaling the differential, $X_{-1}$ is a unique hyperelliptic flat surface in $\cR_0(a,a,-b_1,b_1-2a-2)$. The top level component $X_0$ is a \MIN hyperelliptic flat surface of genus $g$ with $2g+2$ fixed marked points. It is contained in a hyperelliptic component of dimension $2g$, whose ramification profile $\cP_0$ is induced by $\cP$. By induction hypothesis, there exists a unique connected component with ramification profile $\cP_0$. Since there is a unique node, there is also a unique prong-matching equivalence class. Therefore we can conclude that $\cC$ is unique for given $\cP$. 
\begin{figure}
    \centering
    \tikzset{every picture/.style={line width=0.75pt}} 

\begin{tikzpicture}[x=0.75pt,y=0.75pt,yscale=-1,xscale=1]

\draw   (372.17,306.51) .. controls (372.17,304.3) and (373.96,302.51) .. (376.17,302.51) .. controls (378.38,302.51) and (380.17,304.3) .. (380.17,306.51) .. controls (380.17,308.72) and (378.38,310.51) .. (376.17,310.51) .. controls (373.96,310.51) and (372.17,308.72) .. (372.17,306.51) -- cycle ;
\draw    (376.17,327.51) -- (376.17,310.51) ;
\draw    (374.58,260.84) .. controls (362.85,273.77) and (363.52,285.77) .. (376.17,302.51) ;
\draw    (374.58,260.84) .. controls (387.52,271.77) and (387.52,289.11) .. (376.17,302.51) ;
\draw    (358.83,212.84) -- (375.58,231.84) ;
\draw    (388.83,212.34) -- (375.58,231.84) ;
\draw    (138.17,212.67) -- (154.92,231.67) ;
\draw    (168.17,212.17) -- (154.92,231.67) ;
\draw    (153.83,274.77) -- (153.83,257.77) ;
\draw    (188.7,246.64) -- (253.7,246.15) ;
\draw [shift={(255.7,246.14)}, rotate = 179.57] [color={rgb, 255:red, 0; green, 0; blue, 0 }  ][line width=0.75]    (10.93,-3.29) .. controls (6.95,-1.4) and (3.31,-0.3) .. (0,0) .. controls (3.31,0.3) and (6.95,1.4) .. (10.93,3.29)   ;
\draw    (387.82,295.46) -- (376.17,302.51) ;
\draw   (369.57,121.51) .. controls (369.57,119.3) and (371.36,117.51) .. (373.57,117.51) .. controls (375.78,117.51) and (377.57,119.3) .. (377.57,121.51) .. controls (377.57,123.72) and (375.78,125.51) .. (373.57,125.51) .. controls (371.36,125.51) and (369.57,123.72) .. (369.57,121.51) -- cycle ;
\draw    (373.57,142.51) -- (373.57,125.51) ;
\draw    (371.98,75.84) .. controls (360.25,88.77) and (360.92,100.77) .. (373.57,117.51) ;
\draw    (371.98,75.84) .. controls (384.92,86.77) and (384.92,104.11) .. (373.57,117.51) ;
\draw    (356.23,27.84) -- (372.98,46.84) ;
\draw    (386.23,27.34) -- (372.98,46.84) ;
\draw    (138.23,25.68) -- (154.98,44.68) ;
\draw    (168.23,25.18) -- (154.98,44.68) ;
\draw    (153.9,87.77) -- (153.9,70.77) ;
\draw    (186.43,59.97) -- (251.43,59.49) ;
\draw [shift={(253.43,59.47)}, rotate = 179.57] [color={rgb, 255:red, 0; green, 0; blue, 0 }  ][line width=0.75]    (10.93,-3.29) .. controls (6.95,-1.4) and (3.31,-0.3) .. (0,0) .. controls (3.31,0.3) and (6.95,1.4) .. (10.93,3.29)   ;
\draw    (29.33,383.24) -- (672.53,384.84) ;
\draw    (30.05,198.44) -- (673.25,200.04) ;

\draw (374.82,216.34) node   [align=left] {...};
\draw (376.82,333.51) node    {$z$};
\draw (301.08,243.91) node    {$\sigma _{0} \curvearrowright $};
\draw (412.33,235.09) node [anchor=north west][inner sep=0.75pt]    {$\in \mathcal{C}_{\mathcal{P}_{0}}$};
\draw    (154.07, 244.91) circle [x radius= 12.9, y radius= 12.9]   ;
\draw (154.07,244.91) node  [font=\Large] [align=left] {$\displaystyle g$};
\draw (154.16,217.68) node   [align=left] {...};
\draw (116.24,245.91) node    {$\sigma \curvearrowright $};
\draw (154.49,280.77) node    {$z$};
\draw (219.52,231.59) node   [align=left] {shrink};
\draw (303.24,298.91) node    {$\sigma _{-1} \curvearrowright $};
\draw (392.1,288.1) node  [font=\scriptsize]  {$p_{1}$};
\draw (410.67,287.42) node [anchor=north west][inner sep=0.75pt]    {$\in \mathcal{R}_{0}\left( a,-b_{1} ;-\frac{a-b_{1}}{2} -1,\frac{a-b_{1}}{2} -1\right)$};
\draw    (372.77, 61.24) circle [x radius= 14.12, y radius= 14.12]   ;
\draw (372.77,61.24) node  [font=\scriptsize,color={rgb, 255:red, 255; green, 255; blue, 255 }  ,opacity=1 ] [align=left] {$\displaystyle g-1$};
\draw (372.22,31.34) node   [align=left] {...};
\draw (374.22,148.51) node    {$z$};
\draw    (154.14, 57.91) circle [x radius= 12.9, y radius= 12.9]   ;
\draw (154.14,57.91) node  [font=\Large] [align=left] {$\displaystyle g$};
\draw (154.22,30.68) node   [align=left] {...};
\draw (116.31,58.91) node    {$\sigma \curvearrowright $};
\draw (154.56,93.77) node    {$z$};
\draw (217.25,46.93) node   [align=left] {shrink};
\draw (336,86.14) node [anchor=north west][inner sep=0.75pt]    {$Pr$};
\draw (300.88,60.31) node    {$\sigma _{0} \curvearrowright $};
\draw (411.13,49.49) node [anchor=north west][inner sep=0.75pt]    {$\in \mathcal{C}_{\mathcal{P}_{0}}$};
\draw (304.04,115.31) node    {$\sigma _{-1} \curvearrowright $};
\draw (412.47,108.82) node [anchor=north west][inner sep=0.75pt]    {$\in \mathcal{H}_{0}\left( a,-\frac{a}{2} -1,-\frac{a}{2} -1\right)$};
\draw (54.59,365.4) node    {$2g+2$};
\draw (135.03,365.8) node   [align=left] {fixed marked points};
\draw (58.97,180.6) node    {$< 2g+2$};
\draw (146.96,181) node   [align=left] {fixed marked points};
\draw (170,106) node [anchor=north west][inner sep=0.75pt]    {$X$};
\draw (367,167) node [anchor=north west][inner sep=0.75pt]    {$\overline{X}$};
\draw (180,289) node [anchor=north west][inner sep=0.75pt]    {$X$};
\draw (370,353) node [anchor=north west][inner sep=0.75pt]    {$\overline{X}$};
\draw (220.3,69.83) node   [align=left] {simple s.c.};
\draw (223.3,259.23) node   [align=left] {pair of s.c.};
\draw    (374.84, 246.47) circle [x radius= 14.12, y radius= 14.12]   ;
\draw (374.84,246.47) node  [font=\scriptsize,color={rgb, 255:red, 255; green, 255; blue, 255 }  ,opacity=1 ] [align=left] {$\displaystyle g-1$};
\draw (376.95,61) node    {$-1$};
\draw (365.95,62) node    {$g$};
\draw (377.95,245) node    {$-1$};
\draw (366.95,246) node    {$g$};
\draw (337,273.14) node [anchor=north west][inner sep=0.75pt]    {$Pr$};

\end{tikzpicture} 
    \caption{Navigating in the boundary of a \NMIN hyperelliptic component} \label{fig804}
\end{figure}
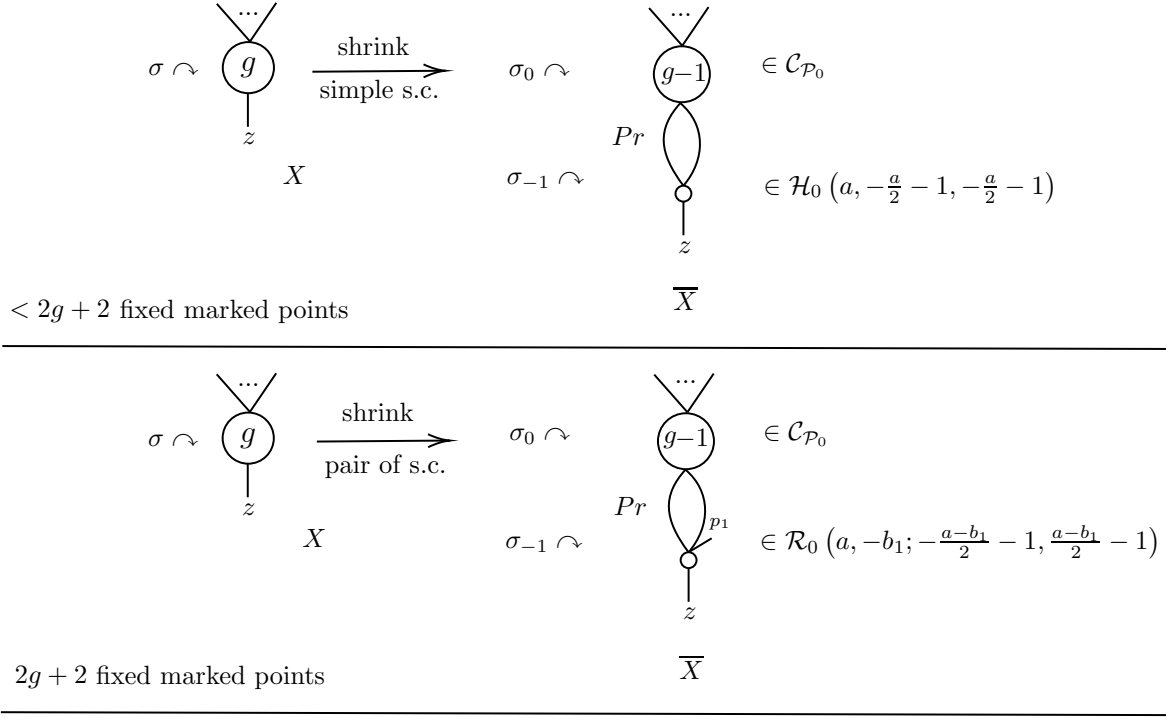
\end{proof}

\section{Genus one multiple-zero strata} \label{sec:g1n}

In this section, we will complete the proof of \Cref{main2}. In \cite{kozo1} and \cite{boissy}, the first step in enumerating the connected components of multiple-zero strata is proving that each of their connected components is adjacent to a connected component of the \MIN stratum, obtained by merging all zeroes. Since merging zeroes can be done by a combination of a certain $GL^{+}(2,\mathbb{R})$ action and a local isoperiodic surgery, it does not affect the residue conditions. So we will follow the same strategy here for the residueless stratum $\cR_g(\mu)$. In \Cref{subsec:merge}, we will prove that every non-hyperelliptic component of $\cR_g(\mu)$, except for the special case $\cR_1(2n,-2^n)$, is adjacent to a non-hyperelliptic component of a \MIN stratum. Then we will classify all non-hyperelliptic components of genus one multiple-zero strata in \Cref{subsec:g1n}.

\subsection{Merging zeroes} \label{subsec:merge}

For a given $\mu=(a_1,\dots,a_m,-b_1,\dots,-b_n)$, we denote $a\coloneqq a_1+\dots+a_m$ and $\mu'\coloneqq (a, -b_1,\dots, -b_n)$. Recall that we assume $b_1\leq \dots\leq b_n$ by default. Also, recall that a connected component $\cD$ of $\cR_g(\mu)$ is {\em adjacent} to a connected component $\cC$ of $\cR_g(\mu')$ if some flat surface in $\cD$ can be obtained by breaking up a zero from some flat surface in $\cC$. In other words, $\cD$ is adjacent to $\cC$ if some flat surface in $\cC$ can be obtained by merging all zeroes of some flat surface in $\cD$. 

We can define a {\em breaking up the zero} map
$$
B:\{\text{non-hyperelliptic components of }\cR_g(\mu')\} \to \{\text{non-hyperelliptic components of }\cR_g(\mu)\}
$$ 
by $B(\cC)=\cD$ if $\cD$ is adjacent to $\cC$. The goal of this subsection is to prove that $B$ is surjective when $g>1$, or $g=1$ and $b_n>2$. In other words,

\begin{proposition} \label{nonhypermerge}
Assume that $g>1$, or $b_n>2$. Let $\cD$ be a non-hyperelliptic component of a stratum $\cR_g (\mu)$ of genus $g>0$. Then $\cD$ is adjacent to a non-hyperelliptic component of the \MIN stratum $\cR_g (\mu')$.
\end{proposition}

As a result of \Cref{nonhypermerge}, we obtain an upper bound for the number of non-hyperelliptic components.

\begin{corollary}
Assume that $g>1$ or $b_n>2$. The number of non-hyperelliptic components of $\cR_g (\mu)$ is less than or equal to the number of non-hyperelliptic components of the corresponding \MIN stratum $\cR_g (\mu')$.
\end{corollary}

We will prove \Cref{nonhypermerge} by induction on $\dim \cR_g (\mu) = 2g+m-1>2$. First, we deal with the base case, when $(g,m)=(1,2)$ and $b_n>2$. So $B$ is given by breaking up the zero of order $a$ into two zeroes of orders $a_1$ and $a_2$. 

\begin{proposition} \label{nonhypermerge1}
Let $\cD$ be a non-hyperelliptic component of a genus one stratum $\cR_1 (a_1,a_2,-b_1,\dots,-b_n)$ with $b_n>2$. Then $\cD$ is adjacent to a non-hyperelliptic component of the \MIN stratum $\cR_1 (a,-b_1,\dots,-b_n)$.
\end{proposition}

Before proving this proposition, we explain the strategy of the proof. Again we move around in the boundary of $\overline{\cD}$, until we end up in $\overline{B(\cC)}$ for some non-hyperelliptic component $\cC$ of $\cR_1 (a,-b_1,\dots,-b_n)$. By \Cref{simple} and \Cref{merging}, we already know that $\cD=B(\cC)$ for some (possibly hyperelliptic) $\cC$. To ensure that we can find a non-hyperelliptic $\cC$, we first need to prove several lemmas that describe the conditions when $B(\cC_1)=B(\cC_2)$ for distinct connected components $\cC_1,\cC_2$ of $\cR_1 (a,-b_1,\dots,-b_n)$. Then for each hyperelliptic component $\cC_1$, we will find a non-hyperelliptic component $\cC_2$ such that $B(\cC_1)=B(\cC_2)$. 

Choose a two-level multi-scale differential ${\bf X}=X(Id,{\bf C},[u,v])\in \partial\overline{\cC}_1$. Then the bottom level component $X_{-1}$ is contained in the stratum $\cH_0(a,-Q_1-1,-Q_2-1)$. This flat surface has distinguished prongs $v^-_1$ and $w^-_1$, at two poles $s_1$ and $s_2$, respectively (See \Cref{fig601}). By breaking up the zero, we obtain $X'_{-1}\in \cH_0(a_1,a_2,-Q_1-1,-Q_2-1)$, and it still has the prongs $v'$ and $w'$, deformed from $v^-_1$ and $w^-_1$. As $X'_{-1}$ deforms in the stratum $\cH_0(a_1,a_2,-Q_1-1,-Q_2-1)$, we can keep track of the prongs at the poles. In other words, we keep the information $(X',v',w')$, where $v',w'$ are the prongs at the poles. The moduli space of this data is called the moduli space of {\em differentials with marked prongs (separatrices)}, which is defined and studied in full generality in \cite{boissy}. Now assume that $X'_{-1}$ has a multiplicity one saddle connection joining $z_1$ and $z_2$, then we can shrink the saddle connection and merge two zeroes, obtaining a flat surface in $\cH_0(a,-Q_1-1,-Q_2-1)$ isomorphic to $X_{-1}$ up to scaling. As a result, the prongs $v'$ and $w'$ will deform to $v^-_i$ and $w^-_j$ of $X_{-1}$ for some $i,j$. Thus we can obtain a multi-scale differential $X(Id,{\bf C},[u-i,v-j])$, in the boundary of a stratum adjacent to $B(\cC_1)$. See \Cref{fig901}.

\begin{figure}
    \centering
    \input{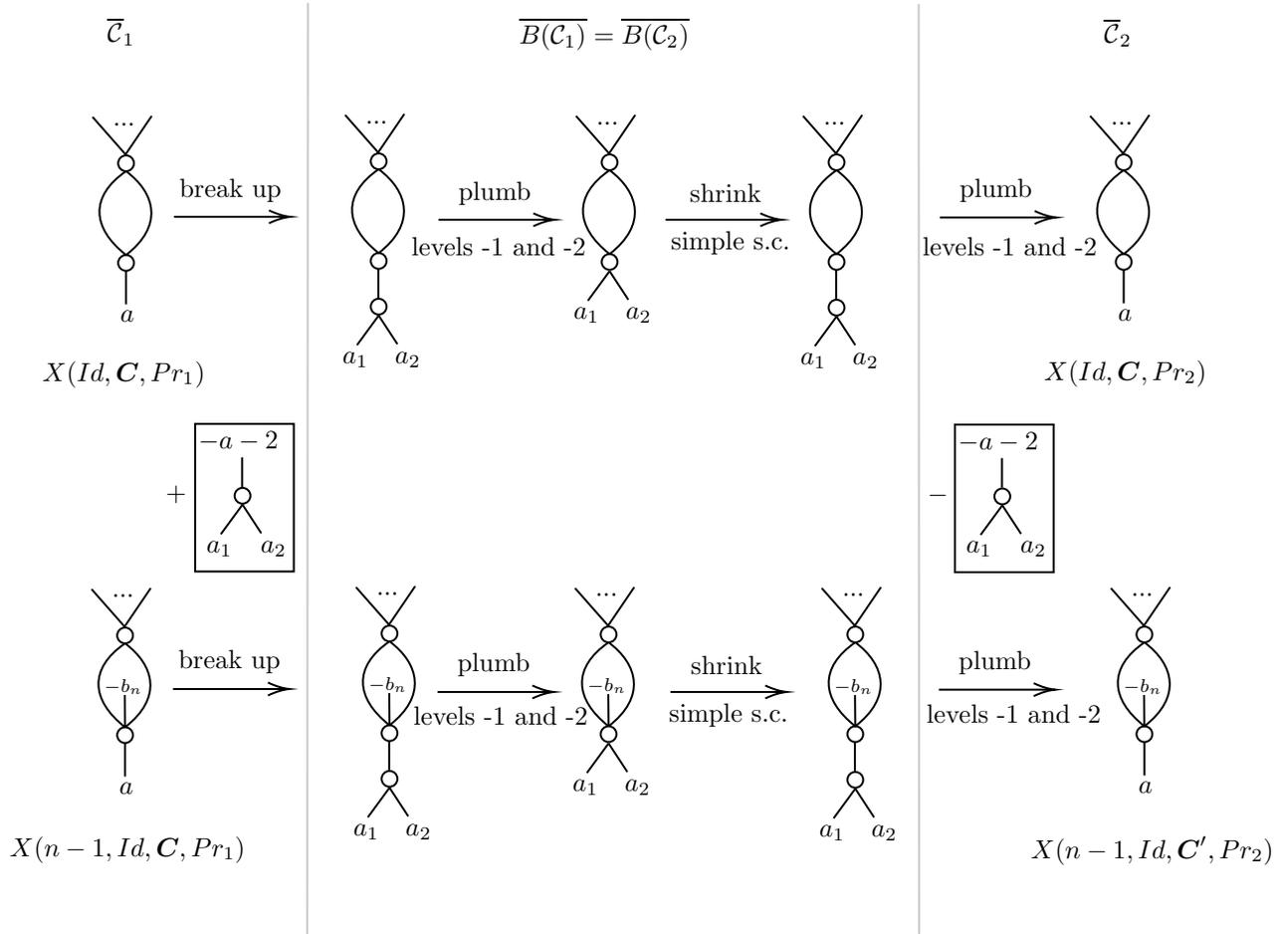} 
    \caption{Strategy of proving $\overline{B(\cC_1)}=\overline{B(\cC_2)}$} \label{fig901}    
\end{figure}

\begin{lemma} \label{breakinglemma1}
Let $\cC_1, \cC_2$ be two connected components of a \MIN stratum $\cR_1(a,-b_1,\dots,-b_n)$. Suppose that $X(Id,{\bf C},[(0,v)])\in \partial\overline{\cC}_1$ and $X(Id,{\bf C},[(0,v+a_1)])\in \partial\overline{\cC}_2$. Then $B(\cC_1)=B(\cC_2)$. 
\end{lemma}

\begin{proof}
Let $\overline{X}=X(Id,{\bf C},[(0,v)])$ and $\beta$ be the unique saddle connection in $X_{-1}$. By breaking up the zero of $X_{-1}$, we obtain a flat surface with a new saddle connection $\alpha$ joining the two distinct zeroes $z_1,z_2$. The saddle connection $\beta$ also becomes a saddle connection, also denoted by $\beta$, in the new flat surface. By \Cref{suitableprong}, we may assume that $\beta$ is joining $z_1,z_2$. By shrinking $\beta$, we obtain a flat surface isomorphic to $X_{-1}$, up to scaling. Along this deformation, we obtain $(X_{-1},v^-_{a_1+1},w^-_1)$ from $(X_{-1},v^-_1,w^-_1)$, as illustrated in \Cref{fig902}. Therefore, we obtain $X(Id,{\bf C},[(0,v+a_1)])$ in the boundary of a stratum adjacent to $B(\cC_1)$. Thus $B(\cC_1)=B(\cC_2)$.
\end{proof}

\begin{figure}
    \centering
    \input{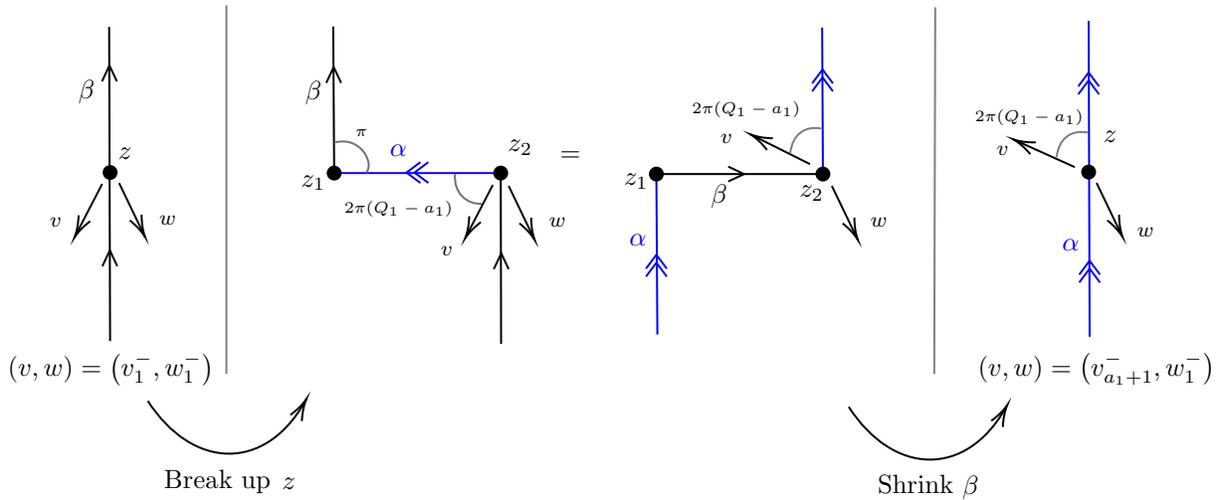} 
    \caption{Prongs and saddle connections along the deformation in $\overline{B(\cC)}$} \label{fig902}
\end{figure}

\begin{lemma} \label{breakinglemma2}
Let $\cC_1, \cC_2$ be two connected components of a \MIN stratum $\cR_1(a,-b_1,\dots,-b_n)$. Suppose that $b_n=2$ and $X(n-1,Id,{\bf C},[(0,v)])\in \partial\overline{\cC}_1$. Note that the pole $p_n$ of order $2$ is contained in the bottom level component.

(1) Suppose that $a_1<\frac{a}{2}-1$ or $Q_1< Q_2$. If $X(n-1,Id,{\bf C},[(0,v+a_1)])\in \partial\overline{\cC}_2$, then $B(\cC_1)=B(\cC_2)$. 

(2) Suppose that $a_1=\frac{a}{2}-1$ and $Q_1= Q_2=\frac{a}{2}-1$. If $X(n-1,Id,{\bf C},[(0,v-2)])\in \partial\overline{\cC}_2$, then $B(\cC_1)=B(\cC_2)$. 
\end{lemma}

\begin{proof}
There are $a+1$ outgoing prongs in $X_{-1}$ at $z$, denoted by $u^+_1,\dots, u^+_{a+1}$. We can label them in the clockwise order so that $\beta_1$ encloses from $u^+_1$ to $u^+_{Q_1}$ and $\beta_2$ encloses from $u^+_{Q_1+2},u^+_a$. When we break up the zero $z$, we identify $z$ with the pole of a flat surface in $\cH_0(a_1,a_2,-a-2)$. There are $a+1$ incoming prongs at the pole, denoted by $u^-_1,\dots,u^-_{a+1}$. We can label them in the counterclockwise order so that $u^-_1,\dots, u^-_{a_1}$ are coming from $z_1$, and the others are coming from $z_2$. The direction of breaking up the zero is equivalent to the choice of a prong-matching at $z$. 

First we prove (1). Suppose that $a_1<\frac{a}{2}-1$ or $Q_1< Q_2$. 

By breaking up the zero, we obtain a flat surface with new saddle connection $\alpha$ joining two distinct zeroes $z_1,z_2$. The saddle connections $\beta_1,\beta_2$ also deform to saddle connections in the new flat surface, also denoted by $\beta_1,\beta_2$. Since $Q_1+Q_2+2=a_1+a_2=a$, we have $a_1<Q_2$ in particular. Choose a prong-matching that sends $u^-_1$ to $u^+_{Q_1+1}$. Then breaking up the zero, $\beta_2$ is joining $z_1,z_2$ and $\beta_1$ is joining $z_2$ to itself. So $\beta_2$ is a multiplicity one saddle connection. By shrinking $\beta_2$, we obtain a flat surface isomorphic to $X_{-1}$, up to scaling. Along this deformation, we obtain $(X_{-1},v^-_1,w^-_{Q_2-a_1})$ from $(X_{-1},v^-_1,w^-_1)$, as illustrated in \Cref{fig903}. Therefore, we obtain $X(n-1,Id,{\bf C},[(0,v+a_1)])$ in the boundary of a stratum adjacent to $B(\cC_1)$. Thus $B(\cC_1)=B(\cC_2)$.

\begin{figure}
    \centering
    \input{diagram903} 
    \caption{} \label{fig903}
\end{figure}

Now we prove (2). Suppose that $a_1=Q_1=Q_2=\frac{a}{2}-1$. As before, we break up the zero of $X_{-1}$. Choose a prong-matching that sends $u^-_1$ to $u^+_t$, such that $t\neq \frac{a}{2}$. Then $\beta_1,\beta_2$ are parallel saddle connections, joining $z_1,z_2$. By shrinking $\beta_1$ and $\beta_2$, we obtain $\overline{Y}$. The top level component is contained in $\cH_0 (a-2,-\frac{a}{2},-\frac{a}{2})$. This has two poles of order $-\frac{a}{2}$ at the nodes, as in $X_{-1}$. However, under this deformation, we obtain $(Y_{-1},v^-_1{Q_1},w^-_1)$ from $(X_{-1},v^-_1,w^-_1)$, as illustrated in \Cref{fig904}. So the levels 0 and -1 form $X(n-1,Id,{\bf C},[0,v+1])\in \partial\overline{\cR}_1(a-2,-b_1,\dots,-b_{n-1})$. The bottom level component is contained in $\PP\cR_0(a_1,a_2,-2,-a)$, which is a singleton. Similarly, if we start from $X(n-1,Id,{\bf C},[(0,v+2)])$, then we can obtain the same multi-scale differential as above. See \Cref{fig906}. Therefore, the boundaries of $B(\cC_1)$ and $B(\cC_2)$ have a common multi-scale differential and thus $B(\cC_1)=B(\cC_2)$.
\begin{figure}
    \centering
    \input{diagram904} 
    \caption{} \label{fig904}
\end{figure}
\begin{figure}
    \centering
    \input{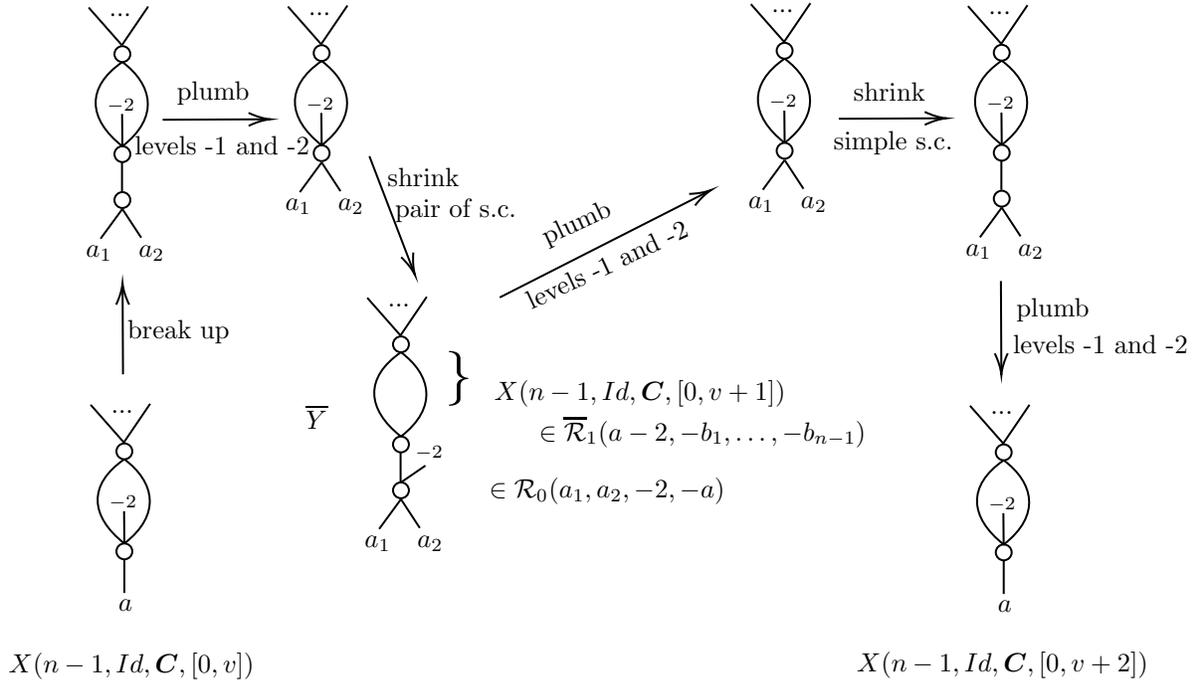} 
    \caption{Navigating the boundary of $B(\cC_1)$, $b_n=2$} \label{fig906}
\end{figure}
\end{proof}

\begin{lemma} \label{breakinglemma3}
Let $\cC_1$ be a connected components of a \MIN stratum $\cR_1(a,-b_1,\dots,-b_n)$. Suppose $b_n>2$ and $a_1<a_2$. Also suppose that $\overline{X}=X(n-1,Id,{\bf C},[(0,v)])\in \partial\overline{\cC}_1$. Then there exists a connected component $\cC_2$ such that $B(\cC_1)=B(\cC_2)$, and there exists $\overline{X'}=X(n-1,Id,{\bf C'},[(0,v')])\in \partial\overline{\cC}_2$, where $C'_n\neq C_n$ and $C'_i=C_i$ for each $i=1,\dots, n-1$.
\end{lemma}

\begin{proof}
The bottom level component $X_{-1}$ is contained in $\cR_0(a,-b_n;-Q_1-1,-Q_2-1)$, for $Q_1\leq Q_2$. First, we assume that $D_n\leq C_n$. In particular, we have $C_n+Q_2\geq \frac{a}{2}> a_1$. 

By \Cref{zerodim2}, $X_{-1}$ has two saddle connections, denoted by $\beta_1$ and $\beta_2$. They are bounding the polar domain of $p_n$, so form two angles equal to $2\pi C_n$ and $2\pi D_n$ at $z$. There are $a+1$ outgoing prongs in $X_{-1}$ at $z$, denoted by $u^+_1,\dots, u^+_{a+1}$. We can label them in the clockwise order so that $\beta_1$ encloses from $u^+_1$ to $u^+_{Q_1}$ and $\beta_2$ encloses from $u^+_{Q_1+C_n+1},u^+_{Q_1+Q_2+C_n}$. When we break up the zero $z$, we identify $z$ with the pole of a flat surface in $\cH_0(a_1,a_2,-a-2)$. There are $a+1$ incoming prongs at the pole, denoted by $u^-_1,\dots,u^-_{a+1}$. We can label them in the counterclockwise order so that $u^-_1,\dots, u^-_{a_1}$ are coming from $z_1$, and the others are coming from $z_2$. The direction of breaking up the zero is equivalent to the choice of a prong-matching at $z$. 

By breaking up the zero, we obtain a flat surface with a new saddle connection $\alpha$ joining two distinct zeroes $z_1,z_2$. The saddle connections $\beta_1,\beta_2$ also become saddle connections in the new flat surface, also denoted by $\beta_1,\beta_2$. Since $C_n+Q_2>a_1$, we can choose a prong-matching that sends $u^-_1$ to $u^+_{Q_1+s}$ for some $1\leq s <C_n$, such that $Q_1+C_n+1\leq Q_1+a_1+s-1\leq Q_1+Q_2+C_n$. Then $\beta_2$ is joining $z_1,z_2$ and $\beta_1$ is joining $z_2$ to itself. So $\beta_2$ is a multiplicity one saddle connection. By shrinking $\beta_2$, we obtain a flat surface $X'_{-1}$ in $\cR_0(a,-b_n;-Q_1-1,-Q_2-1)$. The angles formed by $\beta_2$ and $\alpha$ are equal to  $2\pi s$ and $2\pi (b_n-s)$. Under this deformation, we obtain $(X'_{-1},v^-_1,w^-_1)$ from $(X'_{-1},v^-_1,w^-_{1+a_1+s-C_n})$, as illustrated in \Cref{fig905}. Therefore, we obtain $X(n-1,Id,{\bf C'},[(0,v+C_n-a_1-s)])\in \overline{\cC}_2$ .
\begin{figure}
    \centering
    \input{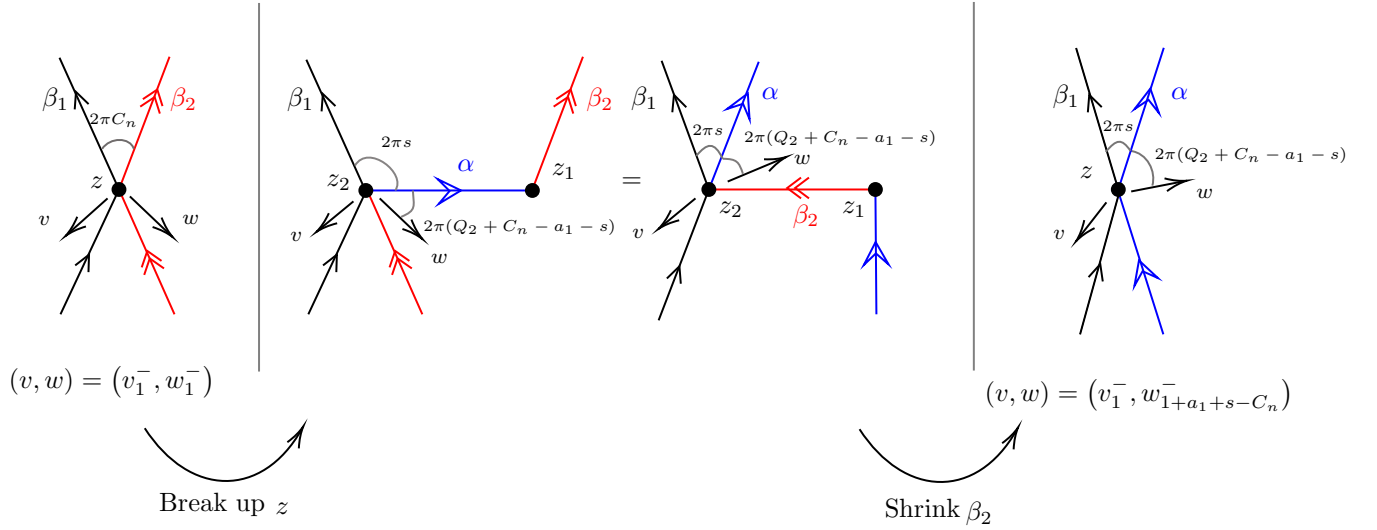} 
    \caption{Navigating the boundary of $B(\cC_1)=B(\cC_2)$} \label{fig905} 
\end{figure}

Now we assume that $C_n<D_n$. Then $Q_2+D_n>a_1$. By the similar argument as above, we can obtain $X(n-1,Id,{\bf C''},[(0,v'')])\in \overline{\cC}_2$ with $C''_n>C_n$, such that $B(\cC_1)=B(\cC_2)$. 
\end{proof}

\begin{proof}[Proof of \Cref{nonhypermerge1}]
By \Cref{simple}, $\cD$ is adjacent to some connected component $\cC$ of $\cR_1 (a_1+a_2,-b_1,\dots,-b_n)$. If $\cC$ is non-hyperelliptic, there is nothing to prove. So assume that $\cC$ is a hyperelliptic component. If $a_1=a_2$, then $\cD$ is also hyperelliptic by \Cref{hyper1}. Thus we assume that $a_1<a_2$. Since $b_n>2$, the ramification profile $\cP$ of $\cC$ satisfies at least one of the following three possibilities:

(1) $\cP$ fixes less than $4$ marked points.

(2) $\cP$ fixes three poles of order $-2$, and interchanges at least one pair of poles of order $-b<-2$. 

(3) $\cP$ fixes some pole of order $-b<-2$.

We will deal with each of these cases separately. 
    
Case (1) --- By \Cref{simple1}, we obtain $X(Id,{\bf C},[(0,v)])\in \partial \overline{\cC}$. Each component of this multi-scale differential has a hyperelliptic involution, and the prong-matching $(0,v)$ is compatible with the involutions. Let $\cC'$ be the connected component of $\cR_1 (a_1+a_2,-b_1,\dots,-b_n)$ containing $X(Id,{\bf C},[(0,v-a_1)])$ in the boundary. By \Cref{breakinglemma1}, we have $\cD=B(\cC)=B(\cC')$. Since $a_1<\frac{a}{2}$, the prong-matching $(0,v-a_1)$ is not compatible with the hyperelliptic involution. Thus $\cC'$ is non-hyperelliptic and $\cD$ is adjacent to some non-hyperelliptic component of $\cR_1 (a_1+a_2,-b_1,\dots,-b_n)$. 

Case (2) --- By relabeling the poles, we may assume that $b_n=2$ and $p_n$ is a fixed pole. By \Cref{double}, We obtain $X(n-1,Id,{\bf C},[(0,v)])\in \partial \overline{\cC}$. Suppose that $a_1< \frac{a}{2}-1$. Let $\cC'$ be the connected component of $\cR_1 (a_1+a_2,-b_1,\dots,-b_n)$ containing $X(Id,{\bf C},[(0,v-a_1)])$ in the boundary. By \Cref{breakinglemma2}, we have $\cD=B(\cC)=B(\cC')$. By the same argument as in the previous case, $\cC'$ is a non-hyperelliptic component. If $a_1=\frac{a}{2}-1$, then we can take $\cC'$ so that $X(Id,{\bf C},[(0,v-2)])\in \partial \overline{\cC'}$. Again by \Cref{breakinglemma2}, we have $\cD=B(\cC)=B(\cC')$, and $\cC'$ is a non-hyperelliptic component.

Case (3) --- By relabeling the poles, we may assume that $b_n>2$ and $p_n$ is a fixed pole. We obtain $X(n-1,Id,{\bf C},[(0,v)])\in \partial \overline{\cC}$. Then $C_n=\frac{b_n}{2}>1$. By \Cref{breakinglemma3}, there exists $X(n-1,Id,{\bf C'},[(0,v')])\in \cC'$ for component $\cC'$ with $C'_n\neq C_n=\frac{b_n}{2}$, such that $\cD=B(\cC)=B(\cC')$. Then the bottom level component of $X(n-1,Id,{\bf C'},[(0,v')])$ does not have an involution. So $\cC'$ is a non-hyperelliptic component. 
\end{proof}

\begin{proof}[Proof of \Cref{nonhypermerge}]
We use the induction on $\dim_{\mathbb{C}} \cR_g(\mu)=2g+m-1\geq 2$. The case $m=1$ is trivial, so assume $m>1$. By \Cref{simple}, there exists a flat surface $X\in \cD$ with a multiplicity one saddle connection $\gamma$ joining two zeroes $z_1$ and $z_2$. By shrinking $\gamma$, we obtain $\overline{X}$. The bottom level component is in a connected stratum $\cH_0 (a_1,a_2,-(a_1+a_2+2))$. The top level component $X_1$ is the flat surface with a zero of order $a_1+a_2$ at the node. Let $\cC$ be the connected component of $\cR_g(a_1+a_2,a_3,\dots,a_m,-b_1,\dots,-b_n)$ containing $X_1$. If $\cC$ is non-hyperelliptic, then we can merge all zeroes by the induction hypothesis. So we may assume that $\cC$ is a hyperelliptic component. If $\cC$ is a \NMIN hyperelliptic component, then $m=3$ and $a_3=a_1+a_2$. By \Cref{simple}, we can deform $X$ continuously so that it contains a multiplicity one saddle connection joining $z_1$ and $z_3$. We merge $z_1$ and $z_3$ by shrinking the multiplicity one saddle connection, obtaining a flat surface $X'_1$ with two zeroes of different orders $a_2$ and $a_1+a_3$. In particular, $X'_1$ is non-hyperelliptic, thus we can merge all zeroes by induction hypothesis. 

Now suppose that $\cC$ is a \MIN hyperelliptic component. In particular, $m=2$. If $a_1=a_2$, then by \Cref{hyper1}, $\cD$ is also hyperelliptic. This is a contradiction to the assumption, and thus we must have $a_1 < a_2$. The base case $g=1$ and $b_n>2$ is already proven in \Cref{nonhypermerge1}. By \Cref{mainhyper}, we can find a multi-scale differential $\overline{Y}\in \partial \overline{\cC}$ with two components intersecting at one node, such that the bottom level component $Y_{-1}$ is a \MIN hyperelliptic flat surface of genus $g_{-1}=1$ and containing the pole $p_n$. By breaking up the zero of $Y_{-1}$, we can obtain an element of $\partial \overline{\cD}$. Let $\cC_{-1}$ be the connected component of a stratum containing $Y_{-1}$. We can apply induction hypothesis to $B(\cC_{-1})$, so there exists a non-hyperelliptic component $\cC'_{-1}$ such that $B(\cC'_{-1})=B(\cC_{-1})$. In other words, $\partial \overline{\cD}$ is adjacent to the boundary of some non-hyperelliptic component $\cC'$. 
\end{proof}

\subsection{Connected components of genus one multiple-zero strata} \label{subsec:g1n}

In this subsection, we will prove that non-hyperelliptic connected components of genus one strata are classified by rotation number. Assume that $\cR_1 (\mu)$ is a genus one \MIN stratum and $\mu\neq(12,-3^4),(2n,-2^n),(2r,-r,-r)$ and $(2r,-2r)$. Recall from \Cref{main2minimal} that the non-hyperelliptic connected components of $\cR_1 (\mu)$ are classified by rotation number, a topological invariant recalled in \Cref{sec:g1m}. For each $r|d\coloneqq\gcd(b_1,\dots,b_n)$, there exists a unique non-hyperelliptic component $\cC_r$ with rotation number $r$. 

Consider a map $$B:\{\cC_r:r|d\} \to\{\text{non-hyperelliptic components of }\cR_1 (a_1,\dots,a_m,-b_1,\dots,-b_n)\}$$ obtained by breaking up the zero. We can easily compute the rotation number of $B(\cC_r)$. For a flat surface $X\in \cC_r$ with symplectic basis $\{\alpha,\beta\}$, the rotation number $r$ is given by $\gcd(d,\Ind\alpha,\Ind\beta)$. While breaking up the zero, $\{\alpha,\beta\}$ still remains to be a symplectic basis, and their indices are not changed. So the rotation number of $B(\cC_r)$ is equal to $\gcd(d,a_1,\dots,a_m,\Ind\alpha,\Ind\beta)=\gcd(a_1,\dots,a_m,r)$. By \Cref{nonhypermerge1}, $B$ is surjective. In order to prove \Cref{main2} for $\cR_1(\mu)$, we need to show that $B(\cC_{r_1})=B(\cC_{r_2})$ if $\gcd(a_1,\dots,a_m,r_1)=\gcd(a_1,\dots,a_m,r_2)$. Instead of proving this directly, we first deal with the simplest case --- double-zero strata $\cR_1(a_1,a_2,-b_1,\dots,-b_n)$. 

\begin{proposition} \label{nonminimal1}
Let $\cR_1 (\mu)$ be a genus one \MIN stratum and $\mu\neq (12,-3^4),(2n,-2^n),(2r,-r,-r)$ and $(2r,-2r)$. Let $r_1,r_2$ be positive integer divisors of $d$. Suppose that $B$ is the map given by breaking up the zero into two zeroes of orders $a_1,a_2$. Then $B(\cC_{r_1})=B(\cC_{r_2})$ if and only if $\gcd(a_1,r_1)=\gcd(a_1,r_2)$.
\end{proposition}

\begin{proof}
One direction of the proposition is immediate. If $B(\cC_{r_1})=B(\cC_{r_2})$, then rotation numbers of $B(\cC_{r_1})$ and $B(\cC_{r_2})$ must be equal, thus $\gcd(a_1,r_1)=\gcd(a_1,r_2)$. 

Conversely, assume that $r=\gcd(a_1,r_1)=\gcd(a_1,r_2)$. Once we prove that $B(\cC_r)=B(\cC_R)$ for any $R|d$ such that $r=\gcd(a_1,R)$, then we can conclude that $B(\cC_r) = B(\cC_{r_1})=B(\cC_{r_2})$. So we can reduce to the case when $r_1=r$. Consider a two-level multi-scale differential $X(\tau,{\bf C},[(0,0)])\in \partial \overline{\cC}_{r_2}$. By \Cref{rotation}, we have $r_2=\gcd(d,Q_1)$. Another multi-scale differential $X(\tau,{\bf C},[(0,a_1)]$ is contained in $\partial \overline{\cC}_r$, since $\gcd(d,Q_1,a_1)=\gcd(a_1,r_2)=r$. By \Cref{breakinglemma1}, we have $B(\cC_r)=B(\cC_{r_2})$ as desired. 
\end{proof}

\subsection{Proof of \Cref{main2}}

We finally complete the proof of \Cref{main2} by using \Cref{nonhypermerge} and \Cref{nonminimal1}. First, we are dealing with the components adjacent to the special strata from \Cref{specialmerge1} to \Cref{nonexistspecial}. 

\begin{proposition} \label{specialmerge1}
The stratum $\cR_1(6^2,-3^4)$ has a unique non-hyperelliptic component with rotation number $3$.
\end{proposition}

\begin{proof}
Let $\cD$ be a non-hyperelliptic component of $\cR_1(6^2,-3^4)$ with rotation number $3$. By \Cref{nonhypermerge1}, $\cD$ is adjacent to some non-hyperelliptic component $\cC$ of $\cR_1(12,-3^4)$. By \Cref{special1}, $\cC$ is equal to one of two possibilities: $\cC_3^1$ containing $\overline{X}=X(Id,(1,1,2,2),[(0,3)])$, or $\cC_3^2$ containing $\overline{X'}=X((1,2),(1,1,2,2),[(0,3)])$ in their boundary. It is sufficient to prove that $B(\cC_3^1)=B(\cC_3^2)$ as connected components of $\cR_1(6^2,-3^4)$. We consider $\overline{Y}=T_2^{(0,3)}\overline{X}=X(2,Id,(1,2,2,2),[(2,0)])$. The bottom level component $Y_{-1}$ is contained in $\cR_0(12,-3,-3;-4,-4)$, with three saddle connections $\beta_1,\beta_2,\beta_3$. By breaking up the zero of $Y_{-1}$ to obtain $Y'_{-1}$. With proper choice of prong-matching, two of them, $\beta_2$ and $\beta_3$, remain to be parallel. By shrinking them, we can further degenerate $Y'_{-1}$ into two-level multi-scale differential $\overline{Z}$. See \Cref{fig904}. The top level component $Z_0$ is contained in $\cR_0(9,-3;-4,-4)$ and the bottom level component $Z_{-1}$ is contained in $\cR_0(6^2,-3,-11)$. By keeping track of the prongs, $(Y_{-1},v^-_1,w^-_1)$ degenerates to $(Z_0,v^-_1,w^-_1)$. Thus the levels 0 and -1 form a multi-scale differential $X(2,Id,(1,2,1),[(2,0)])\in \overline{\cR}_1(9,-3^3)$, containing $p_1,p_2,p_3$. By plumbing the transition between the level 0 and -1, we obtain a flat surface in $\cR_1(9,-3^3)$ with rotation number $\gcd(3,3,3)=3$ by \Cref{rotation}. By swapping the labeling of $p_1$ and $p_2$, we obtain the flat surface in $\cR_1(9,-3^3)$ with the same rotation number. By \Cref{main2minimal}, two flat surfaces are contained in the same connected component of $\cR_1(9,-3^3)$. However, remark that swapping $p_1$ and $p_2$ also swaps the connected components $\cC_3^1$ and $\cC_3^2$ in $\cR_1(12,-3^4)$. Therefore, $B(\cC_3^1)=B(\cC_3^2)$.
\end{proof}

\begin{proposition} \label{specialmerge2}
The stratum $\cR_1(3,9,-3^4)$ has a unique non-hyperelliptic component with rotation number $3$.
\end{proposition}
\begin{proof}
We consider two components $\cC_3^1$ and $\cC_3^2$ of $\cR_1(12,-3^4)$, as in the proof of \Cref{specialmerge1}. Again we need to proof $B(\cC_3^1)=B(\cC_3^2)$ as connected components of $\cR_1(3,9,-3^4)$. Let $\overline{Y}=X(2,Id,(1,2,2,2),[(2,0)])\in \overline{\cC}_3^1$. By breaking up the zero with proper choice of prong-matching, two saddle connections $\beta_2$ and $\beta_3$ remain to be parallel. By shrinking $\beta_2,\beta_3$, we can further degenerate into two-level multi-scale differential $\overline{Z}$. As in the proof of \Cref{specialmerge1}, the top level component $Z_0$ is contained in $\cR_0(9,-3;-4,-4)$. By keeping track of the prongs, $(Y_{-1},v^-_1,w^-_1)$ degenerates to $(Z_0,v^-_1,w^-_1)$. Thus the levels 0 and -1 form a multi-scale differential with rotation number $3$ in $\overline{\cR}_1(9,-3^3)$, containing $p_1,p_2,p_3$. By swapping the labeling of $p_1$ and $p_2$, we obtain the flat surface in $\cR_1(9,-3^3)$ with the same rotation number. By \Cref{main2minimal}, two flat surfaces are contained in the same connected component in $\cR_1(9,-3^3)$. However, remark that swapping $p_1$ and $p_2$ also swaps the connected components $\cC_3^1$ and $\cC_3^2$ in $\cR_1(12,-3^4)$. Therefore, $B(\cC_3^1)=B(\cC_3^2)$.
\end{proof}

If $b_n=2$, then $\mu'=(2n,-2^n)$ and $\cR_1(\mu')$ does not have any non-hyperelliptic component by \Cref{exception0}. So the map $B$ does not give any useful information. In order to analyze the double-zero stratum $\cR_1(a_1,a_2,-2^n)$, we need to consider a slightly modified map $$B':\{\text{connected components of }\cR_1(2n,-2^n)\}\to\{\text{connected components of }\cR_1 (a_1,a_2,-2^n)\}$$ also given by breaking up the zero. If $a_1=a_2=n$, we have the following result. 

\begin{proposition} \label{nnspecial}
The stratum $\cR_1(n,n,-2^n)$ does not have {\em any} non-hyperelliptic connected component.
\end{proposition}

\begin{proof}
Assume the contrary --- let $\cD$ be a non-hyperelliptic component of $\cR_1(n,n,-2^n)$. Then by \Cref{simple}, $\cD$ contains a flat surface with a multiplicity one saddle connection joining $z_1$ and $z_2$. By shrinking the saddle connection, we obtain a two-level multi-scale differential $\overline{X}$. The top level component is contained in some connected component $\cC$ of $\cR_1(2n,-2^n)$. By \Cref{exception0}, $\cC$ is hyperelliptic. The bottom level component is contained in the stratum $\cR_0(n,n,-2n-2)$, which is connected and hyperelliptic. By \Cref{hyper1}, we can conclude that $\cD$ is also hyperelliptic, a contradiction. 
\end{proof}

Recall that the hyperelliptic components of $\cR_1(n,n,-2^n)$ are classified by its ramification profiles by \Cref{mainhyper} proved in \Cref{sec:chc}. 
Now suppose that $a_1 < a_2$. Since $\cR_1 (a_1,a_2,-2^n)$ does not have any ramification profiles, it does not have any hyperelliptic component. So $B'$ is surjective by \Cref{simple} and \Cref{merging}. By \Cref{hyperclass1} the domain of $B'$ is equivalent to the set of ramification profiles of $\cR_1(2n,-2^n)$. 

\begin{proposition}  \label{nonminimalexception}
Let $\cC_1,\cC_2$ be the (hyperelliptic) connected components of $\cR_1(2n,-2^n)$ corresponding to the ramification profiles $\cP_1,\cP_2$, respectively. Then $B'(\cC_1)=B'(\cC_2)$ if and only if $a_1$ is odd, or $a_1$ is even and $\cP_1,\cP_2$ fix the same number of marked points. 
\end{proposition}

\begin{proof}
Note that for a given $n$, there are two possible number $k$ of fixed marked points, since $k\equiv n+1 (\textrm{mod}\ 2)$. For example, if $n$ is even, then $k$ is either one or three. 

Suppose that $B'(\cC_1)=B'(\cC_2)$ and $a_1$ is even. We need to prove $\cP_1,\cP_2$ fix the same number of marked points. Note that the rotation numbers of $\cC_1$ and $\cC_2$ are determined by the number of fixed marked points. Let $\cP_1$ fixes $k$ marked points. If $k=1$, then after relabeling the poles, we have $X(Id,{\bf 1}, [(0,0)])\in \partial\overline{\cC}_1$. Here, ${\bf 1}$ means $(1,\dots,1)$. So by \Cref{rotation}, the rotation number of $\cC_1$ is equal to $\gcd(2,n)=2$. If $k=2,3$, then after relabeling the poles, we have $X(Id,{\bf 1}, [(0,1)])\in \partial\overline{\cC}_1$. The rotation number is equal to $\gcd(2,n,1)=1$. If $k=4$, then after relabeling the poles, we have $X(n-1,Id,{\bf 1},[(0,1)])\in \partial\overline{\cC}_1$. The rotation number is equal to $\gcd(2,n-1,n+1)=2$. Since the rotation number of $B'(\cC_1)$ is equal to the rotation number of $\cC_1$, it is also determined by the number of fixed marked points. So if $B'(\cC_1)=B'(\cC_2)$, then $\cP_1,\cP_2$ fix the same number of marked points. 

Conversely, suppose first that $a_1$ is even and $\cP_1,\cP_2$ fix the same number of marked points. 

First, suppose that $a_1<n-1$. We will deal with the case when $\cP_1$ and $\cP_2$ fix four marked points. This is the most complicated case, and the other cases will follow more easily by the same argument. By relabeling the poles, we may assume that $\cP_1$ fixes $n$. Consider a multi-scale differential $X(n-1,\tau, {\bf 1}, [(0,1)])\in \partial\overline{\cC}_2$ containing only one (fixed) pole $p_{\tau(n)}$ in the bottom level component. The other fixed poles are labeled by $\tau(1)$ and $\tau(\frac{n+1}{2})$. Suppose that $\cP_2$ does not fix $n$ (So $n>3$). Then $p_n$ is contained in the top level component. By relabeling the saddle connections, we may assume that $\tau(\frac{a_1}{2}+1)=n$ since $\frac{a_1}{2}+1\neq 1,\frac{n+1}{2}$. Consider another multi-scale differential $X(n-1,\tau, {\bf 1}, [(0,a_1+1)])$. It has a ramification profile $\cP_3$ that fixes $n$. By (1) of \Cref{breakinglemma2}, we have $B'(\cC_2)=B'(\cC_3)$. So we can reduce to the case when $\cP_2$ fixes $n$. By relabeling the poles other than $p_n$, we may assume that $\cP_1(1)=\cP_1(1)$ and $\cP_1(i)=\cP_1(n+1-i)$ for each $i=2,\dots, n-1$. Since $\cP_1$ and $\cP_2$ have the same cycle type, there exists a permutation $\sigma$ fixing $n$, such that $\cP_2 = \sigma\circ \cP_1 \circ \sigma^{-1}$. Therefore, it is enough to show that $B'(\cC_1)=B'(\cC_2)$ whenever $\cP_2 = \sigma\circ \cP_1 \circ \sigma^{-1}$ for each transposition $\sigma=(i,j)$, $1\leq i<j\leq n-1$. It is obvious that $\sigma\circ \cP_1 \circ \sigma^{-1} =\cP_1$ for each $\sigma=(j,n+1-j)$, $2\leq j\leq n-1$. So it remains to show $B'(\cC_1)=B'(\cC_2)$ when $\cP_2 = (1,j)\circ \cP_1 \circ (1,j)$ for each $1<j\leq \frac{n+1}{2}$. If $j=\frac{n+1}{2}$, then we can take $\tau=(j,a_1+1)(n+1-j,n-a_1)$. In particular, $\tau(a_1+1)=j$. We have $X(n-1,\tau, {\bf 1}, [(0,1)])\in \partial\overline{\cC}_1$. Consider a multi-scale differential $X(n-1,\tau, {\bf 1}, [(0,a_1+1)])$, with ramification profile $\cP_4$, interchanging $\tau(1)=1$ and $\tau(a_1+1)=j$. Then by \Cref{breakinglemma2}, we have $B'(\cC_1)=B'(\cC_4)$. Similarly, we can deduce $B'(\cC_2)=B'(\cC_4)$. Therefore $B'(\cC_1)=B'(\cC_2)=B'(\cC_4)$. Symmetrically, we also have $B'(\cC_1)=B'(\cC_2)$ for each $\cP_2 = (\frac{n+1}{2},j)\circ \cP_1 \circ (\frac{n+1}{2},j)$, $1\leq j<\frac{n+1}{2}$. Note that $(1,\frac{n+1}{2})=(1,2)(2,\frac{n+1}{2})(1,2)$. Therefore, we can conclude that $B'(\cC_1)=B'(\cC_2)$.

Now we assume that $a_1=n-1$. In this case, $n$ is odd. We will deal with the case when $\cP_1$ and $\cP_2$ fix four marked points. By relabeling the poles, we may assume that $\cP_1$ fixes $n$. Consider a multi-scale differential $X(n-1,\tau, {\bf 1}, [(0,1)])\in \partial\overline{\cC}_2$ containing only one (fixed) pole $p_{\tau(n)}$ in the bottom level component. The other fixed poles are labeled by $\tau(1)$ and $\frac{n+1}{2}$. Suppose that $\cP_2$ does not fix $n$ (So $n>3$). Then $p_n$ is contained in the top level component. By relabeling the saddle connections, we may assume that $\tau(n-1)=n$ since $n-1\neq 1,\frac{n+1}{2}$. Consider another multi-scale differential $X(n-1,\tau, {\bf 1}, [(0,-1)])$. It has a ramification profile $\cP_3$ that fixes $\tau(n-1)=n$. By (2) of \Cref{breakinglemma2}, we have $B'(\cC_2)=B'(\cC_3)$. So we can reduce to the case when $\cP_2$ fixes $n$. By relabeling the poles other than $p_n$, we may assume that $\cP_1(1)=\cP_1(1)$ and $\cP_1(i)=\cP_1(n+1-i)$ for each $i=2,\dots, n-1$. Since $\cP_1$ and $\cP_2$ have the same cycle type, there exists a permutation $\sigma$ fixing $n$, such that $\cP_2 = \sigma\circ \cP_1 \circ \sigma^{-1}$. Therefore, it is enough to show that $B'(\cC_1)=B'(\cC_2)$ whenever $\cP_2 = \sigma\circ \cP_1 \circ \sigma^{-1}$ for each transposition $\sigma=(i,j)$, $1\leq i<j\leq n-1$. It is obvious that $\sigma\circ \cP_1 \circ \sigma^{-1} =\cP_1$ for each $\sigma=(j,n+1-j)$, $2\leq j\leq n-1$. So it remains to show $B'(\cC_1)=B'(\cC_2)$ when $\cP_2 = (1,j)\circ \cP_1 \circ (1,j)$ for each $1<j\leq \frac{n+1}{2}$. If $j=\frac{n+1}{2}$, then we can take $\tau=(j,n-2)(n+1-j,2)$. In particular, $\tau(n-2)=j$. We have $X(n-1,\tau, {\bf 1}, [(0,1)])\in \partial\overline{\cC}_1$. Consider a multi-scale differential $X(n-1,\tau, {\bf 1}, [(0,-1)])$, with ramification profile $\cP_4$, interchanging $\tau(1)=1$ and $\tau(n-2)=j$. Then by \Cref{breakinglemma2}, we have $B'(\cC_1)=B'(\cC_4)$. Similarly, we can deduce $B'(\cC_2)=B'(\cC_4)$. Therefore $B'(\cC_1)=B'(\cC_2)=B'(\cC_4)$. Symmetrically, we also have $B'(\cC_1)=B'(\cC_2)$ for each $\cP_2 = (\frac{n+1}{2},j)\circ \cP_1 \circ (\frac{n+1}{2},j)$, $1\leq j<\frac{n+1}{2}$. Note that $(1,\frac{n+1}{2})=(1,2)(2,\frac{n+1}{2})(1,2)$. Therefore, we conclude that $B'(\cC_1)=B'(\cC_2)$.

Finally, assume that $a_1$ is odd. We will show that $B'(\cC_1)=B'(\cC_2)$ when $\cP_1$ fixes two marked points and $\cP_2$ fixes four marked points. 

If $\cP_1$ fixes two marked points and $a_1$ is odd, then we consider $X(n-1,\tau, {\bf 1}, Pr)\in \partial\overline{\cC}_1$ containing the only fixed pole in the bottom level component. By relabeling the saddle connections, we may assume that $\cP_1 (\tau(i))=\cP_1(\tau(n+1-i))$ for $i=1,\dots, n$ and $Pr=[(0,0)]$. Consider a multi-scale differential $X(n-1,\tau, {\bf 1}, [(0,a_1)])$ with ramification profile $\cP_3$. Since $a_1$ is odd, $\tau(\frac{a_1+1}{2})$ and $\tau(\frac{a_1+n}{2})$ are also fixed by $\cP_3$. By (1) of \Cref{breakinglemma2}, we have $B'(\cC_1)=B'(\cC_3)$. We already have $B'(\cC_3)=B'(\cC_2)$ since $\cP_3$ fixes four marked points. Therefore, $B'(\cC_1)=B'(\cC_2)=B'(\cC_3)$.
\end{proof}

\begin{proposition}\label{nonexistspecial}
    The strata $\cR_1(r,r,-2r)$ and $\cR_1(r,r,-r,-r)$ does not have any non-hyperelliptic connected component with rotation number $r$. 
\end{proposition}
\begin{proof}
    If $\cR_1(r,r,-2r)$ has a non-hyperelliptic component $\cD$ with rotation number $r$, then by \Cref{simple}, this component is adjacent to a non-hyperelliptic component $\cC$ of $\cR_1(2r,-2r)$ with rotation number $R$. Then $\gcd(R,r)=r$, so $R=r$ or $2r$. This is contradiction to \Cref{exception1}. Therefore $\cD$ does not exist. The same argument works for $\cR_1(r,r,-r,-r)$ with \Cref{exception2}. 
\end{proof}

By combining above propositions and \Cref{nonminimal1}, we have the following

\begin{lemma} \label{nonminimalbase}
Let $D\coloneqq \gcd(a_1,d)$ and $r|D$. Suppose that $\mu\neq (r,r,-2r), (r,r,-r,-r)$ or $(n,n,-2^n)$. The stratum $\cR_1(\mu)=\cR_1 (a_1, a_2, -b_1,\dots, -b_n)$ has a unique non-hyperelliptic connected component with rotation number $r$.
\end{lemma}

\begin{proof}
If $n=1$, then this is a usual meromorphic stratum and therefore proven in \cite{boissymero}. So assume that $n>1$. First, suppose that $\cR_1(\mu')$ is not one of the special strata dealt in \Cref{subsec:special}. Fix any positive integer $R_1|d$ such that $\gcd(a_1,R_1)=r$. Then by \Cref{main2minimal}, there exists a unique non-hyperelliptic component $\cC_{R_1}$ of $\cR_1(\mu')$ with rotation number $R_1$. Let $\cD_1=B(\cC_{R_1})$. Then the rotation number of $\cD_1$ is equal to $r$. Assume the contrary --- that there exists another non-hyperelliptic connected component $\cD_2$ of $\cR_1(\mu)$ with rotation number $r$. Then by \Cref{nonhypermerge}, $\cD_2=B(\cC_{R_2})$ for some $R_2|d$. We have $r=\gcd(a_1,R_2)=\gcd(a_1,R_1)$. By \Cref{nonminimal1}, we have $\cD_1=B(\cC_{R_1})=B(\cC_{R_2})=\cD_2$. 

Now suppose that $\mu'=(2n,-2^n)$. If $a_1$ is odd, then $\cR_1 (a_1,a_2, -2^n)$ is connected by \Cref{nonminimalexception} and the rotation number is equal to $\gcd(a_1,2)=1$. If $a_1$ is even, then $\cR_1(a_1,a_2,-2^n)$ has at most two connected components by \Cref{nonminimalexception}. Since $\gcd(a_1,2)=2$, there are at least two connected components, corresponding to rotation numbers $r=1,2$.  

If $\mu'=(12,-3^4)$ and $r=3$, then there are exactly two possible cases $\mu=(3,9,-3^4)$ and $(6,6,-3^4)$, which is proven by \Cref{specialmerge1} and \Cref{specialmerge2}. If $\mu'=(2r,-2r)$ or $(2r,-r,-r)$, then $r|D$ if and only if $m=2$ and $a_1=a_2=r$, which is excluded by assumption. 
\end{proof}

Now we are ready to prove \Cref{main2}.

\begin{proof}[Proof of \Cref{main2}]
Let $\cR_1(\mu)=\cR_1(a_1,\dots,a_m,-b_1,\dots,-b_n)$. We denote $a=a_1+\dots+a_m$, $d=\gcd(a,b_1,\dots,b_n)$ and $D=\gcd(a_1,\dots,a_m,d)$. For each $r|D$, we need to prove that there exists a unique non-hyperelliptic component of $\cR_1(\mu)$ with rotation number $r$. Denote the corresponding \MIN stratum by $\cR_1(\mu')$.

If $\mu'=(r,-r)$, then $r|D$ only if $m=1$. If $\mu'=(2r,-2r)$ or $(2r,-r,-r)$, then $r|D$ if and only if $m=1$ or $m=2$ and $a_1=a_2=r$. These cases are exceptional cases and proven in \cite{boissymero}, \Cref{exception1}, \Cref{exception2} and \Cref{nonexistspecial}. 

If $\mu'=(12,-3^4)$ and $r=3$, then the case $m=1$ is exceptional and there are in fact exactly two non-hyperelliptic components. This is proven in \cite{LT}. The cases when $m=2$ are proven in \Cref{specialmerge1} and \Cref{specialmerge2}. If $m>2$, then $\cR_1(\mu)$ has at least one zero of order 3 if $3|D$. By \Cref{simple} and \Cref{merging}, any non-hyperelliptic component $\cD$ of $\cR_1(\mu)$ with rotation number 3 is adjacent to the unique non-hyperelliptic component of $\cR_1(3,9,-3^4)$ with rotation number 3. So $\cD$ is unique. 

Let $\mu'=(2n,-2^n)$. If $m=2$ and $a_1=a_2=n$, then $\cR_1 (\mu)$ has no non-hyperelliptic component by \Cref{nnspecial}. First, suppose that $a_i$ is odd for some $i$. Then $D=1$ and 1 is the only possible rotation number. Let $\cD$ be a non-hyperelliptic component of $\cR_1 (\mu)$. By \Cref{nonhypermerge}, $\cD$ is adjacent to a connected component of $\cC$ of $\cR_1(\mu')$. We can break up the zero of $\cC$ to obtain the stratum $\cR_1(a_i, a-a_i,-2^n)$, which is connected by \Cref{nonminimalbase}. So $\cR_1 (\mu)$ is connected. Now suppose that all $a_i$ are even. Then $D=2$ and there are two possible rotation numbers $r=1,2$. Since $m>2$, there exists $i$ such that $a_i\neq a-a_i$. By \Cref{nonminimalbase}, the stratum $\cR_1(a_i, a-a_i,-2^n)$ has two (non-hyperelliptic) connected components, thus $\cR_1 (\mu)$ has two (non-hyperelliptic) connected components corresponding to $r=1,2$.

Now assume that $\mu'\neq (12,-3^4),(2n,-2^n),(2r,-2r)$ or $(2r,-r,-r)$. Then there exists a unique non-hyperelliptic component $\cC_r$ of $\cR_1 (\mu')$ with rotation number $r$ by \Cref{main2minimal}. By breaking up the zero from $\cC_r$, we obtain a non-hyperelliptic component of $\cR_1(\mu)$ with rotation number $r$, proving the existence. We need to prove the uniqueness. Let $\cD$ be a non-hyperelliptic component of $\cR_1 (\mu)$ with rotation number $r$. By \Cref{nonhypermerge}, $\cD$ is adjacent to a non-hyperelliptic component $\cC_R$ of $\cR_1 (\mu')$ with rotation number $R|d$. Assume that $R$ is the smallest among the components that $\cD$ is adjacent to. Then the rotation number $r$ of $\cD$ is equal to $\gcd(R,D)$. It suffices to prove that $R=r$ because then we have $\cC=B(\cC_r)$ and $\cC$ is unique. For each $i=1,\dots,m$, we can break up the zero of $\cC_R$ into two zeroes of orders $a_i$ and $a-a_i$ to obtain a non-hyperelliptic component $\cD^i$ of $\cR_1 (a_i,a-a_i,-b_1,\dots,-b_n)$. By \Cref{nonminimalbase}, $\cD^i$ is a unique non-hyperelliptic component with rotation number $R_i\coloneqq \gcd(R,a_i)\leq R$. If $R_i<R$, then by breaking up the zero of $\cC_{R_i}$, we can again obtain $\cD^i$. So $\cC$ is adjacent to $\cC_{R_i}$ and this contradicts to the minimality of $R$. Thus $R_i=R$ for each $i$ and therefore $R=\gcd(R,a_1,\dots,a_m)=\gcd(R,D)=r$, as desired.  
\end{proof}

\section{Higher genus strata} \label{sec:hg}

In this section, we will classify the connected components of strata $\cR_g (\mu)$ for genus $g>1$, completing the proof of \Cref{main1}

\subsection{Higher genus \MIN strata} \label{sec:hgm}

First, we work with a \MIN stratum of genus $g>1$. We will prove that any non-hyperelliptic connected component of $\cR_g (\mu)$ can be obtained by a surgery called {\em bubbling a handle} from a connected component of a genus $g-1$ \MIN stratum. Here is one difference from \cite{boissymero} and \cite{kozo1}: a flat surface in the genus one residueless \MIN stratum $\cR_1(\mu)$ cannot be obtained by bubbling a handle from a genus zero residueless flat surface (because there do not exist such flat surfaces). So our base case of the induction has to be $g=1$, not $g=0$. This is why we treated genus one strata separately in \Cref{sec:g1m}. Even though we have different base cases, bubbling machinery will still allow us to enumerate the connected components of $\cR_g (\mu)$ similarly to \cite{boissymero} and \cite{kozo1}. 

\subsection{Unbubbling a handle}

Recall that {\em bubbling a handle at} $z$ operation $\oplus_{z}$ is given in \Cref{subsec:bubble}, as in \cite{boissy}. Since we are dealing with \MIN strata in this section, we will drop $z$ in the notation and simply write it as $\oplus$. Recall that for a connected component $\cC$ of a \MIN stratum $\cR_g(\mu)$, we have $\cC\oplus s_1 = \cC\oplus s_2$ if $\gcd(a+2,s_1)=\gcd(a+2,s_2)$. This is because the multi-scale differential used for bubbling a handle with angle $2\pi s$ is contained in $\overline{\cH}_1(a+2,-a-2)$, and its rotation number is equal to $\gcd(a+2,s)$. So we can always assume that $s|a+2$, by replacing $s$ by $\gcd(a+2,s)$. 

In fact, we can further extend the range of multi-scale differentials used for bubbling a handle. Consider a two-level multi-scale differential $X(Id,C_1,[(0,v)])\in \overline{\cH}_1(a+2,-a-2)$. By \Cref{rotation}, the rotation number of this multi-scale differential is given by $\gcd(a+2,C_1,v)$. If $\gcd(a+2,C_1,v)=s$, then the component $\cC\oplus s$ can be obtained by gluing the pole of $X(Id,C_1,[(0,v)])$ to the zero of a flat surface in $\cC$ and plumbing all the nodes. 

Here we give the residueless version of \cite[Lemma~14]{kozo1} and \cite[Prop~6.1]{boissymero}.

\begin{lemma} \label{unbubble}
Let $\cC$ be a non-hyperelliptic connected component of a \MIN stratum $\cR_g (\mu)$ of genus $g>1$. Then there exists a connected component $\cC'$ of $\cR_{g-1} (a-2,-b_1,\dots,-b_n)$ such that $\cC = \cC' \oplus s$ for some $1\leq s\leq a-1$. If $b_n>2$, then $\cC'$ can be chosen to be non-hyperelliptic.
\end{lemma}

\begin{proof}
By \Cref{simple}, there exists a flat surface $X\in \cC$ that has a multiplicity one saddle connection $\gamma$. By shrinking $\gamma$, we obtain $\overline{X}\in \partial\overline{\cC}$. The top level component $X_0$ has zeroes at the two nodes and its genus is equal to $g-1$. If $X_0$ is contained in a non-hyperelliptic stratum, then we can merge two zeroes of $X_0$ and obtain a curve $X'$ in a component $\cC'$ of $\cR_{g-1} (a-2,-b_1,\dots,-b_n)$. If $b_n>2$, then $\cC'$ can be chosen to be non-hyperelliptic by \Cref{nonhypermerge}. Therefore $\cC=\cC'\oplus s$ for some $s$, see \Cref{fig1001}. 

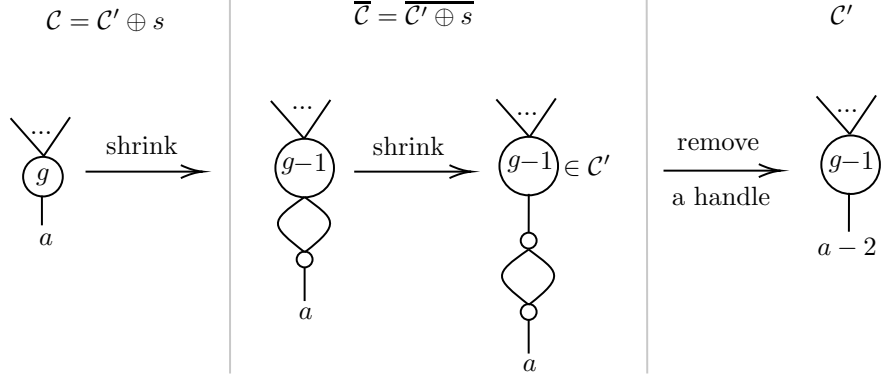
\begin{figure}
    \centering
    \tikzset{every picture/.style={line width=0.75pt}} 

\begin{tikzpicture}[x=0.75pt,y=0.75pt,yscale=-1,xscale=1]

\draw    (45.64,120.23) -- (45.64,135.23) ;
\draw    (29.83,81.17) -- (46.58,100.17) ;
\draw    (59.83,80.67) -- (46.58,100.17) ;
\draw    (67.64,108.13) -- (124.64,108.13) ;
\draw [shift={(126.64,108.13)}, rotate = 180] [color={rgb, 255:red, 0; green, 0; blue, 0 }  ][line width=0.75]    (10.93,-3.29) .. controls (6.95,-1.4) and (3.31,-0.3) .. (0,0) .. controls (3.31,0.3) and (6.95,1.4) .. (10.93,3.29)   ;
\draw    (203.14,108.13) -- (260.14,108.13) ;
\draw [shift={(262.14,108.13)}, rotate = 180] [color={rgb, 255:red, 0; green, 0; blue, 0 }  ][line width=0.75]    (10.93,-3.29) .. controls (6.95,-1.4) and (3.31,-0.3) .. (0,0) .. controls (3.31,0.3) and (6.95,1.4) .. (10.93,3.29)   ;
\draw    (274.83,71.67) -- (291.58,90.67) ;
\draw    (304.83,71.17) -- (291.58,90.67) ;
\draw   (174.13,152.45) .. controls (174.13,150.25) and (175.92,148.45) .. (178.13,148.45) .. controls (180.34,148.45) and (182.13,150.25) .. (182.13,152.45) .. controls (182.13,154.66) and (180.34,156.45) .. (178.13,156.45) .. controls (175.92,156.45) and (174.13,154.66) .. (174.13,152.45) -- cycle ;
\draw    (178.13,173.45) -- (178.13,156.45) ;
\draw    (178.13,120.95) .. controls (159.66,133.68) and (160.71,135.07) .. (178.18,148.84) ;
\draw    (178.13,120.95) .. controls (193.16,132.18) and (197.21,134.07) .. (178.18,148.84) ;
\draw    (161.13,72.45) -- (177.88,91.45) ;
\draw    (191.13,71.95) -- (177.88,91.45) ;
\draw   (287.13,142.95) .. controls (287.13,140.75) and (288.92,138.95) .. (291.13,138.95) .. controls (293.34,138.95) and (295.13,140.75) .. (295.13,142.95) .. controls (295.13,145.16) and (293.34,146.95) .. (291.13,146.95) .. controls (288.92,146.95) and (287.13,145.16) .. (287.13,142.95) -- cycle ;
\draw    (291.18,119.34) -- (291.13,138.95) ;
\draw   (287.13,178.95) .. controls (287.13,176.75) and (288.92,174.95) .. (291.13,174.95) .. controls (293.34,174.95) and (295.13,176.75) .. (295.13,178.95) .. controls (295.13,181.16) and (293.34,182.95) .. (291.13,182.95) .. controls (288.92,182.95) and (287.13,181.16) .. (287.13,178.95) -- cycle ;
\draw    (291.13,199.95) -- (291.13,182.95) ;
\draw    (291.13,147.45) .. controls (272.66,160.18) and (273.71,161.57) .. (291.18,175.34) ;
\draw    (291.13,147.45) .. controls (306.16,158.68) and (310.21,160.57) .. (291.18,175.34) ;
\draw    (359.14,107.63) -- (416.14,107.63) ;
\draw [shift={(418.14,107.63)}, rotate = 180] [color={rgb, 255:red, 0; green, 0; blue, 0 }  ][line width=0.75]    (10.93,-3.29) .. controls (6.95,-1.4) and (3.31,-0.3) .. (0,0) .. controls (3.31,0.3) and (6.95,1.4) .. (10.93,3.29)   ;
\draw    (436.33,70.17) -- (453.08,89.17) ;
\draw    (466.33,69.67) -- (453.08,89.17) ;
\draw    (452.68,118.84) -- (452.63,138.45) ;
\draw [color={rgb, 255:red, 200; green, 200; blue, 200 }  ,draw opacity=1 ]   (350.65,20.55) -- (350.93,210.02) ;
\draw [color={rgb, 255:red, 200; green, 200; blue, 200 }  ,draw opacity=1 ]   (140.45,20) -- (140.53,210.02) ;

\draw    (46.16, 110.67) circle [x radius= 10, y radius= 10]   ;
\draw (46.16,110.67) node   [align=left] {$\displaystyle g$};
\draw (45.32,90.67) node [anchor=south] [inner sep=0.75pt]   [align=left] {...};
\draw (47.78,142.17) node    {$a$};
\draw (96.13,94.5) node   [align=left] {shrink};
\draw (230.63,95) node   [align=left] {shrink};
\draw (290.32,81.17) node [anchor=south] [inner sep=0.75pt]   [align=left] {...};
\draw (46.6,24.38) node [anchor=north west][inner sep=0.75pt]    {$\mathcal{C} =\mathcal{C} '\oplus s$};
\draw (202,20.48) node [anchor=north west][inner sep=0.75pt]    {$\overline{\mathcal{C}} =\overline{\mathcal{C} '\oplus s}$};
\draw (442.1,22.28) node [anchor=north west][inner sep=0.75pt]    {$\mathcal{C} '$};
\draw (177.12,75.96) node   [align=left] {...};
\draw (178.79,179.45) node    {$a$};
\draw (291.79,205.95) node    {$a$};
\draw (386.63,94.5) node   [align=left] {remove};
\draw (451.82,79.67) node [anchor=south] [inner sep=0.75pt]   [align=left] {...};
\draw (451.79,146.45) node    {$a-2$};
\draw (307.1,98.28) node [anchor=north west][inner sep=0.75pt]    {$\in \mathcal{C} '$};
\draw (388.13,120.36) node   [align=left] {a handle};
\draw    (291.17, 105.31) circle [x radius= 14.77, y radius= 14.77]   ;
\draw (291.17,105.31) node  [font=\scriptsize,color={rgb, 255:red, 255; green, 255; blue, 255 }  ,opacity=1 ] [align=left] {$\displaystyle g-1$};
\draw    (453.5, 104.64) circle [x radius= 14.77, y radius= 14.77]   ;
\draw (453.5,104.64) node  [font=\scriptsize,color={rgb, 255:red, 255; green, 255; blue, 255 }  ,opacity=1 ] [align=left] {$\displaystyle g-1$};
\draw    (178.17, 106.31) circle [x radius= 14.77, y radius= 14.77]   ;
\draw (178.17,106.31) node  [font=\scriptsize,color={rgb, 255:red, 255; green, 255; blue, 255 }  ,opacity=1 ] [align=left] {$\displaystyle g-1$};
\draw (180.95,104) node    {$-1$};
\draw (169.95,105) node    {$g$};
\draw (294.95,103) node    {$-1$};
\draw (283.95,104) node    {$g$};
\draw (456.95,103) node    {$-1$};
\draw (445.95,104) node    {$g$};

\end{tikzpicture} 
    \caption{Unbubbling a handle} \label{fig1001}
\end{figure}

If $X_0$ is hyperelliptic, we will move around more in $\partial\overline{\cC}$ until we land on the case when $X_0$ is non-hyperelliptic. First, assume that $X_0$ still has a multiplicity one saddle connection. Then by shrinking this saddle connection, we can merge two zeroes of $X_0$ and degenerate to a \MIN hyperelliptic component $\cC_1$ of $\cR_{g-1} (a-2,-b_1,\dots,-b_n)$. Therefore, $\cC=\cC_1\oplus s$. If $a$ is even and $s=\frac{a}{2}$, then the rotation number of the flat surface $E$ in $\overline{\cR}_1(a,-a)$ used for bubbling a handle is equal to $\frac{a}{2}$. In particular, the flat surface in hyperelliptic. So $\cC$ is hyperelliptic since it is obtained by gluing two hyperelliptic components. This is a contradiction, so $s\neq \frac{a}{2}$. In particular, $E$ is non-hyperelliptic and thus we can degenerate $E$ into a two-level multi-scale differential in $\overline{\cR}_1(a,-a)$ such that $Q_1<Q_2$ by \Cref{unbalance}. By plumbing the level transition between the levels 0 and -1, we obtain a two-level multi-scale differential with the top level component contained in the stratum $\cR_{g-1} (Q_1-1,Q_2-2,-b_1,\dots,-b_n)$. This is exactly the case when $X_0$ is non-hyperelliptic. See the upper line of \Cref{fig1002}.

\begin{figure}
    \centering
    \input{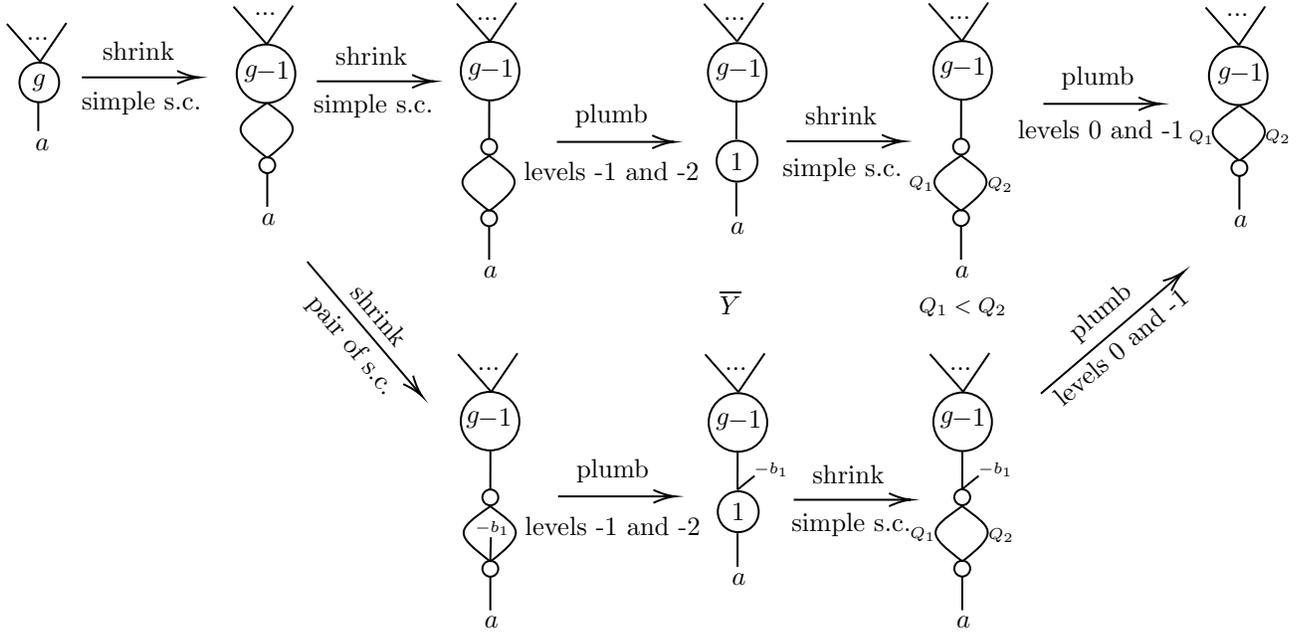} 
    \caption{Finding a non-hyperelliptic $X_0$} \label{fig1002}
\end{figure}

Now we assume that any continuous deformation of $X_0$ does not have a multiplicity one saddle connection. By \Cref{simple}, this means $X_0$ is contained in a \NMIN hyperelliptic component with $2g+2$ fixed marked points. Let $p_1$ be a fixed pole of the smallest order. By \Cref{double}, we may assume that $X_0$ has a pair of parallel saddle connections with multiplicity two bounding the polar domain of $p_1$. By shrinking this pair of saddle connections, we obtain a three-level multi-scale differential in $\partial\overline{\cC}$. By plumbing the level transition between the levels -1 and -2, we obtain a two-level multi-scale differential $\overline{Y}$ whose bottom level component $Y_{-1}$ is contained in $\cR_1 (a,-b_1,-(a-b_1))$. By assumption on $p_1$, we have $a-b_1>b_1$. Therefore, if the bottom level component $Y_{-1}$ is hyperelliptic, then the ramification profile of $Y_{-1}$ must fix both poles. In particular, it fixes the node between the two levels. Since the top level component $Y_0$ is still hyperelliptic, and two components of $\overline{Y}$ intersect at one node, we can conclude that $\cC$ is hyperelliptic by \Cref{hyper1}. This is a contradiction, so $Y_{-1}$ is not hyperelliptic. Then by \Cref{unbalance}, $Y_{-1}$ degenerates into a two-level differential such that $Q_1< Q_2$. Again, by plumbing the level transition between the levels 0 and -1, we land on the case when $X_0$ is non-hyperelliptic. See the lower line of \Cref{fig1002}.
\end{proof}

By applying this lemma repeatedly, we obtain the following

\begin{corollary} \label{unbubblebubble}
Assume that $\mu \neq (2n+2g-2,-2^n)$. Let $\cC$ be a non-hyperelliptic component of a \MIN stratum $\cR_g (\mu)$ of genus $g>1$. Then there exists a non-hyperelliptic component $\cC'$ of $\cR_1 (a-2(g-1), -b_1,\dots, -b_n)$ and integers $1\leq s_i \leq a-2(g-1)+2i-1$ for $i=1,\dots,g-1$ such that $$\cC=\cC' \oplus s_1 \oplus \dots \oplus s_{g-1}.$$
\end{corollary}

Non-hyperelliptic components of a genus one stratum $\cR_1(\mu)$ are classified by rotation number by \Cref{main2}, which we already proved in \Cref{sec:g1n}. So the formula in the above corollary can be rewritten as 
\begin{equation} \label{unbubblebubbleformula}
\cC=\cC_r \oplus s_1 \oplus \dots \oplus s_{g-1}.  
\end{equation}
where $\cC_r$ is the unique non-hyperelliptic component with rotation number $r$.

We prove the following proposition, which is the residueless version of \cite[Prop.~6.2]{boissymero}.

\begin{proposition} \label{bubblereduction}
Assume that $b_n>2$. Then any non-hyperelliptic component of the \MIN stratum $\cR_g(\mu)$ of genus $g>1$ is obtained by repeatedly bubbling a handle as one of the following two possibilities:
$$\cC_1 \oplus 1 \oplus \dots \oplus 1\oplus 1\,,$$
$$\cC_1 \oplus 2 \oplus \dots \oplus 1\oplus 1\,,$$ 
where $\cC_1$ is the unique non-hyperelliptic component of $\cR_1 (a-2(g-1), -b_1,\dots, -b_n)$ with rotation number one. Moreover, the two components above are the same if and only if any $b_i$ is odd. 
\end{proposition}

\begin{proof}
First, we will prove that we can reduce to the case $g=2$. Assume that in genus 2, $\cC_r\oplus s$ for any $r|\gcd(b_i)$, $1\leq s\leq a-2g+2$ is equal to one of the two possibilities ---$\cC_1 \oplus 1$ and $\cC_1 \oplus 2$. By \Cref{unbubblebubbleformula}, then in any genus $g>1$, we can write $\cC=\cC_1\oplus 1\oplus s_2\oplus \dots \oplus s_{g-1}$ or $\cC_1\oplus 2\oplus s_2\oplus \dots \oplus s_{g-1}$. Consider the $g$-level graph we obtain when bubbling $g-1$ handles from $\cC_1$, depicted in \Cref{fig1004}. Let $\overline{X}$ be a $g$-level multi-scale differential corresponding to the graph. The top level component of $\overline{X}$ is contained in $\cC_1$, a connected component of $\cR_1(a-2(g-1),-b_1,\dots,-b_n)$. If $a-2(g-1)\leq 4$, then this stratum is either $\cR_1(2,-2)$, $\cR_1(3,-3)$, $\cR_1(4,-4)$ or $\cR_1(4,-2^2)$. Each of these strata are already treated in \cite{boissymero} (for the first three strata) and in \Cref{exception0} (for the last one). So we may assume that $a-2(g-1)>4$. At level -1 of $\overline{X}$, we have a flat surface in $\cH_1(a-2(g-1),-a+2(g-1))$ of rotation number one or two. So this is not hyperelliptic, as the rotation number of the hyperelliptic component of $\cH_1(a-2(g-1),-a+2(g-1))$ is equal to $\frac{a-2(g-1)}{2}>2$. We can plumb the all level transitions of $\overline{X}$ except for the top level. Then we obtain at level -1 a non-hyperelliptic flat surface $Y$ in $\cH_{g-1}(a,2g-2-a)$. Again by \cite{boissymero}, we can conclude that the connected component of $\cH_{g-1}(a,2g-2-a)$ containing $Y$ is one of the two possibilities: $$\cH_0 \oplus 1\oplus\dots\oplus 1\,,$$ $$\cH_0 \oplus 2\oplus\dots\oplus 1\,,$$ where $\cH_0$ denotes the connected stratum $\cH_0(a-2g,-a+2(g-1))$. This ends the proof of the reduction to $g=2$. 

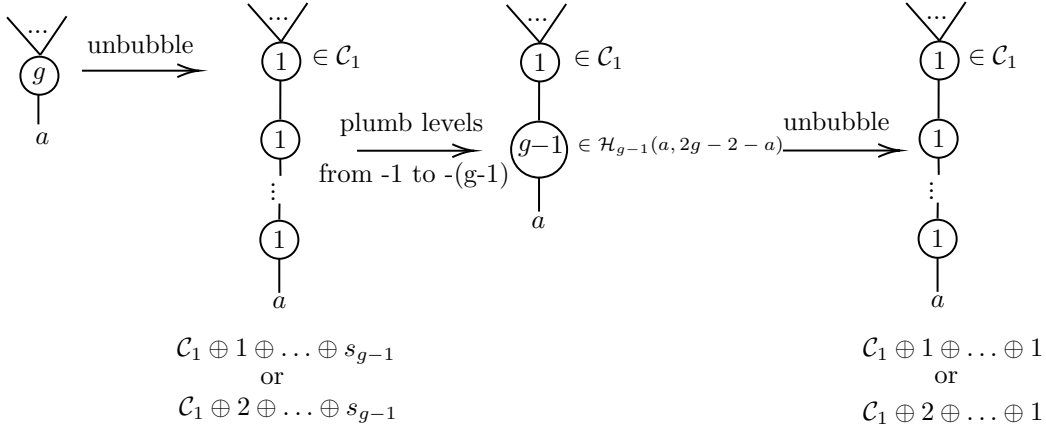
\begin{figure}
    \centering
    \tikzset{every picture/.style={line width=0.75pt}} 

\begin{tikzpicture}[x=0.75pt,y=0.75pt,yscale=-1,xscale=1]

\draw    (45.64,67.23) -- (45.64,82.23) ;
\draw    (29.83,28.17) -- (46.58,47.17) ;
\draw    (59.83,27.67) -- (46.58,47.17) ;
\draw    (67.64,55.13) -- (124.64,55.13) ;
\draw [shift={(126.64,55.13)}, rotate = 180] [color={rgb, 255:red, 0; green, 0; blue, 0 }  ][line width=0.75]    (10.93,-3.29) .. controls (6.95,-1.4) and (3.31,-0.3) .. (0,0) .. controls (3.31,0.3) and (6.95,1.4) .. (10.93,3.29)   ;
\draw    (151.33,21.47) -- (168.08,40.47) ;
\draw    (181.33,20.97) -- (168.08,40.47) ;
\draw    (167.68,60.15) -- (167.63,79.76) ;
\draw    (167.73,108.22) -- (167.63,99.76) ;
\draw    (206.34,95.53) -- (263.34,95.53) ;
\draw [shift={(265.34,95.53)}, rotate = 180] [color={rgb, 255:red, 0; green, 0; blue, 0 }  ][line width=0.75]    (10.93,-3.29) .. controls (6.95,-1.4) and (3.31,-0.3) .. (0,0) .. controls (3.31,0.3) and (6.95,1.4) .. (10.93,3.29)   ;
\draw    (421.14,95.43) -- (478.14,95.43) ;
\draw [shift={(480.14,95.43)}, rotate = 180] [color={rgb, 255:red, 0; green, 0; blue, 0 }  ][line width=0.75]    (10.93,-3.29) .. controls (6.95,-1.4) and (3.31,-0.3) .. (0,0) .. controls (3.31,0.3) and (6.95,1.4) .. (10.93,3.29)   ;
\draw    (167.23,122.22) -- (167.13,130.26) ;
\draw    (167.13,167.26) -- (167.13,150.26) ;
\draw    (281.83,21.97) -- (298.58,40.97) ;
\draw    (311.83,21.47) -- (298.58,40.97) ;
\draw    (298.18,60.65) -- (298.13,80.26) ;
\draw    (298.13,126.26) -- (298.13,109.26) ;
\draw    (483.33,20.97) -- (500.08,39.97) ;
\draw    (513.33,20.47) -- (500.08,39.97) ;
\draw    (499.68,59.65) -- (499.63,79.26) ;
\draw    (499.73,107.72) -- (499.63,99.26) ;
\draw    (499.23,121.72) -- (499.13,129.76) ;
\draw    (499.13,166.76) -- (499.13,149.76) ;

\draw    (46.16, 57.67) circle [x radius= 9.6, y radius= 9.6]   ;
\draw (46.16,57.67) node   [align=left] {$\displaystyle g$};
\draw (45.32,37.67) node [anchor=south] [inner sep=0.75pt]   [align=left] {...};
\draw (47.78,89.17) node    {$a$};
\draw (97.13,41.5) node   [align=left] {unbubble};
\draw (114.1,188.38) node [anchor=north west][inner sep=0.75pt]    {$\mathcal{C}_{1} \oplus 1\oplus \dotsc \oplus s_{g-1}$};
\draw (166.82,30.97) node [anchor=south] [inner sep=0.75pt]   [align=left] {...};
\draw (167.29,172.26) node    {$a$};
\draw    (167.66, 89.97) circle [x radius= 9.6, y radius= 9.6]   ;
\draw (167.66,89.97) node   [align=left] {$\displaystyle 1$};
\draw (233.83,82.41) node   [align=left] {plumb levels};
\draw (182,41.33) node [anchor=north west][inner sep=0.75pt]    {$\in \mathcal{C}_{1}$};
\draw    (168.16, 50.47) circle [x radius= 9.6, y radius= 9.6]   ;
\draw (168.16,50.47) node   [align=left] {$\displaystyle 1$};
\draw    (167.16, 140.47) circle [x radius= 9.6, y radius= 9.6]   ;
\draw (167.16,140.47) node   [align=left] {$\displaystyle 1$};
\draw (297.32,31.47) node [anchor=south] [inner sep=0.75pt]   [align=left] {...};
\draw (297.79,131.76) node    {$a$};
\draw    (298.66, 50.97) circle [x radius= 9.6, y radius= 9.6]   ;
\draw (298.66,50.97) node   [align=left] {$\displaystyle 1$};
\draw (161.82,116.17) node [anchor=south] [inner sep=0.75pt]  [rotate=-90] [align=left] {...};
\draw (448.13,80.5) node   [align=left] {unbubble};
\draw (498.82,30.47) node [anchor=south] [inner sep=0.75pt]   [align=left] {...};
\draw (499.29,171.76) node    {$a$};
\draw    (499.66, 89.47) circle [x radius= 9.6, y radius= 9.6]   ;
\draw (499.66,89.47) node   [align=left] {$\displaystyle 1$};
\draw    (500.16, 49.97) circle [x radius= 9.6, y radius= 9.6]   ;
\draw (500.16,49.97) node   [align=left] {$\displaystyle 1$};
\draw    (499.16, 139.97) circle [x radius= 9.6, y radius= 9.6]   ;
\draw (499.16,139.97) node   [align=left] {$\displaystyle 1$};
\draw (493.82,115.67) node [anchor=south] [inner sep=0.75pt]  [rotate=-90] [align=left] {...};
\draw (314,41.83) node [anchor=north west][inner sep=0.75pt]    {$\in \mathcal{C}_{1}$};
\draw (316,86.83) node [anchor=north west][inner sep=0.75pt]  [font=\scriptsize]  {$\in \mathcal{H}_{g-1}( a,2g-2-a)$};
\draw (460.1,188.53) node [anchor=north west][inner sep=0.75pt]    {$\mathcal{C}_{1} \oplus 1\oplus \dotsc \oplus 1$};
\draw (513,41.33) node [anchor=north west][inner sep=0.75pt]    {$\in \mathcal{C}_{1}$};
\draw (460.1,221.03) node [anchor=north west][inner sep=0.75pt]    {$\mathcal{C}_{1} \oplus 2\oplus \dotsc \oplus 1$};
\draw (496.14,205.31) node [anchor=north west][inner sep=0.75pt]   [align=left] {or};
\draw (115.1,218.38) node [anchor=north west][inner sep=0.75pt]    {$\mathcal{C}_{1} \oplus 2\oplus \dotsc \oplus s_{g-1}$};
\draw (156.14,206.05) node [anchor=north west][inner sep=0.75pt]   [align=left] {or};
\draw (235.13,108.14) node   [align=left] {from -1 to -(g-1)};
\draw    (298.17, 94.97) circle [x radius= 14.12, y radius= 14.12]   ;
\draw (298.17,94.97) node  [font=\scriptsize,color={rgb, 255:red, 255; green, 255; blue, 255 }  ,opacity=1 ] [align=left] {$\displaystyle g-1$};
\draw (301.95,93) node    {$-1$};
\draw (290.95,94) node    {$g$};

\end{tikzpicture} 
    \caption{Reduction to $g=2$} \label{fig1004}
\end{figure}

Therefore, if we prove that in genus 2 any component $\cC_r \oplus s$ is equal to one of the two possibilities $\cC_1 \oplus 1$ and $\cC_1 \oplus 2$, then we can complete the proof. In fact, we will prove that for $g=2$, any nonhyperelliptic component is equal to $\cC_r\oplus1$ for some $r=1,2$, and $\cC_2\oplus 1=\cC_1\oplus 2$ when $\cC_2$ exists. 

Let $\cC=\cC_r\oplus s$. Since $\cC_r \oplus s_1=\cC_r \oplus s_2$ if $\gcd(s_1,a+2)=\gcd(s_2,a+2)$ (see \Cref{subsec:bubble}), we may assume that $s|a+2$. Thus $\gcd(d,s)\leq \gcd(a,s)\leq 2$. 

First, assume that $\cC_r$ is not a connected component of the special strata $\cR_1(12,-3^4)$ or $\cR_1(2n+2,-2^{n-1},-4)$ for odd $n$, studied in \Cref{subsec:special}. By \Cref{unbalance}, then there exists a multi-scale differential $\overline{X}\in \partial\overline{\cC_r}$ with $Q_1<Q_2$. By plumbing the nodes of $\overline{X}$, we obtain a flat surface $X\in \cC_r$ with a multiplicity one saddle connection $\alpha$ of index $R=Q_1$. In particular, $R\neq\frac{a}{2}$. We can find another simple closed curve $\gamma$ so that $\{\alpha,\gamma\}$ form a symplectic basis of $X$. Then $r=\gcd(d,\Ind\alpha,\Ind\gamma)$ and $R=\Ind\alpha$ is divisible by $r$. 

Since bubbling a handle is a local surgery, the surface $X'\in \cC_r \oplus s$ still has a multiplicity one saddle connection $\alpha'$ deformed from $\alpha$. The index of $\alpha'$ depends on the direction that we bubble the handle, i.e., on the choice of prong-matching used to plumb the node in the middle of \Cref{fig406}. Let $\beta$ be the cross curve of the cylinder that we bubbled at $z$. The index of $\beta$ is equal to $s$. After bubbling a handle, $\beta$ deforms to a saddle connection in $X'$, denoted by $\beta'$. So $\alpha'$ and $\beta'$ intersect at $z$. Since $R\neq \frac{a}{2}$ and $s\leq \frac{a+2}{2}$, it is possible to choose a direction of bubbling a handle so that $\alpha'$ and $\beta'$ intersect non-transversely, as in \Cref{fig1005}. This is not possible only when $\cC_r$ is a special connected component introduced in \Cref{subsec:special}, and $s=\frac{a+2}{2}$. These cases will be treated separately in \Cref{subsec:specialbubble}.

\begin{figure}
    \centering
    \tikzset{every picture/.style={line width=0.75pt}} 

\begin{tikzpicture}[x=0.75pt,y=0.75pt,yscale=-1,xscale=1]

\draw [color={rgb, 255:red, 255; green, 0; blue, 0 }  ,draw opacity=1 ]   (209.09,45.5) -- (257.34,104.88) ;
\draw [color={rgb, 255:red, 255; green, 0; blue, 0 }  ,draw opacity=1 ]   (257.34,104.88) -- (212.29,163.1) ;
\draw [color={rgb, 255:red, 0; green, 0; blue, 255 }  ,draw opacity=1 ]   (305.09,46.7) -- (257.34,104.88) ;
\draw [color={rgb, 255:red, 0; green, 0; blue, 255 }  ,draw opacity=1 ]   (257.34,104.88) -- (305.6,164.26) ;
\draw  [color={rgb, 255:red, 255; green, 0; blue, 0 }  ,draw opacity=1 ] (231.93,77.97) -- (230.23,71.88) -- (235.9,74.7) ;
\draw  [color={rgb, 255:red, 255; green, 0; blue, 0 }  ,draw opacity=1 ] (232.72,131.98) -- (238.54,129.48) -- (236.49,135.48) ;
\draw  [color={rgb, 255:red, 0; green, 0; blue, 255 }  ,draw opacity=1 ] (290.15,72.47) -- (284.03,73.06) -- (284.85,66.96)(287.27,75.24) -- (281.15,75.83) -- (281.97,69.74) ;
\draw  [color={rgb, 255:red, 0; green, 0; blue, 255 }  ,draw opacity=1 ] (282.83,130.3) -- (282.56,136.44) -- (276.64,134.77)(285.18,133.54) -- (284.9,139.68) -- (278.98,138.01) ;
\draw  [fill={rgb, 255:red, 0; green, 0; blue, 0 }  ,fill opacity=1 ] (254.85,104.88) .. controls (254.85,103.53) and (255.97,102.43) .. (257.34,102.43) .. controls (258.72,102.43) and (259.84,103.53) .. (259.84,104.88) .. controls (259.84,106.23) and (258.72,107.33) .. (257.34,107.33) .. controls (255.97,107.33) and (254.85,106.23) .. (254.85,104.88) -- cycle ;
\draw    (251.38,97.77) .. controls (244.58,100.17) and (244.98,110.17) .. (251.38,112.17) ;
\draw    (262.8,98.51) .. controls (269.8,101.01) and (271.8,107.01) .. (262.8,111.51) ;

\draw (258.16,120.2) node    {$z$};
\draw (239.55,62.92) node  [color={rgb, 255:red, 255; green, 0; blue, 0 }  ,opacity=1 ]  {$\alpha '$};
\draw (280.21,62.92) node  [color={rgb, 255:red, 0; green, 0; blue, 255 }  ,opacity=1 ]  {$\beta '$};
\draw (222.02,105.27) node  [font=\scriptsize]  {$2\pi R+\pi $};
\draw (293.02,105.27) node  [font=\scriptsize]  {$2\pi s+\pi $};

\end{tikzpicture} 
    \caption{Saddle connections $\alpha'$ and $\beta'$ in $X'$} \label{fig1005}
\end{figure}
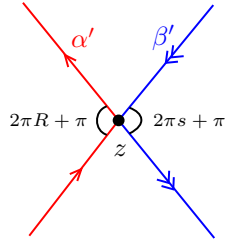

Then $\Ind \alpha' =R$, since the angle at $z$ on the left side of $\alpha$ is unchanged. By shrinking $\alpha'$, we obtain a two-level multi-scale differential $\overline{X}\in \partial \overline{\cC_r\oplus s}$. The top level component $X_0$ has two zeroes of orders $R-1$ and $a+1-R$ at the nodes $s_1,s_2$, respectively. The index of $\beta'$ is equal to $s$. So the rotation number of $Y_0$ divides $\gcd(d,R-1,s)$. If $Y_0$ is contained in a hyperelliptic component, then two zeroes are interchanged by the involution $\sigma$. Therefore, $R-1=a+1-R$, a contradiction. So $Y_0$ is contained in a non-hyperelliptic component. Therefore by \Cref{nonminimal1}, the connected component containing $Y_0$ is adjacent to $\cC_{r'}$, where $r'|\gcd(d,R-1,s)$. Thus $\cC_r \oplus s = \cC_{r'}\oplus s'$ for some $s'|\gcd(R,a+2)$. Note that $r'\leq \gcd(d,s)\leq \gcd(a,a+2)\leq 2$. So we always have $\cC=\cC_r\oplus s$ for some $r=1,2$ and $s|a+2$. However, if $r=2$, then $R$ were also even throughout the discussion. So $r'|\gcd(2,R-1)=1$ and finally we always have $\cC=\cC_1\oplus s$ for some $s|a+2$. See \Cref{fig1006}.

\begin{figure}
    \centering
    \tikzset{every picture/.style={line width=0.75pt}} 

\begin{tikzpicture}[x=0.75pt,y=0.75pt,yscale=-1,xscale=1]

\draw    (52.64,96.23) -- (52.64,111.23) ;
\draw    (36.83,57.17) -- (53.58,76.17) ;
\draw    (66.83,56.67) -- (53.58,76.17) ;
\draw    (67.64,86.13) -- (124.64,86.13) ;
\draw [shift={(126.64,86.13)}, rotate = 180] [color={rgb, 255:red, 0; green, 0; blue, 0 }  ][line width=0.75]    (10.93,-3.29) .. controls (6.95,-1.4) and (3.31,-0.3) .. (0,0) .. controls (3.31,0.3) and (6.95,1.4) .. (10.93,3.29)   ;
\draw    (130.33,47.47) -- (147.08,66.47) ;
\draw    (160.33,46.97) -- (147.08,66.47) ;
\draw    (146.68,86.15) -- (146.63,105.76) ;
\draw    (146.63,130.26) -- (146.63,113.26) ;
\draw    (171.34,86.53) -- (228.34,86.53) ;
\draw [shift={(230.34,86.53)}, rotate = 180] [color={rgb, 255:red, 0; green, 0; blue, 0 }  ][line width=0.75]    (10.93,-3.29) .. controls (6.95,-1.4) and (3.31,-0.3) .. (0,0) .. controls (3.31,0.3) and (6.95,1.4) .. (10.93,3.29)   ;
\draw    (244.03,48.07) -- (260.78,67.07) ;
\draw    (274.03,47.57) -- (260.78,67.07) ;
\draw    (260.33,103.86) -- (260.33,86.86) ;
\draw   (363.73,106.16) .. controls (363.73,103.95) and (365.52,102.16) .. (367.73,102.16) .. controls (369.94,102.16) and (371.73,103.95) .. (371.73,106.16) .. controls (371.73,108.37) and (369.94,110.16) .. (367.73,110.16) .. controls (365.52,110.16) and (363.73,108.37) .. (363.73,106.16) -- cycle ;
\draw    (367.73,127.16) -- (367.73,110.16) ;
\draw    (367.73,74.66) .. controls (349.26,87.39) and (350.31,88.77) .. (367.78,102.55) ;
\draw    (367.73,74.66) .. controls (382.76,85.89) and (386.81,87.77) .. (367.78,102.55) ;
\draw    (350.73,47.16) -- (367.48,66.16) ;
\draw    (380.73,46.66) -- (367.48,66.16) ;
\draw    (279.64,85.93) -- (336.64,85.93) ;
\draw [shift={(338.64,85.93)}, rotate = 180] [color={rgb, 255:red, 0; green, 0; blue, 0 }  ][line width=0.75]    (10.93,-3.29) .. controls (6.95,-1.4) and (3.31,-0.3) .. (0,0) .. controls (3.31,0.3) and (6.95,1.4) .. (10.93,3.29)   ;
\draw    (421.94,86.53) -- (478.94,86.53) ;
\draw [shift={(480.94,86.53)}, rotate = 180] [color={rgb, 255:red, 0; green, 0; blue, 0 }  ][line width=0.75]    (10.93,-3.29) .. controls (6.95,-1.4) and (3.31,-0.3) .. (0,0) .. controls (3.31,0.3) and (6.95,1.4) .. (10.93,3.29)   ;
\draw    (143.63,107.76) .. controls (124.01,95.55) and (126.46,125.47) .. (144.08,113.18) ;
\draw   (142.63,109.76) .. controls (142.63,107.55) and (144.42,105.76) .. (146.63,105.76) .. controls (148.84,105.76) and (150.63,107.55) .. (150.63,109.76) .. controls (150.63,111.97) and (148.84,113.76) .. (146.63,113.76) .. controls (144.42,113.76) and (142.63,111.97) .. (142.63,109.76) -- cycle ;
\draw    (251.13,74.76) .. controls (231.51,62.55) and (233.96,92.47) .. (251.58,80.18) ;
\draw   (363.48,70.16) .. controls (363.48,67.95) and (365.27,66.16) .. (367.48,66.16) .. controls (369.69,66.16) and (371.48,67.95) .. (371.48,70.16) .. controls (371.48,72.37) and (369.69,74.16) .. (367.48,74.16) .. controls (365.27,74.16) and (363.48,72.37) .. (363.48,70.16) -- cycle ;
\draw    (364.48,67.16) .. controls (344.86,54.95) and (347.31,84.87) .. (364.93,72.58) ;
\draw   (503.73,117.16) .. controls (503.73,114.95) and (505.52,113.16) .. (507.73,113.16) .. controls (509.94,113.16) and (511.73,114.95) .. (511.73,117.16) .. controls (511.73,119.37) and (509.94,121.16) .. (507.73,121.16) .. controls (505.52,121.16) and (503.73,119.37) .. (503.73,117.16) -- cycle ;
\draw    (507.73,138.16) -- (507.73,121.16) ;
\draw    (507.73,85.66) .. controls (489.26,98.39) and (490.31,99.77) .. (507.78,113.55) ;
\draw    (507.73,85.66) .. controls (522.76,96.89) and (526.81,98.77) .. (507.78,113.55) ;
\draw    (490.73,47.16) -- (507.48,66.16) ;
\draw    (520.73,46.66) -- (507.48,66.16) ;
\draw    (553.14,85.93) -- (610.14,85.93) ;
\draw [shift={(612.14,85.93)}, rotate = 180] [color={rgb, 255:red, 0; green, 0; blue, 0 }  ][line width=0.75]    (10.93,-3.29) .. controls (6.95,-1.4) and (3.31,-0.3) .. (0,0) .. controls (3.31,0.3) and (6.95,1.4) .. (10.93,3.29)   ;
\draw   (634.73,143.14) .. controls (634.73,140.93) and (636.52,139.14) .. (638.73,139.14) .. controls (640.94,139.14) and (642.73,140.93) .. (642.73,143.14) .. controls (642.73,145.35) and (640.94,147.14) .. (638.73,147.14) .. controls (636.52,147.14) and (634.73,145.35) .. (634.73,143.14) -- cycle ;
\draw    (638.73,164.14) -- (638.73,147.14) ;
\draw    (638.73,111.64) .. controls (620.26,124.37) and (621.31,125.75) .. (638.78,139.53) ;
\draw    (638.73,111.64) .. controls (653.76,122.87) and (657.81,124.75) .. (638.78,139.53) ;
\draw    (621.73,47.14) -- (638.48,66.14) ;
\draw    (651.73,46.64) -- (638.48,66.14) ;
\draw   (634.73,107.64) .. controls (634.73,105.43) and (636.52,103.64) .. (638.73,103.64) .. controls (640.94,103.64) and (642.73,105.43) .. (642.73,107.64) .. controls (642.73,109.85) and (640.94,111.64) .. (638.73,111.64) .. controls (636.52,111.64) and (634.73,109.85) .. (634.73,107.64) -- cycle ;
\draw    (638.73,103.64) -- (638.73,86.64) ;

\draw    (53.16, 86.67) circle [x radius= 9.6, y radius= 9.6]   ;
\draw (53.16,86.67) node   [align=left] {$\displaystyle 1$};
\draw (52.32,66.67) node [anchor=south] [inner sep=0.75pt]   [align=left] {...};
\draw (54.78,118.17) node    {$a$};
\draw (96.13,72.5) node   [align=left] {bubble};
\draw (47.1,12.38) node [anchor=north west][inner sep=0.75pt]    {$\mathcal{C}_{r}$};
\draw (615.5,9.98) node [anchor=north west][inner sep=0.75pt]    {$\overline{\mathcal{C}_{1} \oplus s'}$};
\draw (145.82,56.97) node [anchor=south] [inner sep=0.75pt]   [align=left] {...};
\draw (147.29,136.26) node    {$a+2$};
\draw (198.83,73.41) node   [align=left] {plumb};
\draw (259.52,57.57) node [anchor=south] [inner sep=0.75pt]   [align=left] {...};
\draw (260.99,109.86) node    {$a+2$};
\draw (366.72,50.66) node   [align=left] {...};
\draw (368.39,133.16) node    {$a+2$};
\draw (298.63,74.81) node   [align=left] {shrink};
\draw (342.9,85.3) node [anchor=north west][inner sep=0.75pt]  [font=\tiny]  {$R$};
\draw (392.16,88.8) node  [font=\tiny]  {$ \begin{array}{l}
a+2\\
-R
\end{array}$};
\draw (449.43,73.41) node   [align=left] {plumb};
\draw    (147.16, 76.67) circle [x radius= 9.6, y radius= 9.6]   ;
\draw (147.16,76.67) node   [align=left] {$\displaystyle 1$};
\draw    (260.66, 77.17) circle [x radius= 9.6, y radius= 9.6]   ;
\draw (260.66,77.17) node   [align=left] {$\displaystyle 1$};
\draw (506.72,50.66) node   [align=left] {...};
\draw (508.39,144.16) node    {$a+2$};
\draw (482.9,96.3) node [anchor=north west][inner sep=0.75pt]  [font=\tiny]  {$R$};
\draw (532.16,99.8) node  [font=\tiny]  {$ \begin{array}{l}
a+2\\
-R
\end{array}$};
\draw    (507.66, 76.17) circle [x radius= 9.6, y radius= 9.6]   ;
\draw (507.66,76.17) node   [align=left] {$\displaystyle 1$};
\draw (583.63,75.31) node   [align=left] {shrink};
\draw (637.72,50.64) node   [align=left] {...};
\draw (639.39,170.14) node    {$a+2$};
\draw (613.9,122.28) node [anchor=north west][inner sep=0.75pt]  [font=\tiny]  {$R$};
\draw (663.16,125.78) node  [font=\tiny]  {$ \begin{array}{l}
a+2\\
-R
\end{array}$};
\draw    (638.66, 76.15) circle [x radius= 9.6, y radius= 9.6]   ;
\draw (638.66,76.15) node   [align=left] {$\displaystyle 1$};
\draw (126.1,9.5) node [anchor=north west][inner sep=0.75pt]    {$\overline{\mathcal{C}_{r} \oplus s}$};
\draw (315.73,67) node [anchor=north west][inner sep=0.75pt]    {$\alpha '$};
\draw (201.13,98.05) node   [align=left] {levels 0 and -1};
\draw (448.13,104.05) node   [align=left] {horizontal\\node};
\draw (586.63,96.5) node   [align=left] {simple s.c.};

\end{tikzpicture} 
    \caption{Navigating the boundary of $\cC_r\oplus s$} \label{fig1006}
\end{figure}
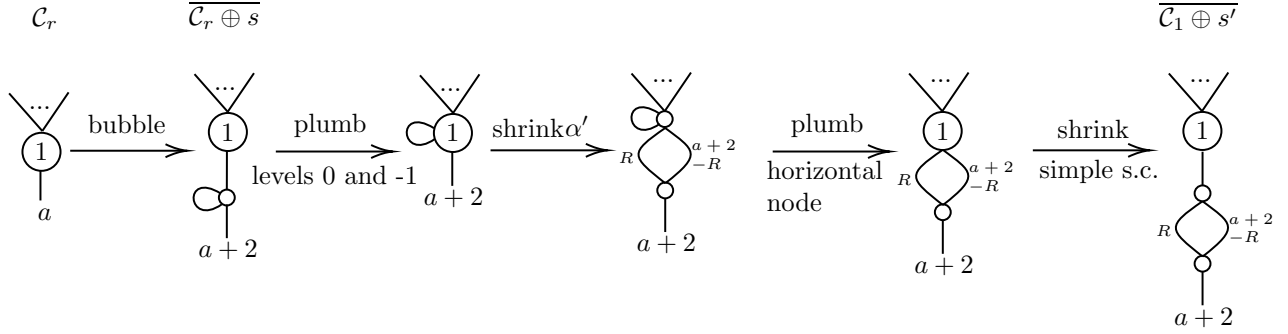

For any $n\leq R\leq a-n$, there exists a multi-scale differential $X(Id,{\bf C},[(0,1)])\in \partial\overline{\cC}_1$ such that $Q_1=\sum_{i=1}^n C_i=R$. Therefore, we can find a flat surface $X\in \cC_1$ with a multiplicity one saddle connection of index $R$. Similarly for $\cC_2$, for any even $R$ such that $n\leq R\leq a-n$, we can find a flat surface $X\in \cC_2$ with a multiplicity one saddle connection of index $R$. 

Assume that $\cC=\cC_1\oplus s$ and $a$ is even. If $s=\frac{a+2}{2}$, then we choose $R=\frac{a}{2}-1$ and by the same argument in the previous paragraphs, $\cC=\cC_r\oplus s'$ for some $r=1,2$ and  $s'|\gcd(R,a+2)|4$. We claim that $s'<s$ for any case. If $s'=s=\frac{a+2}{2}$, then $a=0,2$ or $6$. First two cases are impossible ($a=0$) and has only one trivial case $\cH(2,-2)$ which is already excluded ($a=2)$.  So $a=6$, but this is also impossible because $\gcd(R,a+2)=\gcd(2,8)=2$ and therefore $a=2s'-2\leq 2$. So we can assume that $\cC=\cC_r\oplus s$ for some $r=1,2$ and $s<\frac{a+2}{2}$. 

If $\frac{a}{2}$ is even, then we can choose $R=\frac{a}{2}$. Then $\gcd(d,R-1,s)=1$ and $\gcd(R,a+2)=2$. Therefore $\cC= \cC_1\oplus s'$ for some $s'=1,2$. Now for $\cC_1\oplus 2$, by taking $R=\frac{a}{2}-1$, we have $\gcd(R,a+2)=1$ since $R$ is odd. Thus $\cC_1\oplus 2 = \cC_r\oplus 1$ for some $r=1,2$. 

If $\frac{a}{2}$ is odd, then we first consider $\cC_2\oplus s$. We choose even $R=\frac{a}{2}-1$, then $\gcd(d,R-1,s)=1$ and $\gcd(R,a+2)=4$. Thus $\cC_2\oplus s= \cC_1\oplus s'$ for some $s'=1,2,4$. For
$\cC_1\oplus s$, we choose $R=\frac{a}{2}$, then $\gcd(R,a+2)=1$. Thus $\cC_1\oplus s=\cC_r\oplus 1$ for some $r=1,2$. 

So in any case, we can reduce to $\cC_r\oplus 1$ for $r=1,2$. If all $b_i$ are even, then both $\cC_1$ and $\cC_2$ exist. Furthermore two components $\cC_1\oplus 1$ and $\cC_2\oplus 1$ are distinguished by spin parity which will be recalled in \Cref{subsec:spin}. Moreover by \Cref{bubblereduction}, we have $\cC_2\oplus 1=\cC_1\oplus 2$. If any $b_i$ is odd, then $\cC_2$ does not exists because $d$ is odd. Therefore we always have $\cC=\cC_1\oplus 1$. 
\end{proof}

\subsection{Spin structure of flat surfaces}\label{subsec:spin}

Recall that in the classification of the connected components of usual strata in \cite{boissymero},\cite{kozo1} with $g\geq 2$, the non-hyperelliptic connected components are distinguished by {\em spin parity}, which can be expressed in terms of the flat geometry as follows by the work of Johnson \cite{johnsonspin}.

Let $\alpha_1,\beta_1,\dots,\alpha_g,\beta_g$ be simple closed curves on a flat surface $X$ that form a symplectic basis of $H_1 (X,\ZZ)$. Then the {\em spin parity} of $X$ is:

\begin{equation}\label{roteq} 
\sum_{i=1}^g (\Ind \alpha_i +1)(\Ind \beta_i +1)\in \ZZ/2\ZZ.
\end{equation}

\begin{proposition} \label{rotationspin}
Suppose for $X\in\cR_1$ that all poles are of even order, i.e. all $b_i$ are even. then the spin parity of $X$ is equal to 1+(the rotation number of $X$) in $\ZZ/2\ZZ$.
\end{proposition}

\begin{proof} 
Let $\alpha,\beta$ be simple closed curves in $X$ that form a symplectic basis of $H_1(X,\ZZ)$. Then the rotation number $r$ of $X$ is equal to $\gcd(d,\Ind a, \Ind b)$ where $d=\gcd(a,b_1,\dots,b_n)$ is even by assumption. By \eqref{roteq}, the spin parity is equal to 1+(the rotation number of $X$) in $\ZZ/2\ZZ$.
\end{proof}

\begin{lemma} \label{bubblespin}
Let $\cC$ be a non-hyperelliptic component of $\cR_g(\mu)$. Then the spin parity of $\cC \oplus s$ is equal to $s+1+$(the spin parity of $\cC$) in $\ZZ/2\ZZ$.
\end{lemma}

\begin{proof}
Choose $X\in \cC$ and closed curves $\alpha_1,\beta_1,\dots,\alpha_g,\beta_g$ on $X$ forming a symplectic basis of $H_1 (X,\ZZ)$. Let $X'\in \cC\oplus s$ be the flat surface obtained by bubbling a handle from $X$. Then $X'$ contains two saddle connections $\alpha_{g+1},\beta_{g+1}$ that are the boundary and the cross curve, respectively, of the flat cylinder used for bubbling a handle. We have $\Ind \alpha_{g+1}=0$, $\Ind \beta_{g+1}=s$. The curves $\alpha_1,\beta_1,\dots,\alpha_{g+1},\beta_{g+1}$ now form a symplectic basis of $X'$, so the spin parity of $\cC \oplus s$ is equal to $s+1+$(the spin parity of $\cC$).
\end{proof}

\subsection{Proof of \Cref{main1} for \MIN strata: special cases} \label{subsec:specialbubble}

It remains to deal with the special cases, to complete the proof of \Cref{bubblereduction}. In particular, we need to classify the connected components written as $\cC\oplus \frac{a+2}{2}$, where $\cC$ is a special non-hyperelliptic connected component introduced in \Cref{subsec:special}.

\begin{proposition}
Let $\cC$ be a non-hyperelliptic component of $\cR_1(12,-3^4)$ with rotation number $3$. Then $\cC\oplus 7 =\cC_1\oplus 1$. 
\end{proposition}

\begin{proof}
We will find a multi-scale differential contained in $\overline{\cC\oplus 7}\cap \overline{\cC_1\oplus 1}$. Recall from \Cref{specialmerge1} that $B(\cC)\subset\cR_1(6^2,-3^4)$ contains a two-level multi-scale differential $\overline{Y}$ with the top level component in $\cR_1(9,-3^3)$ with rotation number $3$, and the bottom level component in $\cR_0(6,6,-3,-11)$, containing $p_4$. Therefore, $\overline{\cC\oplus 7}$ contains a two-level multi-scale differential $\overline{Z}$ with the same top level component and the bottom level component contained in the connected stratum $\cR_1(14,-3,-11)$. It remains to show that $\overline{\cC_1\oplus 1}$ also contains $\overline{Z}$. 

Let $\overline{X} = X(3,Id,(1,1,1,1),[(0,0)])\in \overline{\cC_1}\subset \overline{\cR}_1(12,-3^4)$. The bottom level component is in $\cR_0(12,-3;-4,-7)$, containing $p_4$. It has two saddle connections $\beta_1,\beta_2$. By breaking up the zero into two zeroes of orders $10$ and $2$, with a proper choice of prong-matching, the saddle connections $\beta_1,\beta_2$ remain parallel. By shrinking them, we obtain a two-level multi-scale differential $\overline{W}$ with top level component $W_0\in \cH_0(9,-4,-7)$. By keeping track of the prongs of $(X_{-1},v^-_1,w^-_1)$, we obtain $(W_0,v^-_1,w^-_1)$. Thus the levels 0 and -1 form a multi-scale differential $X(3,Id,(1,1,1),[(0,0)])\in \cR_1(9,-3^3)$ with rotation number $3$. Therefore we can conclude that $\overline{Z}\in \overline{\cC_1\oplus 1}$. See \Cref{fig1007}.
\end{proof}

\begin{figure}
    \centering
    \input{diagram1007} 
    \caption{Finding $\overline{Z}\in \overline{\cC\oplus 7}\cap \overline{\cC_1\oplus1}$} \label{fig1007}
\end{figure}

\begin{proposition}
Let $\cC_r$ be the non-hyperelliptic component of $\cR_1(2n+2,-2^{n-1},-4)$ for odd $n$, with rotation number $r=1,2$. Then $\cC_2\oplus (n+2)=\cC_1\oplus 1$.
\end{proposition}

\begin{proof}
Let $\overline{X}=X(n-1,(n-1,n),{\bf 1},[(0,-1)])\in \overline{\cC}$. The bottom level component $X_{-1}$ is in $\cR_0(2n+2,-2;-n,-n-2)$, containing two saddle connections $\beta_1,\beta_2$ bounding the polar domain of $p_{n-1}$. By breaking up the zero into two zeroes of order $n+1$, with a proper choice of prong-matching, $\beta_1,\beta_2$ remain to be parallel. By shrinking them, we obtain a two-level multi-scale differential $\overline{W}$ with top level component $W_0\in \cH_0(2n,-n,-n-2)$. By keeping track of the prongs of $(X_{-1},v^-_1,w^-_1)$, we obtain $(W_0,v^-_{\frac{n+3}{2}},w^-_{\frac{n+3}{2}})$. Thus the levels 0 and -1 form a multi-scale differential $X(Id,{\bf 1},[(0,-n-4)])\in \overline{\cR}_1(2n,-2^{n-2},-4)$ with rotation number $1$. Thus $\overline{\cC_2\oplus (n+2)}$ contains a two-level multi-scale differential $\overline{Z}$ with the top level component $Z_0\in \cR_1(2n,-2^{n-2},-4)$ with rotation number $1$, and the bottom level component $Z_{-1}\in \cR_1(2n+4,-2n-2,-2)$ with rotation number $1$. It remains to prove that $\overline{\cC_1\oplus 1}$ also contains $\overline{Z}$. 

Consider $\overline{X'}=X(n-1,(n-1,n),(1,\dots,1,2),[(0,0)])\in \overline{\cC}_1$. By breaking up the zero of the bottom level component $X'_{-1}$ into two zeroes of orders $2$ and $2n$, and shrinking the saddle connections $\beta'_1$ and $\beta'_2$, we obtain $\overline{W'}$ with the top level component $W'_0\in \cH_0(2n,-n-1,-n-1)$ and the bottom level component $W'_{-1}\in \cR_0(2,2n,-2,-2n-2)$. By keeping track of the prongs of $(X'_{-1},v^-_1,w^-_1)$, we obtain $(W'_0,v^-_1,w^-_2)$. So the levels 0 and -1 form a multi-scale differential $X(Id,(1,\dots,1,2),[(0,-1)])\in \overline{\cR}_1(2n,-2^{n-2},-4)$ with rotation number $1$. Thus $\overline{\cC_1\oplus 1}$ contains a two-level multi-scale differential $\overline{Z'}$ with the top level component $Z'_0\in \cR_1(2n,-2^{n-2},-4)$ with rotation number $1$, and the bottom level component $Z'_{-1}\in \cR_1(2n+4,-2n-2,-2)$ with rotation number $1$. So $\overline{Z}$ and $\overline{Z'}$ are contained in the boundary of the same connected component, proving that $\cC_2\oplus (n+2)=\cC_1\oplus 1$.
\end{proof}

Now, we need to prove the following result for the hyperelliptic stratum $\cR_1(2n,-2^n)$, similar to \Cref{bubblereduction}. 

\begin{proposition} \label{bubblereductionhyper}
Let $\cC_1$ be a hyperelliptic component of $\cR_1(2n,-2^n)$, $n\geq 2$, with ramification profile $\cP_1$. For any other (hyperelliptic) component $\cC_2$ of $\cR_1(2n,-2^n)$ and any $1\leq s\leq n$, the component $\cC_2\oplus s$ is equal to $\cC_1\oplus 1$ or $\cC_1\oplus 2$
\end{proposition}

Note that $\cC_2 \oplus (n+1)$ is hyperelliptic by \Cref{hyper1}. Also, two components $\cC_1\oplus 1$ and $\cC_1\oplus 2$ are distinguished by spin parity. 

\begin{proof}
We will construct an element of $\overline{\cC_2\oplus s}$ in the following way. We will consider a two-level multi-scale differential $\overline{W}\in \overline{\cR_1}(2n+2,-(2n+2))$ defined as follows. If $s>2$, then $\overline{W}\in \cR_1(Id, s,[(0,0)])$. If $s=2$, then $\overline{W}\in \cR_1(Id, 4,[(0,2)])$. If $s=1$, then $\overline{W}\in \cR_1(Id, 3,[(0,1)])\in \overline{\cR_1}(2n+2,-(2n+2))$. In any case, the rotation number of $\overline{W}$ is equal to $s$. Also, the top level component $W_0$ has two zeroes of orders at least two. By identifying the pole of $\overline{W}$ and the unique zero of a flat surface in $\cC_2$, we obtain a three-level multi-scale differential contained in $\overline{\cC_2 \oplus s}$. 
 
First assume that $\cC_2$ has a fixed marked pole, say $p_n$. By \Cref{double}, we can find a flat surface $X$ in $\cC_2$ with a pair of parallel saddle connections $\beta_1,\beta_2$ bounding the polar domain of $p_n$. By breaking up the zero of $X$ into two zeroes with proper choice of prong-matching, the saddle connection $\beta_1,\beta_2$ remain parallel. This is because the total order around a zero is at least $6\pi$ and an angle bounded by $\beta_1,\beta_2$ is equal to $2\pi$. By shrinking them, we obtain a two-level multi-scale differential $\overline{Z}$ with the top level component $Z_0\in \cR_1(2n-2,-2^{n-1})$ and the bottom level component $Z_{-1}\in \cR_0(a_1,a_2,-2,-2n)$. Therefore, $\overline{\cC_2\oplus s}$ contains a two-level multi-scale differential $\overline{Y}$ with the top level component $Y_0\in \cR_1(2n-2,-2^{n-1})$ and the bottom level component $Y_{-1}\in \cR_1(2n+2,-2,-2n)$ with rotation number $r$ dividing 2. Since $Y_0$ is hyperelliptic, $Y_{-1}$ must be non-hyperelliptic. So $Y_{-1}$ can degenerate to $X(Id,(1,1),[(0,r)])\in \overline{\cR}_1(2n+2,-2,-2n)$ where $r|2$ is its rotation number. By merging the level transition between level 0 and level -1, the new top level component is contained in the connected stratum $\cR_1(1,n-1,-2^n)$. Since this stratum is adjacent to $\cC_1$, we can conclude that $\cC_2\oplus s = \cC_1\oplus s'$ for some $s'=1,2$. 

Now assume that $\cC_2$ does not have a fixed marked pole, but a pair $p_{n-1},p_n$ of poles interchanged by hyperelliptic involution. If $n=2$, then we are only taking care of $\cC_2\oplus 1$. In this case, by the previous paragraph, $\cC\oplus 1=\cC_2\oplus 1 $ for other component $\cC$ with a fixed marked pole. So we can assume that $n>2$ and $\overline{X}=X(n-2,\tau,{\bf 1},[(0,v)]\in \overline{\cC}_2$. The bottom level component $X_{-1}$ has three parallel saddle connection. By breaking up the zero of $X_{-1}$ into two zeroes with proper choice of prong-matching, these three saddle connections remain parallel. This is because the total order around a zero is at least $6\pi$ and the sum of two angles bounded by three saddle connections is also equal to $4\pi$. After plumbing all level transitions and shrinking these three saddle connections, we obtain a two-level multi-scale differential $\overline{Z}$ with the top level component $Z_0\in \cR_1(2n-4,-2^{n-2})$ and the bottom level component $Z_{-1}\in \cR_0(a_1,a_2,-2^2,-(2n-2))$. Therefore, $\overline{\cC_2\oplus s}$ contains a two-level multi-scale differential $\overline{Y}$ with the top level component $Y_0\in \cR_1(2n-4,-2^{n-2})$ and the bottom level component $Y_{-1}\in \cR_1(2n+2,-2^2,-(2n-2))$ with rotation number $r$ dividing 2. Since $Y_0$ is hyperelliptic, $Y_{-1}$ must be non-hyperelliptic. So $Y_{-1}$ can degenerate to $X(Id,(1,1,2),[(0,r)])\in \overline{\cR}_1(2n+2,-2,-2n)$ where $r|2$ is its rotation number. By merging the level transition between level 0 and level -1, the new top level component is contained in the connected stratum $\cR_1(3,n-3,-2^n)$. Since this stratum is adjacent to $\cC_1$, we can conclude that $\cC_2\oplus s = \cC_1\oplus s'$ for some $s'=1,2$. 
\end{proof}

In conclusion, combining \Cref{bubblereduction}, \Cref{bubblereductionhyper}, \Cref{rotationspin} and \Cref{bubblespin}, we can classify all non-hyperelliptic components of \MIN strata of genus $g>1$, completing the proof of \Cref{main1} for \MIN strata.

\subsection{Proof of \Cref{main1} for multiple-zero strata} \label{sec:hgn}

Now we will complete the proof of \Cref{main1} by classifying the connected components of a multiple-zero stratum $\cR_g(\mu)=\cR_g(a_1,\dots,a_m,-b_1,\dots,-b_n)$ of genus $g>1$. We denote $a\coloneqq a_1+\dots+a_m$ and $\mu'\coloneqq(a,-b_1,\dots,-b_n)$ as before. In \Cref{nonhypermerge}, we have shown that every non-hyperelliptic component $\cR_g(\mu)$ is adjacent to a non-hyperelliptic component of a \MIN stratum. In other words, the breaking up the zero map $$B:\{\text{non-hyperelliptic components of }\cR_g(\mu')\} \to \{\text{non-hyperelliptic connected components of }\cR_g(\mu)\}$$ is surjective. This gives an upper bound for the number of non-hyperelliptic components of any multiple-zero stratum. 

In order to completely classify the non-hyperelliptic components, we need to determine when $B(\cC_1)=B(\cC_2)$ for two non-hyperelliptic components $\cC_1,\cC_2$ of $\cR_g(\mu')$, similarly to how it is done in \cite[Prop.~7.3]{boissymero} for the usual meromorphic strata. The proof there only uses the fact that any connected component can be obtained by bubbling a handle, which is also true for non-hyperelliptic components of residueless strata of genus $g>1$ by \Cref{unbubble}. Thus we have

\begin{lemma} \label{breakinglemmahigh}
Let $\cR_g (\mu)$ be a \MIN stratum of genus $g>1$ of even type. Assume that $a=a_1+a_2$ for an odd $a_1$. Consider the function $$B:\{\text{connected components of }\cR_g(\mu)\} \to \{\text{connected components of }\cR_g(a_1, a_2, -b_1, \dots, -b_n)\}$$ obtained by breaking up the zero. Let $\cC_{odd}$, $\cC_{even}$ be the two non-hyperelliptic components of $\cR_g (\mu)$ distinguished by spin parity. Then $B(\cC_{odd})=B(\cC_{even})$.
\end{lemma}

We can finally prove \Cref{main1}.

\begin{proof}[Proof of \Cref{main1}]
We use induction on $m>0$. If $m=1$, then $\cR_g(\mu)$ is a \MIN stratum, already treated in \Cref{sec:hgm}. For $m>1$, denote the corresponding \MIN stratum by $\cR_g(\mu')$, as in \Cref{sec:g1n}. We can consider the {\em breaking up the zero} function $$B:\{\text{non-hyperelliptic components of }\cR_g(\mu')\} \to \{\text{non-hyperelliptic components of }\cR_g(\mu)\}.$$ By \Cref{nonhypermerge}, $B$ is surjective. If $\cR_g(\mu)$ is of even type, then so is $\cR_g(\mu')$, and thus $\cR_g(\mu')$ has exactly two non-hyperelliptic components. So $\cR_g(\mu)$ has at most two non-hyperelliptic components. Note that in this case, breaking up the zero preserves the spin parity computed by \eqref{roteq} since it is a local surgery. Therefore, we can construct two non-hyperelliptic flat surfaces in $\cR_g(\mu)$ with odd and even spin parities. Since the spin parity is invariant within a connected component, we can conclude that $\cR_g(\mu)$ has exactly two non-hyperelliptic components, distinguished by spin parity. 

If $\cR_g(\mu)$ is not of even type, and $\cR_g(\mu')$ is also not of even type, then $\cR_g(\mu')$ has a unique non-hyperelliptic component. Then obviously $\cR_g(\mu)$ has a unique non-hyperelliptic component. If $\cR_g(\mu)$ is not of even type while $\cR_g(\mu')$ is, then all $b_i$ are even and $\cR_g(\mu')$ has two non-hyperelliptic components $\cC_1$ and $\cC_2$. Since $\cR_g(\mu)$ is not of even type, by relabeling the zeroes, we may assume that $a_1$ is odd. Let $\mu''=(a_1,a-a_1,-b_1,\dots,-b_n)$. We can consider two functions $$B':\{\text{non-hyperelliptic components of }\cR_g(\mu')\} \to \{\text{non-hyperelliptic components of }\cR_g(\mu'')\}$$ $$B'':\{\text{non-hyperelliptic components of }\cR_g(\mu'')\} \to \{\text{non-hyperelliptic components of }\cR_g(\mu)\}$$ given by breaking up the zero. Then $B=B''\circ B'$. By (1) of \Cref{breakinglemmahigh}, $B'(\cC_1)=B'(\cC_2)$. Therefore, $B(\cC_1)=B(\cC_2)$ and $\cR_g(\mu)$ has a unique non-hyperelliptic component.
\end{proof}

\bibliography{bib}
\bibliographystyle{siam}

\end{document}